\documentclass[a4paper,reqno]{amsart}

\pdfoutput=1

\usepackage{lmodern}
\usepackage[final]{microtype}
\usepackage{bbm,stmaryrd}
\usepackage{amsmath,amssymb,amsthm,url,scalerel}
\usepackage[utf8]{inputenc}
\usepackage[T1]{fontenc}
\usepackage{libertine}
\usepackage[libertine,cmintegrals,cmbraces]{newtxmath}
\usepackage{verbatim}
\usepackage[mathscr]{euscript}
\usepackage[shortlabels]{enumitem}
\usepackage{mathtools}
\usepackage[dvipsnames,svgnames,x11names,hyperref]{xcolor}
\definecolor{darkred}{RGB}{139,0,0}
\definecolor{darkblue}{RGB}{0,0,139}
\definecolor{darkgreen}{RGB}{0,100,0}
\usepackage{hyperref}
\hypersetup{
	linktoc=all,
	colorlinks   = true, 
	urlcolor     = darkred, 
	linkcolor    = darkred, 
	citecolor   = darkgreen}
\usepackage{tikz}
\usetikzlibrary{matrix,calc,positioning,arrows,decorations.pathreplacing,decorations.pathmorphing,patterns,arrows}
\usepackage{tikz-cd}
\usepackage[capitalise,noabbrev]{cleveref}
\usepackage{multicol}
\usepackage[sort=def,symbols,nomain,numberline]{glossaries}
\newcommand{\listofsymbolsname}{List of Symbols} 
\makeglossaries

\AtBeginDocument{%
	\def\MR#1{}
}
\setenumerate[1]{label=(\roman*),} 
\setitemize[1]{label=\raisebox{0.25ex}{\tiny$\bullet$}}

\DeclareMathAlphabet{\mathpzc}{OT1}{pzc}{m}{it}
\newcommand{\catsingle}[1]{\ensuremath{\mathscr{#1}}}
\newcommand{\cat}[1]{\ensuremath{\mathsf{#1}}}
\newcommand{\icat}[1]{\ensuremath{\mathscr{#1}}}

\newcommand{\oC}{\ensuremath{\mathrm{C}}}

\newcommand{\oH}{\ensuremath{\mathrm{H}}}

\newcommand{\oO}{\ensuremath{\mathrm{O}}}

\newcommand{\bfN}{\ensuremath{\mathbf{N}}}

\newcommand{\bfR}{\ensuremath{\mathbf{R}}}
\newcommand{\bfZ}{\ensuremath{\mathbf{Z}}}
\newcommand{\bfQ}{\ensuremath{\mathbf{Q}}}
\newcommand{\bfS}{\ensuremath{\mathbf{S}}}

\newcommand{\cA}{\ensuremath{\catsingle{A}}}
\newcommand{\cB}{\ensuremath{\catsingle{B} }}
\newcommand{\cC}{\ensuremath{\catsingle{C}}}
\newcommand{\cD}{\ensuremath{\catsingle{D}}}
\newcommand{\cE}{\ensuremath{\catsingle{E}}}

\newcommand{\cM}{\ensuremath{\catsingle{M}}}

\newcommand{\cO}{\ensuremath{\catsingle{O}}}
\newcommand{\cP}{\ensuremath{\catsingle{P}}}

\newcommand{\cS}{\ensuremath{\catsingle{S}}}

\newcommand{\cU}{\ensuremath{\catsingle{U}}}

\newcommand{\cW}{\ensuremath{\catsingle{W}}}

\newcommand{\catA}{\ensuremath{\cat{A}}}
\newcommand{\catB}{\ensuremath{\cat{B}}}
\newcommand{\catC}{\ensuremath{\cat{C}}}

\newcommand{\catS}{\ensuremath{\cat{S}}}

\newcommand{\catTop}{\ensuremath{\cat{Top}}}

\newcommand{\Diff}{\ensuremath{\mathrm{Diff}}}

\newcommand{\BDiff}{\ensuremath{\mathrm{BDiff}}}

\newcommand{\BHomeo}{\ensuremath{\mathrm{BHomeo}}}
\newcommand{\Homeo}{\ensuremath{\mathrm{Homeo}}}

\newcommand{\hAut}{\ensuremath{\mathrm{hAut}}}
\newcommand{\BhAut}{\ensuremath{\mathrm{BhAut}}}

\newcommand{\Emb}{\ensuremath{\mathrm{Emb}}}

\newcommand{\holim}{\ensuremath{\mathrm{holim}}}

\newcommand{\OO}{\mathrm{O}}

\newcommand{\Hom}{\mathrm{Hom}}

\newcommand{\GL}{\mathrm{GL}}

\newcommand{\Map}{\mathrm{Map}}

\newcommand{\id}{\mathrm{id}}
\newcommand{\fr}{\mathrm{fr}}

\newcommand{\onefr}{\mathrm{1{-}fr}}

\newcommand{\im}{\mathrm{im}}

\newcommand{\Bun}{\mathrm{Bun}}

\newcommand{\hofib}{\mathrm{hofib}}

\newcommand{\inc}{\mathrm{inc}}
\newcommand{\diag}{\mathrm{diag}}
\newcommand{\pr}{\mathrm{pr}}
\newcommand{\stst}{\mathrm{st}}

\newcommand{\colim}{\mathrm{colim}}
\newcommand{\hocolim}{\mathrm{hocolim}}
\newcommand{\Aut}{\mathrm{Aut}}
\newcommand{\BAut}{\mathrm{BAut}}

\newcommand{\HCG}{\mathrm{HCG}}
\newcommand{\GC}{\mathrm{GC}}
\newcommand{\op}{\mathrm{op}}
\newcommand{\geo}{\mathrm{geo}}

\newcommand{\Fr}{\mathrm{Fr}}

\newcommand{\sur}{\mathrm{sur}}
\newcommand{\interior}{\mathrm{int}}

\newcommand{\ev}{\mathrm{ev}}

\newcommand{\Wh}{\mathrm{Wh}}

\newcommand{\tr}{\mathrm{tr}}
\newcommand{\nonunital}{\mathrm{nu}}
\newcommand{\quasiunital}{\mathrm{qu}}

\newcommand{\fib}{\mathrm{fib}}
\newcommand{\Sing}{\mathrm{Sing}}
\newcommand{\res}{\mathrm{res}}
\newcommand{\ind}{\mathrm{ind}}
\newcommand{\glue}{\mathrm{glue}}
\newcommand{\Latch}{\mathrm{Latch}}
\newcommand{\Match}{\mathrm{Match}}

\newcommand{\TOP}{\mathrm{Top}}
\newcommand{\BTOP}{\mathrm{BTop}}

\newcommand{\BO}{\mathrm{BO}}
\newcommand{\BSpin}{\mathrm{BSpin}}

\newcommand{\End}{\mathrm{End}}

\newcommand\unl{}
\DeclareRobustCommand\unl[1]{{\underline{#1}}}

\newcommand{\inj}{\mathrm{inj}}
\newcommand{\can}{\mathrm{can}}

\newcommand{\wall}{\mathrm{wall}}
\newcommand{\ch}{\mathrm{ch}}
\newcommand{\tch}{\mathrm{tch}}
\newcommand{\lab}{\mathrm{lab}}
\newcommand{\wlab}{\mathrm{wlab}}
\newcommand{\coll}{\mathrm{coll}}
\newcommand{\coh}{\mathrm{coh}}

\newcommand{\act}{\mathrm{act}}

\newcommand{\yon}{y}
\newcommand{\grtp}{\mathrm{grtp}}
\newcommand{\rep}{\mathrm{rep}}
\newcommand{\ab}{\mathrm{ab}}
\newcommand{\Mul}{\mathrm{Mul}}

\newcommand{\rev}{\mathrm{rev}}

\newcommand{\Dend}{\Omega}
\newcommand{\rDend}{\smash{\overline{\Omega}}}

\newcommand{\PSh}{\ensuremath{\mathrm{PSh}}}

\newcommand{\DiscInf}{\ensuremath{\icat{D}\mathrm{isc}}}
\newcommand{\Man}{\ensuremath{\cat{Man}}}
\newcommand{\Manc}{\ensuremath{\cat{Man}^{\mathrm{c}}}}
\newcommand{\ManInf}{\ensuremath{\icat{M}\mathrm{an}}}
\newcommand{\Fin}{\ensuremath{\cat{Fin}}}
\newcommand{\Fun}{\ensuremath{\mathrm{Fun}}}
\newcommand{\ALG}{\ensuremath{\mathrm{ALG}}}
\newcommand{\Alg}{\ensuremath{\mathrm{Alg}}}
\newcommand{\Mon}{\ensuremath{\mathrm{Mon}}}
\newcommand{\CMon}{\ensuremath{\mathrm{CMon}}}
\newcommand{\Nat}{\ensuremath{\cat{Nat}}}

\newcommand{\BordInf}{\ensuremath{\icat{B}\mathrm{ord}}}
\newcommand{\ncBord}{\ensuremath{\cat{ncBord}}}
\newcommand{\ncBordInf}{\ensuremath{\mathrm{nc}\icat{B}\mathrm{ord}}}
\newcommand{\ModInf}{\ensuremath{\icat{M}\mathrm{od}}}
\newcommand{\Gap}{\ensuremath{\cat{Gap}}}

\newcommand{\Kan}{\ensuremath{\cat{Kan}}}
\newcommand{\GapInf}{\ensuremath{\icat{G}\mathrm{ap}}}
\newcommand{\sCat}{\ensuremath{\cat{sCat}}}
\newcommand{\Cat}{\ensuremath{\mathrm{Cat}}}
\newcommand{\Bimod}{\ensuremath{\mathrm{BMod}}}
\newcommand{\BMod}{\ensuremath{\mathrm{BMod}}}
\newcommand{\Cocart}{\ensuremath{\mathrm{Cocart}}}
\newcommand{\Assoc}{\ensuremath{\icat{A}\mathrm{ssoc}}}
\newcommand{\Ass}{\ensuremath{\mathrm{Ass}}}
\newcommand{\CatInf}{\ensuremath{\icat{C}\mathrm{at}_\infty}}
\newcommand{\CatInfS}{\ensuremath{\cat{Cat}_\infty}}
\newcommand{\CatInfTwo}{\ensuremath{\icat{C}\mathrm{at}_{(\infty,2)}}}

\newcommand{\SPAN}{\ensuremath{\mathrm{SPAN}}}
\newcommand{\COSPAN}{\ensuremath{\mathrm{COSPAN}}}

\newcommand{\CSS}{\ensuremath{\icat{C}\mathrm{SS}}}

\newcommand{\Mod}{\ensuremath{\icat{M}\mathrm{od}}}
\newcommand{\grp}{\ensuremath{{\mathrm{grp}}}}
\newcommand{\OpdInf}{\ensuremath{{\icat{O}\mathrm{pd}_\infty}}}
\newcommand{\sOp}{\ensuremath{{s\icat{O}\mathrm{p}}}}

\newcommand{\ra}{\rightarrow}
\newcommand{\la}{\leftarrow}
\newcommand{\lra}{\longrightarrow}
\newcommand{\lla}{\longleftarrow}
\newcommand{\xra}[1]{\xrightarrow{#1}}

\newcommand{\xlra}[1]{\overset{#1}{\longrightarrow}}
\newcommand{\xlla}[1]{\overset{#1}{\longleftarrow}}
\newcommand{\longhookrightarrow}{\lhook\joinrel\longrightarrow}

\newcommand{\rbr}{\hspace{.1em}\rrparenthesis}
\newcommand{\lbr}{\llparenthesis\hspace{.1em}}

\newcommand\smallsquare{{\mathbin{\text{\raise0.17ex\hbox{\scalebox{.7}{$\blacksquare$}}}}}}

\newcommand{\modulo}[1]{\ (\mathrm{mod}\ #1)}

\newcommand{\circled}[1]{\raisebox{.5pt}{\textcircled{\raisebox{-.9pt} {#1}}}}

\newtheorem{bigthm}{Theorem}
\newtheorem{bigcor}[bigthm]{Corollary}

\newtheorem{thm}{Theorem}[section]

\newtheorem{lem}[thm]{Lemma}
\newtheorem{prop}[thm]{Proposition}
\newtheorem{cor}[thm]{Corollary}

\theoremstyle{definition}
\newtheorem{dfn}[thm]{Definition}

\newtheorem*{nconvention}{Convention}
\newtheorem{convention}[thm]{Convention}
\theoremstyle{remark}
\newtheorem{ex}[thm]{Example}
\newtheorem*{nex}{Example}
\newtheorem{rem}[thm]{Remark}
\newtheorem*{nrem}{Remark}
\newtheorem{construction}[thm]{Construction}

\calclayout

\newglossaryentry{catinf}{%
	name={\ensuremath{\CatInf}},
	description={$\infty$-category of $\infty$-categories},
	type=symbols
}
\newglossaryentry{sinf}{%
	name={\ensuremath{\cS}},
	description={$\infty$-category of spaces},
	type=symbols
}
\newglossaryentry{ncoh}{%
	name={\ensuremath{N_\coh}},
	description={Coherent nerve},
	type=symbols
}
\newglossaryentry{homcat}{%
	name={\ensuremath{h}},
	description={Homotopy category},
	type=symbols
}
 \newglossaryentry{cocartpush}{%
	name={\ensuremath{(-)_!}},
	description={Cocartesian pushforward},
	type=symbols
}
\newglossaryentry{fin}{%
	name={\ensuremath{\Fin_*}},
	description={Category of pointed finite sets},
	type=symbols
}
\newglossaryentry{langlerangle}{%
	name={\ensuremath{\langle p \rangle}},
	description={Pointed finite set $\{1,2,\ldots,p,\ast\}$ in $\Fin_\ast$},
	type=symbols
}
\newglossaryentry{delta}{%
	name={\ensuremath{\Delta}},
	description={Category of non-empty totally ordered finite sets},
	type=symbols
}
\newglossaryentry{deltainj}{%
	name={\ensuremath{\Delta_\inj}},
	description={Wide sub-category of $\Delta$ of injective maps},
	type=symbols
}
\newglossaryentry{brpbr}{%
	name={\ensuremath{[p]}},
	description={Totally ordered finite set $(0<1<\cdots<p)$ in $\Delta$},
	type=symbols
}
\newglossaryentry{gap}{%
	name={\ensuremath{\Gap}},
	description={Category of gaps},
	type=symbols
}
\newglossaryentry{lbrrbr}{%
	name={\ensuremath{\lbr p \rbr}},
	description={Totally ordered set $(L<1<\cdots<p<R)$ in $\Gap$},
	type=symbols
}
\newglossaryentry{gapsur}{%
	name={\ensuremath{\Gap_\sur}},
	description={Wide sub-category of $\Gap$ of surjective maps},
	type=symbols
}
\newglossaryentry{caticat}{%
	name={\ensuremath{\Cat(\icat{C})}},
	description={$\infty$-category of double $\infty$-categories},
	type=symbols
}
\newglossaryentry{monicat}{%
	name={\ensuremath{\Mon(\icat{C})}},
	description={$\infty$-category of monoid objects in $\icat{C}$},
	type=symbols
}
\newglossaryentry{catnuicat}{%
	name={\ensuremath{\Cat_\nonunital(\icat{C})}},
	description={$\infty$-category of non-unital double $\infty$-categories},
	type=symbols
}
\newglossaryentry{monnuicat}{%
	name={\ensuremath{\Mon_\nonunital(\icat{C})}},
	description={$\infty$-category of non-unital monoid objects in $\icat{C}$},
	type=symbols
}
\newglossaryentry{cmonnuicat}{%
	name={\ensuremath{\CMon(\icat{C})}},
	description={$\infty$-category of commutative monoid objects in $\icat{C}$},
	type=symbols
}
\newglossaryentry{cabmappingcat}{%
	name={\ensuremath{\cC_{A,B}}},
	description={Mapping $\infty$-category from $A$ to $B$ in double $\infty$-category $\icat{C}$},
	type=symbols
}
\newglossaryentry{catquc}{%
	name={\ensuremath{\Cat_\quasiunital(\cC)}},
	description={$\infty$-category of quasi-unital category objects in $\cC$},
	type=symbols
}
\newglossaryentry{simeq2}{%
	name={\ensuremath{(-)^{(\infty,1)}}},
	description={Functor extracting $\infty$-category from double $\infty$-category},
	type=symbols
}
\newglossaryentry{infty2}{%
	name={\ensuremath{(-)^{(\infty,2)}}},
	description={Functor extracting $(\infty,2)$-category from double $\infty$-category},
	type=symbols
}
\newglossaryentry{infty1}{%
	name={\ensuremath{(-)^{\simeq_2}}},
	description={Functor extracting $\infty$-category from $(\infty,2)$-category},
	type=symbols
}
\newglossaryentry{psh}{%
	name={\ensuremath{\PSh(\cC)}},
	description={$\infty$-category of $\cS$-valued presheaves on $\cC$},
	type=symbols
}
\newglossaryentry{yon}{%
	name={\ensuremath{\yon}},
	description={Yoneda embedding},
	type=symbols
}

\newglossaryentry{mul}{%
	name={\ensuremath{\Mul_\cO}},
	description={Spaces of multi-operations of an $\infty$-operad $\cO$},
	type=symbols
}
\newglossaryentry{circo}{%
	name={\ensuremath{\circ_\cO}},
	description={Operadic composition map of an $\infty$-operad $\cO$},
	type=symbols
}
\newglossaryentry{endc}{%
	name={\ensuremath{\End_\cC(x)}},
	description={Endomorphism operad of $x$ in $\cC$},
	type=symbols
}

\newglossaryentry{lambda}{%
	name={\ensuremath{\Lambda_{/[p]}}},
	description={Full subcategory of cellular maps in $\Delta_{/[p]}$},
	type=symbols
}
\newglossaryentry{assc}{%
	name={\ensuremath{\Ass(\cC)}},
	description={$\infty$-category of associative algebra objects in $\cC$},
	type=symbols
}
\newglossaryentry{bimodc}{%
	name={\ensuremath{\Bimod(\cC)}},
	description={$\infty$-category of bimodule objects in $\cC$},
	type=symbols
}

\newglossaryentry{uab}{%
	name={\ensuremath{U_{A,B}}},
	description={Forgetful functor from $(A,B)$-modules},
	type=symbols
}
\newglossaryentry{fab}{%
	name={\ensuremath{F_{A,B}}},
	description={Free $(A,B)$-bimodule functor},
	type=symbols
}
\newglossaryentry{prealg}{%
	name={\ensuremath{\overline{\ALG}(\cC)}},
	description={Pre-Morita double $\infty$-category of $\cC$},
	type=symbols
}
\newglossaryentry{alg}{%
	name={\ensuremath{\ALG(\cC)}},
	description={Morita double $\infty$-category of $\cC$},
	type=symbols
}
\newglossaryentry{rhd}{%
	name={\ensuremath{(-)^{\rhd}}},
	description={Cocone construction},
	type=symbols
}

\newglossaryentry{cospan}{%
	name={\ensuremath{\COSPAN^+(\cC)}},
	description={Double $\infty$-category of cospans in $\cC$},
	type=symbols
}

\newglossaryentry{Efunctor}{%
	name={\ensuremath{E}},
	description={Functor from bordism category of presheaf Morita category},
	type=symbols
}
\newglossaryentry{ncbordnu}{%
	name={\ensuremath{\ncBord(d)^{\nonunital}}},
	description={Non-compact $d$-bordism non-unital double $\Kan$-enriched category},
	type=symbols
}
\newglossaryentry{ncbordinfnu}{%
	name={\ensuremath{\ncBordInf(d)^{\nonunital}}},
	description={Non-compact $d$-bordism non-unital double $\infty$-category},
	type=symbols
}
\newglossaryentry{wmu}{%
	name={\ensuremath{(W,\mu)}},
	description={$[p]$-walled $d$-manifold $W$},
	type=symbols
}
\newglossaryentry{mand}{%
	name={\ensuremath{\Man_d^\otimes}},
	description={Monoidal $\Kan$-enriched category of $d$-manifolds},
	type=symbols
}
\newglossaryentry{maninf}{%
	name={\ensuremath{\ManInf_d^\otimes}},
	description={Monoidal $\infty$-category of $d$-manifolds},
	type=symbols
}
\newglossaryentry{wall}{%
	name={\ensuremath{\wall(W,\mu)}},
	description={Submanifold of $(W,\mu)$ of walls},
	type=symbols
}
\newglossaryentry{chamber}{%
	name={\ensuremath{\ch(W,\mu)}},
	description={Submanifold of $(W,\mu)$ of chambers},
	type=symbols
}
\newglossaryentry{thickchamber}{%
	name={\ensuremath{\tch(W,\mu)}},
	description={Submanifold of $(W,\mu)$ of thickened chambers},
	type=symbols
}
\newglossaryentry{collar}{%
	name={\ensuremath{\coll(W,\mu)}},
	description={Submanifold of $(W,\mu)$ of collars},
	type=symbols
}
\newglossaryentry{lab}{%
	name={\ensuremath{\lab_\alpha(W,\mu)}},
	description={Submanifold of $(W,\mu)$ of pieces labelled by $\alpha$},
	type=symbols
}
\newglossaryentry{wlab}{%
	name={\ensuremath{\wlab_\alpha(W,\mu)}},
	description={Submanifold of $(W,\mu)$ of thick walls labelled by $\alpha$},
	type=symbols
}
\newglossaryentry{gapkan}{%
	name={\ensuremath{\unl{\Gap}_{\lbr p \rbr/}}},
	description={$\Kan$-enriched thickening of $\Gap_{\lbr p \rbr/}$},
	type=symbols
}
\newglossaryentry{gapinf}{%
	name={\ensuremath{\unl{\GapInf}_{\lbr p \rbr/}}},
	description={$\infty$-categorical thickening of $\Gap_{\lbr \bullet \rbr/}$},
	type=symbols
}
\newglossaryentry{psi}{%
	name={\ensuremath{E^{\geo}}},
	description={Functor from bordism category to manifold pre-Morita category},
	type=symbols
}
\newglossaryentry{overlinee}{%
	name={\ensuremath{\overline{E}}},
	description={Functor from bordism category to presheaf pre-Morita category},
	type=symbols
}
\newglossaryentry{borddcat}{%
	name={\ensuremath{\BordInf(d)}},
	description={Compact $d$-dimensional bordism double $\infty$-category},
	type=symbols
}
\newglossaryentry{borddcatbdync}{%
	name={\ensuremath{\ncBordInf(d)^{\partial}}},
	description={Noncompact $d$-bordism double $\infty$-category with boundary},
	type=symbols
}
\newglossaryentry{borddcatbdy}{%
	name={\ensuremath{\BordInf(d)^{\partial}}},
	description={Compact $d$-bordism double $\infty$-category with boundary},
	type=symbols
}
\newglossaryentry{theta}{%
	name={\ensuremath{\theta}},
	description={Tangential structure},
	type=symbols
}
\newglossaryentry{ncborddcattheta}{%
	name={\ensuremath{\ncBordInf^{\theta}(d)}},
	description={Noncompact $d$-bordism double $\infty$-category with $\theta$-structure},
	type=symbols
}
\newglossaryentry{borddcattheta}{%
	name={\ensuremath{\BordInf^{\theta}(d)}},
	description={Compact $d$-bordism double $\infty$-category with $\theta$-structure},
	type=symbols
}
\newglossaryentry{modd}{%
	name={\ensuremath{\ModInf(d)}},
	description={Double $\infty$-category of bimodules in $\PSh(\DiscInf_d)$},
	type=symbols
}
\newglossaryentry{nullbordism}{%
	name={\ensuremath{{\BordInf(d)_{P}}}},
	description={$\infty$-groupoid of null bordisms of a compact $(d-1)$-manifold $P$},
	type=symbols
}
\newglossaryentry{rightbordism}{%
	name={\ensuremath{{\Mod(d)_{A}}}},
	description={$\infty$-category of right-modules over $A$ in $\PSh(\DiscInf_d)$},
	type=symbols
}

\newglossaryentry{sdisc}{%
	name={\ensuremath{S^{\DiscInf}_P(X)}},
	description={$\DiscInf$-structure space of a right-$E_{P \times I}$-module $X$},
	type=symbols
}
\newglossaryentry{sdiscpartial}{%
	name={\ensuremath{S^{\DiscInf}_\partial(W)}},
	description={$\DiscInf$-structure space of the right-$E_{\partial W \times I}$-module $E_W$},
	type=symbols
}
\newglossaryentry{bordism}{%
	name={\ensuremath{\leadsto}},
	description={A bordism},
	type=symbols
}
\newglossaryentry{whalf}{%
	name={\ensuremath{W(m,k]}},
	description={Part of bordism $W$ with $i$-handles for $m < i \leq k$},
	type=symbols
}
\newglossaryentry{wfull}{%
	name={\ensuremath{W[m,k]}},
	description={Part of bordism $W$ with $i$-handles for $m \leq i \leq k$},
	type=symbols
}
\newglossaryentry{hmanc}{%
	name={\ensuremath{h\Manc_d}},
	description={$1$-category of compact $d$-manifolds},
	type=symbols
}

\newglossaryentry{finn}{%
	name={\ensuremath{\Fin^{\bfN_0}_\ast}},
	description={$\infty$-operad of pointed finite sets with $\bfN_0$-grading on morphisms},
	type=symbols
}
\newglossaryentry{ppproduct}{%
	name={\ensuremath{\Ydown}},
	description={Pair-of-pants product},
	type=symbols
}
\newglossaryentry{tloc}{%
	name={\ensuremath{(-)_T}},
	description={Localisation at a set of primes $T$},
	type=symbols
}
\newglossaryentry{ratloc}{%
	name={\ensuremath{(-)_\bfQ}},
	description={Rationalisation},
	type=symbols
}
\newglossaryentry{mapopd}{%
	name={\ensuremath{\Map^h_\cat{Opd}(-,-)}},
	description={Derived mapping space of operads},
	type=symbols
}
\newglossaryentry{mapdend}{%
	name={\ensuremath{\Map_{\PSh(\Dend)}^h(-,-)}},
	description={Derived mapping space of dendroidal spaces},
	type=symbols
}
\newglossaryentry{maprdend}{%
	name={\ensuremath{\Map_{\PSh(\rDend)}^h(-,-)}},
	description={Derived mapping space of reduced dendroidal spaces},
	type=symbols
}
\newglossaryentry{dend}{%
	name={\ensuremath{\Dend}},
	description={Category of trees},
	type=symbols
}
\newglossaryentry{rdend}{%
	name={\ensuremath{\rDend}},
	description={Category of reduced trees},
	type=symbols
}
\newglossaryentry{dendn}{%
	name={\ensuremath{N_d}},
	description={Dendroidal nerve},
	type=symbols
}
\newglossaryentry{rdendk}{%
	name={\ensuremath{\rDend_{\leq k}}},
	description={Subcategory of trees whose vertices have $\leq k$ incoming edges},
	type=symbols
}
\newglossaryentry{latchk}{%
	name={\ensuremath{\Latch_k(\cO)}},
	description={$k$th latching object of reduced dendroidal space $\cO$},
	type=symbols
}
\newglossaryentry{matchk}{%
	name={\ensuremath{\Match_k(\cO)}},
	description={$k$th matching object of reduced dendroidal space $\cO$},
	type=symbols
}
\newglossaryentry{tower}{%
	name={\ensuremath{\{G_k\}}},
	description={Tower of groups $G_0 \leftarrow G_1 \leftarrow \cdots$},
	type=symbols
}

\begin{document}

\title{The Disc-structure space}

\author{Manuel Krannich}
\address{Department of Mathematics, Karlsruhe Institute of Technology, 76131 Karlsruhe, Germany}
\email{krannich@kit.edu}

\author{Alexander Kupers}
\address{Department of Computer and Mathematical Sciences, University of Toronto Scarborough, 1265 Military Trail, Toronto, ON M1C 1A4, Canada}
\email{a.kupers@utoronto.ca}


\begin{abstract}
We study the $\DiscInf$-structure space $S^{\DiscInf}_\partial(M)$ of a compact smooth manifold $M$. Informally speaking, this space measures the difference between $M$, together with its diffeomorphisms, and the diagram of ordered framed configuration spaces of $M$ with point-forgetting and point-splitting maps between them, together with its derived automorphisms. As the main results, we show that in high dimensions, the $\DiscInf$-structure space a) only depends on the tangential $2$-type of $M$, b) is an infinite loop space, and c) is nontrivial as long as $M$ is spin. The proofs involve intermediate results that may be of independent interest, including an enhancement of embedding calculus to the level of bordism categories, results on the behaviour of derived mapping spaces between operads under rationalisation, and an answer to a question of Dwyer and Hess in that we show that the map $\BTOP(d)\ra \BAut(E_d)$ is an equivalence if and only if $d$ is at most $2$.
\end{abstract}

\maketitle
\vspace{-.5cm}
\tableofcontents

\vspace{-0.7cm}
\section{Introduction}\label{sec:introduction}
The classification of closed smooth $d$-manifolds and families thereof---smooth fibre bundles---is one of the guiding problems of geometric topology. From a homotopy-theoretic perspective, it is the study of the $\infty$-groupoid\footnotemark[1] $\ManInf(d)^{\cong}$ of smooth closed $d$-manifolds and spaces of diffeomorphisms between them. A historically successful approach to relate---and partially reduce---the study of $\ManInf(d)^{\cong}$ in high dimensions to more homotopy-theoretic and algebraic questions goes by comparison to the $\infty$-groupoid $\cS^{\simeq}$ of spaces via the functor $\ManInf(d)^{\cong}\ra \cS^{\simeq}$ that assigns a manifold its homotopy type. For a given homotopy type $X$, one studies the fibre
\[
	S^{\cS}(X)\coloneqq \fib_X\big(\ManInf(d)^{\cong} \ra \cS^{\simeq}\big),
\]
which can be thought of as the space of manifold structures on $X$. The path components of this \emph{structure space} are equivalence classes of manifolds with a homotopy equivalence to $X$,
\[
	\pi_0\,S^{\cS}(X) = \frac{\left\{\text{\parbox{7cm}{\centering pairs $(M,\varphi)$ of a closed smooth $d$-manifold $M$ and \newline a homotopy equivalence $\varphi \colon M \to X$}}\right\}}{\parbox{7.5cm}{\centering $(M,\varphi) \sim (M',\varphi')\Leftrightarrow$ there exists a diffeomorphism $\alpha \colon M \to M'$ with $[\varphi' \circ \alpha]=[\varphi]\in\pi_0\,\Map_\cS(M,X),$}}
\] 
and the path component of $S^{\cS}(X)$ corresponding to such a pair $(M,\varphi)$ agrees with the identity component of the fibre $\hAut(M)/\Diff(M)$ of the map $\BDiff(M)\ra\BhAut(M)$ induced by considering diffeomorphisms as homotopy equivalences,
\[
	S^{\cS}(X)_{(M,\phi)} \simeq \big(\hAut(M)/\Diff(M)\big)_{\id}.
\]
Surgery theory and pseudoisotopy theory combine to provide an approximation to the structure space $S^{\cS}(X)$ up to extensions in terms of three infinite loop spaces---one in the realm of each, \emph{algebraic $K$-theory},  \emph{algebraic $L$-theory}, and \emph{stable homotopy theory} (see \cite{WWsurvey} for a survey). The unfortunate defect of this approach is that it really is only an approximation, in the sense that it can only capture a finite Postnikov truncation of $S^{\cS}(X)$ depending on the dimension.

\footnotetext[1]{This work is written $\infty$-categorically, so we treat homotopy types and $\infty$-groupoids as indistinguishable. In this introduction, readers unfamiliar with this principle may substitute topologically enriched categories or groupoids for $\infty$-categories or -groupoids; the former being related to homotopy types by taking classifying spaces.}

\medskip

Motivated by Goodwillie--Weiss' embedding calculus and factorisation homology, we pursue a different approach to relate the study of $\ManInf(d)^{\cong}$ to more homotopy-theoretic and algebraic questions, and we establish three fundamental properties of this alternative. Observing that the homotopy type of a manifold $M$ can be viewed as that of the space of ordered configurations of $k$ points in $M$ for $k=1$, this approach is motivated by the idea to remember the homotopy types of the configuration spaces for \emph{all} values of $k$, together with the natural point-forgetting maps between them. It is in fact beneficial to consider configuration spaces of thickened points which admit more natural maps between them, by ``splitting points''. To make this precise, one considers the $\infty$-category $\DiscInf_d$ of finite disjoint unions of $d$-dimensional Euclidean spaces, i.e.~$T\times \bfR^d$ for finite sets $T$, and spaces of smooth embeddings between them. A $d$-manifold $M$ gives rise to a presheaf $E_M\colon \DiscInf_d^{\op}\ra \cS$ on $\DiscInf_d$ with values in the $\infty$-category $\cS$ of spaces via
\begin{equation}\label{equ:presheaf-dfn}
	\smash{\DiscInf_d^\op\ni T\times \bfR^d\xmapsto{E_M}\Emb(T\times \bfR^d,M)\in\cS.}
\end{equation}
By taking derivatives at the centres, the space $E_M(T\times \bfR^d)$ is equivalent to the ordered configuration space of $k=|T|$ points in $M$ together with framings of the tangent space of $M$ at each of these points, and the homotopy type of the ordinary ordered configuration space of $k$ points in $M$ (in particular that of $M$ itself for $k=1$) can be recovered as the quotient by the $\Diff(\bfR^d)^{T} \simeq \OO(d)^T$-action on $E_M(T\times \bfR^d)$ obtained by functoriality. The assignment $M\mapsto E_M$ as in \eqref{equ:presheaf-dfn} is natural in embeddings of $M$, so it in particular defines a functor $E\colon \ManInf(d)^{\cong}\ra \PSh(\DiscInf_d)^\simeq$ to the $\infty$-groupoid of $\cS$-valued presheaves on $\DiscInf_d$. The fibre of this functor at a presheaf $X \colon \DiscInf^{\op}\ra \cS$
\begin{equation*}\label{equ:def-disc-structure-space}
	\smash{S^{\DiscInf}(X)\coloneqq \fib_X\big(\ManInf(d)^{\cong}\xra{E} \PSh(\DiscInf_d)^\simeq\big)}
\end{equation*}
is the eponymous \emph{$\DiscInf$-structure space of $X$}. Analogous to the more traditional structure space $S^{\cS}(X)$, the $\DiscInf$-structure space $S^{\DiscInf}(X)$ can be thought as a space of manifold structures, this time on a presheaf as opposed to just a homotopy type. Similar to before, the path components $\pi_0\,S^{\DiscInf}(X)$ are represented by pairs of a manifold with an equivalence of its presheaf to $X$,
\[
	\pi_0\,S^{\DiscInf}(X) = \frac{\left\{\text{\parbox{7cm}{\centering pairs $(M,\varphi)$ of a closed smooth $d$-manifold $M$ and \newline an equivalence of presheaves $\varphi \colon E_M \to X$}}\right\}}{\parbox{9cm}{\centering $(M,\varphi) \sim (M',\varphi')\Leftrightarrow$ there exists a diffeomorphism\newline $\alpha \colon M \to M'$ with $[\varphi' \circ E_\alpha]=[\varphi]\in\pi_0\,\Map_{\PSh(\DiscInf_d)}(E_M,X)$,}}
\] 
and the path component of $S^{\DiscInf}(X)$ corresponding to such a pair $(M,\varphi)$ agrees with the identity component of the fibre $\Aut(E_M)/\Diff(M)$ of the map $\BDiff(M)\ra \BAut(E_M)$ induced by $E$,
\[
	S^{\DiscInf}(X)_{(M,\varphi)} \simeq \big(\Aut(E_M)/\Diff(M)\big)_{\id}.
\]
In particular, the space $S^{\DiscInf}(X)$ is nonempty if and only if $X\simeq E_M$ for some closed smooth $d$-manifold $M$. If so, then $S^{\DiscInf}(X)\simeq S^{\DiscInf}(E_M)$ so nothing is lost by assuming $X=E_M$ in which case we abbreviate $S^{\DiscInf}(M)\coloneqq S^{\DiscInf}(E_M)$. These are the spaces we focus on in this work. Informally speaking, they measure by how many manifolds the presheaf $X=E_M$ is realised, and how much their diffeomorphism groups differ from the automorphism group of $X$.

\medskip

As the main results of this work, we establish three structural properties of $S^{\DiscInf}(M)$ that one could summarise by saying that for most choices of $M$
\begin{enumerate}[A)]
	\item $S^{\DiscInf}(M)$ depends only little on the manifold $M$,
	\item $S^{\DiscInf}(M)$ is an infinite loop space, and
	\item $S^{\DiscInf}(M)$ is nontrivial.
\end{enumerate}
We state these results in terms of a more general version $S_\partial^{\DiscInf}(M)$ for manifolds that may have boundary, which is crucial for our methods. We postpone its definition to \cref{sec:boundary-intro} below.

\renewcommand\thesubsection{\Alph{subsection})} 

\subsection{Tangential $2$-type invariance}
To make the first property precise, recall that two manifolds $M$ and $N$, possibly with boundary, have the \emph{same tangential $2$-type} if there is a map $B\ra \BO$ so that the maps $M\ra \BO$ and $N\ra\BO$ classifying the stable tangent bundles of $M$ and $N$ admit lifts to maps $M\ra B$ and $N\ra B$ that are $2$-connected.

\begin{nex}
Choosing $B=\BSpin \times K(\pi,1)$, one sees that two spin manifolds $M$ and $N$ have the same tangential $2$-type if and only if their fundamental groupoids are equivalent. In particular, all simply connected spin manifolds have the same tangential $2$-type.
\end{nex} 

Our first main result is that in high dimensions, the $\DiscInf$-structure space $S^{\DiscInf}_\partial(M)$ depends only on the dimension $d$ and the tangential $2$-type of $M$.

\begin{bigthm}\label{bigthm:2-type-invariance}
For compact $d$-manifolds $M$ and $N$ with $d\ge5$ that have the same tangential $2$-type, there exists an equivalence $S_\partial^{\DiscInf}(M)\simeq S_\partial^{\DiscInf}(N)$.
\end{bigthm}

In particular, the $\DiscInf$-structure space of a spin $d$-manifold $M$ with $d\ge5$ only depends on the fundamental groupoid, so we in particular have $\smash{S_\partial^{\DiscInf}(M) \simeq S_\partial^{\DiscInf}(D^d)}$ if $M$ is simply connected. \cref{bigthm:2-type-invariance} also implies that $\smash{S_\partial^{\DiscInf}(M)}$ for a compact $d$-manifold $M$ does not depend on the smooth structure of $M$, since homeomorphic manifolds have equivalent tangential $2$-types (see \cref{rem:classification-2-types}).

\begin{nrem}
One ingredient in the above mentioned approximation to the conventional structure space $S^{\cS}_\partial(M)$ has a similar invariance property (namely, the $L$-theory part depends only on the fundamental groupoid), but the others depend more substantially on the homotopy type of $M$.
\end{nrem}

\begin{nrem}
Reformulated in terms of embedding calculus (see \cref{sec:emb-calc-intro2} for an outline of this relation), \cref{bigthm:2-type-invariance} is an extension of a result of Knudsen--Kupers \cite[6.23]{KnudsenKupers} which applies to certain path components of $S^{\DiscInf}_\partial(M)$ if $M$ is $2$-connected, of dimension $d\ge6$, and $\partial M=S^{d-1}$.
\end{nrem}

\subsection{Infinite loop space structure}
As previously mentioned, the more traditional structure space $\smash{S^{\cS}_\partial(M)}$ is an infinite loop space \emph{after a certain truncation} and \emph{up to extensions}. The $\DiscInf$-structure space $S^{\DiscInf}_\partial(M)$ on the other hand is in high dimensions an actual infinite loop space---no truncations or extensions are necessary. This is our second main result.

\begin{bigthm}\label{bigthm:infinite-loop-space}
For a compact manifold $M$ of dimension $d\ge8$, the space $S^{\DiscInf}_\partial(M)$ admits the structure of an infinite loop space.
\end{bigthm}

\begin{nrem}\,
\begin{enumerate}
\item The bound $d\ge8$ in \cref{bigthm:infinite-loop-space} is not optimal. It can for example be improved to $d\ge6$ if $M$ is simply connected and spin (see \cref{thm:oo-loop-general}). Further improvements are likely possible.
\item The $\DiscInf$-structure space $S_{\partial}^{\DiscInf}(M)$ extends to a space-valued functor on an $\infty$-category of compact $d$-manifolds and embeddings between them (see \cref{sec:sdisc-for-manifolds}), but our construction of the infinite loop space structure on $S_{\partial}^{\DiscInf}(M)$ has less functoriality (see \cref{rem:drawbacks-loop-space-structure}).
\end{enumerate}
\end{nrem}

\subsection{Nontriviality}
At this point a very optimistic reader may wonder whether the $\DiscInf$-structure spaces $S^{\DiscInf}_\partial(M)$ are just contractible, which would in particular say that the diffeomorphism group $\Diff(M)$ of a closed manifold $M$ is equivalent to the automorphism group $\Aut(E_M)$ of the associated presheaf. As our third main result, we show that this is never the case if one assumes the manifold to be spin and of dimension $d\ge5$.

\begin{bigthm}\label{bigthm:nontrivial}
For a compact spin $d$-manifold $M\neq\varnothing$ with $d\ge5$, the space $S^{\DiscInf}_\partial(M)$ is not contractible.
\end{bigthm}

\begin{nrem}
There are partial results in low dimensions that complement \cref{bigthm:nontrivial}.
\begin{enumerate}
	\item For $d\le 2$, Theorem A of \cite{KrannichKupersSurfaces} implies $\smash{S^{\DiscInf}_\partial}(M)\simeq \ast$ (see Remark 1.1 (ii) loc.cit.).
	\item For $d=3$, we give several examples for which $\smash{S^{\DiscInf}_\partial}(M)$ is nontrivial, including $M=D^3$ and $M=S^3$ (see \cref{rem:3-manifolds}).
	\item For $d=4$, Theorem B of \cite{KnudsenKupers} implies that $\pi_0\,\smash{S^{\DiscInf}_\partial}(M)$ surjects onto the set of isotopy classes of smooth structures on $M$ as long as $M$ is $1$-connected and closed, so $S^{\DiscInf}_\partial(M)$ is nontrivial for all such $M$ that admit more than one smooth structure.
\end{enumerate}
\end{nrem}

This concludes the summary of our three main results. In the remainder of this introduction, we briefly indicate how $S^{\DiscInf}_\partial(M)$ relates to embedding calculus, the little $d$-discs operad, and factorisation homology, and then give a summary of the proofs of the main results, where we also make good for the omitted definition of $S_\partial^{\DiscInf}(M)$ for manifolds with boundary.

\renewcommand\thesubsection{\thesection.\arabic{subsection}}
\setcounter{subsection}{0}

\subsection{Relation to embedding calculus, the $E_d$-operad, and factorisation homology}

\subsubsection{Embedding calculus}\label{sec:emb-calc-intro2}
Goodwillie and Weiss' \emph{embedding calculus} \cite{WeissImmersion,GoodwillieWeiss} is a device to study embeddings via their restrictions to submanifolds of the source that are diffeomorphic to $T\times\bfR^d$ for finite sets $T$. It has the form of an approximation to the space of embeddings
\begin{equation}\label{equ:emb-calc-intro}
	\Emb(W,W')\lra T_\infty\Emb(W,W')
\end{equation}
whose target is the limit of a tower of maps whose fibres admit a description in terms of the configurations spaces and frame bundles of $W$ and $W'$. The main result in this context, due to Goodwillie--Klein \cite{GoodwillieKlein}, says that \eqref{equ:emb-calc-intro} is an equivalence if the handle codimension (dimension of $W'$ minus handle dimension of $W$) is at least three. In general, the map \eqref{equ:emb-calc-intro} can fail to be an equivalence, and in a sense the $\DiscInf$-structure spaces may be seen as the ``correction terms'' to \eqref{equ:emb-calc-intro} being an equivalence in codimension zero. Let us make this more precise.

The relation of the map \eqref{equ:emb-calc-intro} to $\DiscInf$-structure spaces is a reformulation of a result of Boavida de Brito--Weiss \cite{BdBWSheaf}, at least if $M$ is closed (c.f.\,\cref{rem:BdPW-other-boundary}). They show that \eqref{equ:emb-calc-intro} is equivalent to the map $\Emb(W,W')\ra \Map_{\PSh(\DiscInf_d)}(E_W,E_{W'})$ induced by the naturality of $E_W$ in embeddings, which---for closed $W$ and $W'$ and after discarding non-invertible components in source and target---is the map on mapping spaces induced by the functor $E\colon \ManInf(d)^{\cong}\ra \PSh(\DiscInf_d)^{\simeq}$ used to define the $\DiscInf$-structure space. Since the path space of a $\infty$-groupoid between two objects is naturally equivalent to the space of morphisms between the respective objects, this shows that the loop space of $S^{\DiscInf}(M)$ at $(M,\id_{E_M}) \in \pi_0\,S^{\DiscInf}(M)$ is equivalent to the fibre at $\id$ of \eqref{equ:emb-calc-intro} for $W=W'=M$, so
\begin{equation}\label{equ:rel-to-tinfty-looped}
	\Omega S^{\DiscInf}(M)\simeq \hofib_{\id}(\Emb(M,M)\ra T_\infty\Emb(M,M)).
\end{equation}

\begin{rem}\label{rem:BdPW-other-boundary}
A similar discussion applies if $M$ has boundary, but this does not follow directly from \cite{BdBWSheaf} since we deal with boundary conditions differently to loc.cit. (see \cref{sec:boundary-intro}). 
\end{rem}

Specialising Properties A--C to spin manifolds, they in particular imply:
\begin{bigcor}\label{cor:emb-calc-nontriviality}
For compact connected spin $d$-manifolds $M\neq\varnothing$ with $d\ge5$, the fibre
\[
	\hofib_{\id}\big(\Diff_\partial(M)=\Emb_\partial(M,M)\ra T_\infty\Emb_\partial(M,M)\big)
\] 
is nontrivial and depends only on the fundamental group of $M$. It is an infinite loop space for $d\ge8$.
\end{bigcor}

\subsubsection{The operad $E_d$ of little $d$-discs}
We continue by mentioning two connections between $S^{\DiscInf}_\partial(M)$ and the operad $E_d$ of little $d$-discs. The first is that $\DiscInf_d$ agrees with the monoidal envelope (also known as the associated PROP) of the framed $E_d$-operad, so $\PSh(\DiscInf_d)$ can be identified with the $\infty$-category of right-modules over this operad and hence the definition of $S_\partial^{\DiscInf}(M)$ for closed manifolds can be rephrased in these terms. There is a similar reformulation if $M$ has boundary.

The second relation is less obvious and once more a result of work of Boavida de Brito and Weiss \cite{BdBWConf}. To explain it, observe that the standard action of $\oO(d)$ on the disc $D^d$ induces an $\oO(d)$-action on the operad $E_d$ of little $d$-discs. This action extends to the topological group $\TOP(d)$ of homeomorphisms of $\bfR^d$, so there is a map
\begin{equation}
	\label{equ:topd-to-ed-intro}\BTOP(d)\lra \BAut(E_d)
\end{equation}
with $\Aut(E_d)$ the automorphism group of the $E_d$-operad. Reformulated in our setting, their work (or alternatively work of Ducoulombier--Turchin \cite{DucoulombierTurchin}) moreover implies that there is an equivalence 
\begin{equation}\label{equ:pedro-michael-equivalence}
	\Omega^{d+1}(\Aut(E_d)/\TOP(d))\simeq  S^{\DiscInf}_{\partial}(D^d).
\end{equation}
In particular \cref{bigthm:infinite-loop-space} and \cref{bigthm:nontrivial} for $M=D^d$ (or rather certain refinements of them) imply:

\begin{bigcor}\label{bigcor:top-vs-auted}
The map $\BTOP(d)\ra \BAut(E_d)$ is an equivalence if and only if $d\le2$. Moreover, its fibre admits for $d\ge6$ the structure of an infinite loop space after taking $(d+1)$-fold loop spaces.
\end{bigcor}

\begin{rem}A couple of remarks on the equivalence \eqref{equ:pedro-michael-equivalence} and \cref{bigcor:top-vs-auted} are in order.
	\begin{enumerate}
		\item Dwyer and Hess asked whether the map \eqref{equ:topd-to-ed-intro} is an equivalence \cite[58 min]{Dwyer}. The first part of \cref{bigcor:top-vs-auted} gives an answer.
		\item The cases $d\le 2$ of the first part of \cref{bigcor:top-vs-auted} are not due to us: Horel \cite{Horel} proved the case $d=2$. The case $d=1$ is folklore and can be proved via Horel's approach.
		\item The equivalence \eqref{equ:pedro-michael-equivalence} can strictly speaking only be deduced from \cite{BdBWConf} or \cite{DucoulombierTurchin} after passing to certain components (see \cref{thm:sdisc-auted-topd}), but a different proof that does not require this was given as part of \cite{KrannichKupersOperadic} (see \cref{rem:different-delooping-proof}).
	\end{enumerate}
\end{rem}

\subsubsection{Factorisation homology}\label{sec:factorisation-homology}
The final relation of $S^{\DiscInf}_\partial(M)$ we would like to mention is one to \emph{factorisation homology} (or \emph{topological chiral homology}) \cite{Salvatore, Francis, Andrade, AyalaFrancisTop, LurieHA}. In its simplest instance, this connection amounts to the (quite tautological) observation that for a framed $E_d$-algebra $A$ in a suitable $\infty$-category $\cC$, there is a commutative diagram 
\[\begin{tikzcd} 
	\ManInf(d)^{\cong}\rar{E}\arrow[dr,"\int_{(-)}A",swap]&\PSh(\DiscInf_d)\dar{(-)\otimes_{\DiscInf_{d}}A}\\
	&\cC
\end{tikzcd}\]
of $\infty$-categories in which the diagonal arrow is given by factorisation homology with coefficients in $A$ and the vertical arrow by taking coends, using that $A$ is in particular a functor $A\colon \DiscInf_d\ra \cC$. In fact, the functor $E$ itself is an instance of factorisation homology, namely with coefficients in the framed $E_d$-algebra $E_{D^d}\in \PSh(\DiscInf_d)$, so $E$ may be viewed as the universal factorisation homology invariant on $\ManInf(d)^{\cong}$, and the study of $\DiscInf$-structure spaces as closely related to the question to which extent the theory of manifolds can be captured by factorisation homology.

\subsection{Summary of proofs}
We conclude with a summary of the proofs of Theorems~\ref{bigthm:2-type-invariance}--\ref{bigthm:nontrivial}. \smallskip

\begin{center}\textit{Some steps may be of independent interest. We highlight them with the Roman numerals \ref{enum:bord-emb-calc}--\ref{enum:rationalisation-operads-intro}.}\end{center}

\subsubsection{The case with boundary}\label{sec:boundary-intro}
The more general $\DiscInf$-structure spaces $\cS_\partial^{\DiscInf}(M)$ for manifolds $M$ with boundary play a central role in the proofs of all main results of this work, even when specialised to closed manifolds, so we first make good on omitting its definition earlier. 

Fixing a closed $(d-1)$-manifold $Q$, one replaces $\smash{\ManInf(d)^{\cong}}$ with the $\infty$-groupoid $\ManInf(d)^{\cong}_Q$ of compact $d$-manifolds with an identification of their boundary with $Q$, and spaces of diffeomorphisms preserving these identifications. The definition \eqref{equ:presheaf-dfn} of the presheaf $E_M$ still makes sense if $M$ has boundary $Q$ and thus yields a functor $\ManInf(d)^{\cong}_Q\ra\PSh(\DiscInf_d)^\simeq$, but if $Q\neq\varnothing$ then the presheaf $E_M$ carries additional structure. Indeed, stacking cylinders induces an associative algebra structure on the presheaf $E_{Q\times I}\in \PSh(\DiscInf_d)$ with respect to the symmetric monoidal structure on $\PSh(\DiscInf_d)$ given by Day convolution, induced by taking disjoint unions in $\DiscInf_d$. Similarly, fixing a collar $Q\times I\hookrightarrow M$ of the boundary of $M$, the presheaf $E_{M}$ becomes a right-$E_{Q\times I}$-module. Made precise, this enhances the functor $E\colon \ManInf(d)^{\cong}_Q\ra\PSh(\DiscInf_d)^\simeq$ to a functor 
\begin{equation}\label{equ:e-functor-for-left-modules}
	E\colon \ManInf(d)^{\cong}_Q\lra\Mod(d)_{E_{Q\times I}}^\simeq
\end{equation}
with target the $\infty$-groupoid $\Mod(d)_{E_{Q\times I}}^\simeq$ of right-$E_{Q\times I}$-modules. The \emph{$\DiscInf$-structure space of a right-$E_{Q\times I}$-module} $X$ is then defined as the fibre
\[
	S^{\DiscInf}_{Q}(X)\coloneqq \fib_{X}\big(\ManInf(d)^{\cong}_Q\xra{E}\Mod(d)_{E_{Q\times I}}^\simeq\big);
\]
that this recovers the previous definition in the case $Q=\varnothing$ follows by observing that $E_{\varnothing\times I}$ is the monoidal unit. As in the closed case, we abbreviate $\smash{S^{\DiscInf}_\partial}(M)\coloneqq \smash{S^{\DiscInf}_{Q}}(E_M)$ if the right-$E_{Q \times I}$-module $X=E_M$ is induced by a manifold $M$ with identified boundary $\partial M\cong Q$. This is the generalisation of $S^{\DiscInf}(M)$ for manifolds with boundary in terms of which we stated Theorems~\ref{bigthm:2-type-invariance}--\ref{bigthm:nontrivial} above.
\subsubsection{Extension to the bordism category}\label{sec:intr-bordism}
For the proofs of these results, we need to generalise the functor \eqref{equ:e-functor-for-left-modules} further. Given another closed $(d-1)$-manifold $P$, we write $\BordInf(d)_{P,Q}$ for the $\infty$-groupoid of compact bordisms $W\colon P\leadsto Q$ and spaces of diffeomorphisms preserving the identifications of the ends. For such a bordism, the associated presheaf $E_W$ becomes a $(E_{P\times I},E_{Q\times I})$-bimodule in $\PSh(\DiscInf_d)$ and we have a functor 
\begin{equation}\label{equ:functor-bimodule}
	E\colon \BordInf(d)_{P,Q}\lra\Mod(d)_{E_{P\times I},E_{Q\times I}}^\simeq
\end{equation}
to the $\infty$-groupoid $\smash{\Mod(d)^\simeq_{E_{P\times I},E_{Q\times I}}}$ of $(E_{P\times I},E_{Q\times I})$-bimodules, generalising the case $P=\varnothing$ discussed in the previous subsection. Given another closed $(d-1)$-manifold $R$, one can show that there is a commutative square of $\infty$-groupoids
\[\begin{tikzcd}[column sep=2cm]
	\BordInf(d)_{P,Q}\times\BordInf(d)_{Q,R}\rar{(-)\cup_Q (-)}\dar[swap]{E\times E}&\BordInf(d)_{P,R}\dar{E}\\
	\Mod(d)^{\simeq}_{P,Q}\times \Mod(d)^{\simeq}_{Q,R}\rar{(-)\otimes_{E_{Q\times I}}(-)}& \Mod(d)^{\simeq}_{P,R},
\end{tikzcd}\]
whose horizontal functors are induced by gluing bordisms and tensoring bimodules respectively; this is essentially an instance of what is known as $\otimes$-excision in the theory of factorisation homology. These squares suggest that the functors \eqref{equ:functor-bimodule} might in fact arise as the maps induced on mapping spaces by a functor of $\infty$-categories
\begin{equation}\label{equ:e-on-compact-bordisms-intro}
	E\colon \BordInf(d)^{(\infty,1)}\lra \Mod(d)^{(\infty,1)}
\end{equation}
from the $d$-dimensional bordism category to a Morita category whose objects are associative algebras in $\PSh(\DiscInf_d)$ and whose morphisms are bimodules. This turns out to be the case, but to prove our results, we need even more functoriality. For this, one notes that the presheaf $E_M$ of a manifold makes equal sense if $M$ is noncompact, so \eqref{equ:e-on-compact-bordisms-intro} ought to extend to a functor 
\begin{equation}\label{equ:e-on-noncompact-bordisms-intro}
	E\colon \ncBordInf(d)^{(\infty,2)}\lra \Mod(d)^{(\infty,2)}
\end{equation}
of $(\infty,2)$-categories from a larger bordism category of possibly noncompact manifolds that has codimension $0$ embeddings as $2$-morphisms, not just diffeomorphisms, to a larger Morita category $\Mod(d)^{(\infty,2)}$ that has morphisms of bimodules as $2$-morphisms, not just invertible ones.

In \cref{sec:the-functor}, relying on work of Haugseng \cite{HaugsengMorita}, we carefully construct such a functor \eqref{equ:e-on-noncompact-bordisms-intro} of $(\infty,2)$-categories and show that it can be enhanced to a functor of \emph{symmetric monoidal} $(\infty,2)$-categories. As part of \cref{sec:functor-e-disc-structure}, we show that for (possibly noncompact) bordisms $W,W'\colon P\leadsto Q$ one can identify the map between mapping spaces of $2$-morphisms induced by \eqref{equ:e-on-noncompact-bordisms-intro}
\[
\begin{tikzcd}[row sep=0.2cm,ar symbol/.style = {draw=none,"\textstyle#1" description,sloped},
	equivalent/.style = {ar symbol={\simeq}}]
\Map_{\ncBordInf(d)_{P,Q}}(W,W')\rar{E}\arrow[d,equivalent] &\Map_{\Mod(d)_{P,Q}}(E_W,E_{W'})\arrow[d,equivalent] \\
\Emb_\partial(W,W')\rar& T_\infty\Emb_\partial(E_W,E_{W'})
\end{tikzcd}
\]
with Goodwillie--Weiss' embedding calculus approximation, so one might view the functor \eqref{equ:e-on-noncompact-bordisms-intro} as an enhancement of embedding calculus to the level of bordism categories. In particular,

\begin{enumerate}[label={(\Roman*)},leftmargin=0.8cm]
	\item \label{enum:bord-emb-calc} the functor \eqref{equ:e-on-noncompact-bordisms-intro} of symmetric monoidal $(\infty,2)$-categories equips the limit of the embedding calculus tower with homotopy coherent gluing, composition, and disjoint union maps.
\end{enumerate}
The functor \eqref{equ:e-on-noncompact-bordisms-intro} and its relation to embedding calculus forms the technical backbone of the proofs of Theorems~\ref{bigthm:2-type-invariance}--\ref{bigthm:nontrivial} in the later chapters, whose proof strategies we summarise now.

\begin{rem}As part of \cite{KrannichKupersOperadic}, the functor \eqref{equ:e-on-noncompact-bordisms-intro} was generalised in several directions.
\end{rem}

\subsubsection{\cref{bigthm:2-type-invariance}: tangential $2$-type invariance}\label{sec:intr-2-type-invariance}
The functor \eqref{equ:e-on-compact-bordisms-intro} in particular extends the $\DiscInf$-structure space of a manifold $S^{\DiscInf}_\partial(M)$ to a space-valued functor of $\infty$-categories 
\begin{equation}\label{equ:functor-on-nullbordism-cat-intro}
	S_\partial^{\DiscInf}(-)\colon \BordInf(d)^{(\infty,1)}_{\varnothing/}\lra \cS
\end{equation}
defined on the $\infty$-category of null bordisms, i.e.\,the undercategory of $\varnothing\in \BordInf(d)^{(\infty,1)}$. Relying on the relation to embedding calculus via \eqref{equ:e-on-noncompact-bordisms-intro}, a version of an isotopy extension theorem for embedding calculus due to Knudsen--Kupers \cite{KnudsenKupers}, and Goodwillie--Klein's above mentioned convergence theorem, we show that the functor \eqref{equ:functor-on-nullbordism-cat-intro} sends a bordism $W\colon P\leadsto Q$ to an equivalence if $W$ can be built from a collar on $P$ by attaching handles of index $\ge3$. This leads to a proof of \cref{bigthm:2-type-invariance}, since it turns out that the value of \emph{any} functor of the form \eqref{equ:functor-on-nullbordism-cat-intro} with this property depends up to equivalence only on the tangential $2$-type. This is an instance of 
\begin{enumerate}[label={(\Roman*)},leftmargin=0.8cm,resume]
	\item \label{enum:general-k-invariance} a general tangential $k$-type invariance result for the values of certain functors on the category $\BordInf(d)^{(\infty,1)}_{\varnothing/}$ of null bordisms.
	\end{enumerate}
The proof of \ref{enum:general-k-invariance} amounts to a sequence of surgery arguments that we became aware of through the literature on the space of metrics of positive scalar curvature, in particular \cite{EbertRWbordism,EbertWiemeler}.

\subsubsection{\cref{bigthm:infinite-loop-space}: infinite loop space}\label{sec:intr-infinite-loop-space}
To construct an infinite loop space structure on $S^{\DiscInf}_\partial(M)$, we first use the tangential $2$-type invariance to show that it suffices to consider manifolds of the form $M=P\times D^{2n}$ for $P$ a closed manifold and $2n\ge4$. From the definition
\begin{equation}\label{equ:map-defining-disc-intro}
	S^{\DiscInf}_\partial(P\times D^{2n})=\fib_{E_{P\times D^{2n}}}\big(\BordInf(d)_{P\times S^{2n-1}}\xra{E}\Mod(d)_{E_{P\times S^{2n-1}\times I}}^\simeq\big),
\end{equation}
it is clear that it suffices to prove that the right-hand map is a map of infinite loop spaces. After restriction to certain path-components that does not affect the fibre, this is what we do. More precisely, in the target, we restrict to modules equivalent to $\smash{E_{P\times W_{g,1}}}$ for $g\ge0$ where $W_{g,1}$ is short for the bordism $(S^n\times S^n)^{\sharp g}\backslash\interior(D^{2n})\colon \varnothing\leadsto S^{2n-1}$. In the source, we restrict to bordisms whose induced presheaf is equivalent to $E_{P\times W_{g,1}}$ for $g\ge0$ as a bimodule. We then use the full coherence provided by the functor \eqref{equ:e-on-compact-bordisms-intro} to enhance the restricted map to one of algebras over a certain higher-dimensional version $\cW$ of Tillmann's surface operad  \cite{Tillmann}, constructed out of bordisms of the form $\sqcup^{k} S^{2n-1}\leadsto \sqcup^{l} S^{2n-1}$ for $k,l\ge0$ that are obtained from the manifolds $W_{g,1}$ by creating more boundary spheres. A variant of this operad has already appeared in work of Basterra--Bobkova--Ponto--Tillmann--Yaekel \cite{BBPTY} on \emph{operads with homological stability}. They proved that algebras over this operad are $E_1$-spaces (via a  ``pair-of-pants'' product) which group-complete to infinite loop spaces, the main ingredient being a stable homological stability result of Galatius--Randal-Williams \cite{GRWII}. Translated to our setting, this implies that the fibre of the group completion of the restricted map is an infinite loop space. Using tangential $2$-type invariance once more, we then show that in this case group completion commutes with taking fibres. This only shows that $S^{\DiscInf}_\partial(P\times D^{2n})$ is an infinite loop space \emph{after group-completion}, but we also show that this $E_1$-space is already group-complete, using the $s$-cobordism theorem.

\subsubsection{\cref{bigthm:nontrivial}: nontriviality}\label{sec:intr-nontriviality}
To show that $\smash{S^{\DiscInf}_\partial}(M)$ is nontrivial for all compact spin manifolds $M$ of dimension $d\ge5$, we first reduce to the case $M=D^d$ using tangential $2$-type invariance. The equivalence \eqref{equ:pedro-michael-equivalence} then further reduces this to showing that the fibre $\Aut(E_d)/\TOP(d)$ of \eqref{equ:topd-to-ed-intro} has a nontrivial homotopy group in sufficiently high degree, which we do by showing that the individual homotopy groups of $\Aut(E_d)$ and $\TOP(d)$ are sufficiently different. While quite a bit is known about the homotopy groups of $\TOP(d)$, especially rationally, so far almost nothing is known about the homotopy groups of $\Aut(E_d)$ besides for small values of $d$. This is in stark contrast to the automorphism group $\Aut((E_d)_{\bfQ})$ of the \emph{rationalised} $E_d$-operad, whose homotopy groups have a complete description in terms of graph complexes à la Kontsevich due to work of Fresse--Turchin--Willwacher \cite{FTW}. Thus, to learn something about the homotopy groups of $\Aut(E_d)$, one could try to study the comparison map $\Aut(E_d)\ra \Aut((E_d)_{\bfQ})$ on homotopy groups. This is what we do. More generally, 
\begin{enumerate}[label={(\Roman*)},leftmargin=0.8cm,resume]
	\item \label{enum:rationalisation-operads-intro} we study the effect on homotopy groups of the map $\Map(\cO,\cP)\ra \Map(\cO_\bfQ,\cP_\bfQ)$ for operads $\cO$ and $\cP$, induced by rationalisation. 
\end{enumerate}
For this, we first use work of Göppl and Weiss \cite{Goppl} to decompose the mapping spaces as a limit of a tower of mapping spaces between truncated operads and show that under mild assumptions, the maps analogous to that in \ref{enum:rationalisation-operads-intro} between the stages of this tower are componentwise rationalisations. Rationalisation does \emph{not} commute with sequential limits in general, so this does \emph{not} imply that the map in \ref{enum:rationalisation-operads-intro} has the same property. However, we then show that this can only fail in an extreme way, namely when some of the homotopy groups of $\Map(\cO,\cP)$ are uncountable. We also explain similar results for more general localisations and for more general towers of spaces.

Applied to $\cO=\cP=E_d$, this shows that the homotopy groups of $\Aut(E_d)$ either agree rationally with those of $\Aut((E_d)_{\bfQ})$, as described in Fresse--Turchin--Willwacher's work, or some of them are uncountable. In either case, we can conclude that they are different from that of $\TOP(d)$: in the former by comparing them with known partial computations of the rational homotopy groups of $\TOP(d)$, and in the latter by using that $\TOP(d)$ has countable homotopy groups.

\subsection*{Acknowledgements} Our thanks go to Fabian Hebestreit and Markus Land for answering several questions on $\infty$-categories, to Rune Haugseng and Claudia Scheimbauer for helpful conversations on their work, to Calista Bernard for sharing her view on manifold calculus and bordism categories, and to Oscar Randal-Williams for general discussions.

MK was partially funded by the ERC under the European Union’s Horizon 2020 research and innovation programme (grant agreement No.\,756444), and partially by the Deutsche Forschungsgemeinschaft (DFG, German Research Foundation) under Germany's Excellence Strategy EXC 2044 –390685587, Mathematics Münster: Dynamics–Geometry–Structure. 

AK acknowledges the support of the Natural Sciences and Engineering Research Council of Canada (NSERC) [funding reference number 512156 and 512250], as well as the Research Competitiveness Fund of the University of Toronto at Scarborough.

This material is partially based on work supported by the Swedish Research Council under grant no.\,2016-06596 while the authors were in residence at Institut Mittag-Leffler in Djursholm, Sweden during the semester \emph{Higher algebraic structures in algebra, topology and geometry}.

\section{$\infty$-categorical preliminaries} \label{sec:preliminaries}
Except for the final two sections (see \cref{conv:no-more-infty}), we work in the setting of $\infty$-categories. This section---which may be skipped on first reading and referred back to when necessary---serves to establish some notation and to recall definitions and facts used in later sections, as well as to prove a few technical results that we could not find in the literature. The topics are:

\begin{minipage}[c]{\textwidth}
\begin{multicols}{2}
\begin{enumerate}[leftmargin=0.1cm]
	\item[\ref{sec:conventions}] Conventions.
	\item[\ref{sec:scat-vs-qcat}] The coherent nerve.
	\item[\ref{sec:straightening}] Cocartesian fibrations.
	\item[\ref{section:delta-gap}] The categories $\Delta$, $\Gap$, and $\Fin_*$.
	\item[\ref{sec:cat-objects}] Category and monoid objects.
	\item[\ref{sec:presheaf-category}] Presheaves and the Yoneda embedding.
	\item[\ref{sec:gen-infty-operads}] $\infty$-operads and generalised $\infty$-operads.
	\item[\ref{sec:assalg-bimodules}] Associative algebras and bimodules.
	\item[\ref{sec:haugseng-morita}] Haugseng's Morita category.
	\item[\ref{sec:span-cospan-cats}] Span and cospan categories.
\end{enumerate}
\end{multicols}
\end{minipage}

\subsection{Conventions}\label{sec:conventions}
Unless mentioned otherwise we follow the conventions and notation of Lurie \cite{LurieHTT,LurieHA}. In particular:
\begin{itemize}
	\item An \emph{$\infty$-category} is a \emph{quasi-category} \cite[1.1.2.4]{LurieHTT}. The \emph{$\infty$-category of $\infty$-categories} $\gls*{catinf}$ is the coherent nerve $\CatInf\coloneqq N_{\coh}(\CatInfS)$ of the $\Kan$-enriched category $\CatInfS$ of small $\infty$-categories \cite[3.0.0.1]{LurieHTT}. We consider 1-categories as $\infty$-categories via their nerve.
	\item A \emph{space} is a Kan complex. If topological spaces appear, we implicitly replace them by their singular simplicial sets. The category of simplicial sets is denoted $\catS$ and the full subcategory of Kan-complexes by $\Kan \subset \cat{S}$. Both are enriched over themselves. The \emph{$\infty$-category of spaces} $\gls*{sinf}$ is the coherent nerve $\cS \coloneqq N_{\coh}(\Kan)$  \cite[1.2.16.1]{LurieHTT}.
\end{itemize}

\noindent We use the following notational conventions:
\begin{itemize}
	\item The letters $\cA$, $\cB$, $\cC$, $\ldots$ typically stand for $\infty$-categories, whereas the letters $\catA$, $\catB$, $\catC$, $\ldots$ usually stand for $\catS$-enriched, $\Kan$-enriched, or 1-categories.
	\item Given an $\infty$-category $\cC$ and object $c$ of $\cC$, $\cC_{c/}^\op$ is short for $(\cC_{c/})^\op$ and similarly $\cC^\op_{/c}$ is short for $(\cC_{/c})^\op$. In other words, slices are taken \emph{before} opposite categories.
\end{itemize}

\subsection{The coherent nerve and the homotopy category}\label{sec:scat-vs-qcat} 
\label{sec:coherent-nerve-props}
The \emph{coherent nerve}  $\gls*{ncoh} \colon \sCat\ra \cat{S}$ is a functor from the $1$-category $\sCat$ of $\catS$-enriched categories to the $1$-category of simplicial sets \cite[1.1.5]{LurieHTT}. Some of its properties are:
\begin{enumerate}
	\item \label{enum:bergner-model-structure} It is the right-adjoint in a Quillen equivalence \cite[2.2.5.1]{LurieHTT}, where $\sCat$ is equipped with the Bergner model structure whose 
	\begin{enumerate}
		\item fibrant objects are $\Kan$-enriched categories \cite[A.3.2.24]{LurieHTT},
		\item weak equivalences are \emph{Dwyer--Kan equivalences}, so simplicial functors that induce weak homotopy equivalences on each mapping space and are an equivalence (of $1$-categories) on homotopy categories \cite[A.3.2.4]{LurieHTT},
		\item fibrations are simplicial functors that are Kan fibrations on each mapping space and isofibrations on homotopy categories \cite[A.3.2.24, A.3.2.25]{LurieHTT},
	\end{enumerate}
	and $\cat{S}$ is equipped with the Joyal model structure of which we only need to know that its fibrant objects are precisely $\infty$-categories \cite[2.4.6.1]{LurieHTT}. In particular, the coherent nerve of a $\Kan$-enriched category is an $\infty$-category.
	\item \label{enum:objects-morphisms}Taking coherent nerves preserves objects and morphisms, in the sense that the $0$- and $1$-simplices of $N_{\coh}(\cat{C})$ are the sets of objects and morphisms of $\cat{C}$ \cite[p.\,23]{LurieHTT}.
	\item \label{enum:simplicial-mapping-spaces} Taking coherent nerves preserves mapping spaces of $\Kan$-enriched categories in that for a $\Kan$-enriched category $\cat{C}$ we have $\Map_{\cat{C}}(c,c')\simeq \Map_{N_{\coh}(\cat{C})}(c,c')$ \cite[2.2]{LurieHTT}.
	\item \label{enum:coherent-opposite} There is a natural equivalence $N_\coh(\cat{C}^\op) \simeq N_\coh(\cat{C})^\op$. This is a consequence of the natural isomorphisms  $\mathfrak{C}([n]^\op) \cong \mathfrak{C}([n])^\op$, where $\mathfrak{C}(-)$ is the left adjoint to $N_{\coh}(-)$.
	\item \label{enum:coherent-functor} There is a canonical map $N_\coh(\Fun(\cat{C},\cat{D})) \to \Fun(N_\coh(\cat{C}),N_\coh(\cat{D}))$ obtained by appling $N_\coh$ to the evaluation $\Fun(\cat{C},\cat{D}) \times \cat{C} \ra \cat{D}$, using that as a right adjoint $N_\coh(-)$ preserves products to get $N_\coh(\Fun(\cat{C},\cat{D})) \times N_\coh(\cat{C}) \to N_\coh(\cat{D})$, and adjoining over $N_\coh(\cat{C})$.\end{enumerate}
Restricting $N_\coh$ to $\cat{Cat} \subset \sCat$ gives a fully faithful functor of $1$-categories from ordinary $1$-categories to $\infty$-categories. Applying $N_{\coh}$, we obtain a functor $\Cat\ra \CatInf$ of $\infty$-categories. This has a left-adjoint $\gls*{homcat} \colon \CatInf \to \Cat$ that assigns an $\infty$-category its \emph{homotopy category}. As described in \cite[1.2.3]{LurieHTT}, $h\cC$ has the same objects as $\cC$, morphism sets given by the path components of the respective mapping spaces in $\cC$, and composition is induced by the composition maps of mapping spaces. Some of its further properties are:
\begin{enumerate}
	\item The functor $h$ preserves products.
	\item The functor $h$ preserves pullbacks if one of the maps is between 1-categories.
	\item The functor $h$ preserves cocartesian morphisms when the target is an 1-category.
\end{enumerate}
These follow from the facts that taking mapping spaces in $\infty$-categories preserves pullbacks, and that taking components preserves pullbacks in $\cS$ whose bottom right corner is discrete. 

\subsection{Cocartesian fibrations} \label{sec:straightening}
Lurie's \emph{straightening equivalence} \cite[3.2]{LurieHTT}
\begin{equation}\label{equ:st-un}
		\Fun(\icat{C},\CatInf) \simeq \Cocart(\icat{C})
\end{equation}
identifies the $\infty$-category $\Fun(\icat{C},\CatInf)$ for an $\infty$-category $\cC$ with the  $\infty$-category of \emph{cocartesian fibrations}, which is the sub $\infty$-category $\Cocart(\cC)\subset (\CatInf)_{/\cC}$ with objects \emph{cocartesian fibrations} with target $\cC$ and whose morphisms are \emph{maps of cocartesian fibrations}, in the following sense:
\begin{dfn}\label{dfn:cocartesian-fibration}Let $\varphi\colon \cE \ra \cB $ be a functor between $\infty$-categories. 
	\begin{enumerate}
		\item 
		A morphism $f\colon e\ra e'$ in $\cE $ is \emph{$\varphi$-cocartesian} if for every $x\in \cE $ the square
		\[
		\begin{tikzcd}
			\Map_{\icat{E}}(e',x)\rar{f^*}\dar[swap]{\varphi}&\Map_{\icat{E}}(e,x)\dar{\varphi}\\
			\Map_{\icat{B}}(\varphi(e'),\varphi(x))\rar{\varphi(f)^*}&\Map_{\icat{B}}(\varphi(e),\varphi(x))
		\end{tikzcd}
		\]
		is homotopy cartesian.
		\item The functor $\varphi$ is a \emph{cocartesian fibration} if for every object $e\in\cE $ and morphism $f\colon \varphi(e)\ra b$ there exists a \emph{cocartesian lift} of $f$, i.e.\,a $\varphi$-cocartesian morphism $\tilde{f}\colon e\ra \tilde{b}$ for some $\tilde{b}$ in $\cE $ such that $\varphi(\tilde{f})=f$.
		\item A \emph{map of cocartesian fibrations} from $\varphi\colon \cE \ra \cB $ to $\varphi'\colon \cE '\ra \cB $ is a functor $\cE \ra \cE '$ over $\cB $ that sends $\varphi$-cocartesian morphisms to $\varphi'$-cocartesian morphisms.
	\end{enumerate}
\end{dfn}
Given a cocartesian fibration $\varphi\colon \cE \ra \cB $ and an object $b\in\cB $, one writes $\cE _b\in\CatInf$ for the fibre of $\varphi$ at $b$. Under the straightening equivalence \eqref{equ:st-un}, this corresponds to the value at $b$ of the associated functor $\cB \ra \CatInf$. Moreover, the value of this functor on a morphism $b\ra b'$ in $\cB $ corresponds to a functor $\cE_b\ra \cE_b'$ induced by choosing cocartesian lifts of $b\ra b'$. 

\begin{rem}\label{fact:cocart-from-s-cat}
	 \cref{dfn:cocartesian-fibration} makes equal sense for a functor $\varphi\colon \cat{E}\ra \cat{B}$ of $\Kan$-enriched categories. In view of the natural equivalence  $\Map_{\cat{C}}(c,c')\simeq \Map_{N_{\coh}(\cat{C})}(c,c')$, one sees that a morphism $f\colon e\ra e'$ in $\cat{E}$ is $\varphi$-cocartesian if and only if is $N_{\coh}(\varphi)$-cocartesian.
\end{rem}
\glsadd{cocartpush}
Given a cocartesian fibration $\varphi\colon \cE \ra \cB $ and an $\infty$-category $\cC$, the functor $\varphi_*\colon \Fun(\cC,\cE) \ra \Fun(\cC,\cB)$ is again a cocartesian fibration  \cite[3.1.2.1]{LurieHTT}. In particular, given a functor $f \colon \cC\ra \cE$ and a natural transformation $\eta \colon (\varphi\circ f)\ra *_b$ to the constant functor $*_b\colon \cC\ra \cB$ with value $b\in B$ (equivalently, an extension of $(\varphi\circ f)$ to a functor $\cC^\rhd \ra \cB$ on the right-cone whose value at the cone point is $b$), we can use that $\varphi_*$ is a cocartesian fibration to obtain a cocartesian lift to a functor $f_!\colon \cC\ra \cE_b$ into  the fibre over $b$. The functor $f_!$ is called a \emph{cocartesian pushforward} of $f$ along $\eta$.

\subsection{The categories $\Fin_*$, $\Delta$, and $\Gap$} \label{section:delta-gap}
Recall the 1-category \gls*{fin} of pointed finite sets and pointed maps in between, with skeleton $\gls*{langlerangle} = \{1,\ldots,p,\ast\}$ for $p \geq 0$. We write $\langle \mathring{p} \rangle \coloneqq \langle p \rangle \setminus \{\ast\}$ for the \emph{interior} of $\langle p \rangle$. Three special types of morphisms are relevant for us: a morphism $\alpha\colon \langle p \rangle \ra \langle q \rangle$ is 
\begin{enumerate}
	\item \emph{active} if it satisfies $\alpha^{-1}(\ast)=\{\ast\}$, 
	\item \emph{inert} if $\alpha^{-1}(i)$ consists of a single element for all $i \in \langle \mathring{q} \rangle$,
	\item \emph{Segal} if it agrees with $\rho_i\colon \langle p \rangle \ra \langle 1 \rangle$ for some $1\le i\le p$ where $\rho_i(i)=1$ and $\rho(j)=\ast$ otherwise. Note this is equivalent to being inert with target $\langle 1 \rangle$.
\end{enumerate}
A closely related 1-category is the \emph{simplex category}  $\gls*{delta}$ of non-empty finite totally ordered sets and weakly order-preserving maps between them. We mostly work with its skeleton given by $\gls*{brpbr}=(0<1<\ldots<p)$ for $p\ge0$. The wide subcategory of injective maps is denoted $\gls*{deltainj} \subset \Delta$. Four special types of morphisms are relevant for us: a morphism $\alpha\colon [p]\ra [q]$ is 
\begin{enumerate}
	\item \emph{active} if it satisfies $\alpha(0)=0$ and $\alpha(p)=q$, 
	\item \emph{cellular} if $\alpha(i+1)\le \alpha(i)+1$ for all $i$,
	\item \emph{inert} if it is the inclusion of a subinterval, i.e.\,$\alpha(i)=\alpha(0)+i$ for all $i$, and
	\item \emph{Segal} if it agrees with $\rho_i\colon [1]\ra [q]$ for some $1\le i\le q$ where $\rho_i(0)=i-1$ and $\rho(1)=i$. Note this is equivalent to being inert with domain $[1]$.
\end{enumerate}
Occasionally we work with a different model for $\Delta^{\op}$, given as follows. For $p\ge0$ we write $\gls*{lbrrbr}$ for the totally ordered set $(L<1<\ldots< p<R)$ and call $L$ and $R$ the \emph{left end} and \emph{right end} of $\lbr p\rbr$ respectively. The sets $\lbr p\rbr$ for $p\ge0$ form the objects of the category $\gls*{gap}$ whose morphisms are weakly order-preserving maps that are the identity on ends. There is an isomorphism \begin{equation}\label{equ:gap-iso}
	c\colon \Delta^\op\xlra{\cong} \Gap
\end{equation} 
that sends $[p]$ to $\lbr p \rbr$ and a morphism $\alpha\colon [p]\ra [q]$ to the morphism $c(\alpha)\colon \lbr q\rbr\ra  \lbr p\rbr$ given by
\begin{align*}
	{i}&\longmapsto \begin{cases}L& i\le  \alpha(0),\\
		j& \exists j\colon i\in [\alpha(j-1)+1,\alpha(j)],\\
		R& i>\alpha(p).\end{cases}
\end{align*}
This isomorphism maps $\smash{\Delta_\inj^\op}\subset \Delta^\op$ isomorphically onto the wide subcategory $\gls*{gapsur} \subset \Gap$ of surjective maps. Introducing the notation $\lbr \mathring{p}\rbr\coloneqq \lbr p\rbr\backslash\{L,R\}$ for the \emph{interior} of $\lbr p\rbr$, a morphism $\alpha\colon \lbr q\rbr \ra \lbr p\rbr$, when considered as a morphism $c^{-1}(\alpha)\colon [p]\ra [q]$ in $\Delta$, is 
\begin{enumerate}
	\item \emph{active} if $\alpha^{-1}\lbr \mathring{p}\rbr=\lbr \mathring{q}\rbr$ (we omit the parentheses in $\alpha^{-1}(\lbr \mathring{p}\rbr)$ for legibility),
	\item \emph{cellular} if the restriction $\alpha\colon \alpha^{-1}\lbr \mathring{p}\rbr\ra \lbr \mathring{p}\rbr$ is injective,
	\item \emph{inert} if the restriction $\alpha\colon \alpha^{-1}\lbr \mathring{p}\rbr\ra \lbr \mathring{p}\rbr$ is bijective,
	\item \emph{Segal} if it agrees with $\rho_i'\colon\lbr q\rbr \ra \lbr 1\rbr$ for some $1\le i\le q$ where $\rho_i'(j)=L$ if $j<i$, $\rho_i'(j)=1$ if $j=i$, and  $\rho_i'(j)=R$ if $j>i$.
\end{enumerate}

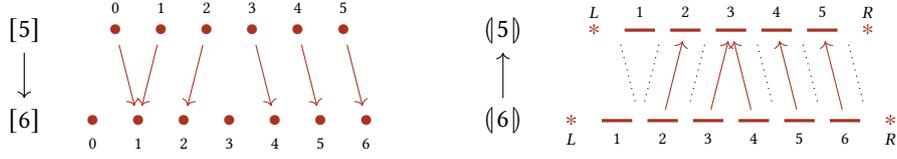
\begin{figure}
	\begin{tikzpicture}[scale=.6]
		
		\begin{scope}
			\node at (-2,0) {$[5]$};
			\node at (-2,-2) {$[6]$};
			\draw[->,shorten >=.3cm,shorten <=.3cm] (-2,0) -- (-2,-2);
			
			\foreach \i in {0,...,5}
			{
				\node [Mahogany] at (\i,0) {$\bullet$};
				\node at (\i,.5) {\tiny $\i$};
			}
			\foreach \i in {0,...,6}
			{
				\node [Mahogany] at ({\i-.5},-2) {$\bullet$};
				\node at ({\i-.5},-2.5) {\tiny $\i$};
			}
			\draw[Mahogany,->,shorten >=.2cm,shorten <=.2cm] (0,0) -- ({1-.5},-2);
			\draw[Mahogany,->,shorten >=.2cm,shorten <=.2cm] (1,0) -- ({1-.5},-2);
			\draw[Mahogany,->,shorten >=.2cm,shorten <=.2cm] (2,0) -- ({2-.5},-2);
			\draw[Mahogany,->,shorten >=.2cm,shorten <=.2cm] (3,0) -- ({4-.5},-2);
			\draw[Mahogany,->,shorten >=.2cm,shorten <=.2cm] (4,0) -- ({5-.5},-2);
			\draw[Mahogany,->,shorten >=.2cm,shorten <=.2cm] (5,0) -- ({6-.5},-2);
		\end{scope}
		
		\begin{scope}[xshift=11cm]
			\node at (-2.5,0) {$\lbr 5 \rbr$};
			\node at (-2.5,-2) {$\lbr 6 \rbr$};
			\draw[<-,shorten >=.3cm,shorten <=.3cm] (-2.5,0) -- (-2.5,-2);
			
			\node [Mahogany] at (-.5,0) {$\ast$};
			\node at (-.5,.4) {\tiny $L$};
			\node [Mahogany] at (5.5,0) {$\ast$};
			\node at (5.5,.4) {\tiny $R$};
			\foreach \i in {1,...,5}
			{
				\draw[Mahogany,very thick,shorten >=.1cm,shorten <=.1cm] ({\i-1},0) -- (\i,0);
				\node at ({\i-.5},.4) {\tiny $\i$};
			}
			
			\node [Mahogany] at (-1,-2) {$\ast$};
			\node at (-1,-2.4) {\tiny $L$};
			\node [Mahogany] at (6,-2) {$\ast$};
			\node at (6,-2.4) {\tiny $R$};
			\foreach \i in {1,...,6}
			{
				\draw[Mahogany,very thick,shorten >=.1cm,shorten <=.1cm] ({\i-1.5},-2) -- ({\i-.5},-2);
				\node at ({\i-1},-2.4) {\tiny $\i$};
			}
			\draw[dotted,shorten >=.2cm,shorten <=.2cm] (0,0) -- ({1-.5},-2);
			\draw[dotted,shorten >=.2cm,shorten <=.2cm] (1,0) -- ({1-.5},-2);
			\draw[Mahogany,<-,shorten >=.15cm,shorten <=.15cm] (1.5,0) -- (1,-2);
			\draw[dotted,shorten >=.2cm,shorten <=.2cm] (2,0) -- ({2-.5},-2);
			\draw[Mahogany,<-,shorten >=.15cm,shorten <=.15cm] (2.5,0) -- (2,-2);
			\draw[Mahogany,<-,shorten >=.15cm,shorten <=.15cm] (2.5,0) -- (3,-2);
			\draw[dotted,shorten >=.2cm,shorten <=.2cm] (3,0) -- ({4-.5},-2);
			\draw[Mahogany,<-,shorten >=.15cm,shorten <=.15cm] (3.5,0) -- (4,-2);
			\draw[dotted,shorten >=.2cm,shorten <=.2cm] (4,0) -- ({5-.5},-2);
			\draw[Mahogany,<-,shorten >=.15cm,shorten <=.15cm] (4.5,0) -- (5,-2);
			\draw[dotted,shorten >=.2cm,shorten <=.2cm] (5,0) -- ({6-.5},-2);
		\end{scope}
	\end{tikzpicture}
	\caption{The isomorphism \eqref{equ:gap-iso} between $\Delta^\op$ (on the left) and $\Gap$ (on the right, but we omitted the elements that map to $L$ or $R$). The morphism indicated is not active, cellular, inert, or Segal.}
	\label{figure:delta-to-gap}
\end{figure}

\begin{rem}\,
\begin{enumerate}
\item We think of $i \in \lbr p \rbr$ as the ``gap'' between $i-1$ and $i$ in $[p]$, and observe that $\alpha \colon [p] \to [q]$ induces a map the other way between these gaps; see \cref{figure:delta-to-gap} for an example.
\item The functor $c\colon \Delta^\op\ra \Gap$ is related to the functor $\mathrm{Cut} \colon \Delta^\op \to \Assoc^\otimes$ of \cite[4.1.2.9]{LurieHA}: the pointed set $\mathrm{Cut}([n]) \cong \langle n \rangle$ can be obtained from the set $c([n])$ by identifying $L$ and $R$.\end{enumerate}
\end{rem}

The three 1-categories $\Fin_*$, $\Delta$, and $\Gap$ are related by a sequence of functors 
\begin{equation}\label{equ:delta-to-fin}
	\Delta^{\op} \lra\Gap \lra \Fin_*
\end{equation}
where the first arrow is the isomorphism \eqref{equ:gap-iso}, and the second arrow is obtained by identifying the left and right ends $L$ and $R$ of objects in $\Gap$ and forgetting that morphisms are order-preserving.

\subsection{Category and monoid objects}\label{sec:cat-objects} Fix an $\infty$-category $\icat{C}$ with finite limits. 

\subsubsection{Category objects and monoid objects}
A \emph{category object} in $\icat{C}$ is a simplicial object $X\in \Fun(\Delta^{\op},\icat{C})$ satisfying the \emph{Segal condition}, i.e.\,the map 
		\begin{equation}\label{equ:segal-maps}
			X_{[p]} \lra X_{[1]} \times_{X_{[0]}} \cdots \times_{X_{[0]}} X_{[1]}
		\end{equation}
		induced by the Segal maps $\rho_i\colon [1] \to [p]$ for $1\le i\le p$, is an equivalence for all $p\ge0$. We call $X_{[1]}$ the \emph{underlying object} of $X$. A \emph{monoid object} is a category object $X$ for which the map $X_{[0]}\ra *$ to the terminal object is an equivalence; equivalently it is a simplicial object for which the analogues of the maps \eqref{equ:segal-maps} with pullbacks replaced by products are equivalences. We write 
\[
	\gls*{caticat} \subset \Fun(\Delta^{\op},\icat{C})\quad\text{and}\quad
	\gls*{monicat} \subset \Fun(\Delta^{\op},\icat{C})
\] 
for the full subcategories of category objects and monoid objects. Replacing simplicial by semisimplicial objects in this definition yields the categories 
\[
	\gls*{catnuicat} \subset \Fun(\Delta_\inj^{\op},\icat{C})\quad\text{and}\quad
	\gls*{monnuicat} \subset \Fun(\Delta_\inj^{\op},\icat{C})
\]
of \emph{non-unital category objects} and \emph{non-unital monoid objects}. 

\subsubsection{Commutative monoid objects} 
We may replace the role of the category $\Delta^{\op}$ in the definition of a monoid object with $\Fin_*$ to arrive at the notion of a \emph{commutative monoid object}: a functor $X\in\Fun(\Fin_*,\cC)$ for which the maps $X_{\langle p \rangle}\ra X_{\langle 1 \rangle}\times\ldots\times X_{\langle 1 \rangle}$ induced by the Segal maps $\rho_i\colon \langle p \rangle\ra \langle 1 \rangle$ for $1\le i\le p$ are equivalences for all $p\ge0$. These span the full subcategory
\[
	\gls*{cmonnuicat}\subset\Fun(\Fin_*,\cC)
\] 
of commutative monoid objects. Precomposition with the composition $\Delta^\op \to \Fin_*$ of \eqref{equ:delta-to-fin} induces a functor $\CMon(\cC)\ra\Mon(\cC)$ that ``forgets commutativity''.

\begin{rem}\label{rem:comm-mon-iterative}
There is a different perspective on commutative monoid objects in the form of an equivalence of $\infty$-categories $\Mon_{\infty}(\cC)\simeq \CMon(\cC)$ where $\Mon_{\infty}(\cC)$ is the limit in $\CatInf$ \[\Mon_{\infty}(\cC)\simeq \lim\big(\cdots \ra \Mon(\Mon(\Mon(\cC)))\ra \Mon(\Mon(\cC))\ra \Mon(\cC)\ra \cC\big)\]  over the maps induced $\ev_{[1]}\colon \Mon(\cC)\ra \cC$ (combine \cite[Proposition 10.11]{HaugsengSpans} with \cite[5.1.1.5, 2.4.2.5]{LurieHA}). In particular, there is an equivalence $\CMon(\Mon(\cC))\simeq \CMon(\cC)$.
\end{rem}

\subsubsection{Monoidal categories and double categories}\label{sec:monoidal-cats}
For $\icat{C}=\CatInf$, (non-unital) monoid objects in $\icat{C}$ are also called \emph{(non-unital) monoidal $\infty$-categories}, (non-unital) category objects in $\icat{C}$ are called \emph{(non-unital) double $\infty$-categories}, and (commutative) monoid objects in $\CatInf$ or $\Cat(\CatInf)$ are \emph{(symmetric) monoidal $\infty$- or double $\infty$-categories}. Via the straightening equivalence of \cref{sec:straightening}, these can be described equivalently as cocartesian fibrations $\cM\ra\Delta^{\op}$ (or $\cM\ra\Delta^{\op}_\inj$ in the non-unital case, or $\cM\ra\Fin_*$ in the commutative case) such that the functors
\[
	\cM_{[p]}\lra \cM_{[1]}\times \ldots \times \cM_{[1]}\quad\text{respectively}\quad \cM_{[p]}\lra \cM_{[1]}\times_{\cM_{[0]}} \ldots\times_{\cM_{[0]}} \cM_{[1]}
\]
induced by the cocartesian lifts of the Segal maps $\rho_i$ are equivalences.

\begin{ex}\label{ex:cartesian-structure}For an $\infty$-category $\cC$ with finite products, taking products induces a symmetric monoidal structure $\cC^{\times}\ra\Fin_*$ on $\cC$, the \emph{cartesian structure} \cite[2.4.1]{LurieHA}. Dually, if $\cC$ has finite coproducts, it carries a \emph{cocartesian} symmetric monoidal structure $\cC^{\sqcup}\ra\Fin_*$ \cite[2.4.3]{LurieHA}.
\end{ex}

\begin{rem}\label{rem:lurie-model-oo-cat}The definition of a monoidal $\infty$-category given in \cite[4.1.1.10]{LurieHA} is different from the one given above, but the resulting $\infty$-categories turn out to be equivalent \cite[4.1.3]{LurieHA}.
\end{rem}

\subsubsection{Mapping $\infty$-categories}\label{sec:mapping-infinity-category}
Given a double $\infty$-category $\cC\in\Cat(\CatInf)$ and objects $A,B\in\cC_{[0]}$, we define the \emph{mapping $\infty$-category} from $A$ to $B$ to be the $\infty$-category given as the fibre in $\CatInf$
\[
	\gls*{cabmappingcat} \coloneqq \fib_{(A,B)}\big((d_0,d_1)\colon \cC_{[1]}\ra \cC_{[0]} \times \cC_{[0]}\big).
\]
These mapping $\infty$-categories come with composition functors $\cC_{A,B}\times \cC_{B,C}\ra\cC_{A,C}$ defined by taking vertical fibres in the commutative diagram in $\CatInf$ 
\[\begin{tikzcd}[row sep=0.4cm]
	\cC_{[1]} \times_{\cC[0]} \cC_{[1]} \dar & \lar[swap]{\simeq} \rar \cC_{[2]} \dar & \cC_{[1]} \dar \\
	\cC_{[0]} \times \cC_{[0]} \times \cC_{[0]}  & \lar[equal] \cC_{[0]} \times \cC_{[0]} \times \cC_{[0]} \rar{\pr_{1,3}} & \cC_{[0]}  \times \cC_{[0]}\end{tikzcd}\]
with top-left horizontal map induced by the Segal morphisms, top-right horizontal map by the unique active morphism $[2] \to [1]$, and vertical map by the face maps.

\subsubsection{Quasi-unital monoid and category objects}\label{sec:quasi-unital}
A non-unital category object $X\in\Cat_{\nonunital}(\cC)$ is \emph{quasi-unital} if it admits a \emph{quasi-unit}, which is by definition a morphism $u\colon X_{[0]}\ra X_{[1]}$ together with a commutative diagram in $\cC$
\[\begin{tikzcd}[row sep=0.2cm]
	X_{[0]}\arrow[dr,"\diag",swap]\arrow[rr,"u"]&&X_{[1]}\arrow[dl,"{(d_0,d_1)}"]\\[-3pt]
	&X_{[0]}\times X_{[0]}&
\end{tikzcd}\]
such that the following two compositions are equivalent to the identity
\begin{equation}\label{equ:composition-with-qunit}
	\begin{gathered}X_{[1]} \simeq  X_{[0]}\times_{X_{[0]}}X_{[1]}\xrightarrow{(u,\id)}X_{[1]}\times_{X_{[0]}}X_{[1]}\simeq X_{[2]}\xlra{d_1}X_{[1]},\\ X_{[1]} \simeq  X_{[1]}\times_{X_{[0]}}X_{[0]}\xrightarrow{(\id,u)}X_{[1]}\times_{X_{[0]}}X_{[1]} \simeq  X_{[2]}\xlra{d_1}X_{[1]}.\end{gathered}
\end{equation}
Quasi-units are unique up to equivalence \cite[Remark 4.8]{HaugsengSegal}. A morphism $\phi\colon X\ra Y$ of non-unital category objects is \emph{quasi-unital} if there exists a commutative diagram in $\cC$ of the form
\begin{equation}\label{equ:morphism-quasiunital}
	\begin{tikzcd}[row sep=0.3cm]
		X_{[0]}\arrow[ddr,"\diag", near end,swap]\arrow[rrr,"\phi_0"]\arrow[drr,"u_X"]&&& Y_{[0]}\arrow[drr,"u_Y"] &&\\[-3pt]
		&& X_{[1]}\arrow[rrr,"\phi_1",near start]\arrow[dl,"{(d_0,d_1)}",swap]&&& Y_{[1]}\arrow[dl,"{(d_0,d_1)}"]\\[-3pt]
		& X_{[0]}\times X_{[0]}\arrow[rrr,"{(\phi_0,\phi_0)}",swap, near start]&&& Y_{[0]}\times Y_{[0]}\arrow[from=uul,"\diag", crossing over, near end, swap]&
	\end{tikzcd}
\end{equation}
such that the outer triangles are quasi-units for $X$ and $Y$. As a result of the uniqueness of quasi-units, the composition of two quasi-unital morphisms is quasi-unital. We write $	\gls*{catquc} \subset\Cat_{\nonunital}(\cC)$ for the subcategory of \emph{quasi-unital category objects} in $\cC$, generated by quasi-unital objects and morphisms. Every category object is quasi-unital ($s_0\colon X_{[0]}\ra X_{[1]}$ is a quasi-unit), and by \cite[Theorem 4.14]{HaugsengSegal}, the forgetful functor $\Cat(\cC)\ra\Cat_{\nonunital}(\cC)$ induces an equivalence 
\begin{equation}\label{equ:qu-is-good}
	\Cat(\cC)\overset{\simeq}\lra\Cat_{\quasiunital}(\cC).
\end{equation}

\begin{rem}\label{rem:quasi-unital-into-simplicial} Note that if $X$ a quasi-unital category object in $\cC$, $Y$ a simplicial object in $\cC$ (not necessarily a category object), and $f\colon X\ra Y$ a morphism of semisimplicial objects in $\cC$, then
	\begin{enumerate}
		\item\label{enum:quasi-unital-into-simplicial-i}it makes sense to ask for $f$ to be quasi-unital (replace $u_Y$ in \eqref{equ:morphism-quasiunital} by the $0$th degeneracy map). This property is preserved by postcomposition with maps of simplicial objects,
		\item\label{enum:quasi-unital-into-simplicial-ii} if $\cC=\CatInf$ and $Y'\subset Y$ is a levelwise full subcategory that is a quasi-unital category object, then a functor $X\ra Y'$ of non-unital category objects is quasi-unital if and only if the composition $X\ra Y'\subset Y$ is quasi-unital in the sense of \ref{enum:quasi-unital-into-simplicial-i}.
	\end{enumerate}
\end{rem}

\subsubsection{Double $\infty$-, $(\infty,2)$-, and $(\infty,1)$-categories} \label{sec:double-vs-infty2}
A double $\infty$-category has an underlying $(\infty,2)$-category (in fact two, but we will not need this) which in turn has an underlying $\infty$-category. More precisely, there are functors of $\infty$-categories
\vspace{-0.1cm}
\[
	\Cat(\CatInf) \xrightarrow{\gls*{infty2}} \CatInfTwo \xrightarrow{{\gls*{infty1}}} \CatInf.
\]
where $\CatInfTwo$ is the $\infty$-category of $(\infty,2)$-categories. We denote the composition by 
\[
	\gls*{simeq2} \colon \Cat(\CatInf) \lra  \CatInf.
\]
These functors have the following properties:
\begin{enumerate}
\item\label{enum:oo-2:i} The functors $(-)^{(\infty,2)}$ and $(-)^{\simeq_2}$ preserve finite products and hence (symmetric) monoidal structures, and so does their composition $(-)^{(\infty,1)}$.
\item \label{enum:oo-2:ii}For $\cC\in\Cat(\CatInf)$, the objects of $\cC^{(\infty,2)}$ can be identified with those of $\cC$. The analogous property holds for the functor $(-)^{\simeq_2}$ and thus also for their composition $(-)^{(\infty,1)}$.
\item\label{enum:oo-2:iii} For $\cC\in\Cat(\CatInf)$, the mapping $\infty$-category $\cC_{A,B}$ between two objects $A$ and $B$ in $\cC$ can be identified with the corresponding mapping $\infty$-category in $\cC^{(\infty,2)}$. The functor $(-)^{\simeq_2}$ is on mapping $\infty$-categories given by taking cores (hence the notation), and thus the same holds for $(-)^{(\infty,1)}$, so we have $\cC_{A,B}^\simeq\simeq\Map_{\cC^{(\infty,1)}}(A,B)$ for objects $A$ and $B$ in $\cC$.
\end{enumerate}
One way to implement these $\infty$-categories and functors between them is to use the equivalence $\Cat_\infty\simeq \CSS(\cS)$ to Rezk's \emph{complete Segal spaces} (a certain full subcategory of $\Cat(\cS)$ \cite[Section 3]{HaugsengSpans}) and model $\CatInfTwo$ as the $\infty$-category of \emph{$2$-fold complete Segal spaces} $\CSS_2(\cS)$ in the sense of Barwick (a certain full subcategory of $\Cat(\Cat(\cS))$ \cite[Section 4]{HaugsengSpans}). In these models, the functor $(-)^{(\infty,2)}$ is explained in \cite[Remark 3.15]{HaugsengMorita} and the functor $(-)^{(\infty,1)}$ can be constructed via the inductive description as 2-fold Segal spaces as $\CSS_2(\cS)=\CSS_{\CSS(\cS)}(\CSS(\cS))$ \cite[Section 7]{HaugsengSpans} by defining $(-)^{\simeq_2}$ as the right-adjoint $\CSS_{\CSS(\cS)}(\CSS(\cS))\ra \CSS_\cS(\cS)=\CSS(\cS)\simeq \CatInf$ induced by the right-adjoint $\ev_{[0]}\colon \CSS(\cS)\ra \cS$ to the inclusion $c\colon \cS\ra \CSS(\cS)$ as constant simplicial spaces, using \cite[Proposition 7.17]{HaugsengSpans}.

It remains to justify properties \ref{enum:oo-2:i}--\ref{enum:oo-2:iii}. That  \ref{enum:oo-2:i} holds for $(-)^{(\infty,2)}$ is justified in \cite[Remark 3.15]{HaugsengMorita} and for $(-)^{\simeq_2}$ it holds since it is a right adjoint. For \ref{enum:oo-2:ii} and \ref{enum:oo-2:iii}, one uses \cite[Lemma 5.50/5.51]{HaugsengMorita} and that $\ev_{[0]}$ corresponds to taking cores under the equivalence $\CatInf \simeq \CSS(\cS)$. 

\subsection{Presheaves and the Yoneda embedding}\label{sec:presheaf-category}\label{sec:yoneda-day}
Given an $\infty$-category $\cC$, we write $\gls*{psh} \coloneqq \Fun(\cC^{\op},\cS)$ for the $\infty$-category of $\cS$-valued presheaves on $\cC$. This admits all small limits and colimits  \cite[5.1.2.4]{LurieHTT}, and there is a natural fully faithful \emph{Yoneda embedding} \cite[5.1.3.1]{LurieHTT}
\[
	\gls*{yon} \colon \cC\longhookrightarrow \PSh(\cC).
\]
 If $\cC$ is (symmetric) monoidal, then its opposite $\cC^{\op}$ is (symmetric) monoidal \cite[2.4.2.7]{LurieHA}, and $\PSh(\cC)$ carries a (symmetric) monoidal structure by Day convolution \cite[2.2.6.17]{LurieHA} which, firstly, preserves small colimits in each variable, and, secondly, allows for an enhancement of the Yoneda embedding to a (symmetric) monoidal functor \cite[4.8.1.12, 4.8.1.13]{LurieHA}. Explicitly, a formula for Day convolution is given by
$(F \otimes G)(c'') = \colim_{c'' \to c \otimes c'} (F(c) \times G(c'))$
where the colimit is over the category of triples $(c,c',u)$ with $c,c' \in \cC$ and $u \colon c'' \to c \otimes c'$ \cite[2.2.6]{LurieHA}. Moreover, from the construction, one sees that a lax (symmetric) monoidal functor $\cC\ra \cD$ (see \cref{ex:monoidal-cat-as-operads})  induces a lax (symmetric) monoidal functor $\PSh(\cD)\ra\PSh(\cC)$ by precomposition.

\begin{rem}\label{fact:yoneda-comparison}Given a $\Kan$-enriched category $\catC$, there is a similar Yoneda embedding in $\Kan$-enriched categories $\yon_s\colon \catC\ra\Fun(\catC^{\op},\Kan)$. Taking coherent nerves and postcomposing with the map $N_{\coh}(\Fun(\catC^{\op},\Kan))\ra \Fun(N_{\coh}(\catC)^{\op},N_{\coh}(\Kan))=\PSh(N_{\coh}(\catC))$ of \cref{sec:coherent-nerve-props} \ref{enum:coherent-functor} yields a functor $N_{\coh}(\catC)\ra \PSh(N_{\coh}(\catC))$ which turns out to agree with $\yon$ up to equivalence, by the construction of $y$ for $\cC=N_{\coh}(\catC)$ in \cite[5.1.3.1]{LurieHTT}.\end{rem}

\subsection{$\infty$-operads}\label{sec:gen-infty-operads} \label{sec:infty-operads} 
Recall the following definition from \cite[2.1.1.10]{LurieHA}.
\begin{dfn} An \emph{$\infty$-operad} $\cO$ is a functor $p \colon \cO^\otimes \to \Fin_\ast$ with the following properties:
	\begin{enumerate}
		\item $\cO^\otimes$ has cocartesian lifts for inert morphisms in $\Fin_\ast$,
		\item\label{enum:operad-ii} the map
		$\sqcap (\rho_i)_!\colon\cO^\otimes_{\langle n \rangle} \rightarrow \sqcap_{i=1}^n  \cO^\otimes_{\langle 1 \rangle}$
		induced by the Segal morphisms is an equivalence,
		\item given an object $x \in \cO^\otimes_{\langle n \rangle}$ and cocartesian lifts $x \to x_i$ of the Segal morphisms $\rho_i \colon \langle n \rangle \to \langle 1 \rangle$, the  following commutative diagram in $\cS$ is cartesian
		\[\begin{tikzcd} \Map_{\cO^\otimes}(y,x) \rar \dar & \sqcap_{i=1}^n \Map_{\cO^\otimes}(y,x_i) \dar \\[-5pt]
			\Map_{\Fin_\ast}(\langle m \rangle,\langle n \rangle) \rar & \sqcap_{i=1}^n \Map_{\Fin_\ast}(\langle m \rangle,\langle 1 \rangle).\end{tikzcd}\]
	\end{enumerate}
	\end{dfn}
	A \emph{map of $\infty$-operads} is a functor over $\Fin_\ast$ that preserves cocartesian lifts over inert morphisms. Such a map $\cO^{\otimes} \to \cP^{\otimes}$ is also called an \emph{$\cO$-algebra in $\cP$}, and we write 
\[\Alg_{\cO}(\cP)\subset \Fun_{\Fin_*}(\cO^{\otimes},\cP^{\otimes})\] 
for the full subcategory of such maps. Given an $\infty$-operad $\cO$, we call the objects of $\cO^\otimes_{\langle 1 \rangle}$  the \emph{colours of $\cO$}. Given colours $x=(x_1,\ldots,x_n) \in \sqcap^n \smash{\cO^\otimes_{\langle 1 \rangle}}\simeq \smash{\cO^\otimes_{\langle n \rangle}}$ and $y \in \smash{\cO^\otimes_{\langle 1 \rangle}}$, the  \emph{space of multi-operations} is the subspace $\gls*{mul}(x;y) \subset \Map_{\cO^\otimes}(x,y)$ covering the unique active morphism $\langle n \rangle \to \langle 1 \rangle$ \cite[2.1.1.16]{LurieHA}. If $\cO^\otimes$ has a single colour $x\in \cO^\otimes_{\langle 1 \rangle}$, we abbreviate $\cO(k)\coloneq \Mul_\cO(x,\ldots,x;x)$ where $x$ appears in the domain $k$ times. These spaces of multi-operations can be composed using \emph{operadic composition maps}, denoted $\gls*{circo}$, that satisfy the axioms of a coloured operad in the classical sense up to coherent homotopies \cite[2.1.1.17]{LurieHA}. In particular, the \emph{homotopy operad} $h\cO^\otimes\ra \Fin_*$ (which is an operad as a result of the properties of $h$ discussed in \cref{sec:scat-vs-qcat} and satisfies $\Mul_{h\cO}(x,y)=\pi_0\,\Mul_{\cO}(x,y)$) gives a coloured operad in the classical sense. By construction, there is a map of $\infty$-operads $\cO^\otimes\ra h\cO^\otimes$. 

\begin{ex}\label{ex:monoidal-cat-as-operads}When viewed as a cocartesian fibration $\cC^{\otimes}\ra\Fin_\ast$ (see \cref{sec:monoidal-cats}), every symmetric monoidal category $\cC$ is an $\infty$-operad. A map of $\infty$-operads between symmetric monoidal categories is called a \emph{lax symmetric monoidal functor}.
\end{ex}

\begin{ex}\label{ex:associative-operad}
Every coloured operad in the category of $\Kan$-complexes in the classical sense gives rise to an $\infty$-operad via the \emph{operadic nerve} \cite[2.1.1.27]{LurieHTT}. For example, the \emph{associative $\infty$-operad} $\Assoc$ \cite[4.1.1.1, 4.1.1.3]{LurieHA} is the operadic nerve of the ordinary operad with a single colour $\ast$, whose $k$-ary multi-operations $\Assoc(k) = \Mul_{\Assoc}(*,\ldots,*;*)$ is the set of linear orders of $\unl{k} = \{1,2,\ldots,k\}$, and where operadic composition is concatenation of linear orders. An $\infty$-operad $\cO$ is equivalent to $\Assoc$ if and only if there is an isomorphism $h\cO \cong h\Assoc$ of operads in the $1$-category of sets and all spaces of operations in $\cO$ are homotopy discrete.
\end{ex}

\subsubsection{Suboperads, endomorphism operads, and algebras over them}\label{sec:suboperad}\label{sec:end-operads} \label{sec:map-as-algebra} 
Let $\cO$ be an $\infty$-operad and $\cat{O}_0 \subseteq h\cO$ be a suboperad of the ordinary operad $h\cO$ in sets. The corresponding \emph{suboperad} $\cO \times_{h \cO} \cat{O}_0$ of $\cO$ is defined as the pullback $\cO^\otimes \times_{h\cO^\otimes} \cat{O}_0^\otimes\ra \Fin_\ast$ in the $\infty$-category $\OpdInf$ of $\infty$-operads, which has limits by \cite[2.1.4]{LurieHA}. In particular, we may restrict $\cO$ to a fixed collection of colours closed under equivalences to obtain a new $\infty$-operad. We call this a \emph{full suboperad}. 

\begin{rem}
The forgetful functor $\OpdInf \to (\CatInf)_{/\Fin_\ast}$ creates limits by \cite[Lemma 1.13]{AyalaFrancisTanakaFact}, so the pullback $\cO \times_{h \cO} \cat{O}_0$ can be computes in $\CatInf$.
\end{rem}

\begin{ex}\label{ex:sub-sym-monoidal}
For a symmetric $\infty$-monoidal category $\cC$ viewed as an $\infty$-operad, its homotopy operad $h\cC$ is a symmetric monoidal $1$-category in the classical sense. Given a sub symmetric monoidal category of $\cat{C}_0\subset h\cC$ in the $1$-categorical sense, the associated sub $\infty$-operad $\cC \times_{h \cC} \cat{C}_0$ is again a symmetric $\infty$-monoidal category. Informally, this is given by restricting the objects and the components of the mapping spaces according to $\cat{C}_0$.
\end{ex}

Fix $\cC$ a symmetric monoidal $\infty$-category $\cC$,  viewed as an $\infty$-operad. The \emph{endomorphism operad} of an object $x$ in $\cC$ is the full sub $\infty$-operad \gls*{endc} obtained by restricting to the colours equivalent to $x$. Writing $\mathbbm{1}$ for the unit in $\cC$, we can form the composition of maps of $\infty$-operads
\begin{equation}\label{equ:algebra-over-end}\End_\cC(x)^{\otimes} \xlra{\subset} \cC^\otimes \overset{\yon}\lra \PSh(\cC)^{\otimes} \overset{\ev_\mathbbm{1}}\lra \cS^\times\end{equation}
to $\cS$ equipped with the cartesian symmetric monoidal structure (see \cref{ex:cartesian-structure}). The first map is induced by the inclusion, the second map the symmetric monoidal Yoneda embedding (see \cref{sec:yoneda-day}), and the third map the evaluation at the unit which is a map of $\infty$-operads by naturality of the Day convolution in lax symmetric monoidal functors (see \cref{sec:yoneda-day}). The composition \eqref{equ:algebra-over-end}  enhances the mapping space $\Map_\cC(\mathbbm{1},x)$ to an $\End_\cC(x)$-algebra in $\cS$.

\subsubsection{Generalised $\infty$-operads}The condition \ref{enum:operad-ii} in the definition of an $\infty$-operad $\cO$ in particular implies that $\smash{\cO^{\otimes}_{\langle0\rangle}}$ is trivial. Sometimes it it useful to relax the notion of an $\infty$-operad to that of a \emph{generalised $\infty$-operad} which need no longer satisfy $\smash{\cO^{\otimes}_{\langle0\rangle}}\simeq\ast$. The precise definition of a generalised $\infty$-operad is not important for us, but it suffices to know that it is a functor $\cO^{\otimes}\ra \Fin_\ast$ satisfying some weaker axioms than those for $\infty$-operads, but that the existence of cocartesian lifts for inert morphisms is still required. \emph{Maps of generalised operads} $\cO\ra\cP$ are defined in the same way as for $\infty$-operads. Generalising the case of $\infty$-operads, we denote the resulting subcategory by $\Alg_{\cO}(\cP)\subset\Fun_{\Fin_\ast}(\cO^{\otimes},\cP^{\otimes})$ and still call its objects $\cO$-algebras in $\cP$.

\subsubsection{(Generalised) nonsymmetric $\infty$-operads}\label{sec:gen-operads}
Replacing the category $\Fin_*$ by $\Delta^\op$ defines \emph{nonsymmetric} variants of all of the above definitions and constructions, e.g.\,(generalised) nonsymmetric operads, maps between them, algebras in them, etc. We use the same notation for the symmetric and nonsymmetric constructions, e.g.\,for (generalised) nonsymmetric $\infty$-operads $\cO$ and $\cP$, we write $\Alg_{\cO}(\cP)\subset \Fun_{\Delta^{\op}}(\cO^{\otimes},\cP^{\otimes})$ for the $\infty$-category of maps of (generalised) nonsymmetric $\infty$-operads aka $\cO$-algebras in $\cP$. 

\begin{ex}\label{example:gen-ns-operad}\label{example:maps-gen-ns-operad}

	The following examples of generalised nonsymmetric $\infty$-operads will be important:
	\begin{enumerate}
		\item Cocartesian fibrations obtained by unstraightening double $\infty$-categories.
		\item The projection $\Delta^{\op}_{/[p]}\ra \Delta^{\op}$ for all $p\ge0$, see \cite[Lemma 4.10]{HaugsengMorita}.
		\item The restriction $\Lambda^{\op}_{/[p]}\ra \Delta^{\op}$ of the projection $\Delta^{\op}_{/[p]}\ra \Delta^{\op}$ to the full subcategory $\gls*{lambda}\subset \Delta_{/[p]}$ spanned by the cellular maps in $\Delta$, see \cite[Lemma 4.14]{HaugsengMorita}.
	\end{enumerate}
	Examples of maps between generalised nonsymmetric $\infty$-operads that will be important are:
	\begin{enumerate}
		\item The map $\Delta^\op_{/[p]} \to \Delta^\op_{/[q]}$ over $\Delta^\op$ induced by a morphism $[p] \to [q]$ of $\Delta$.
		\item The inclusion $\Lambda^\op_{/[p]} \to \Delta^\op_{[p]}$ over $\Delta^\op$ \cite[Lemma 4.14]{HaugsengMorita}.
	\end{enumerate}
\end{ex}

\subsection{Associative algebras and bimodules in the nonsymmetric setting}\label{sec:assalg-bimodules}
Given a monoidal $\infty$-category viewed as a cocartesian fibration $\cC^{\otimes}\ra \Delta^{\op}$ with underlying category $\cC\coloneqq \cC^{\otimes}_{[1]}$, the $\infty$-categories $\gls*{assc}$ and $\gls*{bimodc}$ of \emph{associative algebras in $\cC$} and \emph{bimodules in $\cC$} are defined as
\[
	\Ass(\cC)\coloneqq \Alg_{\Delta^{\op}}(\cC^{\otimes})
	\quad\text{and}\quad
	\Bimod(\cC)\coloneqq \Alg_{\Delta^{\op}_{/[1]}}(\cC^{\otimes}).
\]
These are the $\infty$-categories of $\Delta^{\op}$- and $\Delta_{/[1]}^{\op}$-algebras in $\cC$ as in \cref{sec:gen-operads}. There is a functor
\begin{equation}\label{equ:bimodule-underlying-objects}\Bimod(\cC)\lra \Ass(\cC)\times \cC\times \Ass(\cC)\end{equation} 
consisting of the projections to $\Ass(\cC)$ induced by precomposition with the functors $ \Delta=\Delta_{/[0]}\ra \Delta_{/[1]}$ induced by the $0$th and $1$st face map $[0]\ra[1]$, and the functor to $\smash{\cC^{\otimes}_{[1]}}=\cC$ given by evaluation at $\id_{[1]}\in\Delta_{/[1]}$. The fibre in $\CatInf$
\[\Bimod_{A,B}(\cC)\coloneqq \fib_{(A,B)}\big(\Bimod(\cC)\ra \Ass(\cC)\times \Ass(\cC)\big)\]
at $(A,B)$ for associative algebras $A,B\in \Ass(\cC)$ of the postcomposition $\Bimod(\cC)\ra \Ass(\cC)\times \Ass(\cC)$ of \eqref{equ:bimodule-underlying-objects} with the projection is the \emph{$\infty$-category of $(A,B)$-bimodules}.

\begin{rem}\label{fact:ass-vs-mon}
Associative algebras are closely related to monoid objects in the sense of \cref{sec:cat-objects}: for a category $\cC$ with finite products, equipped with the cartesian monoidal structure (see \cref{ex:cartesian-structure}), we have an equivalence of $\infty$-categories $\Ass(\cC^\times)\simeq\Mon(\cC)$ \cite[2.4.2.5]{LurieHA}.
\end{rem}

\begin{rem}\label{rem:lurie-bimodules} Lurie uses different models for the $\infty$-categories of associative algebras and bimodules in a monoidal $\infty$-category $\cC$ (using the equivalent point of view on monoidal structures mentioned in \cref{rem:lurie-model-oo-cat}), but these turn out to be equivalent to $\Ass(\cC)$ and $\Bimod(\cC)$ as defined above. For $\Ass(\cC)$ this is proved as \cite[4.1.3.19]{LurieHA} and for $\Bimod(\cC)$ it follows from an extension of that argument, or from \cref{rem:other-model-morita} below.
\end{rem}

The following lemma on free $(A,B)$-bimodules will be important later:

\begin{lem}\label{lemma:free-modules}For a monoidal $\infty$-category $\cC$ and associative algebras $A,B\in\Ass(\cC)$, the forgetful functor  $\gls*{uab}\colon \Bimod_{A,B}(\cC)\ra \cC$ given as the composition of the inclusion into $\Bimod(\cC)$ followed by \eqref{equ:bimodule-underlying-objects} and the projection to $\cC$ has the following properties:
	\begin{enumerate}
		\item\label{enum:free-modules-ii} For a fixed $\infty$-category $I$ such that $\cC$ admits all $I$-indexed colimits, the functor $U_{A,B}$ preserves and detects $I$-indexed colimits. The same holds for limits instead of colimits.
		\item\label{enum:free-modules-iii} The functor $U_{A,B}$ reflects equivalences.
		\item\label{enum:free-modules-i} The functor $U_{A,B}$ has a left-adjoint $\gls*{fab}\colon \cC\ra \Bimod_{A,B}(\cC)$ whose unit $M\ra U_{A,B}F_{A,B}(M)$ for $M\in \cC$ agrees with the map $M\ra A\otimes M\otimes B$ given by tensoring with the units of $A$ and $B$.
		\item\label{enum:free-modules-iv} For a functor $\varphi\colon \cC\ra\cD$ of monoidal $\infty$-categories and $M\in\cC$, the canonical morphism $F_{\varphi(A),\varphi(B)}(\varphi(M))\ra \varphi(F_{A,B}(M))$ is an equivalence.
	\end{enumerate}
\end{lem}
\begin{proof}Using \ref{rem:lurie-bimodules}, the first part follows from \cite[4.3.3.3, 4.3.3.9]{LurieHA}. The remaining items follow from \cite[Corollary 4.49]{HaugsengMorita}: The final part of this corollary in particular shows \ref{enum:free-modules-iii} since right adjoints in monadic adjunctions reflect equivalences \cite[4.7.3.5]{LurieHA} and the first part shows \ref{enum:free-modules-i}. This leaves \ref{enum:free-modules-iv}. As a result of \ref{enum:free-modules-iii} it suffices to show that 
	\[U_{\varphi(A),\varphi(B)}F_{\varphi(A),\varphi(B)}(\varphi(M))\lra U_{\varphi(A),\varphi(B)}(\varphi(F_{A,B}(M)))\simeq\varphi(U_{A,B}F_{A,B}(M))\] is an equivalence. Using the second part of \ref{enum:free-modules-i} this follows from the monoidality of $\varphi$.
\end{proof}

\subsection{Haugseng's Morita category}\label{sec:haugseng-morita} In analogy with the classical Morita category of a ring, for a sufficiently nice monoidal $\infty$-category $\cC$ one would expect a double $\infty$-category $\ALG(\cC)$---the \emph{Morita category} of $\cC$---whose $\infty$-category of objects $\ALG(\cC)_{[0]}$ is the $\infty$-category of associative algebras $\Ass(\cC)$, whose $\infty$-category of morphisms $\ALG(\cC)_{[1]}$ is the category of bimodules $\Bimod(\cC)$, and whose composition is given by ``tensoring bimodules''. Haugseng constructed such a Morita category in \cite{HaugsengMorita} (denoted $\ALG_1(\cC)$ therein) under mild assumptions on $\cC$. In what follows, we recall his construction and establish some properties not explicitly stated.

\subsubsection{The pre-Morita category}\label{sec:pre-morita}
For a monoidal $\infty$-category $\cC^{\otimes}\ra \Delta^{\op}$, the \emph{pre-Morita simplicial $\infty$-category of $\cC$} is the simplicial $\infty$-category $\gls*{prealg}\in\Fun(\Delta^\op,\CatInf)$ with
\[
	\overline{\ALG}(\cC )_{[p]}\coloneqq \Alg_{\Delta^{\op}_{/[p]}}(\cC^{\otimes})\subset \Fun_{\Delta^{\op}}(\Delta^{\op}_{/[p]},\cC^{\otimes}).
\] 
The simplicial structure is given by precomposition with the functors $\smash{\Delta_{/[p]}\ra \Delta_{/[q]}}$ induced by postcomposition with morphisms $[p]\ra [q]$ in $\Delta$; this uses \cref{example:maps-gen-ns-operad}. By construction, $\smash{\overline{\ALG}(-)}$ is natural in lax monoidal functors by postcomposition.

This definition extends the $\infty$-categories $\Ass(\cC)=\overline{\ALG}(\cC )_{[0]}$ and $\Bimod(\cC)=\overline{\ALG}(\cC )_{[1]}$ to a simplicial $\infty$-category $\smash{\overline{\ALG}(\cC)}$, but the result is not yet a double $\infty$-category: an object in $\smash{\overline{\ALG}(\cC)_{[p]}}$ gives associative algebras $M(i)$ for $0 \leq i \leq p$ and $(M(i),M(j))$-bimodules $M(i,j)$ for $0 \leq i<j \leq p$ and, informally speaking, we need to enforce that $M(i,j)$ is equivalent to the iterated tensor product $M(i,i+1) \otimes_{M(i+1)} M(i+1,i+2) \otimes_{M(i+2)} \cdots \otimes_{M(j-1)} M(j-1,j)$.

\subsubsection{Composite algebras and the Morita category}\label{sec:composite-algebras}
The condition on the $M(i,j)$ just mentioned can be made precise through the notion of a \emph{composite algebra} from \cite[Section 4.2]{HaugsengMorita}. This requires an assumption on $\cC$ Haugseng calls \emph{having good relative tensor products} \cite[Definition 4.18]{HaugsengMorita}, which is in particular satisfied if the underlying category $\cC$ admits all geometric realisations (colimits indexed over $\Delta^{\op}$) and if they are preserved by tensoring (on either side) with fixed objects of $\cC$; this follows from \cite[Lemma 4.19]{HaugsengMorita}. If $\cC$ has good relative tensor products, then the functor induced by the inclusion $\tau_p\colon \Lambda^\op_{/[p]}\hookrightarrow \Delta^{\op}_{/[p]}$ (see \cref{example:maps-gen-ns-operad})
\[\smash{\tau_p^* \colon \overline{\ALG}(\cC )_{[p]}=\Alg_{\Delta^{\op}_{/[p]}}(\cC^{\otimes})\lra \Alg_{\Lambda^{\op}_{/[p]}}(\cC^{\otimes})}\]
 admits a fully faithful left adjoint
\vspace{-0.2cm}
\[\Alg_{\Lambda^{\op}_{/[p]}}(\cC^{\otimes})\xlra{\tau_{p,!}}  \Alg_{\Delta^{\op}_{/[p]}}(\cC^{\otimes})=\overline{\ALG}(\cC )_p\] by \cite[Corollary 4.20]{HaugsengMorita}, and $M\in\overline{\ALG}(\cC )_{[p]}$ is called \emph{composite} if $M$ is in the essential image of $\tau_{p,!}$, or equivalently if the counit $\tau_{p,!}\tau_p^*M\ra M$ is an equivalence \cite[Definition 4.21]{HaugsengMorita}. By \cite[Corollary 4.38]{HaugsengMorita}, the simplicial structure on the pre-Morita category restricts to a simplicial structure on the full subcategories $\smash{\ALG(\cC)_{[p]}\subset \overline{\ALG}(\cC )_{[p]}}$ of composite objects, and by \cite[Theorem 4.39]{HaugsengMorita}, the result is a double $\infty$-category---the \emph{Morita double $\infty$-category of $\cC$} \[
	\gls*{alg}\in\Cat(\CatInf)\subseteq\Fun(\Delta^\op,\CatInf).
\] 
Note that $\Lambda_{/[p]}=\Delta_{/[p]}$ for $p=0,1$ so $\ALG(\cC)_{[p]}\subseteq \overline{\ALG}(\cC)_{[p]}$ is an equality for $p=0,1$, i.e.\,
\begin{equation}\label{equ:small-simplices-morita}
	\overline{\ALG}(\cC)_{[0]}=\ALG(\cC)_{[0]}=\Ass(\cC)\quad\text{and}\quad  \overline{\ALG}(\cC)_{[1]}=\ALG(\cC)_{[1]}=\Bimod(\cC).
\end{equation}
In particular, the mapping $\infty$-categories between $A,B\in \ALG(\cC)_{[0]}$ in the notation of \cref{sec:mapping-infinity-category} are given as $\ALG(\cC)_{A,B}=\Bimod_{A,B}(\cC)$.

\begin{rem}\label{rem:other-model-morita}In \cite[4.4.3.10, 4.4.3.11]{LurieHA}, Lurie describes a Morita double $\infty$-category $\BMod(\cC)^\circledast $ for monoidal $\infty$-categories $\cC$ that admit geometric realisations which are compatible with tensoring with a fixed object on either side (so they in particular admit good relative tensor products). One advantage of Lurie's model is that it is functorial in all lax monoidal functors, whereas Haugseng's is a priori only functorial in (strong) monoidal functors that are compatible with good relative tensor products (see \cref{sec:morita-functoriality} below). However, it turns out that Haugseng's Morita double $\infty$-category $ \ALG(\cC)$ is equivalent to Lurie's $\BMod(\cC)^\circledast $; see \cite[Corollary 5.14]{HaugsengRemarkMorita}. In particular, on $0$- and $1$-simplices, this comparison shows that Lurie's and Haugseng's models for the category of associative algebras and bimodules in a monoidal $\infty$-category $\cC$ are equivalent (cf.\,\cref{rem:lurie-bimodules}), and on composition functors it shows that Lurie's and Haugseng's models for relative tensor products of bimodules are equivalent.
\end{rem}

\subsubsection{Composite algebras in terms of semisimplicial objects}\label{sec:composite-multisimplicial}
We will now reformulate the condition on an object $M\in\overline{\ALG}(\cC )_{[p]}$ to be composite in a form closer to the informal description mentioned at the end of \cref{sec:pre-morita}, resulting in a convenient criterion for an object $M \in \overline{\ALG}(\cC)_{[p]}$ to be composite. Before turning to the technical details, we describe this criterion informally. As mentioned before, the object $M$ gives associative algebras $M(i)$ for $0 \leq i \leq p$ and $(M(i),M(j))$-bimodules $M(i,j)$ for $0 \leq i<j\leq p$. For each such bimodule, there is an ``iterated bar-construction'' semisimplicial object $M(i,j)_\bullet$ in $\cC$ augmented over $M(i,j)$, with $k$-simplices  \[M(i,j)_{[k]}\simeq  M(i,i+1) \otimes M(i+1)^{\otimes k} \otimes M(i+1,i+2) \otimes M(i+2)^{\otimes k}\otimes \cdots \otimes M(j-1)^{\otimes k}\otimes M(j-1,j).\]
The criterion is then equivalent to requiring that the augmentation geometrically realises to an equivalence for all $0 \leq i<j\leq p$. In fact, for bookkeeping reasons, it is convenient to rephrase this criterion slightly: For each $q\ge0$ and each sequence $\alpha=(0\le i_0 \leq \ldots \leq i_q\le p)$ of integers, we have an augmented semisimplicial object over $M(i_0,i_1)\otimes M(i_1,i_2)\otimes \ldots \otimes M(i_{q-1},i_q)$ given by the diagonal of the $q$-fold semisimplicial object $M(i_0,i_1)_\bullet\otimes M(i_1,i_2)_\bullet\otimes \ldots \otimes M(i_{q-1},i_q)_\bullet$. The criterion is that its augmentation has to realise to an equivalence (see \cref{cor:composite-algebras-multisimplicial}).

\medskip

\noindent To make this precise, we denote by $\Delta^{\act}$ the wide subcategory of $\Delta$ given by the active maps and by $\cC^{\otimes,\act}$ the pullback of $\cC^{\otimes}\ra \Delta^{\op}$ along the inclusion $\Delta^{\act,\op}\ra \Delta^{\op}$. The unique active maps $[1]\ra [p]$ define a natural transformation from the inclusion $\Delta^{\act,\op}\ra \Delta^\op$ to the constant functor at $[1]\in\Delta^\op$, which we can precompose with the projection $\cC^{\otimes,\act}\ra \Delta^{\act,\op}$. Taking a cocartesian pushforward (see \cref{sec:straightening}) of the canonical map $\cC^{\otimes,\act}\ra \cC^{\otimes}$ along this natural transformation gives a functor
\begin{equation}\label{equ:universal-active-pushforward}
	(-)_!\colon \cC^{\otimes,\act}\lra \cC_{[1]}^{\otimes}=\cC.
\end{equation}
Now write $\Delta^{\act}_{/[p]}$ for the wide subcategory of $\Delta_{/[p]}$ of those morphisms that map to $\Delta^{\act}$ under the projection and $	\Lambda^{\act,\op}_{/[p]}\subset \Delta^{\act,\op}_{/[p]}$ for the full subcategory of cellular maps. For $M\in\overline{\ALG}(\cC)_{[p]}\subset\Fun_{\Delta^\op}(\Delta^\op_{/[p]},\cC^{\otimes})$ and an object $\alpha\colon [q]\ra[p]$ of $\Delta^{\act,\op}_{/[p]}$, we consider the composition 
\begin{equation}\label{equ:kan-extension-diagram}
	((\Lambda^{\act,\op}_{/[p]})_{/\alpha})^{\rhd}\xra{\can} (\Delta^{\act,\op}_{/[p]})_{/\alpha}\xra{\pr}\Delta^{\act,\op}_{/[p]}\xra{M}\cC^{\act,\otimes}\xra{(-)_!}\cC
\end{equation}
where $\pr$ is the projection, $\gls*{rhd}$ is the right-cone (this freely adds a terminal object and can be modelled by the join $(-) \ast \Delta^0$), and $\can$ is the extension of the inclusion $\smash{(\Lambda^{\act,\op}_{/[p]})_{/\alpha} \subset (\Delta^{\act,\op}_{/[p]})_{/\alpha}}$ to the cone $((\Lambda^{\act,\op}_{/[p]})_{/\alpha})^{\rhd}$ by sending the terminal object to $\id_\alpha$.

\begin{lem}\label{lem:composite-algebras}For a monoidal $\infty$-category $\cC$ with good relative tensor products,
	an object $M\in \overline{\ALG}(\cC )_{[p]}$ is composite if and only if  the composition \eqref{equ:kan-extension-diagram} is a colimit diagram for all $\alpha$.
\end{lem}

\begin{proof}By definition, $M$ is composite if the counit $\tau_{p,!}\tau_p^*M\ra M$ is an equivalence. Using Proposition 4.16 and Corollary A.60 of \cite{HaugsengMorita} this is equivalent to asking whether the identity $\tau_p^*M\ra \tau_p^*M$ exhibits $M$ as the operadic left Kan extension of $\tau_p^*M$ along $\tau_p$ in the sense of Definition A.56 loc.cit.. By Lemma A.53 loc.cit., this is in turn equivalent to the condition that the functor $((\Lambda^{\act,\op}_{/[p]})\times\Delta^1)\sqcup_{(\Lambda^{\act,\op}_{/[p]})\times \{0\}}(\Delta^{\act,\op}_{/[p]})\times \{0\}\ra \cC$ induced by the restriction of
	\[\smash{(\Delta^{\act,\op}_{/[p]})\times\Delta^1\xlra{\pr_1}\Delta^{\act,\op}_{/[p]}\xra{M}\cC^{\act,\otimes}\xra{(-)_!}\cC}\]
	to  $\Lambda^{\act,\op}_{/[p]}\times \Delta^1$ is a left Kan extension in the sense of \cite[4.3.3.2]{LurieHTT}. As $\Lambda^{\act,\op}_{/[p]}\subset \Delta^{\act,\op}_{/[p]}$ is a full subcategory inclusion, we can use the simpler characterisation of left Kan extensions from \cite[4.3.2.2]{LurieHTT}, which is exactly the condition of the statement. \end{proof}

Unravelling the definitions, the category $(\Lambda^{\act,\op}_{/[p]})_{/\alpha}$ for $\alpha \colon [q] \to [p]$ is the 1-category whose objects are factorisations of $\alpha$ into an active map followed by a cellular map,
\[\begin{tikzcd}[column sep=1.5cm]
	\left[q\right]\rar{\text{active}}\arrow[rr,"\alpha", bend right=20,swap]&\left[r\right]\rar{\text{cellular}}&\left[p\right].
\end{tikzcd}
\] 
Note that if $\alpha$ is active, then so must be $[r] \to [p]$. Given another such factorisation, with middle object $[r']$, a morphism from the factorisation involving $[r]$ to that involving $[r']$ is an active map $[r']\ra[r]$ that makes everything commute:
\[\begin{tikzcd}[column sep=1.5cm] &  {[r]} \arrow{rd}{\text{cellular}} \arrow{dd}[description]{\text{active}}& \\[-15pt] 
	[q] \arrow{ru}{\text{active}} \arrow{rd}[swap]{\text{active}} & & {[p]} \\[-15pt]
	& {[r']} \arrow{ru}[swap]{\text{cellular}} & \end{tikzcd}\]	
Colimits over this category---as appearing in \cref{lem:composite-algebras}---can be rephrased in terms of semisimplicial objects that are easier to handle. Making this precise involves the following construction:

\begin{construction}\label{const:rhoalpha}Let $\alpha \colon [q] \to [p]$ be an object of $\Delta^\op_{/[p]}$ considered as a sequence $(i_0\leq\ldots\leq i_q) \subseteq [p]$. Let $J \subseteq [q-1]$ be the set of indices $j$ for which $i_j < i_{j+1}$ and set $k_{\alpha} \coloneqq \sum_{j \in J} (i_{j+1}-i_j-1)$. Enumerating the indices in the interval $[i_0,i_q]$ that do \emph{not} lie in $(i_0,\ldots,i_q)$ in order as $m_1,\ldots,m_{k_{\alpha}}$, there is a functor 
\[
	\rho_\alpha \colon (\Delta^\op)^{k_{\alpha}} \lra (\Lambda^{\act,\op}_{/[p]})_{/\alpha}
\]
given as follows: writing $\textstyle{k_\alpha^{\vec{a}}\coloneqq q+\sum_{i=1}^{k_{\alpha}}(a_i+1)}$, it sends an object $([a_1],\ldots,[a_q])\in(\Delta^\op)^{k_{\alpha}}$ to
\[\begin{tikzcd}[column sep=0.8cm]
	\left[q\right]\rar{\alpha_0^{\vec{a}}}\arrow[rr,"\alpha", bend right=25,swap]&{[k_\alpha^{\vec{a}}]}\rar{\alpha_1^{\vec{a}}}&\left[p\right]
\end{tikzcd}\]
where $\alpha_1^{\vec{a}}$ is given by the weakly increasing sequence that contains $(i_0\le \ldots\le i_q)$ as well as each $m_j$ repeated $a_j+1$ times. The map $\alpha_0^{\vec{a}}$ is the unique injective map such that  $\alpha^{\vec{a}}_1\circ\alpha^{\vec{a}}_0=\alpha$.\end{construction}

\begin{ex}\label{ex:cofinal-alphas}
If $\alpha \colon [2] \to [p]$ is given by the sequence $(i\le i+2\le i+4)$ then $k_{\alpha}=2$, the map $\alpha^{\vec{a}}_1$ is given by the sequence $(i\le i+1\le \ldots\le i+1\le i+2\le i+3\le \ldots\le i+3\le i+4)$ where $i+1$ appears $a_1+1$ times and $i+3$ appears $a_2+1$ times, and the map $\alpha^{\vec{a}}_0$ is given by the inclusion of $(i\le i+2\le i+4)$ into this sequence.
\end{ex}

For later reference, we spell out how $\rho_\alpha$ translates under the isomorphism $\Delta\cong\Gap$.

\begin{construction}\label{const:rhoalpha-gap}
Let $\alpha \colon \lbr p\rbr \to \lbr q\rbr$ be an object of $\Gap_{\lbr p\rbr/}$ considered as a sequence $(i_1\leq\ldots\leq i_p)$ with $i_j\in\lbr \mathring{q}\rbr$. The quantity $k_\alpha$ is then given by $k_\alpha\coloneqq \#\{j\in\lbr \mathring{p{-}1}\rbr\mid i_j\in\lbr \mathring{q}\rbr\text{ and }i_j=i_{j+1}\}$. We enumerate this set in order by $n_1<\ldots <n_{k_\alpha}$. The functor $\rho_\alpha$ takes the form
\[
	\rho_\alpha \colon \Gap^{k_{\alpha}} \lra (\Gap^{\act,\op}_{\lbr p\rbr/})_{/\alpha}
\]
and can be described as follows: abbreviating $\textstyle{k_\alpha^{\vec{a}}\coloneqq q+\sum_{i=1}^{k_{\alpha}}(a_i+1)}$ as in the $\Delta^\op$-case, for an object $\lbr\vec{a}\rbr=(\lbr a_1\rbr,\ldots,\lbr a_q\rbr)\in\Gap^{k_{\alpha}}$, it sends $\lbr\vec{a}\rbr$ to the factorisation
\[\begin{tikzcd}[column sep=0.8cm]
	\lbr p\rbr\rar{\alpha_1^{\vec{a}}}\arrow[rr,"\alpha", bend right=25,swap]&\lbr k_\alpha^{\vec{a}}\rbr \rar{\alpha_0^{\vec{a}}}&\lbr q\rbr
\end{tikzcd}\]
where $\alpha_1^{\vec{a}}\colon \lbr p\rbr\ra \lbr k_\alpha^{\vec{a}}\rbr$ is given by the sequence obtained from the sequence $(i_1\leq\ldots\leq i_p)$ by inserting a gap of length $a_i$ between $i_{n_j}$ and $i_{n_j+1}$ for $j=1,\ldots, k_\alpha$, and $\alpha_0^{\vec{a}}\colon \lbr k_\alpha^{\vec{a}}\rbr\ra \lbr q\rbr$ is the unique surjective map such that $\alpha_0^{\vec{a}}\circ \alpha_1^{\vec{a}}=\alpha$.
\end{construction}

The following lemma is a generalisation of \cite[Lemma 4.17]{HaugsengMorita}.

\begin{lem}\label{lem:cofinal-multi-simplicial}The functor $\rho_\alpha$ from \cref{const:rhoalpha} is cofinal.
\end{lem}

\begin{proof}By \cite[Theorem 4.1.3.1]{LurieHTT} it suffices to prove that $((\Delta^{\op})^{k_{\alpha}})_{X/}$ has an initial object for all objects $X$ of $(\Lambda^{\act,\op}_{/[p]})_{/\alpha}$, i.e. for all factorisations $X$
\vspace{-0.1cm}
\[
	[q] \overset{\beta} \lra [r] \overset{\alpha'}\lra [p]
\]
of $\alpha$ into an active $\beta$ followed by a cellular $\alpha'$. The category $((\Delta^\op)^{k_{\alpha}})_{X/}$ then has objects given by active maps $\delta\colon [k_\alpha^{\vec{a}}]\ra [r]$ such that in 
\vspace{-0.1cm}
\begin{equation}\label{eqn:factorisation-rho}
	\smash{[q]\xra{\alpha_0^{\vec{a}}}[k_\alpha^{\vec{a}}]\xra{\delta} [r] \xra{\alpha'} [p]}
\end{equation}
the full composition agrees with $\alpha$, the composition of the first two arrows with $\beta$, and the composition of the final two arrows with $\alpha_1^{\vec{a}}$. The morphisms are induced by those of $(\Delta^\op)^{{k_{\alpha}}}$ via $[k_\alpha^{\vec{a}}]$. We now describe an object of this category: define the number $b_j$ by letting $b_j+1$ be the number of times $m_j$ in \cref{const:rhoalpha} appears in $\alpha'$ ($b_j\ge1$ since $\alpha'$ is cellular and $\beta$ is active). Then there is a unique map $\smash{[k_\alpha^{\vec{b}}]\ra [r]}$ that fits in a factorisation as above, and this map is active because $[q] \to [r]$ is active. For another factorisation \eqref{eqn:factorisation-rho}, one checks there is unique morphism $([a_1],\ldots,[a_{k_{\alpha}}]) \to ([b_1],\ldots,[b_{k_{\alpha}}])$ in $\Delta^{k_{\alpha}}$ that induces a morphism of factorisations, so the factorisation we described provides an initial object in $((\Delta^{\op})^{{k_{\alpha}}})_{X/}$ as wished.
\end{proof}

\begin{ex}In the case of \cref{ex:cofinal-alphas} and $X$ given by $[2] \to [6] \to [p]$ with $[6] \to [p]$ given by $(i \leq i \leq i+1 \leq i+1 \leq i+2 \leq i+3 \leq i+4)$ (which determines the active morphism $[2] \to [6]$ uniquely), the initial object in $((\Delta^{\op})^{{k_{\alpha}}})_{X/}$ is given as follows: we get $k_\alpha =2$, $([b_1],[b_2]) = ([1],[0])$, $\smash{k^{\vec{b}}_\alpha} = 5$, and $\delta \colon [5] \to [6]$ is $(0 \leq 2 \leq 3 \leq 4 \leq 5 \leq 6)$. To see it is initial, note that equivalently it is terminal among pairs $([a_1],[a_2])$ with factorisations
	$[2] \to [k^{\vec{a}}_\alpha] \to [6] \to [p]$,
	which is true since $\delta \colon [5] \to [6]$ is bijective onto those elements in $[6]$ which do not get mapped to the image of $\alpha$.
\end{ex}

We now consider the composition
\[
\begin{tikzcd}[column sep=0.6cm]
	(\Delta^\op_{\inj})^{\rhd}\rar{\inc}\arrow[rrrrr,"\eta^\alpha",bend right=10,swap]&(\Delta^\op)^\rhd\rar{\diag}&((\Delta^\op)^{k_\alpha})^\rhd \rar{\rho_\alpha}&((\Lambda^{\act,\op}_{/[p]})_{/\alpha})^\rhd\rar{\can}& (\Delta^{\act,\op}_{/[p]})_{/\alpha}\rar{\pr}&\Delta^{\act,\op}_{/[p]}
\end{tikzcd}
\]
which we abbreviate as $\eta^\alpha$; similar for $\Gap$ instead of $\Delta^\op$. Unravelling the definitions, one checks that $\eta^{\alpha}$ maps an object $[a]\in \Delta^\op_{\inj}$ to $\alpha_1^{\vec{a}}\colon [k_\alpha^{\vec{a}}]\ra [p]$ with $\vec{a}=(a,\ldots,a)$ and the cone point to $\alpha\colon [q]\ra [p]$. The unique map from $[a]$ to the cone point is mapped to $\smash{\alpha_0^{\vec{a}}\colon [q]\ra [k_\alpha^{\vec{a}}]}$. Since the inclusion $\smash{\Delta^\op_{\inj}\subset \Delta^\op}$ is cofinal \cite[4.1.1.8]{LurieHTT}, the category $\Delta^\op$ is sifted and hence the diagonal $\Delta^\op\ra(\Delta^\op)^2$ is cofinal \cite[5.5.8.1, 5.5.8.4]{LurieHTT},  the functor $\rho_\alpha$ is cofinal by the previous lemma, and cofinal functors are closed under composition \cite[4.1.1.3 (2)]{LurieHTT}, the composition $\smash{\Delta^\op_\inj\ra (\Lambda^{\act,\op}_{/[p]})_{/\alpha}}$ is cofinal. As colimits in $\infty$-categories are unaffected by precomposition with cofinal functors \cite[4.1.1.8]{LurieHTT}, we can simplify the condition in \cref{lem:composite-algebras} further to:

\begin{cor}\label{cor:composite-algebras-multisimplicial}For a monoidal $\infty$-category $\cC$ with good relative tensor products, an object $M\in \overline{\ALG}(\cC )_{[p]}$ is composite if and only if for all $\alpha \in \Delta^{\op}_{/[p]}$, the following is a colimit diagram
	\[ (\Delta_\inj^\op)^\rhd\xra{\eta^\alpha}\Delta^{\act,\op}_{/[p]}\xra{M}\cC^{\otimes,\act}\xra{(-)_!}\cC.\]
\end{cor}

\begin{ex} We spell out two exemplary cases of \cref{cor:composite-algebras-multisimplicial} and relate them to the informal description of \cref{cor:composite-algebras-multisimplicial} from the beginning of this subsection, involving the augmented semisimplicial objects $M(i,j)_\bullet$. Firstly, for $\alpha=(0\le 2) \colon [1] \to [2]$, we have $k_\alpha = 1$ and the functor $\smash{(\Delta_\inj^\op)^\rhd \to \Delta^{\act,\op}_{/[2]}}$ sends $[a]$ to the sequence $(0 \leq 1 \leq \ldots \leq 1 \leq 2)$ where $1$ appears $a+1$ times, and it sends the cone point to $(0 \leq 2)$. Applying $M$ and $(-)_!$ we obtain the augmented semisimplicial object corresponding to $M(0,2)_\bullet$. For $\alpha \colon [q] \to [p]$ given by a sequence $(i_0 \leq \cdots \leq i_q) \subseteq [p]$ so that $i_{j+1} = i_j+1$, we have $k_\alpha = 0$, so the composition in the statement is a constant augmented semisimplicial object; this fits with the informal description since $M(i,j)_\bullet$ is constant if $j=i+1$.
\end{ex} 

\subsubsection{Functoriality and monoidality}\label{sec:morita-functoriality}
Postcomposition induces a functor
\begin{equation}\label{equ:premorita-functor}
	\overline{\ALG}(-)\colon\Mon(\CatInf)\lra \Fun(\Delta^\op,\CatInf)
\end{equation}
which is on $p$-simplices given by $\Alg_{{\Delta^\op}_{/[p]}}(-)$. The latter preserves limits  \cite[p.\,1701]{HaugsengMorita}, so \eqref{equ:premorita-functor} does as well and in particular induces a functor between commutative monoid objects
\[
	\overline{\ALG}(-)\colon\CMon(\Mon(\CatInf))\lra \CMon(\Fun(\Delta^\op,\CatInf)).
\]
The situation for $\ALG(-)$ is only slightly more complicated. Let $\Mon(\CatInf)^{\grtp}\subset \Mon(\CatInf)$ the (non-full) subcategory of monoidal categories that admit good relative tensor products and functors that are compatible with them. The latter is made precise in \cite[Definition 4.18]{HaugsengMorita}, but all we need is that (i) monoidal categories admit good relative tensor products if their underlying categories admit all geometric realisations and these are compatible with tensoring with a fixed object on either side, and that (ii) functors of monoidal categories that preserve geometric realisations are compatible with good relative tensor products. Then $\ALG(-)$ gives rise to a functor $\ALG(-)\colon\Mon(\CatInf)^{\grtp}\lra \Cat(\CatInf)$ (see \cite[Corollary 5.41]{HaugsengMorita}). By \cite[Lemma 5.38 (iii)]{HaugsengMorita}, the category $\Mon(\CatInf)^{\grtp}$ admits products and these are preserved by the forgetful functor $\Mon(\CatInf)^{\grtp}\ra \Mon(\CatInf)$. Moreover, $\ALG(-)$ is product-preserving: we may test this on $p$-simplices for $p\ge0$ and since it has values in double $\infty$-categories, it suffices to check this for $p=0,1$ where it follows from \eqref{equ:small-simplices-morita} and the corresponding fact for $\overline{\ALG}(-)$.  $\ALG(-)$ thus induces a functor on commutative monoid objects:
\[\ALG(-)\colon\CMon(\Mon(\CatInf)^{\grtp})\lra \CMon(\Cat(\CatInf)).\]

\subsection{Span and cospan categories}\label{sec:span-cospan-cats} We summarise the construction of a double $\infty$-category $\COSPAN^+(\cC)$ of cospans, following \cite[Section 5]{HaugsengSpans}.

\subsubsection{Categories of (co)spans}
For an $\infty$-category $\cC$, the \emph{pre-span simplicial $\infty$-category of $\cC$} is \[\overline{\SPAN}^+(\cC)\coloneq \Fun(\Sigma^\bullet,\cC) \in \Fun(\Delta^\op,\CatInf)\]
where $\Sigma^\bullet \colon \Delta \ra \Cat$ is defined as follows: on objects it sends $[n] \in \Delta$ the poset $\Sigma^n$ of pairs $(i,j)$ with $0 \leq i \leq j \leq n$, and $(i,j) \preceq (i',j')$ if and only if $i \leq i'$ and $j' \leq j$, and on morphisms sends $\phi \colon [n] \to [m]$ to the functor $\Sigma^n \to \Sigma^m$ given by $(i,j) \mapsto (\phi(i),\phi(j))$. 

If $\cC$ has finite limits then the \emph{span double $\infty$-category of $\cC$}
\begin{equation}\label{equ:span-cat}\SPAN^+(\cC) \in \Cat(\CatInf)\end{equation}
is the levelwise full subcategory $\SPAN^+(\cC)\subset\overline{\SPAN}^+(\cC)$ of \emph{cartesian functors}, where $(\Sigma^p \to \cC)\in \overline{\SPAN}^+(\cC)_{[p]}$ is \emph{cartesian} if the natural map from it to the right Kan extension of its restriction to the full subcategory $\Lambda^p\subset \Sigma^p$ on $(i,j)$ with $j-i \leq 1$ is an equivalence.  By \cite[Proposition 5.14]{HaugsengSpans} $\SPAN^+(\cC)$ is indeed a double $\infty$-category. We have $\SPAN^+(\cC)_{[0]} = \cC$, and arguing as in \cite[Proposition 8.3]{HaugsengSpans} one sees that the mapping $\infty$-categories are $\SPAN^+(\cC)_{A,B} \simeq \cC_{/A \times B}$. 

Dually, one defines the \emph{cospan double $\infty$-category} of $\cC$ and its pre-version as
\[\overline{\COSPAN}^+(\cC^\op) \coloneqq \overline{\SPAN}^+(\cC^\op)^\op\quad\text{and}\quad\gls*{cospan} \coloneqq \SPAN^+(\cC^\op)^\op\]
where the outer $(-)^\op$ denotes taking levelwise opposites, and the second definition requires $\cC$ to have finite colimits. The mapping $\infty$-categories are then given by $\COSPAN^+(\cC)_{A,B}\simeq \cC_{A\sqcup B/}$.

\subsubsection{Relation to Morita categories}\label{sec:alg-cospans}Cospan categories and Morita categories are not unrelated: if $\cC$ has finite colimits, then it has good relative tensor products as in \cref{sec:composite-algebras} when equipped with the cocartesian symmetric monoidal structure $\cC^{\sqcup}$ \cite[Remark 2.5.14]{HaugsengMelaniSafranov}. By Corollaries 2.6.8 and 2.6.10 loc.cit.\,there is an equivalence of double $\infty$-categories
\begin{equation}\label{equ:cospan-are-alg} 
	\COSPAN^+(\cC)\simeq \ALG(\cC^{\sqcup}).
\end{equation}
Moreover, functors that preserve finite colimits are compatible with good relative tensor products, so $\COSPAN^+(\cC)$ inherits the functoriality and monoidality properties from $\ALG(\cC^{\sqcup})$ as discussed in \cref{sec:morita-functoriality} for monoidal categories with finite colimits and functors that preserve those (there is also an a priori description, but we will not need it). Tracing through the proof one sees that under the equivalence \eqref{equ:cospan-are-alg}, an object $M \in \ALG(\cC)_{[p]}$ is sent to the sequence of cospans
\[\begin{tikzcd}[column sep=.5cm,row sep=0.2cm] 
	& M(0,1) & &[5pt] \cdots &[5pt] &[-5pt] M(p{-}1,p) &[-5pt] \\[-6pt]
	M(0) \arrow{ru} & & M(1) \arrow{ru} \arrow{lu} & & M(p{-}1) \arrow{ru} \arrow{lu} & & M(p) \arrow{lu} 
\end{tikzcd}\]
where the map $M(i) \to M(i,i+1)$ is given by $M(i) \xra{\inc} M(i) \sqcup M(i,i+1) \xra{\text{act}} M(i,i+1)$, and similarly for $M(i+1) \to M(i,i+1)$. Here we used the notation from \cref{sec:pre-morita}.

\section{From the bordism to the Morita category}
\label{sec:the-functor}

As part of the introduction, we announced in \cref{sec:intr-bordism} the construction of a functor
\begin{equation}\label{equ:functor-to-morita-oo-2}
	\gls*{Efunctor} \colon \ncBordInf(d)\lra \Mod(d)
\end{equation}
of symmetric monoidal $(\infty,2)$-categories where the domain is an $(\infty,2)$-category of possibly non-compact $(d-1)$-dimensional manifolds with bordisms as $1$-morphisms and embeddings as $2$-morphisms, and the target is a Morita $(\infty,2)$-category of the symmetric monoidal $\infty$-category of presheaves on an $\infty$-category of disjoint unions of $d$-dimensional open discs with embeddings as morphisms and disjoint union as monoidal structure. What we will actually do is to construct \eqref{equ:functor-to-morita-oo-2} as a functor between symmetric monoidal double $\infty$-categories, which is more general by the discussion in \cref{sec:double-vs-infty2}. 

\medskip

\begin{center}\textit{For most of the arguments in the proofs of Theorems~\ref{bigthm:2-type-invariance}--\ref{bigthm:nontrivial}, the precise construction of \eqref{equ:functor-to-morita-oo-2} does not play a role. We summarise the key features in \cref{sec:functor-e-disc-structure}, so readers who mainly care about Theorems \ref{bigthm:2-type-invariance}--\ref{bigthm:nontrivial} may skip this technical section on a first reading.}
\end{center}

\medskip

The steps we take in this section to construct the functor \eqref{equ:functor-to-morita-oo-2}  are as follows:

\begin{enumerate}[leftmargin=1.3cm]
	\item[\ref{step:bordismcat}] Construct a non-unital double $\infty$-category $\ncBordInf(d)^{\nonunital}\in\Cat_{\nonunital}(\CatInf)$
	of possibly non-compact $(d-1)$-manifolds with embeddings and bordisms between them.
	\item[\ref{step:mand}] Construct a monoidal $\infty$-category $\ManInf_d\in \Mon(\CatInf)$ of possibly non-compact $d$-manifolds and embeddings between them, monoidal via disjoint union.
	\item[\ref{step:functor-to-premorita}]  Construct a morphism 
	$E^{\geo} \colon \ncBordInf(d)^{\nonunital}\ra \overline{\ALG}(\ManInf_d)$
	of semisimplicial $\infty$-categories to the pre-Morita category of $\ManInf_d$ from \cref{sec:haugseng-morita}, viewed as a semisimplicial object.
	\item[\ref{step:composite}]Show that the composition
	\[
		\qquad \quad \overline{E} \colon \ncBordInf(d)^{\nonunital} \ra \overline{\ALG}(\ManInf_d)\ra \overline{\ALG}(\PSh(\ManInf_d))\ra\overline{\ALG}(\PSh(\DiscInf_d))
	\]
	lands in the Morita category $\Mod(d)\coloneq \ALG(\PSh(\DiscInf_d))\subset\overline{\ALG}(\PSh(\DiscInf_d))$. The second map is induced by the Yoneda embedding and the third map by the full subcategory $\DiscInf_d\subset \ManInf_d$ on manifolds diffeomorphic to $T\times\bfR^d$ for finite sets $T$.
	\item[\ref{step:functor-to-morita}]Argue that $\ncBordInf(d)^{\nonunital}$ can be enhanced to a (unital) double $\infty$-category $\ncBordInf(d)\in\Cat(\CatInf)$, and that $\overline{E}$ can be enhanced to a functor of double $\infty$-categories as in \eqref{equ:functor-to-morita-oo-2}.
	\item[\ref{step:symmetric-monoidal-structure}] Argue that the resulting functor $E \colon \ncBordInf(d) \to \Mod(d)$ can be enhanced to a functor of \emph{symmetric monoidal} double $\infty$-categories.
\end{enumerate}

\noindent We will conclude the section with some enhancements of the bordism category $\ncBordInf(d)$:

\begin{enumerate}[leftmargin=1.3cm]
	\item[\ref{step:variants}] Construct variants $\BordInf(d)$, $\ncBordInf(d)^\partial$, and $\ncBordInf^\theta(d)$ of $\ncBordInf(d)$ by restricting to compact manifolds and diffeomorphisms instead of embeddings, allowing manifolds with boundary, and adding tangential structures.		
	\item[\ref{step:product}] Construct for a closed $p$-manifold $P$ a map of symmetric monoidal double $\infty$-categories $P \times (-) \colon \ncBordInf(d) \to \ncBordInf(d+p)$ induced by taking cartesian product with $P$, and extend this construction to the variants from \ref{step:variants}.
\end{enumerate}

\begin{rem}Some remarks on the construction of the functor \eqref{equ:functor-to-morita-oo-2}:
\begin{enumerate}
\item One may ask whether this construction can be ``fully extended''; that is, whether one can upgrade $\ncBordInf(d)$ to a symmetric monoidal $(d+1)$-fold $\infty$-category and the functor $E$ to a map of such objects with target the symmetric monoidal $(d+1)$-fold Morita $\infty$-category of $\PSh(\DiscInf_d)$ from \cite[Section 5]{HaugsengMorita}, which would in particular give a functor of symmetric monoidal $(\infty,d+1)$-categories. There are no conceptual issues in doing so, but it would involve additional bookkeeping and make our construction less transparent. Since we do not need it to prove the main results, we did not include it.
\item We construct \eqref{equ:functor-to-morita-oo-2} as a functor of symmetric monoidal double $\infty$-categories, but all later arguments only use the underlying functor of symmetric monoidal $(\infty,2)$-categories.
\item There are at least three constructions of a Morita $(\infty,2)$-category of a sufficiently nice monoidal $\infty$-category $\cC$ that for $\cC=\PSh(\DiscInf_d)$ might serve as potential targets for \eqref{equ:functor-to-morita-oo-2}:
\begin{enumerate}
\item\label{enum:lurie-model} Lurie's model $\BMod(\cC)$ from \cite[4.4.3.11]{LurieHA},
\item\label{enum:haugseng-model}  Haugseng's model $\ALG_1(\cC)$ from \cite[Section 4]{HaugsengMorita}, denoted $\ALG(\cC)$ in \cref{sec:haugseng-morita},
\item\label{enum:Scheimbauer-model}  Scheimbauer's model $\Alg_1(\cC)$ from \cite[Section 3]{Scheimbauer}.
\end{enumerate}
Haugseng's and Lurie's model are known to be equivalent (see \cref{rem:other-model-morita}). For our purposes, Haugseng's model turned out to be the most convenient choice. 
\item For some of our later arguments, it is crucial that $E$ is defined on the bordism category $\ncBordInf(d)$ that involves noncompact manifolds; the restriction to the typically considered subcategory $\BordInf(d)$ that only involves compact manifolds is not sufficient. If one is mainly interested in a functor from the compact variant $\BordInf(d)$ to a Morita category of $\PSh(\DiscInf_d)$, then there are other potential routes to a construction, e.g.\,by modifying a construction of Scheimbauer \cite{Scheimbauer} or relying on the cobordism hypothesis \cite{LurieCobordism}.
\end{enumerate}
\end{rem}

Throughout the following subsections corresponding to the steps above, we generically refer to \cref{sec:preliminaries} for a recollection of the $\infty$-categorical concepts and facts involved.

\renewcommand\thesubsection{\textbf{Step} $\circled{\arabic{subsection}}$}

\subsection{The bordism category via manifolds with walls}\label{step:bordismcat}
We will construct the non-unital double $\infty$-category $\ncBordInf(d)^{\nonunital}\in\Cat_{\nonunital}(\CatInf)$ as the levelwise coherent nerve of a semisimplicial object in $\Kan$-enriched categories 
\begin{equation}\label{equ:bordism-cat-simplicial}
	\gls*{ncbordnu} \in \Fun(\Delta^{\op}_\inj,\sCat).
\end{equation}

\begin{convention}\label{conv:epsilon-conventions}
	Throughout this section, we fix a constant $0<\epsilon<\tfrac{1}{2}$. We write $\tr_\lambda\colon\bfR\ra\bfR$ for the translation by $\lambda\in\bfR$. For a subset $W\subset \bfR\times\bfR^\infty$, we write \[W|_A\coloneqq W\cap(A\times\bfR^\infty)\subset \bfR\times\bfR^\infty\] for subsets $A\subset\bfR$. If $A=\{a\}$ is a singleton, we abbreviate $W|_a\coloneqq W|_{\{a\}}$.
\end{convention}

\begin{figure}
	\begin{tikzpicture}[scale=.9]
	\draw [dotted] (-6,0) -- (7,0);
	\node at (6.5,-.3) {\tiny $\bfR$};
	\draw (-5,0) -- (6,0);
	\node at (-3,0) {$\bullet$};
	\draw [thick,|-|] (-3.5,0) -- (-2.5,0);
	\draw [dashed] (-3,4.5) -- (-3,-1);
	\node at (-3,-.3) [fill=white] {\tiny $\mu(0)$};
	\node at (-1,0) {$\bullet$};
	\draw [thick,|-|] (-1.5,0) -- (-0.5,0);
	\draw [dashed] (-1,4.5) -- (-1,-1);
	\node at (-1,-.3) [fill=white] {\tiny $\mu(1)$};
	\node at (2,0) {$\bullet$};
	\draw [thick,|-|] (1.5,0) -- (2.5,0);
	\draw [dashed] (2,4.5) -- (2,-1);
	\node at (2,-.3) [fill=white] {\tiny $\mu(2)$};
	\node at (5,0) {$\bullet$};
	\draw [thick,|-|] (4.5,0) -- (5.5,0);
	\draw [dashed] (5,4.5) -- (5,-1);
	\node at (5,-.3) [fill=white] {\tiny $\mu(3)$};
	
	\draw [dotted,Mahogany,thick] (-5.5,3.5) -- (-5,3);
	\draw [dotted,Mahogany,thick] (-5.5,2.5) -- (-5,2);
	\draw [dotted,Mahogany,thick] (-5.5,4.5) -- (-5,4);
	\draw [dotted,Mahogany,thick] (6,1) -- (6.5,1);
	\draw [Mahogany,thick](-5,3) to[out=-45,in=180] (-3.5,2) to[out=0,in=180] (-2.5,2) to[out=0,in=90] (-2,1.5) to[out=-90,in=0] (-2.5,1) to[out=180,in=0] (-3.5,1) to[out=180,in=-45] (-5,2);
	\draw [Mahogany,thick] (-5,4) to[out=-45,in=180] (-3.5,3) to[out=0,in=180] (-.25,3) to[out=0,in=180] (1.5,1) -- (6,1);
	\draw [Mahogany,thick] (3.5,2.5) circle (.75cm);
	\node [Mahogany] at (0,3.5) {$W$};
	\end{tikzpicture}
	\caption{A $[3]$-walled $1$-manifold. The vertical projection is $\mu$ and the intervals in $\bfR$ are the $[\mu(i)-\epsilon,\mu(i)+\epsilon]$'s, which are disjoint in accordance with \ref{enum:mfd-wall-i}. The dashed lines indicate the hyperplanes $\{\mu(i)\} \times \bfR^\infty$. Note that $W$ is transverse to these and a product near them, as imposed in \ref{enum:mfd-wall-ii} and \ref{enum:mfd-wall-iii}.}
	\label{fig:mfd-with-walls}
\end{figure}
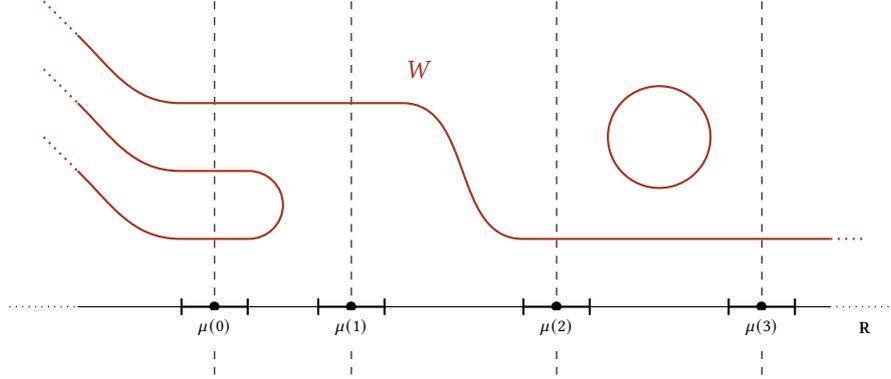

We first set up some language. A \emph{$[p]$-walled $d$-manifold} for $[p]\in\Delta$ is  a pair $\gls*{wmu}$ of a $d$-dimensional smooth submanifold $W\subset\bfR\times \bfR^\infty$ without boundary and an order-preserving map $\mu\colon [p]\ra \bfR$ such that the following is satisfied
\begin{enumerate}
	\item \label{enum:mfd-wall-i}$\mu(i)+\epsilon<\mu(i+1)-\epsilon$ for all $i$,
	\item  \label{enum:mfd-wall-ii}the projection $\pr\colon W\ra \bfR$ to the first coordinate is transverse to $\mu\colon [p]\ra \bfR$,
	\item \label{enum:mfd-wall-iii} $W|_{[\mu(i)-\epsilon,\mu(i)+\epsilon]}=\tr_{\mu(i)}[-\epsilon,+\epsilon]\times W|_{\mu(i)}$ for all $i$;
\end{enumerate}
see \cref{fig:mfd-with-walls} for an example. The space 
\[
	\Emb\big((W,\mu),(W',\mu')\big)\subset \Emb\big(W|_{[\mu(0)-\epsilon,\mu(p)+\epsilon]}, W'|_{[\mu'(0)-\epsilon,\mu'(p)+\epsilon])}\big)
\] 
of \emph{embeddings between $[p]$-walled $d$-manifolds} $(W,\mu)$ and $(W,\mu')$ is the subspace of those embeddings $\varphi$ that satisfy the following properties for all $i$:
\begin{enumerate}
	\item $\varphi^{-1}(W'|_{[\mu'(i)+\epsilon,\mu'(i+1)-\epsilon]})=W|_{[\mu(i)+\epsilon,\mu(i+1)-\epsilon]}$ and $\varphi^{-1}(W'|_{[\mu'(i)-\epsilon,\mu'(i)+\epsilon]})=W|_{[\mu(i)-\epsilon,\mu(i)+\epsilon]}$ 
	\item the embedding $\varphi$ restricts on $W|_{[\mu(i)-\epsilon,\mu(i)+\epsilon]}$ to an embedding of the form 
	\[(\tr_{\mu'(i)-\mu(i)}\times\varphi_i)\colon{\tr_{\mu(i)}[-\epsilon,+\epsilon]}\times W|_{\mu(i)}\longhookrightarrow \tr_{\mu'(i)}[-\epsilon,+\epsilon]\times W'|_{\mu'(i)}
	\] for some embedding $\varphi_i\in\Emb(W|_{\mu(i)}, W'|_{\mu'(i)})$.
\end{enumerate} 
Using this terminology, $\ncBord(d)^\nonunital$ is defined as the semisimplicial $\Kan$-enriched category whose $\Kan$-enriched category $\smash{\ncBord(d)^\nonunital_{[p]}}$ of $p$-simplices has possibly non-compact $[p]$-walled $d$-manifolds $(W,\mu)$ as its objects, spaces of embeddings between $[p]$-walled $d$-manifolds as morphisms, and composition given by composition of embeddings. The semisimplicial structure is given by ``forgetting walls'', i.e.\,by precomposition of $\mu\colon [p]\ra\bfR$ with morphisms in $\Delta_\inj$.

The non-unital double $\infty$-category
	\[\gls*{ncbordinfnu}\in\Cat(\CatInf)\subset\Fun(\Delta^{\op}_\inj,\CatInf)\]
	is now defined as the levelwise coherent nerve of $\ncBord(d)^\nonunital$, i.e.\,we have $(\ncBordInf(d)^{\nonunital})_{[p]}\coloneqq N_{\coh}((\ncBord(d)^\nonunital)_{[p]})$. This implicitly claims that the semisimplicial object $\ncBordInf(d)^{\nonunital}\in \Fun(\Delta^{\op}_\inj,\CatInf)$ is indeed a double $\infty$-category, i.e.\,that it satisfies the Segal condition.

\begin{lem}\label{lem:bord-is-category-object}
$\ncBordInf(d)^{\nonunital}$ is a non-unital double $\infty$-category.
\end{lem}

\begin{proof}
This is straight-forward, so we will only sketch the proof. One first observes that the Segal maps	$\ncBord(d)^\nonunital_{[p]}\ra \ncBord(d)^\nonunital_{[1]}\times_{\ncBord(d)^\nonunital_{[0]}}\ldots \times_{\ncBord(d)^\nonunital_{[0]}}\ncBord(d)^\nonunital_{[1]}$ before taking coherent nerves are Dwyer--Kan equivalences, so weak equivalences in the Bergner model structure from \cref{sec:scat-vs-qcat} \ref{enum:bergner-model-structure}. Since the $\ncBord(d)^\nonunital_{[p]}$ are $\Kan$-enriched, they are fibrant in this model structure. Next, one shows that source and target maps $\ncBord(d)^\nonunital_{[1]}\ra \ncBord(d)^\nonunital_{[0]}$ are Kan fibrations on morphism spaces and isofibrations on homotopy categories, so they are fibrations in the model structure and the pullbacks appearing in the above maps are homotopy pullbacks. Using that the coherent nerve is the right Quillen functor in the Quillen equivalence between the Joyal and the Bergner model structure (see \cref{sec:scat-vs-qcat} \ref{enum:bergner-model-structure}) and therefore preserves homotopy pullbacks and weak equivalences between fibrant objects, it follows that the Segal maps $\ncBordInf(d)^{\nonunital}_{[p]}\ra \ncBordInf(d)^{\nonunital}_{[1]}\times_{\ncBordInf(d)^{\nonunital}_{[0]}}\ldots \times_{\ncBordInf(d)^{\nonunital}_{[0]}}\ncBordInf(d)^{\nonunital}_{[1]}$ in $\Cat_\infty$ are equivalences.
\end{proof}

\subsection{The monoidal category of manifolds and embeddings}\label{step:mand}
We construct the monoidal $\infty$-category of (possibly noncompact) $d$-manifolds and embeddings between them as a cocartesian fibration $\gls*{maninf}\ra \Delta^{\op}\cong \Gap$ obtained as the coherent nerve of a functor 
\begin{equation}\label{equ:man-fibration-simplicial}
	\gls*{mand}\lra \Gap
\end{equation}
of $\Kan$-enriched categories. Objects of $\Man_d^{\otimes}$ are pairs $(\lbr p\rbr,W)$ of $\lbr p\rbr\in\Gap$ and a smooth submanifold $W\subset \lbr \mathring{p}\rbr\times\bfR\times\bfR^\infty$ without boundary; the distinguished $\bfR$-coordinate is not necessary but comes in handy later. To define the space of morphisms, given $A\subset \lbr \mathring{p}\rbr$, we write \[W|^A\coloneqq W\cap (A\times\bfR\times \bfR^\infty)\] to distinguish it from the notation $W|_A$ for $A\subset \bfR$ from \cref{conv:epsilon-conventions}. Using this, we set
\begin{equation}\label{equ:morphisms-man}
	\textstyle{\Map_{\Man_d^\otimes}\big((\lbr p\rbr,W),(\lbr p'\rbr,W')\big)\coloneqq\bigsqcup_ {\varphi\in \Map_\Gap(\lbr p\rbr,\lbr p'\rbr)} \Emb\big(W|^{\varphi^{-1}\lbr\mathring{p}'\rbr}, W')_\varphi}
\end{equation}
where the subscript $(-)_\varphi$ indicates that we restrict to embeddings that cover $\varphi$, i.e.\,that make 
\begin{center}
	\begin{tikzcd}
		\varphi^{-1}\lbr \mathring{p}'\rbr \times \bfR \times \bfR^\infty \supset W|^{\varphi^{-1}\lbr\mathring{p}'\rbr}\arrow[r,hookrightarrow]\dar[swap]{\pr_1}&[10pt] W' \subset \lbr \mathring{p}'\rbr \times \bfR \times \bfR^\infty \dar{\pr_1}\\
		\varphi^{-1}\lbr\mathring{p}'\rbr\rar{\varphi}&\lbr\mathring{p}'\rbr
	\end{tikzcd}
\end{center}
commute. The composition in $\Man_d^\otimes$ is induced by the composition in $\Gap$ and composition of embeddings. The functor \eqref{equ:man-fibration-simplicial} sends $(\lbr p \rbr,W)$ to $\lbr p \rbr$. Taking coherent nerves defines $\ManInf_d^{\otimes}\ra \Gap\cong\Delta^\op$ which one easily checks to be a monoidal $\infty$-category using the description in terms of cocartesian fibrations from \cref{sec:monoidal-cats}.

\begin{rem}By construction, the underlying category $\ManInf_d$ of the monoidal $\infty$-category $\ManInf_d^{\otimes}\ra \Delta^\op$ (the fibre over $[1]\in \Delta^\op$) agrees with the coherent nerve of the $\Kan$-enriched category whose objects are smooth submanifolds $W\subset \bfR^\infty$ without boundary and spaces of embeddings between them. Informally speaking, the monoidal structure is given by taking disjoint unions. Note that this monoidal structure is \emph{not} cocartesian, since typically $\Emb(M\sqcup N,W)\not\simeq \Emb(M,W)\times \Emb(N,W)$.
\end{rem}

\subsubsection{Cocartesian pushforward along active maps}\label{sec:pushforward-for-man}
For later reference, we spell out a model of the cocartesian pushforward from \cref{sec:composite-multisimplicial} in the case $\cC^{\otimes}=\ManInf_d^{\otimes}$
\[
(-)_!\colon \ManInf_d^{\otimes,\act}\lra(\ManInf_d^{\otimes})_{[1]}\eqcolon\ManInf_d.
\]
It is the coherent nerve of a simplicially enriched functor \begin{equation}\label{equ:active-pushforward}\Man_d^{\otimes,\act}\lra(\Man_d^{\otimes})_{[1]}\eqcolon\Man_d,\end{equation} 
defined on the pullback of $\Man^{\otimes}_d$ along the inclusion $\Gap^{\act}\ra\Gap$ of the wide subcategory of active maps as in \cref{section:delta-gap}. The functor \eqref{equ:active-pushforward} is given by taking ``taking disjoint unions'', using that the restriction to active maps in $\Man_d^{\otimes,\act}$ means precisely that the embeddings appearing in \eqref{equ:morphisms-man} are defined on the whole manifold $W$, not just on a subset depending on the maps $\varphi$. As a point-set implementation, one can model this ``disjoint unions''-functor induced by viewing a submanifold $W\subset \lbr \mathring{p}\rbr \times\bfR\times\bfR^\infty$ as a submanifold of $\bfR\times\bfR\times\bfR^\infty$ using the inclusion $\lbr \mathring{p}\rbr=\{1,\ldots,p\}\subset\bfR$ and sending a submanifold $W\subset \lbr \mathring{p}\rbr\times\bfR\times\bfR^\infty\subset \bfR\times \bfR\times\bfR^\infty$ to its image under the diffeomorphism
\begin{equation}\label{equ:flip-shift-functor}s\colon 
\bfR\times\bfR\times\bfR^\infty\xra{\mathrm{flip}\times\id_{\bfR^\infty}}\bfR\times \bfR\times \bfR^\infty\xra{\id_\bfR\times\mathrm{shift}}\bfR\times \bfR^\infty
\end{equation}
with $\mathrm{flip}(x,z)\coloneq (z,x)$ and $\mathrm{shift}(z,(z_1,z_2,\ldots))\coloneq (z,z_1,z_2,\ldots)$. Said differently, the functor \eqref{equ:active-pushforward} is a composition of functors of $\Kan$-enriched categories
\[\smash{\Man_d^{\otimes,\act}\lra \widetilde{\Man_d^{\otimes}}\lra (\Man_d^{\otimes})_{[1]}}\]
where $\smash{\widetilde{\Man_d^{\otimes}}}$ has submanifolds $W\subset \bfR\times\bfR\times\bfR^\infty$ without boundary as objects and all embeddings between them as morphisms. The second functor sends $W$ to $s(W)$ on objects and is on morphisms induced by conjugating embeddings $W\hookrightarrow W'$ with the diffeomorphisms $W\cong s(W)$ and $W'\cong s(W')$ induced by $s$. The first functor sends an object $(\lbr p\rbr,W)$ to $\smash{W\subset \lbr \mathring{p}\rbr\times\bfR\times\bfR^\infty\subset \bfR\times \bfR\times\bfR^\infty}$ and is on morphism spaces \eqref{equ:morphisms-man} induced by the inclusion $\Emb(W, W')_\varphi\subset \Emb(W, W')$ of components.

\subsection{From $\ncBordInf(d)^{\nonunital}$ to the pre-Morita category of manifolds}\label{step:functor-to-premorita}
The construction of the morphism $E^{\geo} \colon \ncBordInf(d)^{\nonunital}\ra\overline{\ALG}(\ManInf_d)$ goes via the following substeps:
\begin{enumerate}[label=(\alph*)]
	\item \label{step:functor-to-premorita-language} Set up preparatory language.
	\item \label{step:functor-to-premorita-thickening} Replace the undercategories $\Gap_{\lbr \bullet \rbr/}\ra \Gap$ by a simplicial thickening $\smash{\underline{\Gap}_{ \lbr \bullet \rbr/}\ra \Gap}$.
	\item \label{step:functor-to-premorita-simplicial}Construct a functor of semisimplicial objects in $\Kan$-enriched categories
	\begin{equation}\label{equ:functor-to-premorita-simplicial}\ncBord(d)^\nonunital_{[\bullet]}\lra \Fun_{\Gap}(\underline{\Gap}_{\lbr \bullet \rbr/ },\Man_d^{\otimes})\end{equation}
	\item\label{step:functor-to-premorita-non-unital} Argue that the resulting functor of semisimplicial $\infty$-categories
	\begin{equation}\label{equ:functor-to-premorita-non-unital}\ncBordInf(d)^{\nonunital}\lra \Fun_{\Gap}(\Gap_{\lbr \bullet \rbr/ },\ManInf_d^{\otimes})\end{equation} lands in the levelwise full subcategory $\overline{\ALG}(\ManInf_d)\subset \Fun_{\Gap}(\Gap_{\lbr  \bullet \rbr/},\ManInf_d^{\otimes})$.
\end{enumerate}

\begin{figure}
	\begin{tikzpicture}
		\draw [dotted] (-6,0) -- (6,0);
		\node at (-4.5,-.3) {\tiny $L$};
		\node at (4,-.3) {\tiny $R$};
		\node at (-.25,-.3) {\tiny $1 \in \lbr \mathring{1} \rbr$};
		\draw (-5,0) -- (5,0);
		\node at (-2.5,0) {$\bullet$};
		\draw [thick,|-|] (-3.5,0) -- (-1.5,0);
		\draw [dashed] (-2.5,4.5) -- (-2.5,-1);
		\node at (-2.5,-.3) [fill=white] {\tiny $0 \in [1]$};
		\node at (2,0) {$\bullet$};
		\draw [thick,|-|] (1,0) -- (3,0);
		\draw [dashed] (2,4.5) -- (2,-1);
		\node at (2,-.3) [fill=white] {\tiny $1 \in [1]$};
		
		\draw [dotted,Mahogany] (-5.5,3.5) -- (-5,3);
		\draw [dotted,Mahogany] (-5.5,2.5) -- (-5,2);
		\draw [dotted,Mahogany] (-5.5,4.5) -- (-5,4);
		\draw [dotted,Mahogany] (5,1) -- (5.5,1);
		\draw [Mahogany](-5,3) to[out=-45,in=180] (-3.5,2) to[out=0,in=180] (-1.5,2) to[out=0,in=90] (-1,1.5) to[out=-90,in=0] (-1.5,1) to[out=180,in=0] (-3.5,1) to[out=180,in=-45] (-5,2);
		\draw [Mahogany] (-5,4) to[out=-45,in=180] (-3.5,3) to[out=0,in=180] (-1.25,3) to[out=0,in=180] (1,1) -- (5,1);
		\draw [Mahogany] (4,2.5) circle (.75cm);
		
		\draw [very thick,densely dotted] (-2,2) -- (-1.5,2);
		\draw [very thick,densely dotted] (-2,1) -- (-1.5,1);
		\draw [very thick,densely dotted] (-2,3) -- (-1.5,3);
		\draw [very thick,densely dotted] (1,1) -- (1.5,1);
		\draw [very thick,Mahogany] (-1.5,2) to[out=0,in=90] (-1,1.5) to[out=-90,in=0] (-1.5,1);
		\draw [very thick,Mahogany] (-1.5,3) to[out=0,in=180] (-1.25,3) to[out=0,in=180] (1,1);
		\node [Mahogany] at (-0.35,.9) {$\ch(W,\mu)$};
		
		\node [Periwinkle] at (-2.5,1) {$\blacksquare$};
		\node [Periwinkle] at (-2.5,2) {$\blacksquare$};
		\node [Periwinkle] at (-2.5,3) {$\blacksquare$};
		\node [Periwinkle] at (2,1) {$\blacksquare$};
		
		\node [Periwinkle,fill=white] at (-3.5,4.5) {$\wall(W,\mu)$};
		\draw [-,white,line width=3pt,shorten >=.4cm,shorten <=.4cm] (-3.5,4.5) -- (-2.5,3);
		\draw [->,Periwinkle,shorten >=.4cm,shorten <=.4cm] (-3.5,4.5) -- (-2.5,3);
		\draw [-,white,line width=3pt,shorten >=.4cm,shorten <=.4cm] (-3,4.5) to[bend left=30] (2,1);
		\draw [->,Periwinkle,shorten >=.4cm,shorten <=.4cm] (-3,4.5) to[bend left=30] (2,1);
		
		\node [fill=white] at (2.5,4.5) {$\coll(W,\mu)$};
		\draw [-,white,line width=3pt,shorten >=.4cm,shorten <=.4cm] (2.1,4.5) to[bend right=15] (-2,3);
		\draw [->,shorten >=.5cm,shorten <=.5cm] (2.1,4.5) to[bend right=15] (-2,3);
		\draw [-,white,line width=3pt,shorten >=.4cm,shorten <=.4cm] (2.25,4.4) -- (1.25,1);
		\draw [->,shorten >=.3cm,shorten <=.3cm] (2.25,4.4) -- (1.25,1);
	\end{tikzpicture}
	\caption{A $[1]$-walled $1$-manifold. Its walls $\wall(W,\mu)$ are the $0$-manifold indicated by the squares, its chambers $\ch(W,\mu)$ are the thick region, its collars $\coll(W,\mu)$ are the dotted regions, and its thickened chambers $\tch(W,\mu)$ are the union of the chambers and the collars.}
	\label{fig:wallsetc}
\end{figure}
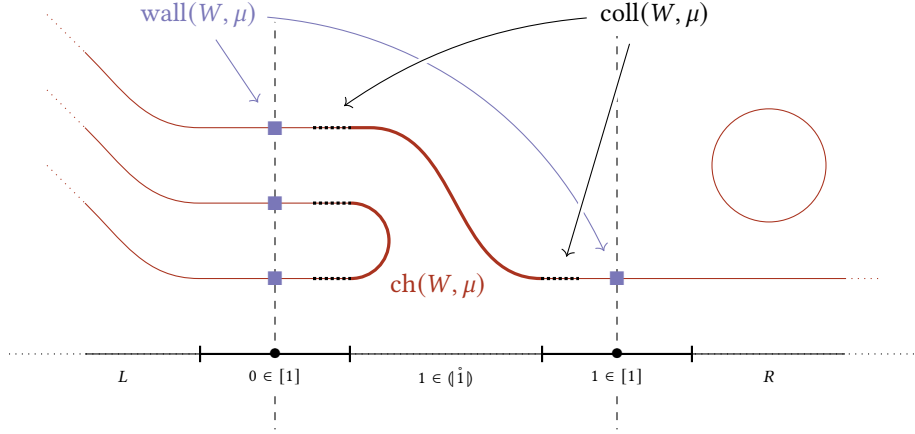

\subsubsection*{Substep \ref{step:functor-to-premorita-language} I: walls and chambers}
For a $[p]$-walled $d$-manifold $(W,\mu)$ as in \ref{step:bordismcat}, we define \[
	\gls*{wall} \subset [p]\times\bfR^\infty, \quad 
	\gls*{chamber} \subset \lbr\mathring{p}\rbr \times\bfR\times\bfR^\infty,\quad\text{and}\quad
	\gls*{thickchamber}\subset \lbr\mathring{p}\rbr \times\bfR\times\bfR^\infty,
\]
 the submanifolds of \emph{walls}, \emph{chambers}, and \emph{thickened chambers} of $(W,\mu)$, as \begin{align*}\wall(W,\mu)&\coloneqq \textstyle{\bigcup_{i\in[p]}\big(\{i\}\times W|_{\mu(i)}}\big),\\ \ch(W,\mu)&\coloneq\textstyle{ \bigcup_{i\in\lbr\mathring{p}\rbr}\big(\{i\}\times W|_{[\mu(i-1)+\epsilon,\mu(i)-\epsilon]}}\big),\\
	\tch(W,\mu)&\coloneqq \textstyle{\bigcup_{i\in\lbr\mathring{p}\rbr}\big(\{i\}\times W|_{(\mu(i-1)+\tfrac{\epsilon}{2},\mu(i)-\tfrac{\epsilon}{2})}}\big).
\end{align*} 
There is an inclusion $\ch(W,\mu)\subset \tch(W,\mu)$ whose complement of the interior we abbreviate as
\[
	\gls*{collar}\coloneqq \tch(W,\mu)\backslash\interior(\ch(W,\mu))\subset \lbr\mathring{p}\rbr \times\bfR\times\bfR^\infty.
\]
We call this the \emph{collars} of $(W,\mu)$. Informally, $\mu$ prescribes hyperplanes $\{\mu(i)\} \times \bfR^\infty$ intersecting $W$ in the walls, the (thickened) chambers are (thickened) regions between the walls, and the collars are collar neighbourhoods in the thickened chambers; see \cref{fig:wallsetc} for an example.

Given in addition a morphism $\alpha\in\Map_\Gap(\lbr p\rbr,\lbr q\rbr)$, we define the submanifold 
\[
	\gls*{lab}\subset \lbr \mathring{q}\rbr\times\bfR\times\bfR^\infty
\]
of \emph{pieces labelled by $\alpha$} as the union $\textstyle{\lab_\alpha(W,\mu)\coloneqq \bigcup_{i\in \lbr \mathring{q}\rbr}\{i\}\times W|_{(\mu(t_{i-1}^\alpha)-\epsilon,\mu(t_i^\alpha)+\epsilon)}}$
where we set $t_i^\alpha\coloneqq c^{-1}(\alpha)(i)$ using the isomorphism \eqref{equ:gap-iso} and thinking of $\lbr \mathring{q} \rbr=\{1<\ldots<q\}$ as a subset of $[q]=\{0<\ldots<q\}$. Informally, $\lab_\alpha(W,\mu)$ is the set $\lbr \mathring{q}\rbr$ labelled by chambers and thickened walls of $W$ as prescribed by $\alpha$; see \cref{fig:lab} for an example.

\begin{figure}
	\begin{tikzpicture}
			\begin{scope}[scale=.7,yshift=12cm,xshift=2cm]
				\node at (-2.5,0) {$\lbr 3 \rbr$};
				\node at (-2.5,-2) {$\lbr 5 \rbr$};
				\draw[->,shorten >=.3cm,shorten <=.3cm] (-2.5,0) -- (-2.5,-2);
				\node at (-2.8,-1) {$\alpha$};
				
				\node [Mahogany] at (0,0) {$\ast$};
				\node at (0,.4) {\tiny $L$};
				\node [Mahogany] at (4,0) {$\ast$};
				\node at (4,.4) {\tiny $R$};
				\foreach \i in {1,...,3}
				{
					\draw[Mahogany,very thick,shorten >=.1cm,shorten <=.1cm] ({\i-.5},0) -- ({\i+.5},0);
					\node at ({\i-0},.4) {\tiny $\i$};
				}
				
				\node [Mahogany] at (-1,-2) {$\ast$};
				\node at (-1,-2.4) {\tiny $L$};
				\node [Mahogany] at (5,-2) {$\ast$};
				\node at (5,-2.4) {\tiny $R$};
				\foreach \i in {1,...,5}
				{
					\draw[Mahogany,very thick,shorten >=.1cm,shorten <=.1cm] ({\i-1.5},-2) -- ({\i-.5},-2);
					\node at ({\i-1},-2.4) {\tiny $\i$};
				}
				\draw[Mahogany,->,shorten >=.15cm,shorten <=.15cm] (1,0) -- (1,-2);
				\draw[dotted,shorten >=.2cm,shorten <=.2cm] (0.5,0) -- (0.5,-2);
				\draw[Mahogany,->,shorten >=.15cm,shorten <=.15cm] (2,0) -- (1,-2);
				\draw[dotted,shorten >=.2cm,shorten <=.2cm] (2.5,0) -- (1.5,-2);
				\draw[dotted,shorten >=.2cm,shorten <=.2cm] (2.5,0) -- (2.5,-2);
				\draw[Mahogany,->,shorten >=.15cm,shorten <=.15cm] (3,0) -- (3,-2);
				\draw[dotted,shorten >=.2cm,shorten <=.2cm] (3.5,0) -- (3.5,-2);
			\end{scope}
			
			\begin{scope}[scale=.9,yshift=2cm]
			\begin{scope} 
				\clip (-4,-1) rectangle (-2,5);
				
				\draw [dotted] (-6,0) -- (6,0);
				\node at (5.5,-.3) {\tiny $\bfR$};
				\draw (-5,0) -- (5,0);
				\node at (-3,0) {$\bullet$};
				\draw [thick,|-|] (-3.5,0) -- (-2.5,0);
				\draw [dashed] (-3,3.5) -- (-3,-1);
				\node at (-3,-.3) [fill=white] {\tiny $\mu(0)$};
				\node at (-1,0) {$\bullet$};
				\draw [thick,|-|] (-1.5,0) -- (-0.5,0);
				\draw [dashed] (-1,3.5) -- (-1,-1);
				\node at (-1,-.3) [fill=white] {\tiny $\mu(1)$};
				\node at (2,0) {$\bullet$};
				\draw [thick,|-|] (1.5,0) -- (2.5,0);
				\draw [dashed] (2,3.5) -- (2,-1);
				\node at (2,-.3) [fill=white] {\tiny $\mu(2)$};
				
				\draw [(-),Mahogany,thick] (-3.5,3) -- (-2.5,3);
				\draw [(-),Mahogany,thick] (-3.5,2) -- (-2.5,2);
				\draw [(-),Mahogany,thick] (-3.5,1) -- (-2.5,1);
				
				\node [fill=white] at (-3,4.2) {$\lab_\alpha(W,\mu)|^{1}$};
			\end{scope}
			
			\begin{scope}[xshift=2.5cm]
				\clip (-3.75,-1) rectangle (2.75,5);
				
				\draw [dotted] (-6,0) -- (6,0);
				\node at (5.5,-.3) {\tiny $\bfR$};
				\draw (-5,0) -- (5,0);
				\node at (-3,0) {$\bullet$};
				\draw [thick,|-|] (-3.5,0) -- (-2.5,0);
				\draw [dashed] (-3,3.5) -- (-3,-1);
				\node at (-3,-.3) [fill=white] {\tiny $\mu(0)$};
				\node at (-1,0) {$\bullet$};
				\draw [thick,|-|] (-1.5,0) -- (-0.5,0);
				\draw [dashed] (-1,3.5) -- (-1,-1);
				\node at (-1,-.3) [fill=white] {\tiny $\mu(1)$};
				\node at (2,0) {$\bullet$};
				\draw [thick,|-|] (1.5,0) -- (2.5,0);
				\draw [dashed] (2,3.5) -- (2,-1);
				\node at (2,-.3) [fill=white] {\tiny $\mu(2)$};
				
				\draw [(-),Mahogany,thick] (-3.5,2) -- (-2.5,2) to[out=0,in=90] (-2,1.5) to[out=-90,in=0] (-2.5,1) -- (-3.5,1);
				\draw [(-),Mahogany,thick] (-3.5,3) to[out=0,in=180] (-.25,3) to[out=0,in=180] (1.5,1) -- (2.5,1);
				
				\node [fill=white] at (-.5,4.2) {$\lab_\alpha(W,\mu)|^{2}$};
			\end{scope}
			
			\begin{scope}[xshift=5cm]
				\clip (1,-1) rectangle (3,5);
				
				\draw [dotted] (-6,0) -- (6,0);
				\node at (5.5,-.3) {\tiny $\bfR$};
				\draw (-5,0) -- (5,0);
				\node at (-3,0) {$\bullet$};
				\draw [thick,|-|] (-3.5,0) -- (-2.5,0);
				\draw [dashed] (-3,3.5) -- (-3,-1);
				\node at (-3,-.3) [fill=white] {\tiny $\mu(0)$};
				\node at (-1,0) {$\bullet$};
				\draw [thick,|-|] (-1.5,0) -- (-0.5,0);
				\draw [dashed] (-1,3.5) -- (-1,-1);
				\node at (-1,-.3) [fill=white] {\tiny $\mu(1)$};
				\node at (2,0) {$\bullet$};
				\draw [thick,|-|] (1.5,0) -- (2.5,0);
				\draw [dashed] (2,3.5) -- (2,-1);
				\node at (2,-.3) [fill=white] {\tiny $\mu(2)$};
				
				\draw [(-),Mahogany,thick] (1.5,1) -- (2.5,1);
				
				\node [fill=white] at (2,4.2) {$\lab_\alpha(W,\mu)|^{3}$};
			\end{scope}
			
			\begin{scope}[xshift=-3cm,yshift=-6cm]
				\clip (1.25,-1) rectangle (5.75,5);
				
				\draw [dotted] (-6,0) -- (7,0);
				\node at (6.5,-.3) {\tiny $\bfR$};
				\draw (-5,0) -- (6,0);
				\node at (2,0) {$\bullet$};
				\draw [thick,|-|] (1.5,0) -- (2.5,0);
				\draw [dashed] (2,3.5) -- (2,-1);
				\node at (2,-.3) [fill=white] {\tiny $\mu(2)$};
				\node at (5,0) {$\bullet$};
				\draw [thick,|-|] (4.5,0) -- (5.5,0);
				\draw [dashed] (5,3.5) -- (5,-1);
				\node at (5,-.3) [fill=white] {\tiny $\mu(3)$};
				
				\draw [Mahogany,thick,(-)] (1.5,1) -- (5.5,1);
				\draw [Mahogany,thick] (3.5,2.5) circle (.75cm);
				\node [Mahogany] at (0,3.5) {$W$};
				
				\node [fill=white] at (3.5,4.2) {$\lab_\alpha(W,\mu)|^{4}$};
			\end{scope}
			
			\begin{scope}[xshift=2.5cm,yshift=-6cm]
				\clip (1,-1) rectangle (3,5);
				
				\draw [dotted] (-6,0) -- (6,0);
				\node at (5.5,-.3) {\tiny $\bfR$};
				\draw (-5,0) -- (5,0);
				\node at (2,0) {$\bullet$};
				\draw [thick,|-|] (1.5,0) -- (2.5,0);
				\draw [dashed] (2,3.5) -- (2,-1);
				\node at (2,-.3) [fill=white] {\tiny $\mu(3)$};
				
				\draw [(-),Mahogany,thick] (1.5,1) -- (2.5,1);
				
				\node [fill=white] at (2,4.2) {$\lab_\alpha(W,\mu)|^{5}$};
			\end{scope}
			\end{scope}
	\end{tikzpicture}
	\caption{Given the $[3]$-walled $1$-manifold $(W,\mu)$ of \cref{fig:mfd-with-walls} and the indicated morphism $\alpha \colon \lbr 3 \rbr \to \lbr 5 \rbr$, this shows the resulting $\lab_\alpha(W,\mu)$. The following informal description may help: $\alpha$ tells  which ``parts'' of $\lbr 3 \rbr$ to put in which ``box'' of $\lbr 5 \rbr$, and if a box is not hit by $\alpha$ it contains a ``connecting part''.}
	\label{fig:lab}
\end{figure}
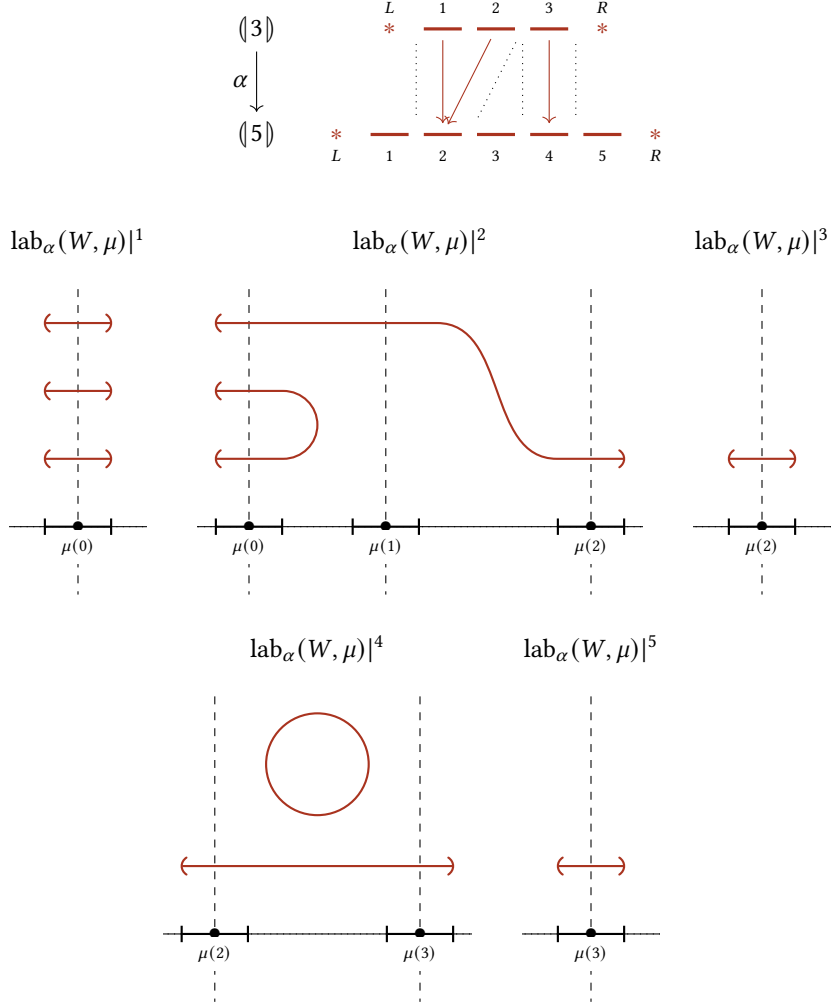

Constructing the functor \eqref{equ:functor-to-premorita-simplicial} will require us to describe embeddings out of $\lab_\alpha(W,\mu)$, for which it is helpful to decompose this manifold into two parts as follows: the map $(\alpha\times\id_{\bfR\times\bfR^\infty})\colon \lbr p\rbr\times  \bfR\times\bfR^\infty\ra\lbr q\rbr\times  \bfR\times\bfR^\infty$ restricts to an embedding 
\begin{equation}\label{equ:ch-to-lab}
	\tch(W,\mu)|^{\alpha^{-1}\lbr\mathring{q}\rbr}\longhookrightarrow \lab_\alpha(W,\mu)
\end{equation}
using which we define a submanifold
\[
	\gls*{wlab} \coloneqq \lab_\alpha(W,\mu)\backslash \interior(\ch(W)|^{\alpha^{-1}\lbr\mathring{q}\rbr})\subset \lbr\mathring{q}\rbr\times\bfR\times\bfR^\infty,
\]
of \emph{thickened walls labelled by $\alpha$}; see \cref{fig:wlab} for an example. We have a preferred decomposition
\begin{equation}\label{eq:decompose-lab}
	\lab_\alpha(W,\mu)\cong \ch(W)|^{\alpha^{-1}\lbr\mathring{q}\rbr} \cup_\partial \wlab_\alpha(W,\mu)
\end{equation}
where the gluing uses the identification $\partial(\ch(W,\mu)|^{\alpha^{-1}\lbr\mathring{q}\rbr})\cong\partial (\wlab_\alpha(W,\mu))$ induced by the restriction of \eqref{equ:ch-to-lab} to the boundary $\partial(\ch(W,\mu)|^{\alpha^{-1}\lbr\mathring{q}\rbr})$. The restriction 
\begin{equation}\label{equ:collar}
	c^\alpha_{(W,\mu)}\colon \coll(W,\mu)|^{\alpha^{-1}\lbr\mathring{q}\rbr}\longhookrightarrow \wlab_\alpha(W,\mu)
\end{equation}
of \eqref{equ:ch-to-lab} to $\coll(W,\mu)|^{\alpha^{-1}\lbr\mathring{q}\rbr}$ provides a collar of this boundary. 

\begin{figure}
	\begin{tikzpicture}[scale=.9]
		\begin{scope} 
			\clip (-4.3,-1) rectangle (-1.7,5);
			
			\draw [dotted] (-6,0) -- (6,0);
			\node at (5.5,-.3) {\tiny $\bfR$};
			\draw (-5,0) -- (5,0);
			\node at (-3,0) {$\bullet$};
			\draw [thick,|-|] (-3.5,0) -- (-2.5,0);
			\draw [dashed] (-3,3.5) -- (-3,-1);
			\node at (-3,-.3) [fill=white] {\tiny $\mu(0)$};
			\node at (-1,0) {$\bullet$};
			\draw [thick,|-|] (-1.5,0) -- (-0.5,0);
			\draw [dashed] (-1,3.5) -- (-1,-1);
			\node at (-1,-.3) [fill=white] {\tiny $\mu(1)$};
			\node at (2,0) {$\bullet$};
			\draw [thick,|-|] (1.5,0) -- (2.5,0);
			\draw [dashed] (2,3.5) -- (2,-1);
			\node at (2,-.3) [fill=white] {\tiny $\mu(2)$};
			
			\draw [(-),Mahogany,thick] (-3.5,3) -- (-2.5,3);
			\draw [(-),Mahogany,thick] (-3.5,2) -- (-2.5,2);
			\draw [(-),Mahogany,thick] (-3.5,1) -- (-2.5,1);
			
			\node [fill=white] at (-3,4.2) {$\wlab_\alpha(W,\mu)|^{1}$};
		\end{scope}
		
		\begin{scope}[xshift=2.5cm]
			\clip (-3.75,-1) rectangle (2.75,5);
			
			\draw [dotted] (-6,0) -- (6,0);
			\node at (5.5,-.3) {\tiny $\bfR$};
			\draw (-5,0) -- (5,0);
			\node at (-3,0) {$\bullet$};
			\draw [thick,|-|] (-3.5,0) -- (-2.5,0);
			\draw [dashed] (-3,3.5) -- (-3,-1);
			\node at (-3,-.3) [fill=white] {\tiny $\mu(0)$};
			\node at (-1,0) {$\bullet$};
			\draw [thick,|-|] (-1.5,0) -- (-0.5,0);
			\draw [dashed] (-1,3.5) -- (-1,-1);
			\node at (-1,-.3) [fill=white] {\tiny $\mu(1)$};
			\node at (2,0) {$\bullet$};
			\draw [thick,|-|] (1.5,0) -- (2.5,0);
			\draw [dashed] (2,3.5) -- (2,-1);
			\node at (2,-.3) [fill=white] {\tiny $\mu(2)$};
			
			\draw [(-|,Mahogany,thick] (-3.5,2) -- (-2.5,2);
			\draw [(-|,Mahogany,thick] (-3.5,1) -- (-2.5,1);
			\draw [(-|,Mahogany,thick] (-3.5,3) -- (-2.5,3);
			\draw [|-|,Mahogany,thick] (-1.5,3) -- (-.5,3);
			\draw [|-),Mahogany,thick] (1.5,1) -- (2.5,1);
			
			\node [fill=white] at (-.5,4.2) {$\wlab_\alpha(W,\mu)|^{2}$};
		\end{scope}
		
		\begin{scope}[xshift=5cm]
			\clip (.7,-1) rectangle (3.3,5);
			
			\draw [dotted] (-6,0) -- (6,0);
			\node at (5.5,-.3) {\tiny $\bfR$};
			\draw (-5,0) -- (5,0);
			\node at (-3,0) {$\bullet$};
			\draw [thick,|-|] (-3.5,0) -- (-2.5,0);
			\draw [dashed] (-3,3.5) -- (-3,-1);
			\node at (-3,-.3) [fill=white] {\tiny $\mu(0)$};
			\node at (-1,0) {$\bullet$};
			\draw [thick,|-|] (-1.5,0) -- (-0.5,0);
			\draw [dashed] (-1,3.5) -- (-1,-1);
			\node at (-1,-.3) [fill=white] {\tiny $\mu(1)$};
			\node at (2,0) {$\bullet$};
			\draw [thick,|-|] (1.5,0) -- (2.5,0);
			\draw [dashed] (2,3.5) -- (2,-1);
			\node at (2,-.3) [fill=white] {\tiny $\mu(2)$};
			
			\draw [(-),Mahogany,thick] (1.5,1) -- (2.5,1);
			
			\node [fill=white] at (2,4.2) {$\wlab_\alpha(W,\mu)|^{3}$};
		\end{scope}
		
		\begin{scope}[xshift=-3cm,yshift=-6cm]
			\clip (1.25,-1) rectangle (5.75,5);
			
			\draw [dotted] (-6,0) -- (7,0);
			\node at (6.5,-.3) {\tiny $\bfR$};
			\draw (-5,0) -- (6,0);
			\node at (2,0) {$\bullet$};
			\draw [thick,|-|] (1.5,0) -- (2.5,0);
			\draw [dashed] (2,3.5) -- (2,-1);
			\node at (2,-.3) [fill=white] {\tiny $\mu(2)$};
			\node at (5,0) {$\bullet$};
			\draw [thick,|-|] (4.5,0) -- (5.5,0);
			\draw [dashed] (5,3.5) -- (5,-1);
			\node at (5,-.3) [fill=white] {\tiny $\mu(3)$};
			
			\draw [Mahogany,thick,(-|] (1.5,1) -- (2.5,1);
			\draw [Mahogany,thick,|-)] (4.5,1) -- (5.5,1);
			
			\node [fill=white] at (3.5,4.2) {$\wlab_\alpha(W,\mu)|^{4}$};
		\end{scope}
		
		\begin{scope}[xshift=2.5cm,yshift=-6cm]
			\clip (.7,-1) rectangle (3.3,5);
			
			\draw [dotted] (-6,0) -- (6,0);
			\node at (5.5,-.3) {\tiny $\bfR$};
			\draw (-5,0) -- (5,0);
			\node at (2,0) {$\bullet$};
			\draw [thick,|-|] (1.5,0) -- (2.5,0);
			\draw [dashed] (2,3.5) -- (2,-1);
			\node at (2,-.3) [fill=white] {\tiny $\mu(3)$};
			
			\draw [(-),Mahogany,thick] (1.5,1) -- (2.5,1);
			
			\node [fill=white] at (2,4.2) {$\wlab_\alpha(W,\mu)|^{5}$};
		\end{scope}
	\end{tikzpicture}
	\caption{The submanifold $\wlab_\alpha(W,\mu)$ for $\lab_\alpha(W,\mu)$ as in \cref{fig:lab}.}	\label{fig:wlab}\end{figure}
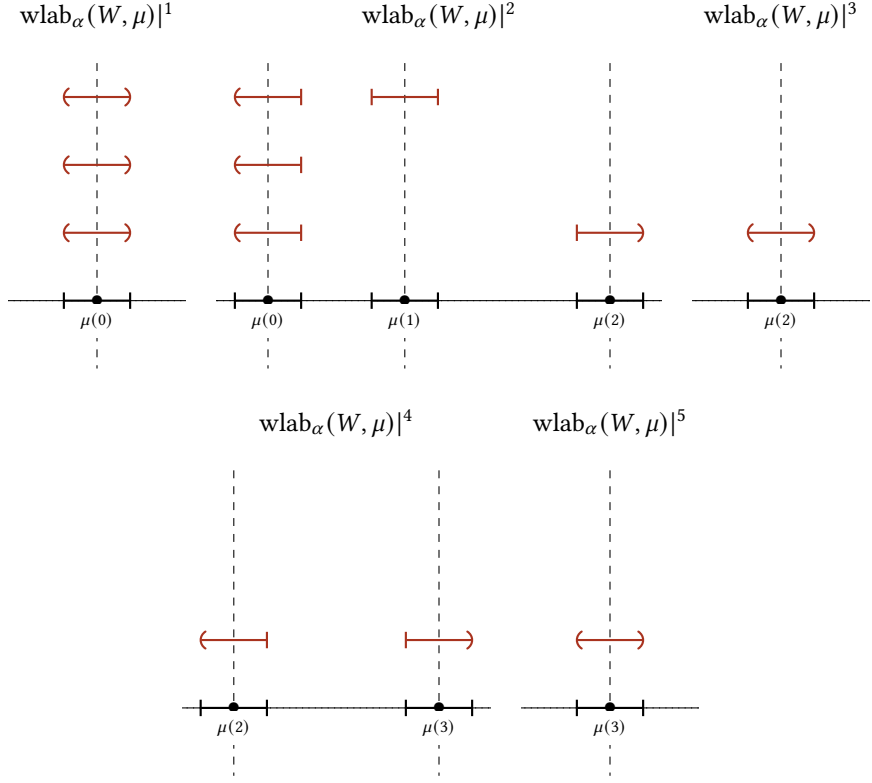

\subsubsection*{Substep \ref{step:functor-to-premorita-language} II: $\wlab_\alpha(-)$ as a pullback}
Unwrapping the definitions, one sees that $\wlab_\alpha(W,\mu)\subset \lbr \mathring{q}\rbr\times\bfR\times\bfR^\infty$ is a disjoint union of products of $W|_{\mu(i)}$ for some $i$ with a (open, half-open, or closed) interval of length $2\cdot \epsilon$. More precisely, for $i\in\lbr\mathring{q}\rbr$ the components $\wlab_\alpha(W,\mu)|^{\{i\}}$ of $\wlab_\alpha(W,\mu)$ lying over $i$ are
$\tr_{\mu(t^\alpha_i)}(-\epsilon,+\epsilon)\times W|_{\mu(t_i^\alpha)}$ for $i\not\in\im(\alpha)$ and 
\[
	\textstyle{\Big(\tr_{\mu(t^\alpha_{i-1})}(-\epsilon,+\epsilon]\times W|_{\mu(t_{i-1}^\alpha)}\Big)\cup \Big(\bigcup_{t_{i-1}^\alpha<j<t_{i}^\alpha}(\tr_{\mu(j)}[-\epsilon,\epsilon]\times W|_{\mu(j)})\Big)\cup \Big(\tr_{\mu(t^\alpha_i) }[-\epsilon,+\epsilon)\times W|_{\mu(t_{i}^\alpha)}\Big)}
\] 
for $i\in\im(\alpha)$. From this description we see in particular that there are preferred maps \[\wlab_\alpha(W,\mu)\ra [p],\quad \wlab_\alpha(W,\mu)\ra \wall(W,\mu),\quad\text{and}\quad \wlab_\alpha(W,\mu)\ra \wlab_\alpha(\bfR,\mu)\]
where we view $(\bfR,\mu)$ as a $[p]$-walled $1$-manifold. Indeed, the first map is given by sending the components of $\wlab_\alpha(W,\mu)$ whose first factor is an interval around $\mu(t_i^{\alpha})\in \bfR$ to $t_i^{\alpha}\in [p]$, the second map is induced by the first map and the projection to $\bfR^\infty$, and the final map is given by the projection to $\lbr \mathring{q}\rbr\times\bfR$. In particular, this exhibits $\wlab_\alpha(W,\mu)$ as the pullback 
\begin{equation}\label{equ:wlab-is-pullback}
	\wlab_\alpha(W,\mu)=\wlab_\alpha(\bfR,\mu)\times_{[p]}\wall(W,\mu),
\end{equation} 
which will be useful to construct embeddings out of $\wlab_\alpha(W,\mu)$; see \cref{fig:wall-as-pullback} for an example. 

It will also be useful to observe that $\wlab_\alpha(\bfR,\mu)$ is related to $\wlab_\alpha(\bfR,\mu')$ for possibly different $\mu'\colon [p]\ra\bfR$ by a preferred diffeomorphism
\begin{equation}\label{equ:wlab-independent-of-walls}
	\wlab_\alpha(\bfR,\mu)\cong \wlab_\alpha(\bfR,\mu'),
\end{equation}
uniquely characterised by requiring it to (i) preserve the order induced by the lexicographical order on $\lbr \mathring{q}\rbr\times\bfR$ and (ii) agree with translation on each component. For convenience we fix a particular choice of $\mu$, namely the inclusion $[p]=\{0,1,\ldots,p\}\subset\bfR$ in which case we omit $\mu$ from the notation, so for instance we abbreviate $\wlab_\alpha(\bfR)=\wlab_\alpha(\bfR,\inc)$.

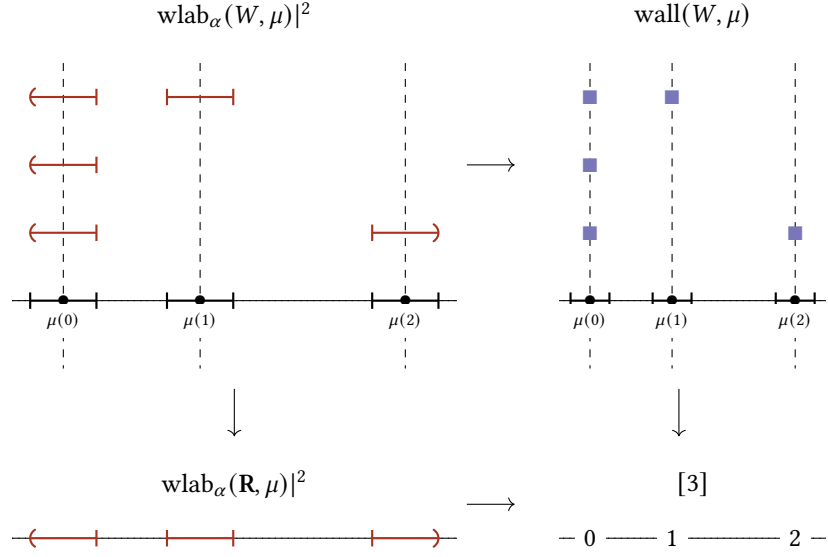
\begin{figure}
	\begin{tikzpicture}[scale=.9]
		\begin{scope}[xshift=2.5cm]
			\clip (-3.75,-1) rectangle (2.75,5);
			
			\draw [dotted] (-6,0) -- (6,0);
			\node at (5.5,-.3) {\tiny $\bfR$};
			\draw (-5,0) -- (5,0);
			\node at (-3,0) {$\bullet$};
			\draw [thick,|-|] (-3.5,0) -- (-2.5,0);
			\draw [dashed] (-3,3.5) -- (-3,-1);
			\node at (-3,-.3) [fill=white] {\tiny $\mu(0)$};
			\node at (-1,0) {$\bullet$};
			\draw [thick,|-|] (-1.5,0) -- (-0.5,0);
			\draw [dashed] (-1,3.5) -- (-1,-1);
			\node at (-1,-.3) [fill=white] {\tiny $\mu(1)$};
			\node at (2,0) {$\bullet$};
			\draw [thick,|-|] (1.5,0) -- (2.5,0);
			\draw [dashed] (2,3.5) -- (2,-1);
			\node at (2,-.3) [fill=white] {\tiny $\mu(2)$};
			
			\draw [(-|,Mahogany,thick] (-3.5,2) -- (-2.5,2);
			\draw [(-|,Mahogany,thick] (-3.5,1) -- (-2.5,1);
			\draw [(-|,Mahogany,thick] (-3.5,3) -- (-2.5,3);
			\draw [|-|,Mahogany,thick] (-1.5,3) -- (-.5,3);
			\draw [|-),Mahogany,thick] (1.5,1) -- (2.5,1);
			
			\node [fill=white] at (-.5,4.2) {$\wlab_\alpha(W,\mu)|^{2}$};
		\end{scope}
		
		\draw [->] (2,-1.3) -- (2,-2);
		\draw [->] (8.5,-1.3) -- (8.5,-2);
		\draw [->] (5.4,2) -- (6.1,2);
		\draw [->] (5.4,-3) -- (6.1,-3);
		
		\begin{scope}[xshift=2.5cm,yshift=-3.5cm]
			\clip (-3.75,-1) rectangle (2.75,2);
			
			\draw [dotted] (-6,0) -- (6,0);
			\node at (5.5,-.3) {\tiny $\bfR$};
			\draw (-5,0) -- (5,0);
			
			\draw [(-|,Mahogany,thick] (-3.5,0) -- (-2.5,0);
			\draw [|-|,Mahogany,thick] (-1.5,0) -- (-.5,0);
			\draw [|-),Mahogany,thick] (1.5,0) -- (2.5,0);
			
			\node [fill=white] at (-.5,.8) {$\wlab_\alpha(\bfR,\mu)|^{2}$};
		\end{scope}
		
		\begin{scope}[xshift=9cm,yshift=-3.5cm,xscale=.6]
			\clip (-3.75,-1) rectangle (2.75,2);
			
			\draw [dotted] (-6,0) -- (6,0);
			\node at (5.5,-.3) {\tiny $\bfR$};
			\draw (-5,0) -- (5,0);
			
			\node at (-3,0) [fill=white] {$0$};
			\node at (-1,0) [fill=white] {$1$};
			\node at (2,0) [fill=white] {$2$};
			
			\node [fill=white] at (-.5,.8) {$[3]$};
		\end{scope}
		
		\begin{scope}[xshift=9cm,xscale=.6]
			\clip (-3.75,-1) rectangle (2.75,5);
			
			\draw [dotted] (-6,0) -- (6,0);
			\node at (5.5,-.3) {\tiny $\bfR$};
			\draw (-5,0) -- (5,0);
			\node at (-3,0) {$\bullet$};
			\draw [thick,|-|] (-3.5,0) -- (-2.5,0);
			\draw [dashed] (-3,3.5) -- (-3,-1);
			\node at (-3,-.3) [fill=white] {\tiny $\mu(0)$};
			\node at (-1,0) {$\bullet$};
			\draw [thick,|-|] (-1.5,0) -- (-0.5,0);
			\draw [dashed] (-1,3.5) -- (-1,-1);
			\node at (-1,-.3) [fill=white] {\tiny $\mu(1)$};
			\node at (2,0) {$\bullet$};
			\draw [thick,|-|] (1.5,0) -- (2.5,0);
			\draw [dashed] (2,3.5) -- (2,-1);
			\node at (2,-.3) [fill=white] {\tiny $\mu(2)$};
			
			\node at (-3,2) [Periwinkle] {$\blacksquare$};
			\node at (-3,1) [Periwinkle] {$\blacksquare$};
			\node at (-3,3) [Periwinkle]{$\blacksquare$};
			\node at (-1,3) [Periwinkle]{$\blacksquare$};
			\node at (2,1) [Periwinkle]{$\blacksquare$};
			
			\node [fill=white] at (-.5,4.2) {$\wall(W,\mu)$};
		\end{scope}
	\end{tikzpicture}
	\caption{The pullback decomposition \eqref{equ:wlab-is-pullback} for one of the part of $\wlab_\alpha(W,\mu)$ from \cref{fig:wlab}. Note that $\wall(W,\mu)$ and $[3]$ are larger than pictured; we have only included the parts that are relevant for this pullback.}
	\label{fig:wall-as-pullback}
\end{figure}

\subsubsection*{Substep \ref{step:functor-to-premorita-thickening}: Thickening}
As a next step, we replace the undercategory functor
\begin{equation}\label{overcat}
	\Gap_{\lbr \bullet \rbr/}\colon \Gap^{\op}\lra \sCat_{/ \Gap}
\end{equation}
by a simplicial thickening after precomposition with the inclusion $\Gap_\sur^{\op}\rightarrow \Gap^{\op}$ of (the opposite of) the wide subcategory of surjective morphisms. By ``simplicial thickening'', we mean a functor whose values need no longer be discrete categories and which comes with a natural transformation to $\Gap_{\lbr \bullet \rbr/}$ that is a levelwise Dwyer--Kan equivalence. 

We first define a $\Kan$-enriched category $\smash{\gls*{gapkan}}$ that is Dwyer--Kan equivalent to $\Gap_{\lbr p \rbr/}$. Its objects are the same as those of $\Gap_{\lbr p \rbr/}$, that is, morphisms $\alpha \colon \lbr p \rbr \to \lbr q \rbr$ in $\Gap$. The space of morphisms from the object $\alpha\colon \lbr p\rbr\ra \lbr q\rbr$ to the object $\alpha'\colon \lbr p\rbr\ra \lbr q'\rbr$ is 
\[
	\textstyle{\Map_{\underline{\Gap}_{\lbr p \rbr/}}(\alpha,\alpha')\coloneqq \bigsqcup_{\gamma\in \Map_{\Gap_{\lbr p \rbr/}}(\alpha,\alpha')}  \Emb\big(\wlab_\alpha(\bfR)|^{{\gamma^{-1}\lbr\mathring{q'}\rbr}},\wlab_{\alpha'}(\bfR)\big)_\gamma}
\]
where the subscript $\gamma$ indicates that we restrict to embeddings $\overline{\gamma}$ that 
\begin{enumerate}
	\item make the diagrams 
	\[\qquad \begin{tikzcd}[column sep=0.4cm,]
		\wlab_\alpha(\bfR)|^{\gamma^{-1}\lbr\mathring{q'}\rbr}\dar\arrow[r,hook,"\overline{\gamma}"]& \wlab_{\alpha'}(\bfR)\dar\\
		\gamma^{-1}\lbr\mathring{q'}\rbr\rar{\gamma}&\lbr \mathring{q'}\rbr
	\end{tikzcd}\quad \begin{tikzcd}[column sep=-0.4cm]
		\wlab_\alpha(\bfR)|^{\gamma^{-1}\lbr\mathring{q'}\rbr}\arrow[rr,"\overline{\gamma}",hook]&& \wlab_{\alpha'}(\bfR)\\
		&\coll(\bfR)|^{\alpha'^{-1}\lbr\mathring{q}\rbr}\arrow[ur,"c^{\alpha'}_{\bfR}",swap,hook']\arrow[ul,"c^\alpha_{\bfR}",hook]&
	\end{tikzcd}
	\]
	commute, i.e.\,they cover $\gamma$ and preserve the collars \eqref{equ:collar}, and
	\item are order-preserving with respect to lexicographical order on $\lbr \mathring{q}\rbr\times\bfR$ and $\lbr \mathring{q'}\rbr\times\bfR$.
\end{enumerate}
The composition in $\underline{\Gap}_{\lbr p \rbr/}$ is induced by the composition in $\Gap_{\lbr p \rbr/}$, forgetting components, and composition of embeddings. By construction, there is a forgetful functor $\underline{\Gap}_{\lbr p \rbr/}\ra\Gap_{\lbr p \rbr/}$ which is a Dwyer--Kan equivalence as a result of the contractibility of the space of monotonous embeddings between connected intervals. Postcomposing this functor with the projection $\Gap_{\lbr p \rbr/}\ra \Gap$ and varying $p$, we obtain a functor \[\underline{\Gap}_{\lbr \bullet \rbr/} \colon \Gap_\sur^{\op}\lra \sCat_{/ \Gap}\] with a natural transformation to \eqref{overcat} that consists of the Dwyer--Kan equivalences just discussed.

\subsubsection*{Substep \ref{step:functor-to-premorita-simplicial}: $E^{\geo}$ on the level of $\Kan$-enriched categories}

We now turn towards the construction of a functor of semisimplicial $\Kan$-enriched categories
\begin{equation}\label{equ:psi-simplicially}
	E^{\geo}_{[\bullet]}\colon \ncBord(d)^\nonunital_{[\bullet]}\lra\Fun_{\Gap}\big(\underline{\Gap}_{\lbr \bullet \rbr/},\Man_d^{\otimes}\big).
\end{equation}
The value of $E^{\geo}_{[p]}$ at $(W,\mu)\in(\ncBord(d)^\nonunital)_{[p]}$ is the functor
\begin{equation}\label{equ:psi-single-manifold}
	E^{\geo}_{[p]}(W,\mu)\colon \underline{\Gap}_{\lbr p \rbr/}\lra \Man_d^{\otimes}
\end{equation}
over $\Gap$ defined as follows: on objects, it maps $(\alpha\colon \lbr p\rbr \to \lbr q\rbr)$ to $(\lbr q\rbr,\lab_\alpha(W,\mu))$. On a morphism given by a pair $(\gamma,\overline{\gamma})$ of a morphism $\gamma \colon \lbr q \rbr \to \lbr q' \rbr$ under $\lbr p \rbr$ in $\Gap$ and an embedding $\overline{\gamma}\in\Emb(\wlab_\alpha(\bfR)|^{{\gamma^{-1}\lbr\mathring{q'}\rbr}},\wlab_{\alpha'}(\bfR))_\gamma$, it is given by the embedding
\[
	E^{\geo}_{[p]}(W,\mu)(\overline{\gamma})\colon\lab_\alpha(W,\mu)|^{{\gamma^{-1}\lbr\mathring{q'}\rbr}}\longhookrightarrow \lab_{\alpha'}(W,\mu)
\] 
over $\gamma$ constructed via the following recipe: using the decomposition \eqref{eq:decompose-lab} and $\alpha^{-1}(\gamma^{-1}\lbr\mathring{q'}\rbr) = \alpha'^{-1}\lbr\mathring{q'}\rbr$, the embedding $E^{\geo}_{[p]}(W,\mu)(\overline{\gamma})$ is of the form
\[
	\ch(W,\mu)|^{\alpha'^{-1}\lbr\mathring{q'}\rbr} \cup_\partial \wlab_\alpha(W,\mu)|^{\gamma^{-1}\lbr \mathring{q'}\rbr}\longhookrightarrow \ch(W,\mu)|^{\alpha'^{-1}\lbr\mathring{q'}\rbr} \cup_\partial \wlab_{\alpha'}(W,\mu).
\]
On $\ch(W,\mu)|^{\alpha'^{-1}\lbr\mathring{q'}\rbr}$ we declare it to be the identity and on the complement we use the pullback description \eqref{equ:wlab-is-pullback} and the translations \eqref{equ:wlab-independent-of-walls} to define it via the commutative diagram
\[\begin{tikzcd}[row sep=0.2cm,column sep=2cm,ar symbol/.style = {draw=none,"\textstyle#1" description,sloped},
	iso/.style = {ar symbol={\cong}}]
	\wlab_\alpha(W,\mu)|^{\gamma^{-1}\lbr \mathring{q'}\rbr}\arrow[r,hookrightarrow,"{E^{\geo}_{[p]}(W,\mu)(\overline{\gamma})}"]\arrow[d,equal]&\wlab_{\alpha'}(W,\mu)\arrow[d,equal]\\
	\wlab_{\alpha}(\bfR,\mu)|^{\gamma^{-1}\lbr \mathring{q'}\rbr}\times_{[p]}\wall(W,\mu)\arrow[d,iso]&\wlab_{\alpha'}(\bfR,\mu)\times_{[p]}\wall(W,\mu)\arrow[d,iso]\\
	\wlab_{\alpha}(\bfR)|^{\gamma^{-1}\lbr \mathring{q'}\rbr}\times_{[p]}\wall(W,\mu)\arrow[r,hookrightarrow,"{(\overline{\gamma},\id)}"]&\wlab_{\alpha'}(\bfR)\times_{[p]}\wall(W,\mu).
\end{tikzcd}\]
This finishes the construction of the functor $\smash{E^{\geo}_{[p]}(W,\mu)\colon \underline{\Gap}_{\lbr p \rbr/}\ra \Man_d^{\otimes}}$. Note that it commutes with the functors to $\Gap$ by construction. 

Having defined $E^{\geo}_{[p]}$ on objects, defining it on morphisms amounts to specifying maps
\[
	\begin{tikzcd}[row sep=0.4cm, column sep=-0.2cm] \Emb\big((W,\mu),(W',\mu')\big) \dar & \\\Nat_{\Gap}(E^{\geo}_{[p]}(W,\mu),E^{\geo}_{[p]}(W',\mu'))&\subset 	\bigsqcup_{q,\alpha\in\Map_{\Gap}(\lbr p\rbr,\lbr q\rbr)}
	\Emb\big(\lab_\alpha(W,\mu),\lab_\alpha(W',\mu')\big)\end{tikzcd}
\]
where $\Nat_\Gap(-,-)$ is the hom-functor in the $\Kan$-enriched category $\Fun_{\Gap}\big(\underline{\Gap}_{\lbr p \rbr/},\Man_d^{\otimes}\big)$, i.e.\,the space of natural transformations covering the identity on $\Gap$. These maps are induced by the evident naturality of the $\lab_\alpha(-)$-construction in embeddings of $[p]$-walled $d$-manifolds. 

To finish the construction of \eqref{equ:psi-simplicially}, we have to argue that the $E^{\geo}_{[p]}$'s assemble to a morphism of semisimplicial objects in $\Kan$-enriched categories as in \eqref{equ:psi-simplicially}. But this is merely a case of going through the definitions; ultimately it amounts to the identity $\lab_{\beta\circ c(\delta)}(W,\mu)=\lab_\beta(W,\mu\circ\delta)$.

\subsubsection*{Substep \ref{step:functor-to-premorita-non-unital}: $E^{\geo}$ on the level of $\infty$-categories}
Taking coherent nerves, we obtain 
\[\gls*{gapinf} \coloneqq N_{\coh}(\underline{\Gap}_{\lbr \bullet \rbr/}) \in \Fun(\Delta^\op_\inj,\CatInf),\]
which comes with an equivalence to $\Gap_{\lbr \bullet \rbr/}\cong \Delta_{/[\bullet]}^\op$ induced by the equivalence $\underline{\Gap}_{\lbr \bullet \rbr/}\simeq \Gap_{\lbr \bullet \rbr/}$ from Substep \ref{step:functor-to-premorita-thickening}. From $E^{\geo} _{[\bullet]}$ we obtain a morphism of semisimplicial objects in $\CatInf$ 
\begin{equation}\label{equ:bord-to-pre-morita-nonalg}
	\ncBordInf(d)^{\nonunital}\lra \Fun_{\Gap}(\underline{\GapInf}_{\lbr \bullet \rbr/},\ManInf_d^{\otimes})\simeq \Fun_{\Delta^\op}(\Delta_{/[\bullet]}^\op,\ManInf_d^{\otimes})
\end{equation}
given by postcomposing the coherent nerve applied to \eqref{equ:psi-simplicially} with the canonical map
\[N_{\coh}\big(\Fun_{\Gap}(\underline{\Gap}_{\lbr \bullet \rbr/},\Man_d^{\otimes})\big)\lra \Fun_{\Gap}(\underline{\GapInf}_{\lbr \bullet \rbr/},\ManInf_d^{\otimes})\simeq \Fun_{\Delta^\op}(\Delta_{/[\bullet]}^\op,\ManInf_d^{\otimes}),\]
from Property \ref{enum:coherent-functor} of \cref{sec:coherent-nerve-props}.

\begin{lem}\label{lem:image-in-premorita}The image of the functor \eqref{equ:bord-to-pre-morita-nonalg} lies in the levelwise full subcategory $\overline{\ALG}(\ManInf_d) \subset \Fun_{\Delta^\op}(\Delta_{/[\bullet]}^\op,\ManInf_d^{\otimes})$ from \cref{sec:pre-morita}.
\end{lem}

\begin{proof}
In view of \cref{fact:cocart-from-s-cat}, it suffices to show that for a $[p]$-walled manifold $(W,\mu)$, and objects $\alpha\colon \lbr p\rbr \ra \lbr q\rbr$ and $\alpha'\colon \lbr p\rbr \ra \lbr q'\rbr$, the functor
$
	E^{\geo}_{[p]}(W,\mu)\colon\underline{\Gap}_{\lbr p \rbr/}\ra \Man_d^{\otimes}
$  of $\Kan$-enriched categories sends embeddings $\overline{\gamma}\in\Map_{\underline{\Gap}_{\lbr p \rbr/}}(\alpha,\alpha')$ whose underlying map $\gamma\colon \lbr q\rbr \ra \lbr q'\rbr$ is inert to cocartesian morphisms in $\Man_d^{\otimes}$ with respect to  the projection $\Man_d^{\otimes}\ra \Gap$. In other words, for objects $( Z,\lbr q''\rbr)\in\Man_d^{\otimes}$, we need to check that the square of Kan-complexes
\[
\hspace{-0.3cm}
\begin{tikzcd}[column sep=0.6cm]
	\Map_{\Man_d^{\otimes}}\Big((\lbr q'\rbr,\lab_{\alpha'}(W,\mu)),(\lbr q''\rbr,Z)\Big)\rar{E^{\geo}_{[p]}(W,\mu)(\overline{\gamma})^*}\arrow[d,equal]&\Map_{\Man_d^{\otimes}}\Big((\lbr q\rbr,\lab_{\alpha}(W,\mu)), (\lbr q''\rbr,Z)\Big)\arrow[d,equal]\\[-0.45cm]
	\underset{\varphi\in \Map_\Gap(\lbr q'\rbr,\lbr q''\rbr)}\bigsqcup \Emb\big(\lab_{\alpha'}(W,\mu)|^{\varphi^{-1}\lbr\mathring{q}''\rbr}, Z)_\varphi\dar& \underset{\psi\in \Map_\Gap(\lbr q\rbr,\lbr q''\rbr)}\bigsqcup \Emb\big(\lab_{\alpha}(W,\mu)|^{\psi^{-1}\lbr\mathring{q}''\rbr}, Z)_\psi\dar\\
	\Map_\Gap(\lbr q'\rbr,\lbr q''\rbr)\rar{\gamma^*}& \Map_\Gap(\lbr q\rbr,\lbr q''\rbr)
	\end{tikzcd}
\]
is homotopy cartesian. Taking vertical homotopy fibres, it suffices to show that the maps
\[
	E^{\geo}_p(W,\mu)(\overline{\gamma})^*\colon \Emb\big(\lab_{\alpha'}(W,\mu)|^{\varphi^{-1}\lbr\mathring{q}''\rbr}, Z)_\varphi\lra\Emb\big(\lab_{\alpha}(W,\mu)|^{(\varphi\circ\gamma)^{-1}\lbr\mathring{q}''\rbr}, Z)_{\varphi\circ\gamma}
\]
are weak equivalences. Since $\gamma$ is inert, the restricted map $\gamma^{-1}\lbr \mathring{q}'\rbr\ra \lbr \mathring{q}'\rbr$ is bijective, so it suffices to show that the embedding $E^{\geo}_{[p]}(W,\mu)(\overline{\gamma})^* \colon \lab_{\alpha}(W,\mu)|^{\gamma^{-1} \lbr\mathring{q}'\rbr}\hookrightarrow \lab_{\alpha'}(W,\mu)$ is an isotopy equivalence over $\gamma$. To see this, note that since $\gamma^{-1}\lbr \mathring{q}'\rbr\ra \lbr \mathring{q}'\rbr$ is bijective, the embedding $\overline{\gamma}\colon \wlab_\alpha(\bfR)|^{\gamma^{-1}\lbr \mathring{q}'\rbr}\hookrightarrow \wlab_{\alpha'}(\bfR)$ is an isotopy equivalence over $\gamma$ and under $\coll(\bfR)|^{\alpha'^{-1}\lbr \mathring{q}'\rbr}$, from which it follows that $E^{\geo}_{[p]}(W,\mu)(\overline{\gamma})$ is an isotopy equivalence over $\gamma$ as claimed.
\end{proof}

By the previous lemma, \eqref{equ:bord-to-pre-morita-nonalg} restricts to a morphism $\gls*{psi} \colon \ncBordInf(d)^{\nonunital}\ra\overline{\ALG}(\ManInf_d)$ of semisimplicial $\infty$-categories. This completes \ref{step:functor-to-premorita}.

\subsection{Composite algebras}\label{step:composite}
We now consider the composition
\begin{equation}
	\label{equ:non-unital-composition}\gls*{overlinee} \colon \ncBordInf(d)^{\nonunital}\xrightarrow{E^{\geo}}\overline{\ALG}(\ManInf_d)\xlra{\yon_*}\overline{\ALG}(\PSh(\ManInf_d))\xlra{\iota^*}\overline{\ALG}(\PSh(\DiscInf_d)).
\end{equation}
Here $E^{\geo}$ is the functor from the previous step, $\yon_*$ is induced by the (monoidal) Yoneda embedding $\yon\colon \ManInf_d\ra \PSh(\ManInf_d)$ (see \cref{sec:presheaf-category}), and $\iota^*$ is the functor induced by the lax monoidal functor $\PSh(\ManInf_d)\ra \PSh(\DiscInf_d)$ which is itself induced by the inclusion $\iota\colon \DiscInf_d\hookrightarrow \ManInf_d$ of the full subcategory spanned by manifolds diffeomorphic to $T \times \bfR^d$ for finite sets $T$ with monoidal structure inherited from $\ManInf$. By the properties of presheaf categories discussed in \cref{sec:presheaf-category}, the monoidal category $\PSh(\DiscInf_d)$ has good relative tensor products in the sense of  \cref{sec:composite-algebras}, so it makes sense to ask whether \eqref{equ:non-unital-composition} lands in the levelwise full subcategory $\ALG(\PSh(\DiscInf_d))\subset \overline{\ALG}(\PSh(\DiscInf_d))$
of \cref{sec:composite-algebras}. This section serves to prove this:

\begin{prop}\label{prop:image-in-Morita}
The functor $\smash{\overline{E}}$ from \eqref{equ:non-unital-composition} factors through $\ALG(\PSh(\DiscInf_d))\subset \overline{\ALG}(\PSh(\DiscInf_d))$.
\end{prop}

We will first explain how \cref{prop:image-in-Morita} follows from a seemingly different result, and then prove that other result.  The argument involves a simplicial thickening
\[
	\underline{\Gap_{\sur}^\rhd}\xlra{\simeq}\Gap_\sur^{\rhd}.
\]
of the right-cone $\Gap_\sur^{\rhd}$ of the category $\Gap_\sur$ (the category obtained by freely adding a terminal object $\infty\in\Gap_\sur^{\rhd}$) in terms of the manifolds 
\[\lbr a\rbr^*\coloneqq L \times [-\epsilon,\epsilon) \cup \lbr \mathring{a} \rbr \times (-\epsilon,\epsilon) \cup R \times (-\epsilon,\epsilon] \subset \lbr a \rbr \times \bfR \quad\text{and} \quad	\lbr \infty\rbr^*\coloneqq [-\epsilon,\epsilon]\subset\bfR\]
where $a\ge0$. The objects of $\underline{\Gap_\sur^{\rhd}}$ are the same as those of $\Gap_\sur^{\rhd}$. The space of morphisms $\lbr a\rbr\ra \lbr b\rbr$ between objects of $\Gap_\sur\subset \Gap_\sur^{\rhd}$ is defined as
\[\textstyle{\Map_{\underline{\Gap_{\sur}^\rhd}}(\lbr a\rbr, \lbr b\rbr)\coloneqq \bigsqcup_{\gamma\in \Map_{\Gap_\sur}(\lbr a\rbr,\lbr b\rbr)} \Emb\big(\lbr a \rbr^\ast,\lbr b \rbr^\ast\big)_{\gamma}}
\]
where the subscript $(-)_{\gamma}$ indicates we restrict to embeddings $\overline{\gamma}$ that cover $\gamma$, are the identity on $L \times [-\epsilon,-\tfrac{\epsilon}{2}) \cup R \times (\tfrac{\epsilon}{2},\epsilon]$ and preserve the lexicographic order inherited from $\lbr a \rbr \times \bfR$ and $\lbr b \rbr \times \bfR$. Finally, the space of morphisms $\lbr a\rbr \ra \lbr \infty\rbr$ is defined as
\[
	\Map_{\underline{\Gap_{\sur}^\rhd}}(\lbr a\rbr, \lbr \infty\rbr)\coloneq
\Emb\big(\lbr a \rbr^\ast,\lbr \infty\rbr^*\big)_{\infty}
\]
where the subscript $(-)_\infty$ indicates that we restrict to embeddings $\overline{\gamma}$ that agree on $L \times [-\epsilon,-\tfrac{\epsilon}{2}) \cup R \times (\tfrac{\epsilon}{2},\epsilon]$ with the projection to the second coordinate and preserve the lexicographical order inherited from $\lbr a\rbr\times\bfR$ and $\bfR$. The space of morphisms $\lbr\infty\rbr \ra \lbr\infty\rbr$ is the space of self-embeddings of $\lbr\infty\rbr^\ast=[-\epsilon,\epsilon]$ that agree with the identity on the complement of $[-\tfrac{\epsilon}{2},\tfrac{\epsilon}{2}]$. This category admits an evident functor to $\Gap_\sur^{\rhd}$ which is an equivalence as a result of the contractibility of the space of order-preserving embeddings between intervals.

\begin{nconvention}In what follows, we occasionally omit the choices of embeddings of manifolds into Euclidean spaces for brevity. For instance, we treat $\Man_d=(\Man_d^{\otimes})_{[1]}$ from \cref{step:mand} as the $\Kan$-enriched category of abstract smooth $d$-manifolds and codimension $0$ embeddings.
\end{nconvention}

Given a (possibly noncompact) $d$-manifold without boundary $V$  equipped with $k$ disjoint codimension $1$ submanifolds $V_i\subset V$ that are topologically closed in $V$ as a subspace, equipped with disjoint bicollars $[-\epsilon,\epsilon]\times V_i\subset V$, we construct a simplicially enriched functor
\[
	V{\lbr -\rbr}\colon \underline{\Gap_\sur^\rhd}\lra \Man_d
\]
which on objects sends $\lbr \infty\rbr $ to $V{\lbr \infty\rbr}\coloneqq V$ and $ \lbr a\rbr\in \underline{\Gap_{\sur}}$ to
\[
\textstyle{V{\lbr a\rbr}\coloneqq V^\ast\sqcup\big( \bigsqcup_{i=1}^k\lbr \mathring{a}\rbr\times(-\epsilon,\epsilon)\times V_i}\big),
\]
where $V^\ast$ is the manifold obtained from $V$ by cutting out $\cup_{i=1}^k[-\tfrac{\epsilon}{2},\tfrac{\epsilon}{2}]\times V_i$ and extending the resulting collars $[-\epsilon,-\tfrac{\epsilon}{2})\times V_i\sqcup (\tfrac{\epsilon}{2},\epsilon]\times V_i$ to collars $[-\epsilon,\epsilon)\times V_i\sqcup (-\epsilon,\epsilon]\times V_i$. Given a morphism $\overline{\gamma} \colon \lbr a\rbr \to \lbr b \rbr$ there is an embedding $V{\lbr a\rbr} \hookrightarrow V{\lbr b\rbr}$ that is the identity of $V^\ast$ outside the extended collars, and agrees on the remaining part with $\overline{\gamma} \times \id_{V_i}$. Finally, for $\lbr a \rbr \to \lbr\infty\rbr$ or $\lbr \infty \rbr \to \lbr\infty\rbr$ one defines embeddings $V{\lbr a\rbr} \hookrightarrow V{\lbr \infty \rbr}$ or  $V{\lbr \infty\rbr} \hookrightarrow V{\lbr \infty \rbr}$ in the same manner. 

Writing $\underline{\Gap_{\sur}}\subset \underline{\Gap^\rhd_\sur}$ for the full subcategory covering the inclusion $\Gap_{\sur}\subset\Gap_{\sur}^{\rhd}$, \cref{prop:image-in-Morita} will be a consequence of the following proposition involving homotopy colimits in the Kan--Quillen model structure on $\cat{S}$.

\begin{prop}\label{prop:cutting-hocolim}For a manifold $D$ diffeomorphic to $T\times \bfR^d$ for a finite set $T$, the map
\[
	\hocolim_{\underline{\Gap_\sur}}\,\Emb(D,V{\lbr-\rbr})\lra \hocolim_{\underline{\Gap_\sur^\rhd}}\,\Emb(D,V{\lbr-\rbr})\simeq \Emb(D,V{\lbr\infty\rbr})
\]
induced by the inclusion $\underline{\Gap_{\sur}} \subset \underline{\Gap_\sur^\rhd}$ is an equivalence.
\end{prop}

We postpone the proof to the next subsection and first explain how it implies  \cref{prop:image-in-Morita}.

\begin{proof}[Proof of \cref{prop:image-in-Morita}]
Consulting the definition of the Morita category, we have to show that the image of any object $\smash{(W,\mu)\in\ncBordInf(d)^{\nonunital}_{[p]}}$ in $\smash{\overline{\ALG}(\PSh(\DiscInf_d))_{[p]}}$ is composite in the sense of \cref{sec:composite-algebras}. By \cref{cor:composite-algebras-multisimplicial} this is equivalent to proving that for each  $\alpha\in\Delta^{\op}_{/[p]}$ 
\vspace{-0.1cm}
\begin{equation}\label{equ:composition-test-composite}
	(\Delta^\op_\inj)^\rhd\xra{\eta^\alpha} \Delta^{\act,\op}_{/[p]} \xra{E^{\geo}_{[p]}(W,\mu)} \ManInf_d^{\otimes,\act}\xra{(-)_!}\ManInf_d
\end{equation}
becomes a colimit diagram when postcomposed with $(\iota^*\circ \yon)\colon \ManInf_d\ra\PSh(\DiscInf_d)$. We first make the composition \eqref{equ:composition-test-composite} more explicit. Recall from \ref{step:functor-to-premorita} \ref{step:functor-to-premorita-simplicial} that \[E^{\geo}_{[p]}(W,\mu)\in\overline{\ALG}(\ManInf_d)_{[p]}\subset \Fun_{\Delta^{\op}}(\Delta^{\op}_{/[p]},\ManInf^{\otimes}_d)\] was obtained from a functor between simplicially enriched categories \begin{equation}\label{equ:psiW-simplicial}
	E^{\geo}_{[p]}(W,\mu)\colon \underline{\Gap}_{\lbr p\rbr/}\lra \Man^{\otimes}_d
\end{equation} 
by taking coherent nerves and using the equivalence $\underline{\Gap}_{\lbr p\rbr/}\simeq \Gap_{\lbr p\rbr/}\cong \Delta^{\op}_{/[p]}$ from \ref{step:functor-to-premorita} \ref{step:functor-to-premorita-thickening}. We now give a similar description of the composition \eqref{equ:composition-test-composite} as a simplicially enriched functor using a simplicial functor to the full subcategory $\underline{\Gap}^\act_{\lbr p \rbr/}\subset \underline{\Gap}_{\lbr p \rbr/}$ covering $\Gap^\act_{\lbr p \rbr/}\subset \Gap_{\lbr p \rbr/}$
 \[\underline{\eta^\alpha}\colon \underline{\Gap_\sur^\rhd}\lra\underline{\Gap}^\act_{\lbr p \rbr/}\] to the pullback $\underline{\Gap}^\act_{\lbr p \rbr/}$ of $\underline{\Gap}_{\lbr p \rbr/}$ along $\Gap^\act_{\lbr p \rbr/}\subset \Gap_{\lbr p \rbr/}$. The functor $\underline{\eta^\alpha}$ will make 
\[\begin{tikzcd}
	\underline{\Gap_\sur^\rhd}\dar[swap]{\simeq}\arrow[r,"\underline{\eta^\alpha}"]&\underline{\Gap}^\act_{\lbr p \rbr/}\dar{\simeq}\\[-2pt]
	\Gap_\sur^\rhd\rar{\eta^\alpha}&\Gap^\act_{\lbr p \rbr/}.
\end{tikzcd}\]
commutative where $\eta^\alpha$ is the functor from \cref{sec:composite-multisimplicial}. The construction involves the notation of \cref{const:rhoalpha-gap} ($k_\alpha$, $\alpha_1^{\vec{a}}$, $n_i$, etc.) and the discussion preceding \cref{cor:composite-algebras-multisimplicial}. On objects, $\smash{\underline{\eta^\alpha}}$ is determined by $\eta^\alpha$. On morphisms it sends $\overline{\gamma}\colon \lbr a\rbr^\ast\hookrightarrow  \lbr b\rbr^\ast$ to the right-hand embedding in a commutative square of embeddings (here $\vec{a}=(a,\ldots,a)$ and $\vec{b}=(b,\ldots,b)$)
\[\begin{tikzcd}[ar symbol/.style = {draw=none,"\textstyle#1" description,sloped},	subset/.style = {ar symbol={\subset}}, supset/.style = {ar symbol={\supset}}]
	\sqcup^{k_\alpha} \lbr a\rbr\times \bfR&[-15pt] \arrow[l,supset]\sqcup^{k_\alpha}\lbr a\rbr^\ast\rar[r,hookrightarrow]\arrow[d,hookrightarrow,"{\sqcup^{k_\alpha}\overline{\gamma}}",swap]& \wlab_{\alpha_1^{\vec{a}}}(\bfR)\arrow[d,hookrightarrow]\arrow[r,subset]& [-15pt]\lbr k_{\alpha}^{\vec{a}}\rbr\times\bfR\\
	\sqcup^{k_\alpha} \lbr b\rbr\times \bfR& \arrow[l,supset] \sqcup^{k_\alpha}\lbr b\rbr^\ast\arrow[r,hookrightarrow]& \wlab_{\alpha_1^{\vec{b}}}(\bfR)\arrow[r,subset]&\lbr k_{\alpha}^{\vec{b}}\rbr\times\bfR.
\end{tikzcd}\]
The $i$th component of the upper horizontal map is the embedding \[\lbr a\rbr\times\bfR \supset \lbr a\rbr^*\longhookrightarrow \wlab_{\alpha_1^{\vec{a}}}(\bfR)\subset\lbr k_{\alpha}^{\vec{a}}\rbr\times\bfR\] that is the unique inclusion of components that preserves the lexicographic order inherited from $\lbr a\rbr\times\bfR$ and $\lbr k_{\alpha}^{\vec{a}}\rbr\times\bfR$ and covers the map $\lbr a\rbr\ra \lbr k_{\alpha}^{\vec{a}}\rbr$ given by the sequence $\alpha_1^{\vec{a}}(n_i)<\alpha_1^{\vec{a}}(n_i)+1<\ldots <\alpha_1^{\vec{a}}(n_i)+a<\alpha_1^{\vec{a}}(n_i+1)$ (note that this is \emph{not} a morphism in $\Gap$ as it does not preserve the endpoints). The bottom horizontal embedding is defined in the same way, and the right hand embedding is defined to agree with $\overline{\gamma}$ on the components hit by the horizontal embedding and on the complement as the unique inclusion of components that covers the map $\eta^\alpha(\gamma)\colon \lbr k_{\alpha}^{\vec{a}}\rbr \ra \lbr k_{\alpha}^{\vec{b}}\rbr $ and preserves the lexicographic order. Similarly, $\underline{\eta}^\alpha$ sends a morphism in $\underline{\Gap_\sur^\rhd}$ given by an embedding $\overline{\gamma}\colon \lbr a\rbr^\ast\hookrightarrow \lbr \infty\rbr^\ast$ to the right-hand embedding in the square
\[\begin{tikzcd}[ar symbol/.style = {draw=none,"\textstyle#1" description,sloped},
		subset/.style = {ar symbol={\subset}}, supset/.style = {ar symbol={\supset}}]
		\sqcup^{k_\alpha} \lbr a\rbr\times \bfR&[-15pt]\arrow[l,supset]\sqcup^{k_\alpha}\lbr a\rbr^\ast\rar[r,hookrightarrow]\arrow[d,hookrightarrow,"{\sqcup^{k_\alpha}\overline{\gamma}}",swap]&\wlab_{\alpha_1^{\vec{a}}}(\bfR)\arrow[d,hookrightarrow]\arrow[r,subset]&[-15pt]\lbr k_{\alpha}^{\vec{a}}\rbr\times\bfR\\
		& \sqcup^{k_\alpha}\lbr \infty\rbr^\ast \arrow[r,hookrightarrow]&\wlab_{\alpha}(\bfR)\arrow[r,subset]&\lbr q\rbr\times\bfR
\end{tikzcd}\]
where the top horizontal embedding is the same as before, the bottom embedding includes the $i$th copy of $\lbr \infty\rbr^\ast=[-\epsilon,\epsilon]$ as the unique $[-\epsilon,\epsilon]$-component in $\wlab_{\alpha}(\bfR)$ that maps to $\alpha(n_i)\in \lbr q\rbr$ under the projection (using the notation from \cref{const:rhoalpha-gap}) and to $n_i\in\lbr \mathring{p-1}\rbr=\{1,\ldots,p-1\}\subset[p]$ under the map $\wlab_{\alpha}(\bfR)\ra [p]$ from \ref{step:functor-to-premorita} \ref{step:functor-to-premorita-language} II. The right vertical embedding is defined via the left vertical one on the components hit by the horizontal map and as the unique inclusion of components that cover the map $\gamma_{\vec{a}}\colon \lbr k_{\alpha}^{\vec{a}}\rbr\ra \lbr q\rbr$ and preserve the lexicographic order on $\lbr k_{\alpha}^{\vec{a}}\rbr\times\bfR$ and $\lbr q\rbr\times\bfR$.
	
By construction the composition \eqref{equ:composition-test-composite} is equivalent to the coherent nerve of the composition
\vspace{-0.1cm}
\begin{equation}\label{equ:colimit-diagram-psiW-simplicial-n}
	\underline{\Gap_\sur^\rhd}\xlra{\underline{\eta^\alpha}} \underline{\Gap}^{\act}_{\lbr p\rbr/} \xrightarrow{E^{\geo}_p(W,\mu)} \Man_d^{\otimes,\act}\xlra{(-)_!}\Man_d
\end{equation}
where $(-)_!$ is the simplicial ``disjoint unions''-functor of \eqref{equ:active-pushforward}. Tracing through the definitions, one checks that this functor agrees up to equivalence with the functor $V\lbr -\rbr$ for the manifold $V=\lab_\alpha(W,\mu)_!$ with the $k_\alpha$ different bicollared submanifolds $[-\epsilon,\epsilon]\times W_{\mu(j)}\cong W_{[\mu(j)-\epsilon,\mu(j)+\epsilon]}\subset \lab_\alpha(W,\mu)_!$ for $j\in\lbr \mathring{p-1}\rbr$ with $\alpha(j)\in\lbr \mathring{q}\rbr$ and $ \alpha(j)=\alpha(j+1)$. Using that a diagram $A\colon K^{\rhd}\ra \cC$ is a colimit diagram if and only if the natural map $\colim_{K}A\ra \colim_{K^{\rhd}}A$ is an equivalence, this implies that it suffices to show that the colimit
\[
	\colim_{N_{\coh}(\underline{\Gap_\sur^\rhd})}\Big(N_{\coh}(\underline{\Gap_\sur^\rhd})) \xra{N_{\coh} (V\lbr-\rbr)}N_{\coh}(\Man_d) \xra{\yon} \PSh(\ManInf_d)\xra{\iota^*}\PSh(\DiscInf_d)\Big)
\]
is unaffected by precomposing the diagram with the functor $N_{\coh}(\underline{\Gap_\sur})\ra N_{\coh}(\underline{\Gap_\sur^\rhd})$ induced by inclusion. Using that (i) equivalences in functor categories are detected objectwise, (ii) colimits in functor categories commute with evaluation at a fixed object $D\in\DiscInf_d$ \cite[5.1.2.3]{LurieHTT}, and (iii) the compatibility of the simplicial and $\infty$-categorical Yoneda embedding (see \cref{fact:yoneda-comparison}), we see that it is enough to show that the colimit
\[
	\colim_{N_{\coh}(\underline{\Gap_\sur^\rhd})}\Big(N_{\coh}(\underline{\Gap_\sur^\rhd})\xra{N_{\coh}(\ev_D\circ \yon_s\circ V\lbr -\rbr)}N_{\coh}(\Kan)\Big)
\]
is unaffected by precomposing the diagram with $N_{\coh}(\underline{\Gap_\sur})\ra N_{\coh}(\underline{\Gap_\sur^\rhd})$ for each object $D\in\DiscInf_d$ where $\yon_s\colon \Man_d\ra\Fun(\Man_d,\Kan)$ is the simplicial Yoneda embedding of the $\Kan$-enriched category $\Man_d$.  Using that model category-theoretic homotopy colimits are compatible with $\infty$-categorical colimits \cite[4.2.4.1]{LurieHTT}, the claim reduces to showing that the natural map between homotopy colimits in the Kan--Quillen model structure
\[
	\hocolim_{\underline{\Gap_\sur}}\big(\underline{\Gap_\sur}\xra{\ev_D\circ \yon_s\circ V\lbr -\rbr}\catS\big)\lra
	\hocolim_{\underline{\Gap_\sur^\rhd}}\big(\underline{\Gap_\sur^\rhd}\xra{\ev_D\circ \yon_s\circ V\lbr -\rbr}\catS\big)
\]
is an equivalence. This is \cref{prop:cutting-hocolim}.
\end{proof}

\begin{proof}[Proof of \cref{prop:cutting-hocolim}]This proof will eventually rely on a microfibration argument, which is why we phrase the argument in the category of topological spaces $\catTop$ as opposed to simplicial sets $\catS$. Relying on the usual Quillen equivalence between the category of simplicial sets $\catS$ and that of topological spaces $\catTop$, the claim has an evident reformulation in terms of homotopy colimits of $\catTop$-enriched $\catTop$-valued functors and it is this reformulation that we shall prove.
	
To begin with, we note that it suffices to show the claim for $D=\underline{n}\times\bfR^d$ for $n\ge0$. Next, we simplify the functor $\Emb(\underline{n}\times\bfR^d,-)\colon \Man_d\ra\catTop$ in terms of the functor $C_n^{\fr}\colon\Man_d\ra\catTop$ given by taking framed configurations, i.e.\,the pullback of functors
\[\begin{tikzcd} 
	C^\fr_n(-) \rar \dar &[-2pt] \Map(\unl{n},\Fr(-)) \dar \\[-2pt]
	\Emb(\unl{n},-) \rar{\subset} & \Map(\unl{n},-)
\end{tikzcd}\]
whose right vertical map is induced by the projection $\Fr(W)\ra W$ of the frame bundle of manifolds $W\in\Man_d$. Taking derivatives at the centres $\underline{n}\times\{0\}\subset \underline{n}\times\bfR^d$ gives a natural transformation $\Emb(\underline{n}\times\bfR,-)\ra C^\fr_n(-)$ which is a componentwise weak equivalence, so we conclude that in order to prove \cref{prop:cutting-hocolim} it suffices to show that the map
\[
	\hocolim_{\underline{\Gap_\sur}}\big(C_n^{\fr}(V\lbr-\rbr)\big)\lra
	\hocolim_{\underline{\Gap_\sur^\rhd}}\big(C_n^{\fr}(V\lbr-\rbr)\big)
\]
is a weak equivalence. This is a map between homotopy colimits in spaces, which we model by a bar construction. In general, given a $\catTop$-enriched category $\cat{C}$ and $\catTop$-enriched functors $F \colon \cat{C} \to \catTop$ and $G \colon \cat{C}^\op \to \catTop$, the \emph{bar-construction} $B_\bullet(F,\cat{C},G) \colon \Delta^\op \ra \catTop$ is the simplicial space $\textstyle{[r] \longmapsto \bigsqcup_{(c_0,\ldots,c_r)} F(c_0) \times \cat{C}(c_0,c_1) \times \cdots \times \cat{C}(c_{r-1},c_r) \times G(c_r)}$ where $(c_0,\ldots,c_r)$ tuns through ordered sequences of $(r+1)$ objects in $\cat{C}$. If $G$ has weakly contractible values, the thick geometric realisation $B(F,\cat{C},G)\coloneqq \|B_\bullet(F,\cat{C},G)\|$ is a model for $\hocolim_\cat{C} F$ (see e.g.\,\cite[Corollary 9.2.7]{Riehl}; since we take thick geometric realisations we do not need to worry about cofibrancy issues). Choosing $\cat{C}=\underline{\Gap_\sur^\rhd}$ and $G=\Map_{\underline{\Gap_\sur^\rhd}}(-,\lbr \infty \rbr)$, it therefore suffices to that
\begin{equation}\label{equ:simplicial-composite-version}
	B_\bullet\big(C_n^{\fr}(V\lbr-\rbr),\underline{\Gap_\sur},\Map_{\underline{\Gap_\sur^\rhd}}(-,\lbr \infty \rbr)\big) \lra B_\bullet\big(C_n^{\fr}(V\lbr-\rbr),\underline{\Gap_\sur^\rhd},\Map_{\underline{\Gap_\sur^\rhd}}(-,\lbr \infty \rbr)\big) 
\end{equation}
induced by $\underline{\Gap_\sur}\subset \underline{\Gap_\sur^\rhd}$ is a weak equivalence on thick realisations. There is an augmentation
\begin{equation}\label{equ:bar-resolution}
	B_\bullet\big(C_n^{\fr}(V\lbr-\rbr),\underline{\Gap_\sur^\rhd},\Map_{\underline{\Gap_\sur^\rhd}}(-,\lbr \infty \rbr)\big) \lra C^\fr_n(V)
\end{equation} 
induced by composition of embeddings and evaluation of $C_n^{\fr}(-)$. This admits an \emph{extra degeneracy} so it induces an equivalence on (thick) realisation (see e.g.\,\cite[Example 4.5.7]{Riehl}). This leaves us with showing that the composition of \eqref{equ:simplicial-composite-version} and \eqref{equ:bar-resolution}
\begin{equation}\label{equ:augmentation-of-bar}
	B_\bullet\big(C_n^{\fr}(V\lbr-\rbr),\underline{\Gap_{\sur}},\Map_{\underline{\Gap_\sur^\rhd}}(-,\lbr \infty \rbr)\big) \lra C^\fr_n(V)
\end{equation}  is an equivalence on thick realisations. To prove this, we consider a semisimplicial space $\wall_\bullet$ whose space of $p$-simplices is the space of order-preserving functions $\tau \colon [p] \to (-\epsilon,\epsilon)$ with simplicial structure by precomposition, and we define an augmented semisimplicial space
\begin{equation}\label{equ:resolve-embeddings-into-infty}
	\Map_{\underline{\Gap_\sur^\rhd}}(\lbr a\rbr,\lbr \infty \rbr)_\smallsquare\lra \Map_{\underline{\Gap_\sur^\rhd}}(\lbr a\rbr,\lbr \infty \rbr)
\end{equation}
 for $a\ge0$ whose space of $p$-simplices 
\begin{equation}\label{equ:thicken-maps-to-infty}
	\Map_{\underline{\Gap_\sur^\rhd}}(\lbr a\rbr,\lbr \infty \rbr)_p\subset \Map_{\underline{\Gap_\sur^\rhd}}(\lbr a\rbr,\lbr \infty \rbr)\times \wall_p
\end{equation} 
is the subspace of pairs of a function $\tau\colon[p]\ra (\epsilon,\epsilon)$ and an embedding $\lbr a\rbr^* \hookrightarrow\lbr \infty \rbr^*$ that is disjoint from the image of $\tau$. Varying $a$, this defines a functor $(\underline{\Gap_\sur})^\op\times\Delta_\inj^{\op}\lra \catTop$ that is compatible with \eqref{equ:resolve-embeddings-into-infty}, so we obtain an augmentation of semisimplicial spaces
\begin{equation}\label{equ:map-on-bar}
	B\big(C_n^{\fr}(V\lbr-\rbr),\underline{\Gap_{\sur}},\Map_{\underline{\Gap_\sur^\rhd}}(-,\lbr \infty \rbr)_\smallsquare\big)\lra B\big(C^\fr_n,\underline{\Gap_{\sur}},\Map_{\underline{\Gap_\sur^\rhd}}(-,\lbr \infty \rbr)\big).
\end{equation}
where we have geometrically realised the semisimplicial direction of the bar-construction. In \cref{lem:microfibration-lemma} \ref{enum:microfibration-lemma-i} below, we will show that \eqref{equ:resolve-embeddings-into-infty} realises to a weak equivalence. Together with the fact that, up to weak equivalence, it does not matter in which direction one realises a bisemisimplicial space first (so we may realise the $\smallsquare$-direction before the bar-direction) and that the geometric realisation of a levelwise weak equivalence is a weak equivalence, this implies that the map in \eqref{equ:map-on-bar} realises to a weak equivalence. It thus remains to show that the augmented semisimplicial space 
\[
	B\big(C_n^{\fr}(V\lbr-\rbr),\underline{\Gap_{\sur}},\Map_{\underline{\Gap_\sur^\rhd}}(-,\lbr \infty \rbr)_\smallsquare\big)\lra C_n^\fr(V)
\]
obtained by combining \eqref{equ:map-on-bar} and \eqref{equ:augmentation-of-bar} realises to a weak equivalence. To prove this remaining claim, we consider the sub-simplicial space $\wall^V_\smallsquare\subset \wall_\smallsquare\times C_n^{\fr}(V)$ consisting of pairs of an order-preserving function $\tau\colon [p]\ra(-\epsilon,\epsilon)$ and a framed configurations $\vec{x}\in C_n(V)$ that is disjoint from the submanifolds $\{\tau(j)\}\times V_i\subset V$ for all $j=0,\ldots,p$ and $i=1,\ldots,k$ (here we used the collars $[-\epsilon,\epsilon]\times V_i\subset V$; see \cref{fig:wall-semi} for an example). The projection to $\wall_p$ in \eqref{equ:thicken-maps-to-infty} and the augmentation to $C_n^\fr(V)$ assemble to a semisimplicial map over $C_n^\fr(V)$
\begin{equation}\label{equ:bar-to-wall}
	B\big(C_n^{\fr}(V\lbr-\rbr),\underline{\Gap_{\sur}},\Map_{\underline{\Gap_\sur^\rhd}}(-,\lbr \infty \rbr)_\smallsquare\big)\lra \wall^V_\smallsquare
\end{equation}
which we show to be a levelwise weak equivalence in \cref{lem:microfibration-lemma} \ref{enum:microfibration-lemma-ii} (levelwise with respect to the $\smallsquare$-direction, in which we did not realise yet). This leaves us with showing that the augmentation $\wall_\smallsquare^V\ra C_n^{\fr}(V)$ realises to a weak equivalence. This is \cref{lem:microfibration-lemma} \ref{enum:microfibration-lemma-iii}.
\end{proof}

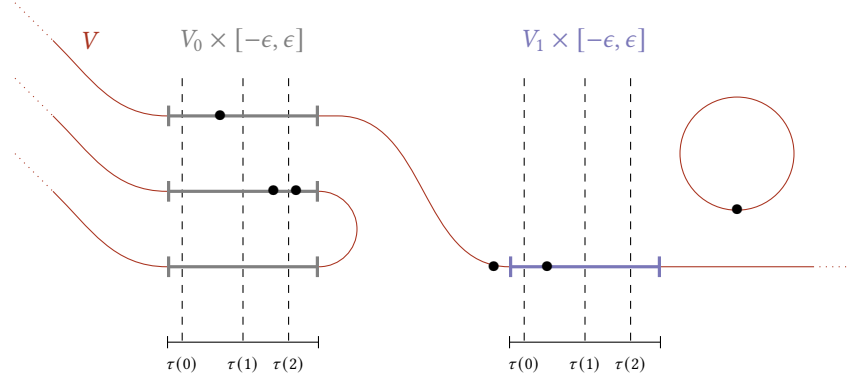
\begin{figure}
		\begin{tikzpicture}
		\draw [dashed] (-3.3,3.5) -- (-3.3,0);
		\draw [dashed] (-2.5,3.5) -- (-2.5,0);
		\draw [dashed] (-1.9,3.5) -- (-1.9,0);
		\draw [|-|] (-3.5,0) -- (-1.5,0);
		\node at (-3.3,-.3) {\tiny $\tau(0)$};
		\node at (-2.5,-.3) {\tiny $\tau(1)$};
		\node at (-1.9,-.3) {\tiny $\tau(2)$};
		
		\draw [dashed] (1.2,3.5) -- (1.2,0);
		\draw [dashed] (2,3.5) -- (2,0);
		\draw [dashed] (2.6,3.5) -- (2.6,0);
		\draw [|-|] (1,0) -- (3,0);
		\node at (1.2,-.3) {\tiny $\tau(0)$};
		\node at (2,-.3) {\tiny $\tau(1)$};
		\node at (2.6,-.3) {\tiny $\tau(2)$};
		
		\draw [dotted,Mahogany] (-5.5,3.5) -- (-5,3);
		\draw [dotted,Mahogany] (-5.5,2.5) -- (-5,2);
		\draw [dotted,Mahogany] (-5.5,4.5) -- (-5,4);
		\draw [dotted,Mahogany] (5,1) -- (5.5,1);
		\draw [Mahogany](-5,3) to[out=-45,in=180] (-3.5,2) to[out=0,in=180] (-1.5,2) to[out=0,in=90] (-1,1.5) to[out=-90,in=0] (-1.5,1) to[out=180,in=0] (-3.5,1) to[out=180,in=-45] (-5,2);
		\draw [Mahogany] (-5,4) to[out=-45,in=180] (-3.5,3) to[out=0,in=180] (-1.25,3) to[out=0,in=180] (1,1) -- (5,1);
		\draw [Mahogany] (4,2.5) circle (.75cm);
		\node at (-4.5,4) [Mahogany] {$V$};
		
		\draw [|-|,very thick,gray] (-3.5,1) -- (-1.5,1);
		\draw [|-|,very thick,gray] (-3.5,2) -- (-1.5,2);
		\draw [|-|,very thick,gray] (-3.5,3) -- (-1.5,3);
		\node at (-2.5,4) [gray,fill=white] {$V_0 \times [-\epsilon,\epsilon]$};
		
		\draw [|-|,very thick,Periwinkle] (1,1) -- (3,1);
		\node at (2,4) [Periwinkle,fill=white] {$V_1 \times [-\epsilon,\epsilon]$};
		
		\node at (-2.8,3) {$\bullet$};
		\node at (-2.1,2) {$\bullet$};
		\node at (-1.8,2) {$\bullet$};
		\node at (4,1.75) {$\bullet$};
		\node at (.8,1) {$\bullet$};
		\node at (1.5,1) {$\bullet$};
	\end{tikzpicture}
	\caption{An element of $\wall_2^V$. We suppressed the framings at the points in the configuration indicated by the black points.}
	\label{fig:wall-semi}
\end{figure}

We now supply the postponed ingredients to the proof of \cref{prop:cutting-hocolim}. This finishes the proof of that proposition and thus also that of \cref{prop:image-in-Morita}.

\begin{lem}\label{lem:microfibration-lemma}\ 
\begin{enumerate}
	\item \label{enum:microfibration-lemma-i} The thick realisation of the map \eqref{equ:resolve-embeddings-into-infty} is a weak equivalence.
	\item\label{enum:microfibration-lemma-ii} The map \eqref{equ:bar-to-wall} is a levelwise weak equivalence.
	\item\label{enum:microfibration-lemma-iii} The augmentation $\varepsilon\colon \wall_\smallsquare^V\ra C_n^{\fr}(V)$ realises to a weak equivalence.
	\end{enumerate}
\end{lem}

\begin{proof}
We begin with a general observation. Let $X$ be a nonempty totally ordered topological poset (by which we mean topological space $X$ with a total order on its underlying set). If the function $\max(x_0,-)\colon X\ra X$ is continuous for some $x_0\in X$, then the nerve of $X$ is weakly contractible, since the sequence of inequalities $x\le \max(x_0,x)\ge x_0$ induces a zig-zag of natural transformations from the identity on $X$ to the constant functor with values $x_0$, so we obtain a homotopy between the identity map of the nerve of $X$ and the constant map.
	
Replacing the (half-)open intervals in the definition of $\lbr a \rbr^\ast$ with closed intervals, we get a weakly equivalent semisimplicial space. Doing so, it follows from a version of the parametrised isotopy extension theorem on restricting embeddings to compact submanifolds (c.f.\,\cite{Palais}) that the augmentation \eqref{equ:resolve-embeddings-into-infty} is a levelwise fibration. Hence to prove \ref{enum:microfibration-lemma-i} it suffices to show that the semisimplicial space given by the fibres over an embedding $e\colon \lbr a\rbr^\ast\hookrightarrow \lbr\infty\rbr^\ast=[-\epsilon,\epsilon]$ realises to a weakly contractible space. This agrees with the nerve of the nonempty totally ordered poset of real numbers $t\in(-\epsilon,\epsilon)$ disjoint from the image of $e$, so the claim follows from the observation.
	
To show part \ref{enum:microfibration-lemma-ii}, we choose for all $p\ge0$ an order-preserving function $\tau\colon [p]\ra(\epsilon,\epsilon)$ and an embedding $\smash{e\in\Map_{\underline{\Gap_\sur^\rhd}}(\lbr p\rbr,\lbr \infty \rbr)}$ such that $\tau$ hits every component of the complement of $e$.  This induces an equivalence
\vspace{-0.2cm}
\[
	(e\circ(-),\tau)\colon \Map_{\underline{\Gap_\sur}}(-,\lbr p\rbr)\xlra{\simeq} \Map_{\underline{\Gap_\sur^\rhd}}(-,\lbr \infty \rbr)_p
\] 
which in turn induces the left vertical equivalence in the commutative diagram
\[\begin{tikzcd}
	B\big(C_n^{\fr}(V\lbr-\rbr),\underline{\Gap_{\sur}},\Map_{\underline{\Gap_\sur}}(-,\lbr p \rbr)\big)\rar{\simeq}\dar{\simeq}&C_n^{\fr}(V\lbr p\rbr)\dar\\
	B\big(C_n^{\fr}(V\lbr-\rbr),\underline{\Gap_{\sur}},\Map_{\underline{\Gap_\sur^\rhd}}(-,\lbr \infty \rbr)_p\big)\rar&\wall_p^V
\end{tikzcd}\]
whose top horizontal map is induced by composition and evaluation. The latter is a weak equivalence for the same reason as \eqref{equ:bar-resolution}. The right vertical map is induced by the function $\tau\colon [p]\ra (-\epsilon,\epsilon)$ and the embedding $e$, and is easily seen to be an equivalence as well: use that it lands in the deformation retract of $\wall_p^V\subset \wall_p\times C_n^\fr(V)$ given by those pairs whose first coordinate agrees with $\tau$ (i.e.\,the space of framed configurations in the complement $V\backslash \cup_{j,i}\tau(j)\times V_i$) and that the vertical map is induced by the embedding $V\lbr p \rbr\hookrightarrow V\backslash \cup_{j,i}\tau(j)\times V_i$ obtained from $e$ which is an isotopy equivalence, so induces an equivalence on framed configuration spaces. It follows that the bottom horizontal map is an equivalence as well, as claimed.
	
To show that $\|\varepsilon\|$ is a weak equivalence, note that its fibre at a framed configuration $\vec{x}\in C_n^\fr(X)$ is the realisation of the nerve of the nonempty totally ordered topological poset of real numbers $t\in (\epsilon,\epsilon)$ such that $\{t\}\times V_i\subset V$ is disjoint from $\vec{x}$ for all $i=1,\ldots k$, so it is weakly contractible by the above observation. We now show that $\|\varepsilon\|$ is a microfibration, which will finish the proof because any microfibration with weakly contractible fibres is a weak equivalence by \cite[Lemma 2]{WeissClassifying}. The remaining task is thus to show that given commutative solid arrows as in
\[\begin{tikzcd}[ar symbol/.style = {draw=none,"\textstyle#1" description,sloped},	subset/.style = {ar symbol={\subset}},row sep=1cm,column sep=0.3cm]
	&D^i\times{\{0\}} \dar \rar{f} &[5pt] {\|\wall_\smallsquare^V\|} \dar{\parallel\varepsilon\parallel}\arrow[r,subset]&{\|\wall_\bullet\|\times C^\fr_n(V)} \\
	D^i \times \left[0,\delta\right]\arrow[urr,dashed,crossing over, end anchor={south west},"\widetilde{\psi}" pos=0.3]\arrow[r,subset] &D^i \times [0,1] \arrow[swap,r,"\psi"] & C^\fr_n(V)&
\end{tikzcd}\]
there is an $0<\delta\le 1$ for which a dashed lift as indicated exists. For this, we note that the necessary data to lift a framed configuration $\vec{x}\in C^\fr_n(V)$ to $\|\wall_\smallsquare^V\| \subset \|\wall_\smallsquare\|\times C^\fr_n(V)$ is a point $z\in\interior(\Delta^p)$ for some number $p\ge0$, a function $\tau\colon [p]\ra (-\epsilon,\epsilon)$ such that $\vec{x}$ is disjoint from $\{\tau(i)\}\times V_j\subset V$ for all $i$ and $j$. For any $\vec{x}'$ close enough to $\vec{x}$ the same data works, so for each $x\in D^i$ we get lifts $\psi(x,t)$ for $t\in[0,\delta_x]$ for some $0<\delta_x\le 1$, uses that the subspaces $V_i\subset V$ are closed. By compactness, we can find a uniform choice of $\delta_x$ for $x\in D^i$. This gives the lift.
\end{proof}

\subsection{Unitality}\label{step:functor-to-morita}The goal of this step is to prove the following proposition, which uses the terminology of \cref{sec:quasi-unital} and its variation from \cref{rem:quasi-unital-into-simplicial} \ref{enum:quasi-unital-into-simplicial-i}.

\begin{prop}\label{prop:quasi-unitality}The non-unital bordism category $\ncBordInf(d)^{\nonunital}\in\Cat_{\nonunital}(\CatInf)$ is quasi-unital and the following morphism of semisimplicial objects in $\CatInf$ is quasi-unital
\[
	E^{\geo}\colon \ncBordInf(d)^{\nonunital}\lra\Fun_{\Delta^\op}(\Delta^\op_{/[\bullet]},\ManInf_d^{\otimes}).
\]
\end{prop}

By the equivalence \eqref{equ:qu-is-good}, the non-unital double $\infty$-category $\ncBordInf(d)^{\nonunital}$ thus extends to a (unital) double $\infty$-category \[\ncBordInf(d)\in\Cat(\CatInf).\] The second part of the proposition together with \cref{rem:quasi-unital-into-simplicial} \ref{enum:quasi-unital-into-simplicial-ii} and \cref{lem:image-in-premorita} implies that the composition \eqref{equ:non-unital-composition} is quasi-unital in the sense of \cref{rem:quasi-unital-into-simplicial} \ref{enum:quasi-unital-into-simplicial-i}, so using the second part of this remark once more, together with \cref{prop:image-in-Morita}, we conclude that the functor of double $\infty$-categories $\ncBordInf(d)^{\nonunital}\ra \ALG(\PSh(\DiscInf_d))$ is quasi-unital and thus extends by the equivalence \eqref{equ:qu-is-good} essentially uniquely to a functor of double $\infty$-categories
\[
	E\colon \ncBordInf(d) \lra \ALG(\PSh(\DiscInf_d)).
\]

\begin{proof}[Proof of \cref{prop:quasi-unitality}]
This is tedious but straight-forward, so we avoid spelling out all details.  Recalling that $\ncBordInf(d)^{\nonunital}$ is the levelwise coherent nerve of a semisimplicial $\Kan$-enriched category $\ncBord(d)^\nonunital$, the quasi-unit is given by the coherent nerve of the simplicial functor $u\colon \ncBord(d)^\nonunital_{[0]}\ra \ncBord(d)^\nonunital)_{[1]}$ which sends a $[0]$-walled $d$-manifold $(W,\mu)$ to $(\bfR\times W|_{\mu(0)},\mu')$ with $\mu'(0)=\mu(0)$ and $\mu'(1)=\mu(0)+1$. On morphisms, it is induced by sending $\varphi\colon W|_{[\mu(0)-\epsilon,\mu(0)+\epsilon]}\ra W'|_{[\mu'(0)-\epsilon,\mu'(0)+\epsilon]}$ to $\id_\bfR\times \varphi_{0}$. 

To prove that the functor $E^{\geo}\colon \ncBordInf(d)^{\nonunital}\ra\Fun_{\Delta^\op}({\Delta^\op}_{/[\bullet]},\ManInf_d^{\otimes})$ is quasi-unital, recall that it was constructed as the coherent nerve of the zig-zag
\vspace{-0.1cm}
\[
\ncBord(d)^\nonunital_{[\bullet]} \xra{E^{\geo}_{[\bullet]}}\Fun_{\Gap}(\underline{\Gap}_{\lbr \bullet\rbr/},\Man_d^\otimes) 	\xlla{\simeq}\Fun_{\Gap}(\Gap_{\lbr\bullet\rbr/},\Man_d^\otimes) \cong \
	\Fun_{\Delta^{\op}}(\Delta^{\op}_{/[\bullet]},\Man_d^\otimes) 
\]
of semisimplicial objects in $\Kan$-enriched categories. We first construct the top horizontal functor in a commutative diagram of $\Kan$-enriched categories
\begin{equation}\label{equ:0th-degeneracy}
\begin{tikzcd}
	\underline{\Gap}_{\lbr 1\rbr/}\rar\dar[swap]{\simeq}& \underline{\Gap}_{\lbr 0\rbr/}\dar{\simeq}\\[-2pt]
	\Gap_{\lbr 1\rbr/}\rar{\iota^*}& \Gap_{\lbr 0\rbr/}
\end{tikzcd}
\end{equation}
where  $\iota\colon \lbr 0\rbr\ra\lbr 1\rbr$ is the unique morphism. On objects, the top arrow agrees with the bottom one. On morphisms, the top arrow is given by sending an embedding $\overline{\gamma}\colon \wlab_{\alpha}(\bfR)|^{\gamma^{-1}\lbr \mathring{q'}\rbr}\hookrightarrow \wlab_{\alpha'}(\bfR)$ to the unique dashed embedding that makes the diagram
\[\begin{tikzcd}[ar symbol/.style = {draw=none,"\textstyle#1" description,sloped},	subset/.style = {ar symbol={\supset}}]
	{\gamma^{-1}\lbr \mathring{q'}\rbr}\times\bfR&[-15pt]\arrow[l,subset]\wlab_{\alpha}(\bfR)|^{\gamma^{-1}\lbr \mathring{q'}\rbr}\arrow[d,hookrightarrow]\arrow[r,two heads]&\wlab_{\alpha\circ\iota}(\bfR)|^{\gamma^{-1}\lbr \mathring{q'}\rbr}\arrow[d,hookrightarrow,dashed]\arrow[r,equal]&[-15pt]{\gamma^{-1}\lbr \mathring{q'}\rbr}\times(-\epsilon,\epsilon)\\
	\lbr \mathring{q'}\rbr\times\bfR&\arrow[l,subset]\wlab_{\alpha'}(\bfR)\arrow[r,two heads]&\wlab_{\alpha'\circ\iota}(\bfR)\arrow[r,equal] &\lbr \mathring{q'}\rbr\times(-\epsilon,\epsilon)
\end{tikzcd}\]
commute where the bottom surjection is the identity if $\alpha'(1)\in\{L,R\}$ and otherwise the union of the identity over $\lbr \mathring{q'}\rbr\backslash\alpha'(1)$ with the map
\[
	\wlab_{\alpha'}(\bfR)|^{\alpha'(1)}=(-\epsilon,\epsilon]\sqcup [1-\epsilon,1+\epsilon)\xra{\tr_{-\epsilon}\sqcup \tr_{-(1-\epsilon)}}(-2\epsilon,2\epsilon)\xra{1/4}(-\epsilon,\epsilon)=\wlab_{\alpha'\circ\iota}(\bfR)|^{\alpha'(1)}
\]
over $\alpha'(1)$; the top arrow is defined in the same way by replacing $\alpha'$ by $\alpha$. 
Applying $\Fun_{\Gap}(-,\Man_d^\otimes)$ to \eqref{equ:0th-degeneracy} results in a commutative diagram of $\Kan$-enriched categories
\[\begin{tikzcd}
	\Fun_{\Gap}(\underline{\Gap}_{\lbr 0\rbr/},\Man_d^\otimes)\rar\dar[swap]{\simeq}& \Fun_{\Gap}(\underline{\Gap}_{\lbr 1\rbr/},\Man_d^\otimes)\dar{\simeq}\\[-2pt]
	\Fun_{\Gap}(\Gap_{\lbr 0\rbr/},\Man_d^\otimes)\rar& \Fun_{\Gap}(\Gap_{\lbr 1\rbr/},\Man_d^\otimes),
\end{tikzcd}\]
so $N_\coh(-)$ applied to the top arrow models the $0$th degeneracy map of $\Fun_{\Delta^{\op}}(\Delta^{\op}_{/[\bullet]},\ManInf_d^\sqcup)$. Using this model for the degeneracy and the above quasi-unit for $\ncBordInf(d)^{\nonunital}=N_{\coh}(\ncBord(d)^{\nonunital})$, it is tedious but straightforward to check that $N_\coh(E^{\geo}_\bullet)$ and thus $E^{\geo}$ is quasi-unital.
\end{proof}

\subsection{Symmetric monoidal structure}\label{step:symmetric-monoidal-structure}In this step we promote the functor of double $\infty$-categories $E\colon \ncBordInf(d) \ra \ALG(\PSh(\DiscInf_d))$ to a functor of \emph{symmetric monoidal} double $\infty$-categories (modelled as commutative monoid objects in $\Cat(\CatInf)$, see \cref{sec:monoidal-cats}). This is not difficult and essentially amounts to adding an index by a finite pointed set $\langle s\rangle\in\Fin_*$ to the previous steps. To avoid being too repetitive, we will not spell out all details.

\begin{nconvention}Given a space $X$, a map $\lambda\colon X\ra \langle s\rangle$ to $\langle s\rangle$, and a subset $A\subset \langle s\rangle$, we denote the preimage of $A$ by ${}^A|X\coloneqq \lambda^{-1}(A)$ to distinguish it from the notation $X|_A$ and $X|^A$ introduced in \cref{conv:epsilon-conventions} and \ref{step:mand}.
\end{nconvention}

\subsubsection*{\ref{step:bordismcat}': the bordism category}
We first extend $\ncBordInf(d)\in\Cat(\CatInf)$ to a symmetric monoidal non-unital double $\infty$-category $\ncBordInf(d)^{\nonunital}\in \CMon(\Cat_\nonunital(\CatInf))$ as follows: firstly, we extend the semisimplicial object $\ncBord(d)^{\nonunital}\in\Fun(\Delta^\op_\inj,\sCat)$ in $\Kan$-enriched categories to an object $\ncBord(d)^{\nonunital}\in\Fun(\Fin_*,\Fun(\Delta_\inj^\op,\sCat))=\Fun(\Fin_*\times \Delta_\inj^\op,\sCat)$; evaluation at $\langle 1\rangle\in\Fin_*$ recovers the previous construction. The value of $\ncBord(d)^{\nonunital}$ at $([p],\langle s\rangle)$ for $\langle s\rangle\in\Fin_*$ is the $\Kan$-enriched category ${\ncBord(d)^{\nonunital}}_{[p],\langle s\rangle}$ whose objects are $[p]$-walled $d$-manifolds $(W,\mu)$ together with a map $\lambda\colon W\ra \langle\mathring{s}\rangle$, which we think of as a way to decompose $W$ into disjoint summands indexed by $\langle\mathring{s}\rangle$. Morphisms from $(W,\mu,\lambda)$ to $(W',\mu',\lambda')$ are embeddings of $[p]$-walled manifolds that are additionally assumed to commute with the maps to $\langle \mathring{s}\rangle$. The functoriality of ${\ncBord(d)^{\nonunital}}_{[p],\langle s\rangle}$ in $p$ is defined as for ${\ncBord(d)^{\nonunital}}_{[p]}$, and that in $\langle s\rangle$ is for $\varphi\in \Fin_*(\langle s\rangle, \langle s'\rangle)$ on objects given by $\smash{(W,\mu,\lambda)\mapsto ({}^{\varphi^{-1}\langle \mathring{s}\rangle}|W,\mu,\varphi \circ \lambda)}$ and on morphisms by restricting embeddings. A mild extension of the proof of \cref{lem:bord-is-category-object} then shows that taking taking coherent nerves yields a commutative monoid object  in double $\infty$-categories, as wished.

\subsubsection*{\ref{step:mand}': the manifold category}
Next, we extend the monoidal $\infty$-category $\ManInf_d$ (thought of as a cocartesian fibration $\ManInf^{\otimes}_d\ra \Gap$) to an \emph{symmetric} monoidal $\infty$-category. It will be convenient to view it as a commutative monoid object in monoidal $\infty$-categories $\ManInf_d\in\CMon(\Mon(\CatInf))\subset \Fun(\Fin_*,\Fun(\Gap,\CatInf))$. To this end, we extend the construction of the functor $\Man^{\otimes}_d\ra \Gap$ of $\Kan$-enriched categories to yield $\Kan$-enriched functors ${\Man_d}^{\sqcup,\langle s\rangle}\ra \Gap$, one for each pointed set $\langle s\rangle\in\Fin_*$. Objects of $\smash{\Man^{\otimes,\langle s\rangle}_d}$ are now triples $(W,\lbr p\rbr,\lambda)$ of $\lbr p\rbr\in\Gap$, a smooth submanifold $W\subset\lbr\mathring{p}\rbr\times\bfR\times\bfR^\infty$ and a map $\lambda\colon W\ra \langle\mathring{s}\rangle$. The space of morphisms is defined as before, with the additional requirement that the embeddings have to commute with the reference maps to $\langle\mathring{s}\rangle$. Given a map $\varphi\colon \langle s\rangle\ra \langle s'\rangle$ in $\Fin_*$, there is a functor ${\Man_d}^{\sqcup,\langle s\rangle}\ra {\Man_d}^{\sqcup,\langle s'\rangle}$ over $\Gap$ which on objects is given by $(W,\lbr p\rbr,\lambda)\mapsto(^{\varphi^{-1}\langle \mathring{s'}\rangle}|W,\lbr p\rbr,\varphi\circ \lambda)$
 and on morphisms is induced by restriction. This yields a functor from $\Fin_*$ to cocartesian fibrations over $\Gap$. Using straightening and taking coherent nerves then gives the desired commutative monoid object in monoidal $\infty$-categories.

\subsubsection*{\ref{step:functor-to-premorita}': from the bordism category to the pre-Morita category of manifolds}
By the discussion in \cref{sec:morita-functoriality}, taking pre-Morita categories of $\ManInf_d\in\CMon(\Mon(\CatInf))$ yields a commutative monoid object $\smash{\overline{\ALG}(\ManInf_d)\in\CMon(\Fun(\Delta^{\op},\CatInf))}$, and our next task is to upgrade the morphism $E^{\geo}\colon \ncBordInf(d)^{\nonunital}\ra \overline{\ALG}(\ManInf_d)$ in $\smash{\Fun(\Delta^{\op}_\inj,\CatInf)}$ from \ref{step:functor-to-premorita} to a morphism in $\CMon(\Fun(\Delta_\inj^{\op},\CatInf))$. To do this, we first define for each $\langle s\rangle\in\Fin_*$ a variant $\smash{E^{\geo}_{[\bullet],\langle s\rangle}\colon (\ncBord(d)^{\nonunital})_{[\bullet],\langle s\rangle}\ra \Fun_{\Gap}(\underline{\Gap}_{\lbr\bullet\rbr/},\Man_d^{\otimes,\langle s\rangle})}$ in $\Fun(\Delta_\inj^\op,\sCat)$ of \eqref{equ:psi-simplicially}. For this, note that in the notation of Substep \ref{step:functor-to-premorita-language} I, projection on $\bfR\times \bfR^\infty$ gives a map $\lab_\alpha(W,\mu)\ra W$ for any $[p]$-walled manifold, so if $W$ comes with a map to $\langle\mathring{s}\rangle$, then so does $\lab_\alpha(W,\mu)$. Based on this observation, the construction of $\smash{E^{\geo}_{[\bullet]}}$ from \eqref{equ:psi-simplicially} directly generalises to a functor $\smash{E^{\geo}_{[\bullet],\langle s\rangle}}$ as desired by incorporating the maps to $\langle s\rangle$. Varying $s$, the maps $\smash{E^{\geo}_{[\bullet],\langle s\rangle}}$ define a morphism in $\Fun(\Fin_*,\Fun(\Delta_\inj^{\op},\sCat))$. Taking coherent nerves gives desired extension of $E^{\geo}$ to a morphism in the full subcategory $\CMon(\Fun(\Delta_\inj^{\op},\CatInf))\subset \Fun(\Fin_*,\Fun(\Delta_\inj^{\op},\CatInf))$,
\begin{equation}\label{equ:psi-monoidal}
	E^{\geo}\colon \ncBordInf(d)^{\nonunital}\lra \overline{\ALG}(\ManInf_d).
\end{equation}

\subsubsection*{\ref{step:composite}': composite algebras}
We claim that the two functors \begin{equation}\label{equ:extend-functor-monoidal}\ManInf_d\lra \PSh(\ManInf_d)\lra\PSh(\DiscInf_d)\end{equation} extend to  morphisms in $\CMon(\CatInf)\simeq \CMon(\Mon(\CatInf))$ (see \cref{rem:comm-mon-iterative}). For the first map, we discussed this in \cref{sec:presheaf-category}. A restriction map on presheaves such as the second map in \eqref{equ:extend-functor-monoidal} is only lax symmetric monoidal in general, but turns out to be actually monoidal in our case:

\begin{lem}\label{lem:iota-monoidal} The lax symmetric monoidal functor $\PSh(\ManInf_d) \to \PSh(\DiscInf_d)$ induced by restriction along the inclusion $\iota^* \colon \DiscInf_d\hookrightarrow \ManInf_d$ is strong monoidal.
\end{lem}

\begin{proof}By the formula for Day convolution it suffices to verify that for finite sets $S$ the inclusion $\smash{(\DiscInf_d \times \DiscInf_d)^\op_{S \times \bfR^d/} \subset (\ManInf_d \times \ManInf_d)^\op_{S \times \bfR^d/}}$ is cofinal (recall the convention to take slices before opposition). By \cite[4.1.3.1]{LurieHTT} it suffices to prove that  $\smash{((\DiscInf_d \times \DiscInf_d)^\op_{S \times \bfR^d/})_{/u}}$ has a terminal object for all triples $(M,M',u)$ of $M,M' \in \ManInf_d$ and $u \colon S \times \bfR^d \hookrightarrow M \sqcup M'$. Such a terminal object is given by the factorisation $\smash{S \times \bfR^d=T \times \bfR^d\sqcup T' \times \bfR^d\xrightarrow{u} M}$ where the decomposition $S=T\sqcup T'$ is so that $T \times \bfR^d = u^{-1}(M)$ and $T' \times \bfR^d = u^{-1}(M')$.
\end{proof}

After applying $\overline{\ALG}(-)$ to \eqref{equ:extend-functor-monoidal} this gives a composition of morphisms in $\CMon(\Fun(\Delta^{\op}_\inj,\CatInf))$ (see \cref{sec:morita-functoriality}) which we may precompose with \eqref{equ:psi-monoidal} to arrive at an enhancement of \eqref{equ:non-unital-composition} to \begin{equation}\label{equ:non-unital-composition-monoidal} 
	\overline{E} \colon \ncBordInf(d)^{\nonunital}\xlra{E^{\geo}}\overline{\ALG}(\ManInf_d)\xlra{\yon_*}\overline{\ALG}(\PSh(\ManInf_d))\xlra{\iota^*}\overline{\ALG}(\PSh(\DiscInf_d))
\end{equation}
in $\CMon(\Fun(\Delta^{\op}_\inj,\CatInf))\subset \Fun(\Fin_*\times \Delta^{\op}_\inj,\CatInf))$. To show that this composition lands in the levelwise full subcategory $\ALG(\PSh(\DiscInf_d))\subset \overline{\ALG}(\PSh(\DiscInf_d))$ (which lies in the full subcategory $\CMon(\Cat(\CatInf))\subset \CMon(\smash{\Fun(\Delta^{\op}_\inj},\CatInf))$, see \cref{sec:morita-functoriality}), by the Segal property it suffices to show this after evaluation at $\langle 1\rangle \in\Fin_*$ where it agrees with the previously variant without symmetric monoidal structures for which we have already checked this property in \ref{step:composite}, so we obtain a map $\ncBordInf(d)^{\nonunital} \ra \ALG(\PSh(\DiscInf_d))$ in $\CMon(\Cat_{\nonunital}(\CatInf))$. Finally, a minor extension of the arguments of \ref{step:functor-to-morita} to incorporate indexing maps to finite sets enhances this to a functor of symmetric monoidal double $\infty$-categories	$E\colon \ncBordInf(d)\ra  \ALG(\PSh(\DiscInf_d))$.

\subsection{Variants}\label{step:variants}
We now define several variants of $\ncBordInf(d)$, related by a diagram 
\begin{equation}\label{equ:variants-of-ncbord}
\begin{tikzcd}[row sep=0.3cm]
	\gls*{borddcat}\dar\rar& \gls*{borddcatbdy}\dar\rar&\BordInf(d-1)\dar \\
	\ncBordInf(d)\rar &\gls*{borddcatbdync}\rar&\ncBordInf(d-1)
 \end{tikzcd}
\end{equation}
of symmetric monoidal double $\infty$-categories. Informally speaking, $\BordInf(d)$ is obtained from $\ncBordInf(d)$ by restricting to compact bordisms between closed manifolds and diffeomorphisms between them, the versions with a $(-)^{\partial}$-subscript allow manifolds to have boundary, all vertical maps and the left horizontal maps are induced by inclusion, and the right horizontal maps are induced by taking boundaries.

\subsubsection{Compact variant}
To define the compact variant $\BordInf(d)$, we say that a $[p]$-walled $d$-manifold $(W,\mu)$ is of \emph{of compact type} if the subspace $W|_{[\mu(0)-\epsilon,\mu(p)+\epsilon]} \subseteq W$ is compact. Restricting to $[p]$-walled $d$-manifolds of compact type in the construction of $\ncBordInf(d)$ and to spaces of diffeomorphisms instead of embeddings defines the symmetric monoidal double $\infty$-category $\BordInf(d)$. By construction, it comes with a levelwise subcategory inclusion into $\ncBordInf(d)$. This is the leftmost vertical map in \eqref{equ:variants-of-ncbord}.

\subsubsection{Variants with boundary}
To define the variant $\ncBordInf(d)^{\partial}$ of $\ncBordInf(d)$ involving manifolds with boundary, we replace $[p]$-walled $d$-manifolds $(W,\mu)$ where $W\subset \bfR\times \bfR^{\infty}$ is required to have no boundary, by \emph{$[p]$-walled $d$-manifolds with boundary}: these are pairs $(W,\mu)$ of a smooth submanifold $W\subset \bfR\times[0,\infty)\times\bfR^{\infty}$, possibly with boundary, together with an order-perserving function $\mu\colon [p]\ra\bfR$ such that
\begin{enumerate}
	\item $(W,\mu)$ satisfies the conditions in the definition of $[p]$-walled $d$-manifolds (see \ref{step:bordismcat}),
	\item\label{enum:boundary-condition} $\partial W=W\cap (\bfR\times\{0\}\times\bfR^{\infty})$ such that $W\cap (\bfR\times[0,\epsilon]\times\bfR^{\infty})=\partial W\times [0,\epsilon]$ under the appropriate identifications.
\end{enumerate}
The space $\Emb((W,\mu),(W',\mu'))$ of embeddings of $[p]$-walled $d$-manifolds with boundary is defined in the same way as in the case without boundary, except that we demand in addition that the embedding $\varphi\colon W|_{[\mu(0)-\epsilon,\mu(p)+\epsilon]}\hookrightarrow W'|_{[\mu'(0)-\epsilon,\mu'(p)+\epsilon]}$ also satisfies
\begin{enumerate}
	\item $\varphi^{-1}(\bfR\times [0,\epsilon]\times\bfR^{\infty})=(W|_{[\mu(0)-\epsilon,\mu(p)+\epsilon]})\cap (\bfR\times [0,\epsilon]\times\bfR^{\infty})$,
	\item\label{enum:emb-collared} under the appropriate identifications, $\varphi$ restricts to an embedding of the form 
	\[
		(\partial \phi\times \id_{[0,\epsilon]})\colon \partial W|_{[\mu(0)-\epsilon,\mu(p)+\epsilon]}\times [0,\epsilon]\longhookrightarrow  \partial W'|_{[\mu'(0)-\epsilon,\mu'(p)+\epsilon]}\times [0,\epsilon]
	\] 
	for some embedding $\partial \varphi\colon \partial W|_{[\mu(0)-\epsilon,\mu(p)+\epsilon]}\hookrightarrow \partial W'|_{[\mu(0)-\epsilon,\mu(p)+\epsilon]}$.
\end{enumerate}
Replacing the $[p]$-walled $d$-manifolds in the construction of $\ncBordInf(d)$ by $[p]$-walled $d$-manifolds with boundaries in the sense just described gives rise to a symmetric monoidal double $\infty$-category $\ncBordInf(d)^{\partial}$ which receives a levelwise full subcategory inclusion from $\ncBordInf(d)$, induced by the inclusion $\bfR\times\bfR^{\infty} \cong \bfR\times\{1\}\times \bfR^{\infty}\subset \bfR\times[0,\infty)\times \bfR^{\infty}$. This inclusion restricts to a functor $\BordInf(d)\ra \BordInf(d)^{\partial}$ where $\BordInf(d)^{\partial}$ is the symmetric monoidal double $\infty$-category given as the levelwise  subcategory of $\ncBordInf(d)^{\partial}$ obtained by restricting to $[p]$-walled $d$-manifolds with boundary of compact type, defined by the same condition as for the variant without boundary, and to diffeomorphisms between them instead of embeddings. This explains \eqref{equ:variants-of-ncbord}, except for the horizontal functor of the right square which are induced by sending a $[p]$-walled $d$-manifold with boundary $W\subset \bfR\times[0,\infty)\times \bfR^\infty$ to its boundary $\partial W=W\cap(\bfR\times\{0\}\times \bfR^\infty)$, with the same walls, and restricting embeddings to the boundary. 

\subsubsection{Tangential structures without boundary}\label{sec:tangential-no-bdy}
Associating to a smooth manifold $M$ its frame bundle $\Fr(M)$ with its canonical right $\GL_d(\bfR)$-action, induces a functor of $\Kan$-enriched categories 
\begin{equation}\label{equ:simplicial-frame-bundle}
	(\Man_d^{\otimes})_{[1]}\lra \Fun(\GL^{\op}_d,\cat{S})^{\circ},
\end{equation}
where $\Man_d^{\otimes}$ is the symmetric monoidal $\infty$-category from \ref{step:mand} and \ref{step:symmetric-monoidal-structure}, and $\GL_d$ is the (singular simplicial set of) the topological group $\GL_d(\bfR)$ viewed as a $\Kan$-enriched groupoid with one object. The superscript $(-)^\circ$ indicates that we pass to the full subcategory on the fibrant-cofibrant objects in the projective model structure on $\Fun(\GL^\op_d,\cat{S})$, as in \cite[A.3.3.2]{LurieHA}. Let us explain why functor $\Fr(-)$ takes values in this subcategory. Firstly $\Fr(M)$ is fibrant: in the projective model structure an object is fibrant if its underlying simplicial set is a Kan complex, and this is the case for $\Fr(M)$ as a singular simplicial set of a topological space. Secondly $\Fr(M)$ is cofibrant: because the map $\Fun(\GL_d^\op,\cat{S}) \to \cat{S}$ that forgets the action is the right adjoint in a Quillen adjunction, each map $\GL_d(\bfR) \times S \to \GL_d(\bfR) \times S'$ with canonical right $\GL_d(\bfR)$-action and $S \to S'$ a monomorphism is a cofibration, and $\Fr(M)$---being locally trivial---is isomorphic to a (possibly transfinite) composition of pushouts against such maps.

Applying coherent nerves to the map \eqref{equ:simplicial-frame-bundle} and viewing $\GL_d$ as an $\infty$-category via the coherent nerve, gives a functor of $\infty$-categories
\[
	(\ManInf_d^{\otimes})_{[1]}\simeq N_{\coh}((\Man_d^{\otimes})_{[1]})\lra N_{\coh}(\Fun(\GL^{\op}_d,\cat{S})^{\circ})\simeq \Fun(N_{\coh}(\GL^{\op}_d),\cS)=\PSh(\GL_d)
\]
where the second equivalence is an instance of \cite[4.2.4.4]{LurieHTT}. Since the unit $\varnothing\in (\ManInf_d^{\otimes})_{[1]}$ is initial and so $\ManInf_d^\sqcup$ is unital as an $\infty$-operad  \cite[2.3.1.1]{LurieHA}, this functor extends uniquely to a lax symmetric monoidal functor $\Fr(-)\colon \ManInf^{\otimes}\ra \PSh(\GL_d)^\sqcup$ where $\PSh(\GL_d)$ carries the cocartesian symmetric monoidal structure \cite[2.4.3.9]{LurieHA}. Note that $\Fr(M)\sqcup \Fr(N)\ra \Fr(M\sqcup N)$ is an equivalence for manifolds $M$ and $N$, so this is actually (strong) symmetric monoidal.

By an easier version of the argument in \ref{step:composite}, the composition
\[
	\ncBordInf(d)\xlra{E^{\geo}}\overline{\ALG}(\ManInf_d)\xrightarrow{\overline{\ALG}(\Fr(-))}\overline{\ALG}(\PSh(\GL_d))
\]
lands in the Morita double $\infty$-category $\ALG(\PSh(\GL_d))\subset \overline{\ALG}(\PSh(\GL_d))$, which is equivalent to $\COSPAN^+(\PSh(\GL_d))$ (see \cref{sec:alg-cospans}). We thus arrive at a functor of symmetric monoidal double $\infty$-categories $\Fr(-)\colon \ncBordInf(d)\ra\COSPAN^+(\PSh(\GL_d))$.
Informally, this is given by sending a bordism $W\colon P \leadsto Q$ to the cospan $\Fr(c(P))\ra \Fr(W)\la\Fr(c(Q))$, where $c(P),c(Q)\subset W$ are collar neighbourhoods of the boundary components. 

\begin{dfn}\label{dfn:bord-tang-structures} Given a \emph{tangential structure} $\gls*{theta}\in \PSh(\GL_d)$, we define $\ncBordInf^\theta(d)$ and $\BordInf^\theta(d)$ by the following pullbacks in symmetric monoidal double $\infty$-categories
\[\begin{tikzcd} 
	\gls*{borddcattheta} \dar\arrow[r,hookrightarrow]& \gls*{ncborddcattheta} \rar \dar &  \COSPAN^+(\PSh(\GL_d)_{/\theta}) \dar\\
	\BordInf(d)\arrow[r,hookrightarrow]&\ncBordInf(d) \rar{\Fr(-)} & \COSPAN^+(\PSh(\GL_d));
\end{tikzcd}\]
here the rightmost vertical map is induced by the forgetful functor $\PSh(\GL_d)_{/\theta}\ra \PSh(\GL_d)$ which preserves colimits \cite[1.2.13.8]{LurieHA} and thus induces a functor on cospan categories.
\end{dfn}

Varying $\theta$ induces functors $\ncBordInf^{(-)}(d),\BordInf^{(-)}(d)\colon  \PSh(\GL_d)\ra \CMon(\Cat(\CatInf))$. In particular, for a map $\theta\ra\theta'$ in $\PSh(\GL_d)$, we have functors
\begin{equation}\label{eqn:bord-theta-naturality}
	\BordInf^\theta(d) \lra \BordInf^{\theta'}(d)\quad\text{and}\quad\ncBordInf^\theta(d) \lra \ncBordInf^{\theta'}(d).
\end{equation}

\subsubsection{Tangential structures with boundary}
To define the version $\ncBordInf(d)^\partial$ that includes tangential structure, one uses a variant 
\begin{equation}\label{equ:egeo-with-boundary}
	E^{\geo}\colon \ncBordInf(d)^\partial\lra \overline{\ALG}(\ManInf^\partial_d)
\end{equation} 
of the map $E^{\geo}\colon \ncBordInf(d)\ra \overline{\ALG}(\ManInf_d)$ between commutative monoid objects in simplicial $\infty$-categories. The symmetric monoidal $\infty$-category $\ManInf^\partial_d$ is defined in the same way as $\ManInf_d$ except that we use submanifolds $W\subset \lbr \mathring{p}\rbr\times\bfR\times[0,\infty)\times\bfR^\infty$ that may have boundary, but have to satisfy the evident analogue of \ref{enum:boundary-condition} in the definition of a $[p]$-walled $d$-manifold with boundary. With this modification, the construction in \ref{step:functor-to-premorita} and its extensions in \ref{step:functor-to-morita} and \ref{step:symmetric-monoidal-structure} extends almost verbatim to give the map \eqref{equ:egeo-with-boundary} in $\smash{\CMon(\Fun(\Delta^{\op},\CatInf))}$.

Assigning to a manifold $W\in (\ManInf^{\partial,\otimes}_d)_{[1]}$ the map $\Fr(\partial W\times[0,\epsilon])\ra \Fr(W)$ induced by the inclusion induces an extension of the functor $(\ManInf_d^\sqcup)_{[1]}\ra\PSh(\GL_d)$ to a functor of $\infty$-categories 
\[
	(\ManInf^{\partial,\otimes}_d)_{[1]}\lra \Fun([1]\times N_{\coh}(\GL_d^\op),\cS)\eqcolon\PSh([1]\times \GL_d)
\] 
which, by the same argument as in the case without boundary, extends to a symmetric monoidal functor $\Fr(-)\colon \ManInf^{\partial,\otimes}_d\ra \PSh([1]\times \GL_d)^\sqcup$ where the target is equipped with the cocartesian symmetric monoidal structure. This functor allows us to extend \cref{dfn:bord-tang-structures} to define symmetric monoidal double $\infty$-categories 
\[
	\ncBordInf^{\theta}(d)^\partial \quad\text{ and }\quad\ncBordInf^{\theta}(d)^\partial
\]
for any \emph{tangential structure $\theta$ with boundary} by which mean a map $\theta=(\theta^\partial{\ra}\theta^\circ)\in \PSh([1]\times \GL_d)$.
 
\subsubsection{Taking boundaries with tangential structures}\label{sec:boudnary-functor-tangential-structure}
Next, we extend the ``taking-boundaries functors'' $\ncBordInf(d)^{\partial}\ra \ncBordInf(d-1)$ and $\BordInf(d)^{\partial}\ra \BordInf(d-1)$ from \eqref{equ:variants-of-ncbord} to include tangential structures. This involves the commutative diagram of $\infty$-categories
\begin{equation}\label{equ:res-ind-tang-structures}
\begin{tikzcd}
	(\ManInf_d^{\partial,\otimes})_{[1]}\dar\rar&\PSh([1]\times \GL_d)\arrow[r,equal]&\PSh([1]\times \GL_d)\dar{\res}\\
	(\ManInf^\otimes_{d-1})_{[1]}\rar&\PSh(\GL_{d-1})\rar{\ind_{d-1}^d}&\PSh(\GL_{d})
\end{tikzcd}
\end{equation}
where the leftmost vertical map is induced by sending a submanifold $W\subset \lbr \mathring{p}\rbr\times\bfR\times[0,\infty)\times\bfR^\infty$ to its boundary, i.e.\,the intersection with $\lbr\mathring{p}\rbr\times\bfR\times\{0\}\times\bfR^\infty$. The arrow labelled $\res$ is induced by precomposition with the inclusion $\{1\}\times\GL_d\subset [1]\times\GL_d$ and arrow labelled $\ind_{d-1}^d$ is the left adjoint to the functor $\res_{d-1}^d\colon \PSh(\GL_{d})\ra \PSh(\GL_{d-1})$ induced by precomposition with the inclusion $\GL_{d-1}(\bfR)\subset\GL_{d}(\bfR)$ using the first $(d-1)$-coordinates. One way to provide the commutativity of \eqref{equ:res-ind-tang-structures} is to recognise this diagram as the coherent nerve of a diagram of $\Kan$-enriched categories (using \cite[5.2.4.6]{LurieHTT} for $\ind_{d-1}^d$) and then use the fact that the extension $\ind_{d-1}^d(\Fr(\partial W))\ra \Fr(\partial W\times[0,\epsilon])$ of the $\GL_{d-1}(\bfR)$-equivariant map $\Fr(\partial W)\ra\Fr(\partial W\times[0,\epsilon])$ induced by the inclusion $\partial W\times\{0\}\subset \partial W\times[0,\epsilon]$ and the canonical non-zero vector field on $[0,\epsilon]$ is an equivalence of $\GL_d(\bfR)$-spaces natural in $W$.

Equipping all categories of presheaves with the cocartesian symmetric monoidal structure and using the universality property as in \ref{sec:tangential-no-bdy}, we can extend \eqref{equ:res-ind-tang-structures} to a commutative diagram of symmetric monoidal $\infty$-categories. Applying $\overline{\ALG}(-)$, using the $E^\geo$-functors, and the equivalence $\ALG(\cC)\simeq\COSPAN^+(\cC)$ for cocartesian $\cC$, this leads to a commutative diagram of symmetric monoidal double $\infty$-categories
\[
	\begin{tikzcd}[column sep=0.3cm]
	\BordInf(d)^\partial\arrow[r,hookrightarrow]\dar&\ncBordInf(d)^\partial\dar\rar&\COSPAN^+(\PSh([1]\times \GL_d))\arrow[r,equal]&\COSPAN^+(\PSh([1]\times \GL_d))\dar{(\res)_*}\\
	\BordInf(d-1)^\partial\arrow[r,hookrightarrow]&\ncBordInf(d-1)\rar&\COSPAN^+(\PSh(\GL_{d-1}))\rar{(\ind_{d-1}^d)_*}&\COSPAN^+(\PSh(\GL_{d}))
	\end{tikzcd}
\]
For a tangential structure with boundary $\theta=(\theta^\partial{\ra}\theta^\circ)\in \PSh([1]\times \GL_d)$, this induces extensions
\[
	\BordInf^{\theta}(d)^\partial\lra \BordInf^{\res_{d-1}^d(\theta^\partial )}(d-1)\quad\text{and}\quad \ncBordInf^{\theta}(d)^\partial\lra \ncBordInf^{\res_{d-1}^d(\theta^\partial )}(d-1)
\]
of the ``taking boundaries'' functors from \eqref{equ:variants-of-ncbord}.

\begin{ex}\label{ex:framing-to-oneframing}
The tangential structure with boundary encoding framings is $\fr\coloneq (\id\colon \GL_d(\bfR)\ra\GL_d(\bfR))$, so the above in particular gives a functor of symmetric monoidal double $\infty$-categories
$\BordInf^\fr(d)^{\partial}\ra \BordInf^{\onefr}(d-1)^{\partial}$
from the compact framed $d$-dimensional bordism category with boundary to the $d$-dimensional bordism category with boundary and the tangential structure $\onefr\coloneq \res_{d-1}^d(\GL_{d}(\bfR))$ encodes framings of the once-stablised tangent bundle. 
\end{ex}

\subsection{Product functors}\label{step:product}
Given a smooth $p$-manifold $P$, possibly with boundary, we now explain the construction of a ``taking products'' functor of symmetric monoidal double $\infty$-categories
\begin{equation}\label{equ:product-functor-noncompact-bord}
	(P\times-)\colon \ncBordInf(d)^\partial\lra \ncBordInf(d+p)^\partial,
\end{equation}
which restricts to product functors of the form 
\[
	\ncBordInf(d)\ra \ncBordInf(d+p),\quad\BordInf(d)^\partial\ra \BordInf(d+p)^\partial,\quad \text{and} \quad\BordInf(d)\ra \BordInf(d+p)
\]
if $P$ has no boundary, is compact, or is closed, respectively. This will involve smoothing corners.

We fix an embedding $P\subset [0,\infty)\times\bfR^N$ for some $N\ge0$ which satisfies the condition \ref{enum:boundary-condition} in the definition of a $[p]$-walled $d$-manifolds with boundary (ignoring the first $\bfR$-factor). Furthermore, we fix once and for all a homeomorphism $\psi\colon [0,\infty)\times[0,\infty)\ra[0,\infty)\times\bfR$ such that
\begin{enumerate}
	\item \label{enum:psi-boundary} $\psi$ agrees with the identity on $[0,\infty)\times \{0\}$ and with the clockwise rotation by $\pi/2$ on $\{0\}\times [0,\infty)$. In particular, it fixes the origin.
	\item \label{enum:psi-diff} $\psi$ is a diffeomorphism away from the origin.
	\item \label{enum:psi-collar} $\psi^{-1}([0,\epsilon]\times\bfR)\subset \big([0,\epsilon]\times[0,\infty)\cup [0,\infty)\times [0,\epsilon]\big)$,
	\item \label{enum:psi-delta} $\psi([0,\delta]\times[0,\infty)\cup [0,\infty)\times [0,\delta])\subset [0,\epsilon]\times \bfR$ for some fixed $0<\delta\le \epsilon$
	\item \label{enum:psi-1-1} $\psi$ fixes the point $(1,1)$.
\end{enumerate}
Using $\psi$ and its properties \ref{enum:psi-boundary}--\ref{enum:psi-collar}, given a $[p]$-walled $d$-manifold with boundary $(W,\mu)$, we obtain a $[p]$-walled $(d+p)$-manifold with boundary $(\Psi(P\times W),\mu)$ with $\Psi(P\times W)$ the image of $P\times W$ under the composition
\begin{align*}
	&[0,\infty)\times\bfR^N\times \bfR\times[0,\infty)\times\bfR^\infty  \xrightarrow{\text{swap}}\bfR\times [0,\infty)\times [0,\infty)\times\bfR^N\times \bfR^\infty\\
	&\qquad \xrightarrow{\id_\bfR\times \psi\times \id_{\bfR^N \times \bfR^\infty}} \bfR\times [0,\infty)\times \bfR \times\bfR^N\times \bfR^\infty \xrightarrow{\id_{\bfR \times [0,\infty)} \times \text{shift}}\bfR\times [0,\infty)\times\bfR^\infty
\end{align*}
where the first map swaps the right $[0,\infty) \times \bfR^N$-factor with the middle $\bfR\times[0,\infty)$-factor. For instance, condition \ref{enum:psi-collar} is used to ensure the condition $\Psi(P\times W)\cap (\bfR\times[0,\epsilon]\times\bfR^\infty)=\partial(\Psi(P\times W))\times[0,\epsilon]$ in the definition of a $[p]$-walled manifold with boundary. Note that $\Psi(P\times W)$ comes with a preferred homeomorphism $P\times W\cong\Psi(P\times W)$ which is a diffeomorphism away from $\partial P\times\partial W$ as a consequence of condition \ref{enum:psi-diff}. Taking products with $P$ and conjugating with $\psi$ induces a map 
\[
	\Emb\big((W,\mu),(W',\mu')\big)\lra \Emb\big((\Psi(P\times W),\mu),(\Psi(P\times W'),\mu')\big),
\]
which is well-defined due to the collaring condition \ref{enum:emb-collared} on embeddings between $[p]$-walled $d$-manifolds with boundaries. Going through the construction of $\ncBordInf(d)^\partial$, one checks that the assignment $(W,\mu)\mapsto (\Psi(P\times W),\mu)$ together with the maps between embedding spaces just discussed leads to functors as desired.

These product functors can be extended to include tangential structures. To this end, one notes that there is a functor of symmetric monoidal categories $(P\times-)\colon \ManInf_d^\partial\ra\ManInf_{p+d}^\partial$ defined as for \eqref{equ:product-functor-noncompact-bord}. On underlying $\infty$-categories, this participates in a diagram of $\infty$-categories
\begin{equation}\label{equ:product-with-tang-structures}
\begin{tikzcd}
	(\ManInf_d^{\partial,\otimes})_{[1]}\arrow[dd,"{(P\times-)}"]\rar&\PSh([1]\times\GL_d)\dar\\[-10pt]
	&\PSh([1]\times\GL_p\times\GL_d)\dar{\ind_{p,d}^{p+d}}\\
	(\ManInf_{p+d}^{\partial,\otimes})_{[1]}\rar&\PSh([1]\times\GL_{p+d})
\end{tikzcd}
\end{equation}
where the upper right vertical arrow is the functor that sends a map $X\ra Y$ of $\GL_d(\bfR)$-spaces to 
\[(\Fr(P)\times X)\cup_{\Fr(\partial P\times [0,\epsilon])\times X}(\Fr(\partial P\times [0,\epsilon])\times Y)\ra \Fr(P)\times Y\] viewed as a map of $(\GL_{p}(\bfR)\times\GL_d(\bfR))$-spaces, and the functor $\ind_{p,d}^{p+d}$ is the left adjoint to the restriction along the inclusion $\GL_p(\bfR)\times\GL_d(\bfR)\subset \GL_{p+d}(\bfR)$. \eqref{equ:product-with-tang-structures} can be extended to a \emph{commutative} square of $\infty$-categories in a way similar to what we did for \eqref{equ:res-ind-tang-structures}: recognise it as the coherent nerve of a diagram of $\Kan$-enriched categories and then use that the two compositions are related by a zig-zag of natural equivalences. In this case, the zig-zag is provided by the commutative diagram
\[\hspace{-.3cm} \begin{tikzcd}[column sep=-0.2cm, row sep=0.3cm]
	\res^{p+d}_{p,d}\Fr(\partial \Psi(P\times W)\times[0,\epsilon])\rar&\res^{p+d}_{p,d}\Fr(\Psi(P\times W))\\
	(\Fr(\interior(P))\times \Fr(c(W))))\cup_{\Fr(c(P))\times \Fr(c(W))}(\Fr(c(P))\times \Fr(\interior(W))\uar\rar\dar&\Fr(\interior(P))\times\Fr(\interior(W))\uar\dar\\
	(\Fr(P)\times \Fr(\partial W\times[0,\epsilon]))\cup_{\Fr(\partial P\times [0,\epsilon]))\times \Fr(W\times [0,\epsilon])}(\Fr(\partial P\times[0,\epsilon])\times \Fr(W)\rar&\Fr(P)\times\Fr(W)
\end{tikzcd}\]
of $(\GL_{p}(\bfR)\times\GL_{d}(\bfR))$-spaces which is natural in $W$ and consists of vertical equivalences when taking adjoints with respect to the $(\smash{\ind_{p,d}^{p+d}},\smash{\res_{p,d}^{p+d}})$-adjunction. Here $c(P)\coloneq \partial P\times(0,\delta)\subset \interior(P)$ and $c(W)\coloneq \partial W\times(0,\delta)\subset \interior(P)$, the lower vertical arrows are induced by the inclusions $\interior(P)\subset P$ and $\interior(W)\subset W$, and the upper vertical arrows by the preferred embedding $\interior(P)\times\interior(W)\hookrightarrow \Psi(P\times W)$ induced by $\psi$; this uses property \ref{enum:psi-delta} of $\psi$. Similarly to the final paragraph of \ref{sec:boudnary-functor-tangential-structure}, \eqref{equ:product-with-tang-structures} yields a commutative diagram of symmetric monoidal double $\infty$-categories
\begin{equation}\label{equ:products-with-tang-structure-bord}
\begin{tikzcd}
	\ncBordInf(d)^{\partial}\arrow[d,"P\times(-)",swap]\rar&\COSPAN^+(\PSh([1]\times\GL_d))\dar\\
	\ncBordInf(p+d)^{\partial}\rar&\COSPAN^+(\PSh([1]\times\GL_{p+d})).
\end{tikzcd}
\end{equation}
Now given a tangential structure with boundary $\lambda=(\lambda^\partial\ra \lambda^\circ)\in \PSh([1]\times\GL_p)$, and a $\lambda$-structure on $P$ in the form of a map $\ell_P\colon (\Fr(P),\Fr(\partial P\times [0,\epsilon]))\ra (\lambda^\circ,\lambda^{\partial})$ in $\PSh([1]\times\GL_p)$, then \eqref{equ:products-with-tang-structure-bord} induces a functor of symmetric monoidal double $\infty$-categories
\[
	((P,\ell_P)\times(-))\colon \ncBordInf^{\theta}(d)^\partial\lra\ncBordInf^{\glue(\theta,\lambda)}(p+d)^\partial
\]
where $\glue(\theta,\lambda)\coloneq \ind_{p,d}^{p+d}\big(\lambda^\circ\times\theta^\partial\cup_{\lambda^\partial\times\theta^\partial}\lambda^\partial\times\theta^\circ \ra \lambda^\circ\times\theta^\circ\big)\in \PSh([1]\times\GL_{p+d})$. This also extends the variants of the product functors mentioned below \eqref{equ:product-functor-noncompact-bord}, where property \ref{enum:psi-1-1} of $\psi$ is used for the variants without boundary.

\begin{ex}\label{ex:product-functor-with-framing}
In the case of framings $\fr_p=\lambda=(\id\colon \GL_p(\bfR)\ra \GL_p(\bfR))$ and $\fr_d=\theta=(\id\colon \GL_d(\bfR)\ra \GL_d(\bfR))$, we have $\glue(\fr_p,\fr_d)\simeq \fr_{d+p}$, so omitting the subscripts, we have a product functor of symmetric monoidal double $\infty$-categories $((P,\ell_P)\times(-))\colon\ncBordInf^\fr(d)^\partial\ra\ncBordInf^\fr(p+d)^\partial$
for framed $p$-manifolds $P$, and similarly for the compact variants.
\end{ex}
 
\renewcommand\thesubsection{\thesection.\arabic{subsection}}

\section{Properties of $E$, embedding calculus, and $\DiscInf$-structure spaces}\label{sec:functor-e-disc-structure}  
The main outcome of the previous section is the construction of a functor 
\[
	E\colon \ncBordInf(d)\lra \gls*{modd} \coloneqq \ALG(\PSh(\DiscInf_d))
\]
of symmetric monoidal double $\infty$-categories, in the sense of \cref{sec:monoidal-cats}, from a bordism category of (possible noncompact) $(d-1)$-manifolds to a Morita category on the category $\PSh(\DiscInf_d)$ of presheaves on a category $\DiscInf_d$ of finite disjoint unions of $d$-dimensional Euclidean spaces and codimension $0$ embeddings between them. We also constructed variants $\BordInf(d)$, $\BordInf(d)^\partial$, and $\ncBordInf(d)^\partial$ of $\ncBordInf(d)$, related by a diagram of symmetric monoidal double $\infty$-categories \eqref{equ:variants-of-ncbord}, as well as enhancements with tangential structures of all of these bordism categories.

This section has several purposes: firstly in \cref{sec:bordism-mapping-oo-cats}, we give more practical descriptions of these double $\infty$-categories by describing their objects and mapping $\infty$-categories in a model-independent and more intuitive manner, and we explain the functor $E$ in these terms. For most of the arguments in the later sections, this discussion is sufficient and there is no need to know the specifics of the construction in \cref{sec:the-functor}. Secondly, we establish three properties of the functor $E$: 
\begin{itemize}
	\item a descent property in \cref{sec:descent}, 
	\item a close relationship to Goodwillie--Weiss' embedding calculus in \cref{sec:embedding-calculus}, and 
	\item an isotopy extension property in \cref{sec:isotopy-extension}.
\end{itemize} 
Finally, in \cref{sec:disc-structure-spaces}, we give the precise definition of the $\DiscInf$-structure spaces.

\subsection{Mapping $\infty$-categories}\label{sec:bordism-mapping-oo-cats}
Recall from \cref{sec:mapping-infinity-category} that a double $\infty$-category $\cC$ has mapping $\infty$-categories $\cC_{A,B}$ for objects $A,B\in\cC_{[0]}$, and these feature in composition functors $\cC_{A,B}\times\cC_{B,C}\ra\cC_{A,C}$. We now spell these out for some of the double $\infty$-categories of the previous section.

\subsubsection{$\ncBordInf(d)$} \label{sec:details-ncbord} 
In short: objects of $\ncBordInf(d)$ are (possibly noncompact) $(d-1)$-manifolds $P$ without boundary, and given two such manifolds $P$ and $Q$, the objects of the mapping $\infty$-category $\ncBordInf(d)_{P,Q}$ are bordisms $W\colon P\leadsto Q$ and the mapping spaces in $\ncBordInf(d)_{P,Q}$ are given by embedding spaces relative to the boundary. The composition in these mapping $\infty$-categories is by composing embeddings, the composition functor
$\ncBordInf(d)_{P,Q} \times \ncBordInf(d)_{Q,R} \ra \ncBordInf(d)_{P,R}$ by gluing bordisms, and the symmetric monoidal structure by disjoint union.

More precisely, given a $(d-1)$-manifold $P$, we may use the weak Whitney embedding theorem to choose an embedding $P\subset\bfR^{\infty}$ and can thus view $P$ as a $[0]$-walled manifold $(\bfR\times P,\mu)$ in the sense of \ref{step:bordismcat} of \cref{sec:the-functor} (and hence as an object in $\ncBordInf(d)_{[0]}$) by setting $\mu(0)=0$. Moreover, it is easy to see that each object in $\ncBordInf(d)_{[0]}$ is equivalent to one of this form, so we will no longer distinguish between abstract $(d-1)$-manifolds and objects in $\ncBordInf(d)_{[0]}$. Similarly, given a bordism $W\coloneq P\leadsto Q$ between $(d-1)$-manifolds, we may embed it suitably collared in $[0,1]\times \bfR^\infty$ so that $((\infty,0]\times P\cup W\cup[1,\infty)\times Q,\mu)$ with $\mu(i)=i$ for $i=0,1$ is a $[1]$-walled manifold and thus an object in the mapping $\infty$-category $\ncBordInf(d)_{P,Q}$. Again, any object is equivalent to one of this form, so we will also no longer distinguish between abstract bordisms $P\leadsto Q$ and objects in $\ncBordInf(d)_{P,Q}$. The identification of the mapping spaces in $\ncBordInf(d)_{P,Q}$ is justified by:

\begin{lem}\label{lem:mapping-space-emb}
Given possibly noncompact bordisms $W,W'\colon P\leadsto Q$ between $(d-1)$-manifolds $P,Q$ without boundary, there is a natural equivalence $	\Map_{\ncBordInf(d)_{P,Q}}(W,W')\simeq\Emb_\partial(W,W')$ in $\cS$.
\end{lem}

\begin{proof}
Using that mapping spaces in a pullback of $\infty$-categories are pullbacks of the mapping spaces, and that coherent nerves of $\Kan$-enriched categories preserve mapping spaces, we see $\Map_{\ncBordInf(d)_{P,Q}}(W,W')$ is the fibre (i.e.\,pullback along the indicated inclusion of a point) in $\cS$ 
\vspace{-0.2cm}
\[
	\Map_{\ncBordInf(d)_{P,Q}}(W,W')=\fib_{\id}\big(\Emb(W,W')_w\xlra{\res} \Emb(P,P)\times \Emb(Q,Q) \big)
\]
where the subscript $(-)_w$ indicates that we restrict to the subspace of embeddings $e$ that in some fixed collars $P \times [0,1] \hookrightarrow W$ and $Q \times [-1,0] \hookrightarrow W$ have the form $\id\times e_P$ and $\id\times e_Q$ for self-embeddings $e_P$ and $e_Q$ of $P$ and $Q$ respectively. The map $\res$ is induced by restriction to $e_P$ and $e_Q$. It is not hard to see that this is a Kan fibration, so the fibre in $\cS$ agrees with the point-set fibre over $(\id,\id)$. The latter is $\Emb_{\partial}(W,W')$, so we obtain an equivalence as claimed.
\end{proof} 

\subsubsection{$\BordInf(d)$}\label{sec:details-bord}
Under the identification of the objects in $\ncBordInf(d)_{[0]}$ as $(d-1)$-manifolds $P$ without boundary, those in the subcategory $\BordInf(d)_{[0]}$ correspond to $(d-1)$-manifolds $P$ without boundary that are also compact. Similarly, the objects in the mapping $\infty$-categories of the levelwise subcategory $\BordInf(d)\subset \ncBordInf(d)$ correspond to \emph{compact} bordisms between closed manifolds. The morphism spaces in the mapping $\infty$-categories are given by spaces of diffeomorphisms fixing the boundary. Hence, unlike for $\ncBordInf(d)$, all mapping $\infty$-categories of $\BordInf(d)$ are $\infty$-groupoids and can be regarded as spaces. Thus, by the discussion of \cref{sec:double-vs-infty2}, not much is lost by applying $(-)^{(\infty,1)}$ and considering the symmetric monoidal $\infty$-category $\BordInf(d)^{(\infty,1)}$ with objects closed $(d-1)$-manifolds and mapping spaces
\[
	\textstyle{\BordInf(d)_{P,Q}\simeq\Map_{\BordInf^{(\infty,1)}(d)}(P,Q) \simeq \bigsqcup_{[W]} \BDiff_\partial(W)},
\]
where $[W]$ ranges over compact bordisms $W \colon P \leadsto Q$ up to diffeomorphism relative to the ends. Composition is by gluing bordisms and the symmetric monoidal structure by disjoint union.

\subsubsection{$\ncBordInf^\theta(d)$}\label{sec:details-tangential-bord} 
In short: given a tangential structure $\theta$ in the form of a $\GL_d(\bfR)$-space $\theta$, the objects of $\ncBordInf^\theta(d)$ are (possibly noncompact) $(d-1)$-manifolds $P$ with a $\theta$-structure on their once-stabilised tangent bundle, i.e.\,a $\GL_d(\bfR)$-equivariant map $\theta_P\colon \Fr(I\times N) \ra \theta$ where $\Fr(-)$ denotes the frame bundle and $I=[0,1]$. The objects of the mapping category $\smash{\ncBordInf^\theta}(d)_{(P,\theta_P),(Q,\theta_Q)}$ are bordisms with $\theta$-structures and the morphisms are $\theta$-embeddings, fixed on the boundary. The composition and monoidal structure is as in $\ncBordInf(d)$, but with the addition of $\theta$-structures.

To justify this, recall from \cref{sec:tangential-no-bdy} that the noncompact bordism category with $\theta$-structures is defined as the pullback of symmetric monoidal double $\infty$-categories
\[
	\ncBordInf^\theta(d) =\ncBordInf(d) \times_{\COSPAN^+(\PSh(\GL_d))} \COSPAN^+(\PSh(\GL_d)_{/\theta}),
\]
so the claimed description of the objects follows by using that forgetting symmetric monoidal structures preserves pullbacks and that pullbacks of double $\infty$-categories are computed levelwise. This also shows that the mapping $\infty$-categories are given by pullbacks of $\infty$-categories
\[\begin{tikzcd}
	\ncBordInf^\theta(d)_{(P,\theta_P),(Q,\theta_Q)} \rar \dar & \COSPAN^+(\PSh(\GL_d)_{/\theta})_{\theta_P,\theta_Q} \dar \\
	\ncBordInf(d)_{P,Q} \rar & \COSPAN^+(\PSh(\GL_d))_{\Fr(I \times P),\Fr(I \times Q)}
\end{tikzcd}\]
which justifies the description of the objects in $\ncBordInf^\theta(d)_{(P,\theta_P),(Q,\theta_Q)}$ when combined with the equivalence $\COSPAN^+(\cC)_{A,B}\simeq \cC_{A\sqcup B/}$ mentioned in \cref{sec:span-cospan-cats}. Combining this discussion with the fact that mapping spaces in a pullback of $\infty$-categories agree with the pullback of the mapping spaces, we arrive at the following precise version of the description of the mapping spaces in the mapping $\infty$-category $\ncBordInf^\theta(d)_{(P,\theta_P),(Q,\theta_Q)}$.

\begin{lem}\label{lem:mapping-space-theta}
Given $(d-1)$-manifolds $(P,\theta_P),(Q,\theta_Q)$ without boundary together with $\theta$-structures on their once-stabilised tangent bundle, and $\theta$-bordisms $(W,\theta_W),(W',\theta_{W'})\colon (P,\theta_P)\leadsto (Q,\theta_Q)$, there is a natural pullback diagram in $\cS$.
\[\hspace{-.25cm}\begin{tikzcd}
	\Map_{\ncBordInf^\theta(d)_{(P,\theta_P),(Q,\theta_Q)}}((W,\theta_W),(W',\theta_{W'})) \dar \rar &[-15pt] \Map_{(\PSh(\GL_d)_{/\theta})_{\theta_P \sqcup \theta_Q/}}(\theta_W,\theta_{W'}) \dar\\
	\Emb_\partial(W,W') \rar & \Map_{\PSh(\GL_d)_{\Fr(I \times P)\sqcup \Fr(I \times Q)/}}(\Fr(W),\Fr(W'))).
\end{tikzcd}\]
\end{lem}

\subsubsection{$\BordInf^\theta(d)$}\label{sec:details-tangential-bord-compact}
The previous discussion of $\ncBordInf^\theta(d)$  applies also to the levelwise subcategory $\BordInf^\theta(d)$ when restricting to compact manifolds throughout. By a minor enhancement of the argument in \cref{sec:details-bord}, the mapping $\infty$-categories $\BordInf^\theta(d)$ are again $\infty$-groupoids, so as for $\BordInf(d)$ not much is lost by applying $(-)^{(\infty,1)}$ and consider the symmetric monoidal $\infty$-category $\BordInf^\theta(d)^{(\infty,1)}$ with closed $(d-1)$-manifolds with $\theta$-structure on their once-stabilised tangent bundle as objects and mapping spaces given by 
\begin{equation}\label{equ:mapping-cat-theta-bord-compact}
	\textstyle{\BordInf^\theta(d)_{(P,\theta_P),(Q,\theta_Q)}\simeq \Map_{\BordInf^\theta(d)^{(\infty,1)}}((P,\theta_P),(Q,\theta_Q))\simeq \bigsqcup_{[W]} \BDiff^\theta_\partial(W,\theta_P \sqcup \theta_Q)}
\end{equation}
where $[W]$ ranges over compact bordisms $W \colon P \leadsto Q$ up to diffeomorphism relative to the ends and $\BDiff^\theta_\partial(W,\theta_P \sqcup \theta_Q)$ is the quotient $\Map_{\PSh(\GL_d)_{\Fr(I \times P) \sqcup \Fr(I \sqcup Q)/}}(\Fr(W),\theta) / \Diff_\partial(W)$ where the action is induced by precomposition (by standard bundle theory, this agrees with other definitions of $\BDiff^\theta_\partial(-)$ in the literature such as that in \cite[Definition 1.5]{GRWstable}). Composition is given by gluing $\theta$-bordisms and the symmetric monoidal structure by disjoint union.

\subsubsection{Variants with boundary}The discussion for the variants $\ncBordInf(d)^\partial$ and $\BordInf(d)^\partial$ with boundary and their enhancements with tangential structures $\ncBordInf^\theta(d)^\partial$ and $\BordInf^\theta(d)^\partial$ is the same as that for the version without boundary, except that we allow the $(d-1)$-manifolds that appear as objects to have boundary and the bordisms $W\colon P\leadsto Q$ to be bordisms of manifolds with boundary. The bordisms thus come with a decomposition $\partial W= \partial_0W \cup \partial^hW\cup\partial_1W$ into codimension $0$ submanifolds where the \emph{ends} $\partial_iW$s are disjoint and come with identifications $P\cong \partial_0W$ and $Q\cong  \partial_1W$, and the \emph{horizontal boundary} $\partial^hW$ meets the ends in a corner. Embeddings between such manifolds are required to preserve this decomposition, map the interior to the interior, and be the identity near the ends, but they are allowed to move the horizontal boundary. The discussion for the variants $\ncBordInf^\theta(d)^\partial$ and $\BordInf^\theta(d)^\partial$ with tangential structures is similar; on the ends the tangential structures are fixed, but not on the horizontal boundary.

\subsubsection{$\DiscInf_d$ and $\ModInf(d)$}\label{sec:details-bimod} 
The objects of the symmetric monoidal $\infty$-category $\DiscInf_d$ can be identified with $d$-manifolds without boundary that are diffeomorphic to a finite disjoint union of $\bfR^d$'s. The mapping spaces are given by codimension $0$ embeddings and the symmetric monoidal structure by disjoint union. Day convolution equips the $\infty$-category $\PSh(\DiscInf_d)$ of $\cS$-valued presheaves with a symmetric monoidal structure, and the objects of $\ModInf(d)=\ALG(\PSh(\DiscInf_d))$ are associative algebras in $\PSh(\DiscInf_d)$ (see \cref{sec:haugseng-morita}). The mapping $\infty$-category between two associative algebras $A,B \in \ModInf(d)$ is the $\infty$-category $\ModInf(d)_{A,B}$ of $(A,B)$-bimodules and bimodule maps between these (see \cref{sec:assalg-bimodules} where this category is denoted $\Bimod_{A,B}(\PSh(\DiscInf_d)$). The composition functors $\ModInf(d)_{A,B} \times \ModInf(d)_{B,C} \ra \ModInf(d)_{A,C}$ are given taking tensor products over $B$, which we denote by $(-)\cup_{B}(-)$ to emphasise the similarity with the bordism category. The symmetric monoidal structure is given by external tensor product.

\subsubsection{The functor $E$}\label{sec:functor-e-summary}
In terms of the identifications of the objects and mapping categories of source and target explained in Sections~\ref{sec:details-ncbord} and \ref{sec:details-bimod}, the functor $E \colon \ncBordInf(d) \ra \Mod(d)$ of symmetric monoidal double $\infty$-categories is on objects given by sending a $(d-1)$-manifold $P$ to the presheaf $E_{P\times I}=\Emb(-,P\times I)$ where $I=[0,1]$, equipped with the algebra structure induced by ``stacking''. On mapping $\infty$-categories, it is given by the functor $\ncBordInf(d)_{P,Q} \ra \ModInf(d)_{E_{P \times I},E_{Q \times I}}$ which sends a bordism $W \colon M \leadsto N$ to the presheaf $E_W=\Emb(-,W)$ with its $(E_{P \times I},E_{Q \times I})$-bimodule structure by ``stacking'', using fixed collars $P\times I\hookrightarrow W$ and $Q\times I\hookrightarrow W$ of both ends, where the convention is that the canonical vector field on $P \times I$ is inwards pointing and that of $Q\times I$ is outwards pointing. On morphisms, it sends an embedding $W \hookrightarrow W'$ that is fixed on the boundary to the map $E_W\ra E_{W'}$ induced by postcomposition. That $E$ is a functor of double $\infty$-categories in particular says that, given bordisms $W\colon P\leadsto Q$ and $W'\colon Q\leadsto R$, we have a preferred equivalence $E_{W\cup_QW'}\simeq E_W\cup_{E_{Q\times I}}E_{W'}$ of $(E_{P\times I},E_{R\times I})$-bimodules.

We will often restrict the functor $E$ to the levelwise subcategory $\BordInf(d)$ of $\ncBordInf(d)$ and pass to underlying symmetric monoidal $\infty$-categories (i.e.\,apply the functor $(-)^{(\infty,1)}$ from \cref{sec:double-vs-infty2}, which has little effect on $\BordInf(d)$; see \cref{sec:details-bord}) to obtain a functor of symmetric monoidal $\infty$-categories $E\colon \BordInf(d)^{(\infty,1)}\ra \Mod(d)^{(\infty,1)}$ . Recall from \cref{sec:double-vs-infty2} that the mapping spaces of $\Mod(d)^{(\infty,1)}$ are given as $\Map_{\Mod(d)^{(\infty,1)}}(A,B)\simeq\Mod(d)^{\simeq}_{A,B}$.

\subsection{Descent with respect to Weiss $\infty$-covers}\label{sec:descent}
We now prove a descent property for the mapping spaces in $\Mod(d)_{E_{P\times I},E_{Q\times I}}$ for (possibly noncompact) $(d-1)$-manifolds $P$ and $Q$ without boundary. To state it, given a bordism $W\colon P\leadsto Q$, we write $\cO(W)$ for the poset of open subsets of $W$ containing a neighbourhood of the boundary, ordered by inclusion. A subposet $\cU\subset\cO(M)$ is a \emph{Weiss $\infty$-cover} of $M$ if any finite subset of $M$ is contained in some $O\in\cU$. Such a cover is \emph{complete} if it contains a Weiss $\infty$-cover for $\bigcap_{O\in \cU'} O$ for any finite subset $\cU'\subset\cU$. A functor $F\colon\cO(W)^\op\ra \cC$ to an $\infty$-category $\cC$ satisfies \emph{descent for Weiss $\infty$-covers} if for every nonempty $O\in\cO(W)$ and every complete Weiss $\infty$-cover $\cU\subset \cO(O)$ the diagram $F(O)\ra\{F(U)\}_{U\in \cU}$ is a limit diagram.

\begin{prop}\label{prop:descent}For a nonempty bordism $W\in \ncBordInf(d)_{P,Q}$ and a bimodule $X\in\ModInf(d)_{E_{P \times I},E_{Q \times I}}$, the functor $\Map_{\ModInf(d)_{E_{P \times I},E_{Q \times I}}}(E_{(-)},X)\colon\cO(W)^\op\ra \cS$ satisfies descent for Weiss $\infty$-covers.
\end{prop}

\begin{proof}
It suffices to show that for a given complete Weiss $\infty$-cover $\cU\subset\cO(O)$ of nonempty open subset $O\in\cO(W)$, the diagram $\{E_U\}_{U\in\cU}\ra E_O$ is a colimit diagram in $\ModInf(d)_{E_{P \times I},E_{Q \times I}}=\Bimod_{E_{P \times I},E_{Q \times I}}(\PSh(\DiscInf_d))$. Since $\cU$ is cofiltered, its nerve is weakly contractible so by \cref{lemma:free-modules} \ref{enum:free-modules-ii}, it suffices to show that the diagram is a colimit diagram after applying the forgetful functor to $\PSh(\DiscInf_d)$. The result is is the diagram $\{\Emb(-,U)\}_{U\in \cU}\ra \Emb(-,O)$ in $\PSh(\DiscInf_d)$, so as colimits in functor categories are computed objectwise \cite[5.1.2.3]{LurieHTT}, it suffices to show  that $\{\Emb(T\times \bfR^d,U)\}_{U\in \cU}\ra \Emb(T\times \bfR^d ,O)$ is a colimit diagram in $\cS$ for all finite sets $T$, or equivalently, that it is a homotopy colimit diagram in the Kan--Quillen model structure on simplicial sets. This holds by a well-known argument; see the proof of \cite[Lemma 6.4]{KnudsenKupers}.
\end{proof}

\begin{rem}The assumption that $W$ is nonempty is necessary: for $W=\varnothing$, the empty cover $\cU = \varnothing$ is a cover of $W$ but $E_W$ is not the colimit of the empty diagram as $E_W(\varnothing) \simeq \ast$.
\end{rem}

\subsection{Relationship to embedding calculus}  \label{sec:embedding-calculus} Using \cref{prop:descent}, we now relate the functor $\ncBordInf_{P,Q}\ra \Mod(d)_{P\times I,Q\times I}$ induced by $E$ on mapping $\infty$-categories to the map $\Emb_\partial(W,W')\ra T_\infty\Emb_\partial(W,W')$ provided by \emph{embedding calculus} as introduced in \cite{WeissImmersion,WeissImmersionErrata}.

\begin{thm}\label{thm:emb-calc}
Given bordisms $W,W'\in\ncBordInf(d)_{P,Q}$, the map
\begin{equation}\label{equ:emb-calc}
	\Map_{\ncBordInf(d)_{P,Q}}(W,W')\lra \Map_{\Mod(d)_{E_{P \times I},E_{Q \times I}}}(E_W,E_{W'})
\end{equation}
agrees up to equivalence with the map $\Emb_\partial(W,W') \to T_\infty\Emb_\partial(W,W')$ from \cite{WeissImmersion}.
\end{thm}

\begin{proof}
We consider the poset $\cU$ of open subsets $U\subset W$ that are unions $U=c(P) \cup D\cup c(Q)$ of three disjoint open subsets of $W$ where $c(P)$ and $c(Q)$ are open collars of the boundary components $P$ and $Q$ and $D$ is diffeomorphic to $T\times\bfR^d$ for some finite set $T$, ordered by inclusion. Considering $U\in\cU$ as an object in $\ncBordInf(d)_{P,Q}$, we obtain a commutative square in $\cS$
\[
	\begin{tikzcd}
	\Map_{\ncBordInf(d)_{P,Q}}(W,W')\dar[swap]{\circled{1}}\rar&\Map_{\Mod(d)_{E_{P \times I},E_{Q \times I}}}(E_W,E_{W'}))\dar{\circled{3}}\\
	\lim_{U\in\cU} \big(\Map_{\ncBordInf(d)_{P,Q}}(U,W')\big)\rar{\circled{2}}& \lim_{U\in\cU} \big(\Map_{\Mod(d)_{E_{P \times I},E_{Q \times I}}}(E_{U},E_{W'})\big).
	\end{tikzcd}
\]
whose vertical arrows are induced by restriction. By \cref{lem:mapping-space-emb} the map $\circled{1}$ agrees with the restriction map $\Emb_\partial(W,W')\ra \lim_{U\in\cU}\Emb_\partial(U,W')$ which in turn agrees with the map $\Emb_\partial(W,W')\ra T_\infty\Emb_\partial(W,W')$ by the discussion in \cite[Sections 5, 10]{WeissImmersion}, so the claim follows once we show that $\circled{2}$ and $\circled{3}$ are equivalences. As $\cU\subset\cO(W)$ is a complete Weiss $\infty$-cover, the map $\circled{3}$ is an equivalence by \cref{prop:descent}. To prove that $\circled{2}$ is an equivalence, we show that for all $U\in\cU$ the individual maps before taking limits
\begin{equation}\label{equ:emb-to-bimodule-map}
E\colon \Emb_\partial(U,W')\simeq \Map_{\ncBordInf(d)_{P,Q}}(U,W')\lra \Map_{\Mod(d)_{E_{P \times I},E_{Q\times I}}}(E_U,E_{W'})
\end{equation}
are equivalences. To give a convincing proof of this, we rely on the specific construction of $E$ from \cref{sec:the-functor} and refer to that section for the notation. Recall that the functor $E$ arose from restricting the codomain of the composition of simplicial objects in $\infty$-categories
\begin{equation}\label{equ:egeo-composition-embcalcproof}\ncBordInf(d) \xlra{E^{\geo}} \overline{\ALG}(\ManInf_d)\xrightarrow{(\iota^*\circ \yon)_*} \overline{\ALG}(\PSh(\DiscInf_d))\end{equation}
where $(\iota^*\circ \yon)\colon \ManInf_d\ra \PSh(\DiscInf_d)$ is the Yoneda embedding followed by restriction along the inclusion $\iota\colon \DiscInf_d \hookrightarrow \ManInf_d$. This factorisation induces a commutative diagram
\[\begin{tikzcd}
	\ncBordInf(d)_{P,Q}\dar{\inc}\rar{E^{\geo}}&[10pt]\Bimod_{E^{\geo}(P),E^{\geo}(Q)}(\ManInf_d)\dar{U_{E^{\geo}(P),E^{\geo}(Q)}}\rar{(\iota^*\circ \yon)_*}&[10pt] \Mod(d)_{E_{P \times I},E_{Q \times I}}\dar{U_{E_P,E_Q}}\\
	\ncBordInf(d)_{[1]}\rar &\ManInf_d\rar{\iota^*\circ \yon}&\PSh(\DiscInf_d).
\end{tikzcd}\]
where the top composition is obtained from \eqref{equ:egeo-composition-embcalcproof} by evaluation at $[1]$ and taking fibres of the face maps $(d_0,d_1)$, the middle and rightmost vertical map are the forgetful maps from \cref{lemma:free-modules} and the bottom left horizontal map is the coherent nerve of the functor $\ncBord(d)_{[1]}\ra \Man_d$ of $\Kan$-enriched categories that sends a $[1]$-walled manifold $(W,\mu)$ to $W|_{(\mu(0)-\epsilon,\mu(1)+\epsilon)}$. In particular, for $U=c(P)\cup D \cup c(Q)\in\cU$ considered as an object in $\ncBordInf(d)_{P,Q}$, the inclusion $D\subset U$ viewed as a morphism in $\ManInf_d$ gives a morphism $D\ra U_{E^{\geo}(P),E^{\geo}(Q)}(E^{\geo}(U))$ in $\ManInf_d$, so by adjunction a morphism $F_{E^{\geo}(P),E^{\geo}(Q)}(D)\ra E^{\geo}(U)$ in $\Bimod_{E^{\geo}(P),E^{\geo}(Q)}(\ManInf_d)$ which we claim to be an equivalence. By \cref{lemma:free-modules} \ref{enum:free-modules-iii}, it suffices to show that the image 
\[U_{E^{\geo}(P),E^{\geo}(Q)}(F_{E^{\geo}(P),E^{\geo}(Q)}(D))\lra U_{E^{\geo}(P),E^{\geo}(Q)}(E^{\geo}(U))=U=c(P)\cup D\cup c(Q)\]
under $U_{E^{\geo}(M),E^{\geo}(N)}$ is an equivalence. This is a consequence of the second part of \cref{lemma:free-modules} \ref{enum:free-modules-i}. Applying $(\iota^*\circ y)$ and using \cref{lemma:free-modules} \ref{enum:free-modules-iv}, it follows that the natural map $F_{E_{P \times I},E_{Q \times I}}(E_D)\ra E_U$ in $\Mod(d)_{E_{P \times I},E_{Q \times I}}$ is an equivalence. As $F_{E_{P \times I},E_{Q \times I}}$ is left-adjoint to the forgetful functor $U_{E_{P \times I},E_{Q \times I}}$, the map \eqref{equ:emb-to-bimodule-map} thus has the form
\[
	\Emb_\partial(c(M)\cup D\cup c(N),W')=\Emb_\partial(U,W')\lra \Map_{\PSh(\DiscInf_d)}(E_D,E_{W'}).
\]
and is given by the restriction map $\Emb_\partial(c(M)\cup D\cup c(N),W')\ra \Emb_\partial(D,W')$ followed by the map induced by the Yoneda embedding. The former is an equivalence by the contractibility of the space of collars and the latter is an equivalence by the Yoneda lemma since $D$ lies in the full subcategory $\DiscInf_d\subset \ManInf_d$, so the composition is an equivalence.
\end{proof}

\begin{rem}\label{rem:initial-among-rep-presheaves}
The first part of the previous proof in particular shows that for bordisms $W\in\ncBordInf_{P,Q}$ that are diffeomorphic, relative to the ends, to $[0,1)\times P\sqcup T\times\bfR^d \sqcup (-1,0]\times Q$ for some finite set $T$, the map \eqref{equ:emb-calc} is an equivalence for all bordisms $W'\in\ncBordInf_{P,Q}$. In particular, for $T=\varnothing$, we see from the contractibility of the space of collars that both the source and target of this map are both contractible.  
\end{rem}

Combining \cref{thm:emb-calc} with the convergence of embedding calculus in handle codimension $\ge3$ due to Goodwillie, Klein, and Weiss (see \cite[Fact 5.1]{GoodwillieWeiss} and \cite{GoodwillieKlein}), we conclude:

\begin{cor}\label{cor:convergence} Fix bordisms $W,W'\in\ncBordInf(d)_{P,Q}$. If $W$ can be obtained from a closed collar of $P\sqcup Q\cong \partial W$ by attaching handles of index $\le d-3$, then the map
\[
	\Emb_\partial(W,W')\simeq\Map_{\ncBordInf(d)_{P,Q}}(W,W')\lra \Map_{\Mod(d)_{E_{P \times I},E_{Q \times I}}}(E_W,E_{W'})\simeq T_{\infty}\Emb_\partial(W,W')
\]
induced by $E$ is an equivalence.
\end{cor}

\subsubsection{Comparison with the model of Boavida de Brito--Weiss'}\label{sec:comparison-to-pedromichael}\cref{thm:emb-calc} shows that the map \eqref{equ:emb-calc} is a model for embedding calculus, so agrees up to weak equivalence with any other model Among the previously established models, that of Boavida de Brito--Weiss \cite{BdBWSheaf} is closest to ours. Like ours, their model enhances the embedding calculus approximation $\Emb_\partial(W,W')\ra T_\infty\Emb_\partial(W,W')$ to a functor on $\ncBordInf(d)_{P,Q}$. This section serves to extend \cref{thm:emb-calc} to a comparison of the \emph{functors} as opposed to just the individual maps on mapping spaces. This will in particular show that the monoid structures on $T_\infty\Emb_\partial(W,W)$ induced by composition in our and their model agree, which we will use in \cref{sec:conf-cats}.

For this, we write $(\DiscInf_d)_{P,Q}\subset \ncBordInf(d)_{P,Q}$ for the full subcategory of those bordisms that are diffeomorphic relative to the boundary to $P\times[0,1)\sqcup T\times\bfR^d\sqcup (-1,0]\times Q$ for some finite set $T$. When translated from $\Kan$-enriched categories to $\infty$-categories, Boavida de Brito's model for the embedding calculus approximation $\Emb_\partial(W,W')\ra T_\infty\Emb_\partial(W,W')$ is the map on mapping spaces between $W$ and $W'$ induced by the composition
\[
	\ncBordInf(d)_{P,Q}\xlra{\yon} \PSh(\ncBordInf(d)_{P,Q})\xlra{\iota^*}\PSh((\DiscInf_d)_{P,Q})
\] 
of the Yoneda embedding with the inclusion  $\iota\colon (\DiscInf_d)_{P,Q}\hookrightarrow \ncBordInf(d)_{P,Q}$ (cf.\,Section 9 loc.cit.).

\begin{prop}\label{prop:comparison-to-pedromichael}
There is an equivalence of $\infty$-categories 
\[
	\varphi\colon\Mod(d)_{E_{P \times I},E_{Q \times I}}\xlra{\simeq} \PSh((\DiscInf_d)_{P,Q})
\] 
which fits into a commutative diagram of $\infty$-categories
\[\begin{tikzcd}[row sep=0.2cm]
	&\ncBordInf(d)_{P,Q}\arrow[dl,"E",swap]\arrow[dr,"\iota^*\circ \yon"]&\\
	\Mod(d)_{E_{P \times I},E_{Q \times I}}\arrow[rr,"\varphi","\simeq"',swap]&& \PSh((\DiscInf_d)_{P,Q})
\end{tikzcd}.\]
\end{prop}

\begin{proof}The functor $\varphi$ is defined as the  composition
\vspace{-0.1cm}
\[\Mod(d)_{E_{P \times I},E_{Q \times I}} \xlra{\yon} \PSh(\Mod(d)_{E_{P \times I},E_{Q \times I}})\xlra{E^*} \PSh(\ncBordInf(d)_{P,Q})\xlra{\iota^*}
	\PSh((\DiscInf_d)_{P,Q}) \]
With this choice, the canonical natural transformation $\yon\rightarrow E^*\circ \yon \circ E$ induces a natural transformation from the right-hand diagonal functor $(\iota^*\circ \yon)$ in the claimed triangle to $(\varphi\circ E)$, and we will first show that this is an equivalence to obtain commutativity of the triangle. On a bordism $W\in \ncBordInf(d)_{P,Q}$ this natural transformation is the map in $\PSh((\DiscInf_d)_{P,Q})$ induced by $E$ 
\[\Emb_\partial(-,W)\simeq \Map_{\ncBordInf(d)_{P,Q}}(-,W)\lra\Map_{\Mod(d)_{E_{P \times I},E_{Q \times I}}}(E_{(-)},E_W),\]
which is an equivalence by \cref{rem:initial-among-rep-presheaves}. Note that this in particular shows that $E\circ \iota$ is fully faithful. 

It remains to show that $\varphi$ is an equivalence. We will use that for $U\in (\DiscInf_d)_{P,Q}$, we have $E_{U}\simeq F_{E_{P \times I},E_{Q \times I}}(E_{T_U\times\bfR^d})\in \Mod(d)_{E_{P \times I}}$ as a result of the final part of the proof of \cref{thm:emb-calc}; here $T_U$ is the finite set such that $U\cong  P\times[0,1)\sqcup T_U\times\bfR^d\sqcup (-1,0]\times Q$. This property in particular implies that $\varphi$ is conservative, using that $U_{E_{P \times I},E_{Q \times I}}$ is conservative by \cref{lemma:free-modules} \ref{enum:free-modules-iii}, and that a map of presheaves is an equivalence if it is one objectwise. To show that $\varphi$ is an equivalence, it thus suffices to prove that it has a fully faithful left adjoint. This left adjoint is given by the colimit preserving extension $\smash{\overline{E\circ \iota}\colon \PSh(\DiscInf_d)_{P,Q})\ra \Mod(d)_{E_{P \times I},E_{Q \times I}}}$ along the Yoneda embedding of the fully faithful functor  $E\circ \iota$, using that $\Mod(d)_{E_{P \times I},E_{Q \times I}}\simeq \BMod_{E_{P \times I},E_{Q \times I}}(\PSh(\DiscInf_d))$ has colimits as a result of \cref{lemma:free-modules}. By \cite[5.1.6.10]{LurieHTT}, this left adjoint $\overline{E\circ \iota}$ is fully faithful if  $E_{U}\in \Mod(d)_{E_{P \times I},E_{Q \times I}}$ is completely compact for all $U\in (\DiscInf_d)_{P,Q}$, i.e.\,if ${\Map_{\Mod(d)_{E_{P \times I},E_{Q \times I}}}(E_U,-)\colon \Mod(d)_{E_{P \times I},E_{Q \times I}}\ra\cS}$ preserves small colimits. Since $E_{U}\simeq F_{E_{P \times I},E_{Q \times I}}(E_{T_U\times\bfR^d})$, this condition is by adjunction equivalent to $\smash{\Map_{\PSh(\DiscInf_d)}(E_{T_U\times\bfR^d},U_{E_{P \times I},E_{Q \times I}}(-))}$
preserving small colimits which indeed holds, by the Yoneda lemma and the fact that $U_{E_{P \times I},E_{Q \times I}}$ preserves colimits by \cref{lemma:free-modules} \ref{enum:free-modules-ii}.
\end{proof}

\begin{rem}
Considering bordisms $W\colon P\leadsto Q$ as bordisms $\varnothing \leadsto P\sqcup Q$ or $P\sqcup Q\leadsto \varnothing$ leads to equivalences between $\ncBordInf(d)_{P,Q}$, $\ncBordInf(d)_{\varnothing,P\sqcup Q},$ and $\ncBordInf(d)_{P\sqcup Q,\varnothing}$, and similarly for $(\DiscInf_d)_{P\sqcup Q}$---compatible with the functor $(\iota^*\circ y)$. It thus follows from \cref{prop:comparison-to-pedromichael} that $E\colon \ncBordInf(d)_{P,Q}\ra  \Mod(d)_{E_{P \times I},E_{Q \times I}}$ agrees up to equivalences with the analogous functors involving $\Mod(d)_{E_{\varnothing},E_{P\times I\sqcup Q \times I}}$ or $\Mod(d)_{E_{P\times I\sqcup Q \times I,E_{\varnothing}}}$. That the latter two categories are equivalent can also be deduced from \cref{rem:lurie-bimodules} and \cite[4.6.3.11]{LurieHA} (no $(-)^\rev$ appears since we implicitly used the anti-homomorphism of $E_{P \times I}$ or $E_{Q \times I}$ by reflection in $I$).
\end{rem}

\subsection{Isotopy extension for $E$}\label{sec:isotopy-extension}
A key input in the proof of \cref{bigthm:2-type-invariance} in \cref{sec:2-type-invariance-sdisc} will be a version of the isotopy extension theorem for the mapping spaces in $\ModInf(d)_{P,Q}$.  In view of \cref{thm:emb-calc}, this amounts to an isotopy extension theorem for embedding calculus. Such a theorem has been proved by Knudsen--Kupers \cite[Theorem 6.1]{KnudsenKupers}, but instead of reducing the version we need from theirs, it is more convenient to give a direct proof based on their strategy.

The setting is as follows. We fix two compact bordisms $W\colon P\leadsto Q$,  $W'\colon R\leadsto S$, two possibly noncompact bordisms $M,N\colon Q\leadsto R$, and an open collar neighbourhood $c(M)\subset M$ viewed as a bordism $Q\leadsto R$. 
Writing $c$ for the inclusion $c(M)\subset M$ viewed as a morphism in $\ncBordInf(d)_{Q,R}$, we have a commutative diagram
\[
\begin{tikzcd}[column sep=1.5cm,ar symbol/.style = {draw=none,"\textstyle#1" description,sloped},
	equivalent/.style = {ar symbol={\simeq}}]
	\Map_{\ncBordInf(d)_{Q,R}}\big(M,N\big)\dar{(-)\circ c}\rar{W\cup_Q(-)\cup_{R}W'}\dar\arrow[d, phantom, shift left=3cm, "\circled{$1$}"]& \Map_{\ncBordInf(d)_{P,S}}\big(W\cup_QM\cup_RW',W\cup_QN\cup_RW'\big)\dar{(-)\circ {(W\cup_Qc\cup_R{W'})}}\\
	\Map_{\ncBordInf(d)_{Q,R}}\big(c(M),N\big)\arrow[d,equivalent]\rar{W\cup_Q(-)\cup_{R}W'} &\Map_{\ncBordInf(d)_{P,S}}\big(W\cup_Qc(M)\cup_RW',W\cup_QN\cup_RW'\big)\\[-15pt]
	\ast&&
\end{tikzcd}\vspace{-0.2cm}
\]
which maps via the functor $E\colon \ncBordInf(d)\ra \Mod(d)$ to the corresponding square for $\Mod(d)$ 
\[\hspace{-0.3cm}
\begin{tikzcd}[column sep=2.5cm,ar symbol/.style = {draw=none,"\textstyle#1" description,sloped},
	equivalent/.style = {ar symbol={\simeq}}]
	\Map_{\Mod(d)_{E_{Q\times I},E_{R\times I}}}\big(E_M,E_N\big)\dar{(-)\circ E_c}\arrow[d, phantom, shift left=3cm, "\circled{$2$}"]\rar{E_{W}\cup_{E_{Q\times I}}(-)\cup_{E_{R\times I}}E_{W'}}\dar& \Map_{\Mod(d)_{E_{P\times I},E_{S\times I}}}\big(E_{W\cup_QM\cup_RW'},E_{W\cup_QN\cup_RW'}\big)\dar{(-)\circ {E_{W\cup_Qc\cup_R{W'}}}}\\
	\Map_{\Mod(d)_{E_{Q\times I},E_{R\times I}}}\big(E_{c(M)},E_N\big)\arrow[d,equivalent]\rar{E_{W}\cup_{E_{Q\times I }}(-)\cup_{E_{R\times I}}E_{W'}} &\Map_{\Mod(d)_{E_{P\times I},E_{S\times I}}}\big(E_{W\cup_Qc(M)\cup_RW'},E_{W\cup_QN\cup_RW'}\big)\\[-15pt]
	\ast&
\end{tikzcd}\vspace{-0.2cm}
\]
Note that the bottom left corners in both squares are contractible by \cref{rem:initial-among-rep-presheaves}. Moreover, in view of \cref{lem:mapping-space-emb} the square $\circled{1}$ has up to equivalence the form
\[\begin{tikzcd}[ar symbol/.style = {draw=none,"\textstyle#1" description,sloped},
	equivalent/.style = {ar symbol={\simeq}}, column sep=1.5cm]
	\Emb_\partial\big(M,N\big)\dar{(-)\circ c}\rar{W\cup_Q(-)\cup_{R}W'}\dar&[10pt]\Emb_\partial\big(W\cup_QM\cup_RW',W\cup_QN\cup_RW'\big)\dar{(-)\circ {(W\cup_Qc\cup_R{W'})}}\\
	\Emb_\partial\big( c(M),N\big)\arrow[d,equivalent]\rar{W\cup_Q(-)\cup_{R}W'} &\Emb_\partial\big(W\cup_Qc(M)\cup_RW',W\cup_QN\cup_RW'\big)\\[-15pt]
	\ast&
\end{tikzcd}\vspace{-0.2cm}\]
As the restriction map $\Emb_\partial(W\cup_Qc(M)\cup_RW',W \cup_Q N\cup_R W')\ra \Emb_{P\sqcup S}(W\sqcup W',W\cup_Q N\cup_R W')$ is an equivalence and $W\sqcup W'$ is compact, it follows from the parametrised isotopy extension theorem that this square is cartesian, so the same holds for the square $\circled{1}$. 

The isotopy extension result we will prove says that the same holds for $\circled{2}$ under a certain condition on the convergence of embedding calculus, namely that the map from the bottom right of $\circled{1}$ to the bottom right corner of $\circled{2}$ is an equivalence if $M$ is replaced by $C_k\coloneqq c(M)\sqcup \unl{k} \times \bfR^d\in \ncBordInf(d)_{Q,R}$ for $\unl{k} = \{1,2,\ldots,k\}$ and all $k\ge0$. We denote by $\circled{2}^\simeq$ the square obtained from $\circled{2}$ by replacing the categories $\Mod(d)_{E_{Q \times I},E_{R \times I}}$ and $\Mod(d)_{E_{P \times I},E_{S\times I}}$ in the top row by their cores. 

\begin{thm}\label{thm:isotopy-extension}
If the map induced by $E$
\[\hspace{-0.1cm}
	\Map_{\ncBordInf(d)_{P,S}}\big(W\cup_QC_k\cup_RW',W\cup_QN\cup_RW'\big)\ra\Map_{\Mod(d)_{E_{P \times I},E_{S\times I}}}\big(E_{W\cup_QC_k\cup_RW'},E_{W\cup_QN\cup_RW'}\big)
\]
is an equivalence for all $k\ge0$, then the square $\circled{2}$ is cartesian. If this assumption in addition holds for $M$ in place of $N$, then the square $\circled{2}^\simeq$ is also cartesian.\end{thm}

\begin{proof}
We first show the claim for $\circled{2}$. We write $\cU$ for the poset of open subsets of $M$ that are unions $U=D\cup c'(M)$ such that $c(M)\subset M$ is an open collar of the boundary that contains the chosen collar $c(M)\subset M$ and $D\subset M$ is diffeomorphic to $T\times \bfR^d$ for some finite set $T$. Considering $U$ as an object in $\ncBordInf(d)_{Q,R}$ we have a functor $\cU\ra \ncBordInf(d)_{Q,R}$. Since the square $\circled{2}$ is natural in $M$, it maps to the limit of the same squares for $M$ replaced by $U\in\cU$
\[\begin{tikzcd}
	\underset{U\in\cU}{\lim\ }\Map_{\Mod(d)_{E_{Q \times I,R\times I}}}\big(E_U,E_{N}\big)\dar\rar\dar&\underset{U\in\cU}{\lim\ }\Map_{\Mod(d)_{E_{P \times I,S\times I}}}\big(E_{W\cup_QU\cup_{R}W'},E_{W\cup_QN\cup_{R}W'}\big)\dar\\
	\underset{U\in\cU}{\lim\ }\Map_{\Mod(d)_{E_{Q \times I,R\times I}}}\big(E_{c(M)},E_{N}\big)\rar &\underset{U\in\cU}{\lim\ }\Map_{\Mod(d)_{E_{P \times I,S\times I}}}\big(E_{W\cup_Qc(M)\cup_RW'},E_{W\cup_QN\cup_RW'}\big).
\end{tikzcd}\]
We claim that it suffices to show this square of limits is cartesian. To justify this, we show that the maps from $\circled{2}$ to the square of limits are all equivalences. For the maps between the bottom left corners and between the bottom right corners, this follows from the fact that the diagram is constant and the category $\cU$ is weakly contractible since it is cofiltered. For the top-right corner and top-left corner it follows from \cref{prop:descent} since the posets $\cU$ and $\{W\cup_QU\cup_RW'\,|\,U\in\cU\}$ are complete Weiss $\infty$-covers of $M$ and $W\cup_QM\cup_RW'$. 
	
To show that the previous square of limits is cartesian, note that it receives a map from the analogous square using $\ncBordInf(d)$ instead of $\ModInf(d)$, and this map of squares consists of equivalences: for the top right and bottom right corner it holds by assumption and for the top left and bottom left corner it holds by \cref{rem:initial-among-rep-presheaves}. The square using $\ncBordInf(d)$ is a limit of squares of the form $\circled{1}$, with $M$ replaced by $U\in\cU$, so it is cartesian since we have already explained that $\circled{1}$ is cartesian and limits of cartesian squares remain cartesian.
	
To show the claim for ${\circled{2}}^\simeq$, we first assume $M=N$ in which case, the claim follows from the following fact: given a monoid $A$ in $\cS$ acting on a space $X$ and $x\in X$, consider the fibre sequence
\[
	\hofib_{x}(A\xlra{(-)\cdot x}X)\lra A\xlra{(-)\cdot x}X 
\]
whose fibre inherits the structure of a monoid in $\cS$ from that of $A$ and the $A$-action on $X$. Then one can check that the sequence obtained by passing to group-like components in fibre and total space is again a fibre sequence.

To deduce the general case of ${\circled{2}}^\simeq$ from that of ${\circled{2}}$, it suffices to prove that if $\varphi \colon E_M \to E_N$ has the property that $\varphi' \coloneqq \id_{E_W} \cup_{E_{Q \times I}}  \varphi \cup_{E_{R \times I}} \id_{E_{W'}}$ is an equivalence, then $\varphi$ is also an equivalence. To prove this, we pick an inverse $\psi' \colon E_{W \cup_Q N \cup_R W'} \to E_{W \cup_Q M \cup_R W'}$ to $\varphi'$ and claim that the image of $\psi'$ under the right-vertical map in the square ${\circled{2}}$ with the role of $M$ and $N$ reversed lies in the component of the bottom horizontal map. To see this, we extend this square to the bottom by
\[\begin{tikzcd} \Map_{\Mod(d)_{E_{Q\times I},E_{R\times I}}}\big(E_{c(M)},E_N\big) \rar \dar{\varphi\circ(-)}[swap]{\simeq} & \Map_{\Mod(d)_{E_{P\times I},E_{S\times I}}}\big(E_{W\cup_Q c(M)\cup_RW'},E_{W\cup_QN\cup_RW'}\big) \dar{\varphi' \circ -}[swap]{\simeq} \\
\Map_{\Mod(d)_{E_{Q\times I},E_{R\times I}}}\big(E_{c(M)},E_M\big) \rar & \Map_{\Mod(d)_{E_{P\times I},E_{S\times I}}}\big(E_{W\cup_Q c(M) \cup_RW'},E_{W\cup_Q M\cup_RW'}\big)\end{tikzcd}\]
where the left vertical map is an equivalence as both source and target are contractible, and the right vertical map is an equivalence because $\varphi'$ is one by assumption. To see whether the image of $\psi'$ in the upper right corner is in the component hit by the upper horizontal map, it thus suffices to show that the image of $\psi'$ in the bottom horizontal corner is in the component hit by the bottom horizontal map. But this follows from the relation $[\varphi'\circ\psi']=[\id]$ in the set of components, which holds by the choice of $\psi'$. Using that the square ${\circled{2}}$ with the role of $M$ and $N$ reverse is a pullback (this is where we use the additional hypothesis for $M$), we conclude that there exists $\psi \colon E_N \to E_M$ such that $[\psi']= [W\cup_Q\psi\cup_RW']$. To finish the proof, it suffices to show that $\varphi\circ\psi$ and $\psi\circ\varphi$ are both equivalences, since then $\varphi$ has to be an equivalence. But this follows from the case $M=N$ treated above, using that both compositions become equivalences after applying $W\cup_Q(-)\cup_RW'$ since this even holds for $\psi$ and $\varphi$ individually.
\end{proof}

\begin{rem}The proof of \cref{thm:isotopy-extension} in particular shows that if the assumption in the statement holds for $M$ and $N$, then the following map detects equivalences:
\vspace{-0.05cm}
\[
	\Map_{\Mod(d)_{E_{Q\times I},E_{R\times I}}}\big(E_M,E_N\big) \xrightarrow{E_{W}\cup_{E_{Q\times I}}(-)\cup_{E_{R\times I}}E_{W'}} \Map_{\Mod(d)_{E_{P\times I},E_{S\times I}}}\big(E_{W\cup_QM\cup_RW'},E_{W\cup_QN\cup_RW'}\big)
\]
\end{rem}

\subsection{$\DiscInf$-structure spaces}\label{sec:disc-structure-spaces} We conclude this section with the definition of the $\DiscInf$-structure spaces and a discussion some of their functoriality. Given objects $P\in \BordInf(d)$ and $A\in\Mod(d)$, i.e.\,a closed $(d-1)$-manifold $P$ and an associative algebra $A$ in $\PSh(\DiscInf_d)$, we abbreviate the $\infty$-category of nullbordisms of $P$ and the analogue for right $A$-modules by
\begin{equation}\label{equ:abbreviate-right-modules}
	\gls*{nullbordism} \coloneqq \BordInf(d)_{\varnothing,P} \qquad \text{and} \qquad \gls*{rightbordism} \coloneqq \ModInf(d)_{E_\varnothing,A}.
\end{equation}

\begin{rem}Note that $E_\varnothing$ is the monoidal unit in $\PSh(\DiscInf_d)$, so $\ModInf(d)_{E_\varnothing,A}$ may be viewed as an $\infty$-category of right-$A$-modules. Using \cref{rem:lurie-bimodules} and \cite[4.3.2.8]{LurieHA} one sees that this agrees with Lurie's model of the $\infty$-category of right-$A$-modules, but we will not use this.\end{rem}

\subsubsection{$\DiscInf$-structure spaces of modules}For $A=E_{P\times I}$ for a closed $(d-1)$-manifold $P$, the functor $E$ induces a functor $\BordInf(d)_{P}\ra \ModInf(d)_{P\times I}$. As the source is an $\infty$-groupoid by the discussion \cref{sec:details-bord}, it lands in the core $\ModInf(d)^\simeq_{P\times I}\subset \ModInf(d)_{P\times I}$. The $\DiscInf$-structure spaces are the fibres of this functor of $\infty$-groupoids:

\begin{dfn}\label{def:disc-structure-space} The \emph{$\DiscInf$-structure space} of a right-$E_{P \times I}$-module $X\in\Mod(d)_{E_{P\times I}}$ is the fibre \[\gls*{sdisc}\coloneqq \fib_{X}(\BordInf(d)_{P}\ra \ModInf(d)^{\simeq}_{E_{P \times I}})\in \cS\]
\end{dfn}
From the description of the object and mapping spaces of $\BordInf(d)$ and $\ModInf(d)$ in \cref{sec:mapping-infinity-category}, we see that the path components of $S^{\DiscInf}_P(X)$ are given by
\[
	\pi_0\,S^{\DiscInf}_P(X) = \frac{\left\{\text{\parbox{11.5cm}{\centering pairs $(M,\varphi)$ of a compact smooth $d$-manifold $M$ with identified boundary $\partial M\cong P$ \newline and an equivalence of right $E_{P \times I}$-modules $\varphi \colon E_M \to X$}}\right\}}{\parbox{9.5cm}{\centering $(M,\varphi) \sim (M',\varphi')\Leftrightarrow$ there exists a diffeomorphism $\alpha \colon M \to M'$ relative to $P$ with $[\varphi' \circ E_\alpha]=[\varphi]\in\pi_0\,\Map_{\Mod(d)^{\simeq}_{E_{P \times I}}}(E_M,X)$}}
\] 
and that the component of a pair $(M,\varphi)$ agrees with the identity component
\[
	S^{\DiscInf}_P(X)_{(M,\varphi)} \simeq \big(\Aut_{\Mod(d)_{P\times I}}(E_M)/\Diff_\partial(M)\big)_{\id}.
\]
of the fibre $\Aut_{\Mod(d)_P}(E_M)/\Diff_\partial(M)$ of the map $\BDiff_\partial(M)\ra \BAut_{\Mod(d)_P}(E_M)$ induced by $E$. This can also be rephrased in the form of an equivalence
\begin{equation}\label{equ:disjoint-union-description-sdisc}
	 \textstyle{S^{\DiscInf}_P(X)\simeq \bigsqcup_{[M]}\Aut_{\Mod(d)_{P\times I}}(E_M)/\Diff_\partial(M)}
\end{equation}
where $[M]$ runs through diffeomorphism classes of compact manifolds $M$ with identified boundary $\partial M\cong P$ for which there exists an equivalence $E_M\ra X$ of right $E_{P\times I}$-modules.

\subsubsection{$\DiscInf$-structure spaces of manifolds}\label{sec:sdisc-for-manifolds}
Given a compact $d$-manifold $W$ with identified boundary $\partial W\cong P$, considered as an object in $\BordInf(d)_{P}$, we abbreviate \[\gls*{sdiscpartial} \coloneqq S^{\DiscInf}_{P}(E_W).\] 
This is natural in $\smash{W \in \BordInf(d)^{(\infty,1)}_{\varnothing/}}$ in that it gives a functor $\smash{S^{\DiscInf}_\partial(-) \colon \BordInf(d)^{(\infty,1)}_{\varnothing/} \ra \cS}$ from the $\infty$-category of nullbordisms to the $\infty$-category of spaces. In particular, for bordisms $W\colon\varnothing\leadsto P$ and $W' \in\BordInf(d)_{P,Q}$ we have a \emph{gluing map} $(- \cup_P W) \colon S^{\DiscInf}_\partial(W) \ra S^{\DiscInf}_\partial(W \cup_P W')$.

\section{\cref{bigthm:2-type-invariance}: $2$-type invariance} \label{sec:2-type-invariance}
The goal of this section is to prove \cref{bigthm:2-type-invariance}, which says that the $\DiscInf$-structure space of a compact $d$-manifold depends for $d\ge5$ only on the tangential 2-type, a notion that we recall in \cref{sec:tangential-k-types}. As outlined in \cref{sec:intr-2-type-invariance}, this will be an application of a general tangential $k$-type invariance result, proved in \cref{sec:k-type-invariance}, about the values of certain functors on a  category of compact null bordisms. That $S^{\DiscInf}_\partial(-)$ satisfies its hypotheses is verified in \cref{sec:2-type-invariance-sdisc}.

\begin{nconvention}\
\begin{enumerate}
\item In contrast to the previous sections, all manifolds---which were already assumed to be smooth---are now also assumed to be compact. Non-empty boundaries are allowed. 
\item In this section we adopt the point of view on $\theta$-structures in terms of bundle maps (always required to be fibrewise injective), which is different but by basic bundle theory equivalent to that in terms of $\GL_d(\bfR)$-spaces from \cref{sec:details-tangential-bord}. For the convenience of the reader, we recall the necessary definitions from scratch in \cref{sec:theta-manifolds}.
\end{enumerate}\end{nconvention}

\subsection{Tangential $k$-types} \label{sec:tangential-k-types} We start with some manifold-theoretic preliminaries.

\subsubsection{$\theta$-manifolds and tangential $k$-types} 
\label{sec:theta-manifolds}
Given a map $\theta\colon B\ra \BO$, a \emph{$\theta$-manifold} $M$ is a manifold with a \emph{$\theta$-structure} on its stable tangent bundle, by which we mean in this section a stable bundle map $\ell_M\colon \tau_M^s\ra \theta^*\gamma$ from the stable tangent bundle of $M$ to the pullback of the universal stable vector bundle $\gamma$ over $\BO$ along $\theta$. A tangential structure is \emph{$k$-connected} if the underlying map $\bar{\ell}_M\colon M\ra B$ is $k$-connected in the usual sense. 

Given a codimension $0$ embedding $e\colon M\hookrightarrow N$ and a $\theta$-structure $\ell_N$ on $N$, we obtain a $\theta$-structure $e^*\ell_N$ on $M$ by precomposition with the stable derivative of $e$. Two $\theta$-manifolds $M$ and $N$ are \emph{$\theta$-diffeomorphic} if there exists a diffeomorphism $\phi\colon M\ra N$ of the underlying manifolds such that $\phi^*\ell_N$ and $\ell_M$ are homotopic as bundle maps. A codimension $0$ embedding $e\colon M\hookrightarrow N$ is an \emph{equivalence on tangential $k$-types} if $N$ admits a $k$-connected $\theta$-structure $\ell_N$ for some $\theta$ such that $e^*\ell_N$ is again $k$-connected. Two manifolds $M$ and $N$ have the \emph{same tangential $k$-type} if there is a $\theta\colon B\ra \BO$ such that  $M$ and $N$ admit $k$-connected $\theta$-structures $\ell_M$ and $\ell_N$ (for the same $\theta$).

\begin{ex}\label{rem:2-connected-maps}
Any codimension $0$ embedding $M\hookrightarrow N$ that is $k$-connected is an equivalence on tangential $k$-types. This is clear from the definition as long as $N$ admits  a $k$-connected $\theta$-structure with respect to \emph{some} $\theta$, and there is indeed always such a choice: pick a Moore-Postnikov factorisation $N\ra B\ra\BO$ of a classifying map for the stable tangent bundle of $N$ into a $k$-connected map followed by a $k$-coconnected map $\theta\colon B\ra\BO$.
\end{ex}

\begin{ex}\label{rem:classification-2-types}The case of most interest to us is $k=2$, where there is a simple recipe to decide whether two $d$-manifolds $M_0$ and $M_1$ have the same tangential $k$-types. If the $M_i$ are disconnected, then they have the same tangential $2$-type if and only if there exists a bijection between their components such that the corresponding components have the same tangential $2$-type. For connected manifolds $M_0$ and $M_1$, one can decide whether they have the same tangential $2$-type as follows (cf.\,\cite[p.\,712--713]{Kreck}; Kreck deals with \emph{normal} $k$-types as opposed to \emph{tangential} $k$-types and has a different indexing convention, but neither of this makes a difference):
	\begin{enumerate}
		\item The functionals $w_2(M_i) \colon \pi_2(M_i)\ra \bfZ/2$ for $i=0,1$ induced by the second Stiefel--Whitney classes need to be both trivial or nontrivial.
		\item If they are both nontrivial, then $M_0$ and $M_1$ have the same tangential $2$-type if and only if there exists an abstract isomorphism $\varphi\colon \pi_1(M_0)\ra\pi_1(M_1)$ such that $\varphi ^*w_1(M_1)=w_1(M_0)$, where $w_1(M_i) \in \oH^1(M_i;\bfZ/2)\cong \oH^1(K(\pi_1M_i,1);\bfZ/2)$ is the first Stiefel--Whitney class.
		\item If they are both trivial, then there are unique classes $w_2(M_i) \in \oH^2(K(\pi_1(M_i),1);\bfZ/2)$ that pull back to the second Stiefel--Whitney classes along the canonical maps $M_i\ra K(\pi_1(M_i),1)$. Then $M_0$ and $M_1$ have the same tangential $2$-type if and only if there exists an abstract isomorphism $\varphi\colon \pi_1(M_0)\ra\pi_1(M_1)$ with $\varphi ^*w_j(M_1)=w_j(M_0)$ for $j=1,2$.
	\end{enumerate}
In particular, if $M_0$ and $M_1$ are spin, $w_i(M)$ and $w_i(N)$ vanish for $i\le2$, so the recipe shows that they have the same tangential $2$-types if and only if their fundamental groupoids are equivalent. It also implies that the tangential $2$-type of a smooth manifold does not depend on the smooth structure, since Stiefel--Whitney classes are defined for topological manifolds. \end{ex}

\begin{lem}\label{lem:nice-representative-k-type} 
Let $M$ be an $m$-manifold and $k\ge 0$ a number. For any $d\ge4$ with $k\le \lfloor \tfrac{d}{2}\rfloor$, there exists a closed $d$-manifold $P$ with the same tangential $k$-type as $M$.
\end{lem}

\begin{proof}
We may assume $k\ge1$ and that $M$ is connected; apply the claim to each connected component otherwise. Choose a Moore--Postnikov $k$-factorisation $M\ra B\ra \BO$ of the stable tangent bundle into a $k$-connected map followed by a $k$-coconnected map $\theta\colon B\ra \BO$. The condition $k\le \lfloor \tfrac{d}{2}\rfloor$ in particular implies that $d \ge k+1$, so the $d$-sphere $S^d$ admits a $\theta$-structure by obstruction theory. Doing surgeries compatible with the $\theta$-structure (see \cite[Proposition\,4]{Kreck}), we obtain a closed $d$-manifold $P$ with a $k$-connected $\theta$-structure.
\end{proof}

\subsubsection{$\theta$-bordism}\label{sec:theta-manifolds}
Given a $\theta$-manifold $M$, a choice of inwards pointing vector field induces a $\theta$-structure on the boundary $\partial M$. Using the canonical vector field $\smash{\frac{\partial}{\partial x}}$ on $[0,1]$, we moreover obtain a $\theta$-structure on $M\times [0,1]$, which restricts to a $\theta$-structure on the \emph{double} $M\cup_{\partial M}\overline{M} \cong \partial(M\times [0,1])$ of $M$. Here $\overline{M}$ is the $\theta$-manifold whose underlying manifold is $M$ but which is equipped with the \emph{opposite $\theta$-structure} obtained by restricting the induced $\theta$-structure on $M\times [0,1]$ to $M\times\{1\}\subset \partial(M\times [0,1])$. A \emph{$\theta$-bordism} from a $d$-dimensional $\theta$-manifold $P$ to another $d$-dimensional $\theta$-manifold $Q$ is a $(d+1)$-dimensional $\theta$-manifold $W$ together with a $\theta$-diffeomorphism $\partial W\cong P\sqcup \smash{\overline{Q}}$; we denote this $W \colon P \gls*{bordism} Q$. A $\theta$-manifold $P$ is \emph{$\theta$-null bordant} if there is a $\theta$-bordism $P \leadsto \varnothing$. Note that, by construction, the double $M\cup_{\partial M} \overline{M}$ of any $\theta$-manifold $M$ is $\theta$-nullbordant.

\subsubsection{Handle decompositions}\label{sec:handle-dec}
Given a compact $d$-dimensional bordism $W\colon P\leadsto Q$ between closed $(d-1)$-manifolds, a \emph{handle decomposition of the bordism} $W$ is a decomposition
\[
	P=W_{-1}\overset{W(-1,0]}{\leadsto}W_{0}\overset{W(0,1]}{\leadsto}\cdots \overset{W(d-2,d-1]}{\leadsto}W_{d-1}\overset{W(d-1,d]}{\leadsto}W_{d}=Q
\]
of $W$ as a union of bordisms between closed $(d-1)$-manifolds $W_i$ such that $W(k-1,k]$ is obtained from a collar on $W_{k-1}$ by attaching finitely many handles of index $k$. Such a decomposition always exists, for instance by choosing a self-indexing Morse function. By construction, $W_{k+1}$ is obtained from $W_{k}$ by finitely many $k$-surgeries. We abbreviate 
\[
	\gls*{whalf} \coloneqq \cup_{m\le i \le k-1}W(i,i+1] \quad \text{and} \quad \gls*{wfull} \coloneqq \cup_{m-1\le i \le k-1}W(i,i+1]
\]
and consider these manifolds as bordisms from $W_m$ to $W_k$ and from $W_{m-1}$ to $W_k$, respectively. The idea behind the notation is that the half-open or closed interval indicates which handles the submanifold contains. Given $m\le k$, we say that $W$ has \emph{handle type} $[m,k]$ if there is a handle decomposition with $W=W[m,k]$. A $d$-manifold $M$ \emph{handle type $[m,k]$} if it has that property when viewed as a bordism $M\colon\varnothing \leadsto \partial M$. It is said to have \emph{handle dimension $\le k$} if it has handle type $[0,k]$. A codimension $0$ submanifold inclusion $N \subset \interior(M)$ has \emph{relative handle type $[m,k]$} if the bordism $M\backslash \interior(N)\colon \partial N\leadsto \partial M$ has handle type $[m,k]$, and $N\subset \interior(M)$ has \emph{relative handle dimension $\le k$} if this bordism has handle type $[0,k]$.

\subsubsection{Handle trading and connectivity}

The following two lemmas are certainly standard, but we could not find references for them in the generality we needed.
\begin{lem}\label{lem:handle-trading} Let $W\colon P\leadsto Q$ be a bordism between closed $d$-manifolds $P$ and $Q$ with $d\ge4$. If both boundary inclusions $P\subset W\supset Q$ are $k$-connected for some $k\ge0$, then the following holds.
\begin{enumerate}
	\item If $2k<d-1$, then $W\colon P\leadsto Q$ has handle type $[k+1,d-k]$ and 
	\item If $2k=d-1$, then $W\sharp (S^{k+1}\times S^{k+1})^{\sharp r}\colon P\leadsto Q$ has type $[k+1,d-k]$ for some $r\ge0$.
\end{enumerate}
\end{lem}

\begin{proof}We begin with the first case. Starting from a handle decomposition of the bordism $W\colon P\leadsto Q$, we obtain a potentially different handle decomposition of $W$ of type $[k+1,d+1]$ by handle trading, without changing the number of $i$-handles for $i\ge k+3$ (see e.g.\,the proof of \cite[Theorem 3]{WallConnectivity}). Now we apply the same procedure to the dual of this new handle decomposition to obtain yet another handle decomposition, this time of type $[0,d-k]$ and with the same number of $i$-handles for $i\le d-k-2$. Since $k\le d-k-2$ and we previously arranged that there are no $i$-handles for $i\le k$, the resulting decomposition has type $[k+1,d-k]$.
	
In the case $2k=d-1$, we may reindex so that the claim reads as follows (set $n=k+1$): given a $2n$-dimensional bordism $W\colon P\leadsto Q$ with $2n\ge6$ such that the inclusions $P\subset W\supset Q$ are $(n-1)$-connected, there exists an $r\ge0$ such that the bordism $(W\sharp (S^n\times S^n)^{\sharp r})\colon P\leadsto Q$ admits a handle decomposition with only $n$-handles. We are not aware of a classical reference for this fact; we learnt it from the proof of \cite[Lemma 6.21]{GRWstable}.
\end{proof}

\begin{lem}\label{lem:bordism-connectivity}
Let $\theta\colon B\ra \BO$ be a map and $(P,\ell_P)$ and $(Q,\ell_Q)$ two closed $d$-dimensional $\theta$-manifolds that are $\theta$-bordant. Assume $d\ge4$ and fix $k\ge0$ with $2k<d$.
\begin{enumerate}
	\item If $\ell_P$ is $k$-connected, then there is a $\theta$-bordism $W\colon P\leadsto Q$ such that $P\subset W$ is $k$-connected.
	\item If also $\ell_Q$ is $k$-connected, then we may assume that $W\supset Q$ is $k$-connected as well.
\end{enumerate}
\end{lem}

\begin{proof}For part (i), we refer to the proof of the correction \cite[Proposition p.\,48]{HebestreitJoachim} to a part of \cite[Proposition 4]{Kreck}. The proof of part (ii) is a minor extension of their argument which we spell out for the convenience of the reader in the case $k\ge1$, leaving $k=0$ as an easy exercise.
	
Starting from a $\theta$-bordism $(W,\ell_W)\colon (P,\ell_P)\leadsto (Q,\ell_Q)$ we can assume that $\ell_W$ is $k$-connected by performing surgery in the interior of $W$. As $\ell_P$ and $\ell_Q$ are $k$-connected, this ensures that the inclusions $P\subset W\supset Q$ induce an isomorphism on homotopy groups at all basepoints in degrees $\le k-1$. By considering each component in $\pi_0(P)\cong\pi_0(W)\cong \pi_0(Q)\cong\pi_0(B)$ separately we may assume that each of $P,W,Q,B$ is connected. We now consider the long exact sequences 
\[\begin{tikzcd}[column sep=0.5cm, row sep=0.1cm]
	\ldots \rar&\begin{cases}\pi_k(P)\\\pi_k(Q)\end{cases}\rar[shorten <=-9pt]\arrow[dr,two heads,shorten <=-4pt]& \pi_k(W)\arrow[d,two heads]\arrow[r,two heads]&\begin{cases}\pi_k(W,P)\\\pi_k(W,Q)\end{cases}\rar[shorten <=-9pt]{0}&\begin{cases}\pi_{k-1}(P)\\\pi_{k-1}(Q)\end{cases}\rar[shorten <=-9pt]{\cong}\arrow[dr,"\cong",swap,shorten <=-4pt]&\pi_{k-1}(W)\dar{\cong}\rar& \ldots \\
	&&\pi_k(B)&&&\pi_{k-1}(B)&
\end{tikzcd}\]
of the pairs $(W,P)$ and $(W,Q)$. We first assume $k\ge2$. By the relative Hurewicz theorem, we have $\pi_k(W,P)\cong \oH_k(\widetilde{W},\widetilde{P})$ and similarly for $\pi_k(W,Q)$ where $\smash{\widetilde{(-)}}$ denotes the universal covers, so these groups are in particular finitely generated as $\pi_1(W)$-modules. Contemplating the diagram shows that there are finite sets of elements $\{p_i\}$ and $\{q_i\}$ of $\pi_k(W)$ that (i) map trivially to $\pi_k(B)$ (and thus trivially to $\pi_k(\BO)$) and (ii) map to sets of generators of $\pi_k(W,P)$ and $\pi_k(W,Q)$ as $\pi_1(W)$-modules respectively. As $2k<d$ we may represent these elements by two disjoint embeddings $\overline{p}\colon {\sqcup^i} S^k\times D^{d+1-k}\hookrightarrow \interior(W)$ and $\overline{q}\colon {\sqcup^i} S^k\times D^{d+1-k}\hookrightarrow \interior(W)$. Doing $\theta$-surgery on these embeddings (see \cite[Lemma 2]{Kreck}) yields a $\theta$-bordism $W'\colon P\leadsto Q$ which we claim to satisfy the requirements of the statement, that is $\pi_i(W',P)=0$ and  $\pi_i(W',Q)=0$ for $i\le k$. The reason being that (i) $\pi_i(W',P)$ vanishes for $i\le k-1$ since it is isomorphic to  $\pi_i(W,P)=0$ and (ii) $\pi_k(W',Q)$ vanishes since it is a quotient of $\pi_k(W,Q)$ by a subgroup that contains the $\pi_1(W)$-orbit of the images of $\{p_i\}$ and $\{q_i\}$ and we chose the $\{p_i\}$ so that their images generate $\pi_k(W',P)$ as $\pi_1(W)$-modules. The same argument applies to the groups $\pi_i(W',Q)$, so the claim in the case $k\ge2$ follows. For $k=1$, the same argument applies even though $\pi_1(W,P)$ and $\pi_1(W,Q)$ need no longer be groups: instead of the relative Hurewicz theorem, one uses that $\pi_1(W)$ is finitely generated, being the fundamental group of a compact manifold.
\end{proof}

Combining the previous two lemmas we get:

\begin{cor}\label{cor:theta-bordism-surgery}
	Let $\theta\colon B\ra \BO$ be a map and $(P,\ell_P)$ and $(Q,\ell_Q)$ two closed $\theta$-manifolds of dimension $d\ge4$ that are $\theta$-bordant. If $\ell_P$ and $\ell_Q$ are $k$-connected for some $k\ge0$ with $2k<d$, then $Q$ can be obtained from $P$ by a finite sequence of $p$-surgeries with $k\le p\le d-k-1$.
\end{cor}

\begin{proof}Lemmas~\ref{lem:handle-trading} and \ref{lem:bordism-connectivity} ensure that there is a bordism $W\colon P\leadsto Q$ of handle type $[k+1,d-k]$, which implies the statement.
\end{proof}

\subsection{$k$-type invariance}\label{sec:k-type-invariance} 
To state the announced tangential $k$-type result, we denote by $\gls*{hmanc}$ the $1$-category whose objects are smooth compact $d$-manifolds (potentially with boundary) and whose morphisms are isotopy classes of codimension $0$ embeddings. Fixing another $1$-category $\catC$, we prove the following result for functors of the form $F\colon h\Manc_d\ra \catC$.

\begin{thm}\label{thm:abstract-k-type-invariance}
Let $d\ge4$ and $F\colon h\Manc_d\ra \catC$ a functor such that $F$ maps codimension $0$ submanifold inclusions of relative handle type $[k+1,d]$ to isomorphisms for some fixed $0\le k<d/2$. Then for any compact $d$-manifolds $M$ and $N$ of the same tangential $k$-type, the following holds 
\begin{enumerate}
	\item\label{k-type-i} There exists an isomorphism $F(M)\cong F(N)$.
	\item\label{k-type-ii} For any codimension $0$ embedding $e\colon L\hookrightarrow M$ where $L$ has handle dimension $\le k$, there is an embedding $e'\colon L\hookrightarrow N$ for which the isomorphism \ref{k-type-i} can be chosen so that the diagram
	\[\begin{tikzcd}[row sep=0.2cm]
		& F(L) \arrow{ld}[swap]{F(e)} \arrow{rd}{F(e')} & \\
		F(M) \arrow{rr}{\cong} & & F(N)
	\end{tikzcd}\]
	is commutative.
	\item\label{k-type-iii} Any embedding $e\colon M\hookrightarrow N$ that is an equivalence on tangential $k$-types induces an isomorphism $F(M)\cong F(N)$ as in \ref{k-type-i}.
\end{enumerate}
\end{thm}

\begin{rem}\label{rem:psc-relation}
\cref{thm:abstract-k-type-invariance} is based on arguments we learned from the literature on the space of positive scalar curvature metrics on a manifold $M$, in particular \cite{EbertRWbordism,EbertWiemeler}. This space shares strong formal properties with the $\DiscInf$-structure space: it is often an infinite loop space \cite[Theorems A-B]{EbertRWbordism}, depends conjecturally only on the tangential $2$-type (see \cite[Conjecture C]{EbertWiemeler} and \cite[Section 9]{EbertRWbordism}), and is often non-trivial (see e.g.~\cite[Remark 1.1.1]{EbertRWbordism}).
\end{rem}

\begin{rem}\label{rem:nulbbordism-cat}
Taking complements, $h\Manc_d$ can be viewed equivalently as the ``homotopy category of null bordisms'' by which we mean the undercategory $\smash{h\BordInf(d)^{(\infty,1)}_{\varnothing/}}$ of the empty manifold $\varnothing$ viewed as an object in the homotopy category $h\BordInf(d)^{(\infty,1)}$, whose objects are closed $(d-1)$-manifolds and whose morphisms are diffeomorphism classes of compact bordisms.
\end{rem}

As preparation to the proof of \cref{thm:abstract-k-type-invariance}, we show that the values of the functor are invariant under certain surgeries.

\begin{lem}\label{lem:surgery-invariance}Let $F$ be as in  \cref{thm:abstract-k-type-invariance}. If two compact $d$-manifolds $M$ and $N$ differ by $p$-surgeries in the interior with $k\le p\le d-k-1$, then there exists an isomorphism $F(M)\cong F(N)$.
\end{lem}

\begin{proof}
	It suffices to show the claim in the case where $N$ is obtained from $M$ by a single $p$-surgery along an embedding $S^p\times D^{d-p}\hookrightarrow \interior(M)$. We consider the zig-zag 
	\[
		F(M)\longleftarrow F(M\backslash \interior(S^p\times D^{d-p}))\lra F(N)
	\]
	induced by the inclusions $M\backslash \interior(S^p\times D^{d-p})\subset M$ and $M\backslash \interior(S^p\times D^{d-p}) \subset N$. The former has relative handle type $[d-p,d]$ and the latter has relative handle type $[p+1,d]$. As $d-p\ge k+1$ and $p+1\ge k+1$ and $F$ sends inclusions of submanifolds of relative handle type $[k+1,d]$ to isomorphisms by assumption, we conclude the claim.
\end{proof}

\begin{proof}[Proof of \cref{thm:abstract-k-type-invariance}]Recall that two $d$-manifolds $M$ and $N$ have the same tangential $k$-type if there exists a map $\theta\colon B\ra \BO$ and $k$-connected $\theta$-structures $\ell_M$ and $\ell_N$ on $M$ and $N$. 
	
Part \ref{k-type-i} of the claim asserts an isomorphism $F(M)\cong F(N)$. In the case that $(M,\ell_M)$ and $(N,\ell_N)$ are closed manifolds that are $\theta$-bordant, this follows directly from \cref{cor:theta-bordism-surgery} and \cref{lem:surgery-invariance}. To show the general case, we pick a handle decomposition of $M$ viewed as a bordism $M\colon \varnothing\leadsto \partial M$ and consider the zig-zag (using the notation from \cref{sec:handle-dec})
\begin{equation}\label{equ:reduction-to-closed-case}
	F\big(M\big)\lla F\big(M{[0,k]}\big)\lra F\big(M[0,k]\cup_{M_k}\overline{M[0,k]}\big)
\end{equation}
whose arrows are induced by the bordisms $M[k+1,d]\colon M_k\leadsto \partial M$ and $\overline{M{[0,k]}}\colon M_k\leadsto \varnothing$. The former is of handle type $[k+1,d]$ and the latter of handle type $[d-k,d]$, so using that $d-k>k$ the two submanifold inclusions inducing the maps in the zig-zag have relative handle type $[k+1,d]$, so the zig-zag consists of isomorphisms. Applying the same reasoning for $N$, we see that the claim follows once we provide an isomorphism between the values of $F$ at the two doubles $M[0,k]\cup_{M_k}\overline{M[0,k]}$ and $N[0,k]\cup_{N_k}\overline{N[0,k]}$. Both of these doubles are closed manifolds that are $\theta$-nullbordant (see \cref{sec:theta-manifolds}), so they are in particular $\theta$-bordant to each other. This implies the claim by the first part as long as we make sure that the induced $\theta$-structures on these doubles are $k$-connected. But this is case, since it holds for $M$ and $N$ by assumptions and the above handle considerations in particular imply that all inclusions in $M\supset M[0,k]\subset M[0,k]\cup_{M_k}\overline{M[0,k]}$ and $N\supset N[0,k]\subset N[0,k]\cup_{N_k}\overline{N[0,k]}$ are $k$-connected. 
	
To prove part~\ref{k-type-ii}, we fix an embedding $L\hookrightarrow M$ as in the claim which we may assume by transversality to be contained in $M[0,k]\subset M$ as the complement $M[k+1,d]\supset \partial M$ has relative handle dimension $\le d-(k+1)$ and $L$ has handle dimension $\le k$ by assumption. The zig-zag \eqref{equ:reduction-to-closed-case} is then compatible with the maps from $F(L)$ induced by inclusion. Now $M[0,k]\cup_{M_k}\overline{M[0,k]}$ differs from $N[0,k]\cup_{N_k}\overline{N[0,k]}$ by surgeries of index $k\le p\le d-k-1$, which we may assume (again by transversality) to be done away from $L$, so there is an embedding $L\hookrightarrow N[0,k]\cup_{N_k}\overline{N[0,k]}$, such that the induced isomorphism $F(M[0,k]\cup_{M_k}\overline{M[0,k]})\cong F(N[0,k]\cup_{N_k}\overline{N[0,k]})$ is compatible with the maps from $F(L)$. Using transversality one last time, we see that we may isotope the embedding $L\hookrightarrow N[0,k]\cup_{N_k}\overline{N[0,k]}$ to land in $N[0,k]$ since $\overline{N[0,k]}\subset $ has handle dimension $\le k$ and $2k< d$. With respect to the isotoped embedding, the zig-zag of equivalences \eqref{equ:reduction-to-closed-case} is compatible with the maps from $F(N)$ and this concludes the proof.
	
For part~\ref{k-type-iii}, we may assume without loss of generality that the embedding is a submanifold inclusion of the form $M\subset M\cup_{\partial M} W$ for $W\colon \partial M\leadsto \partial N$ a bordism. We now consider the commutative square of codimension $0$ submanifold inclusions
\begin{equation}\label{equ:k-type-iii}\begin{tikzcd}
	c(M_k)\arrow[d,hook]\arrow[r,hook]& c(M_k)\cup_{M_k}M[k+1,d]\cup_{\partial M}W\arrow[d,hook]\\
	M\arrow[r,hook]&M\cup_{\partial M}W=N
\end{tikzcd}\end{equation}
where $c(M_k)\subset M$ is a closed bicollar of $M_k\subset M$. The vertical inclusions are of relative handle type $[k+1,d]$ (this uses $d-k\ge k+1$), so we conclude that they map to isomorphisms under $F$. It thus suffices to show that $F$ maps the top horizontal inclusion to an isomorphism. Since the vertical inclusions and the $\theta$-structures $\ell_M$ and $\ell_N$ are $k$-connected, it follows that the top horizontal inclusion is an equivalence on tangential $k$-type equivalence. Since $d\ge4$ we have $d/2\le d-2$, so $k<\min(d/2,d-2)$. Abbreviating $V=M[k,d] \cup_{\partial M} W$, an application of \cite[Lemma 6.10]{KrannichKupersOperadic} shows that we can factor the top horizontal inclusion as a composition of the form $c(M_k)\subset c(M_k)\cup_{M_k}V[k,k]\subset c(M_k)\cup_{M_k}V$ where the first inclusion is obtained by attaching \emph{trivial} $k$-handles and the second by attaching $\ge k+1$-handles. By assumption $F$ sends the second inclusion to an isomorphism, so it suffices to show that the same holds for the first inclusion. By attaching cancelling $(k+1)$-handles, the first inclusion fits into a sequence of inclusions $c(M_k)\subset c(M_k)\cup_{M_k}V[k,k]\subset c'(M_k)$ whose composition is given by attaching a collar (so is an isotopy equivalence) and the second inclusion is obtained by attaching $(k+1)$-handles. Now $F$ sends the second inclusion and the composition to isomorphisms, so also the first.
\end{proof}


\subsection{$2$-type invariance of the $\DiscInf$-structure space}\label{sec:2-type-invariance-sdisc} 
By \cref{sec:sdisc-for-manifolds}, the $\DiscInf$-structure spaces of compact manifolds form the values of a functor $S^{\DiscInf}_\partial(-)\colon \BordInf(d)^{(\infty,1)}_{\varnothing/}\ra \cS$ of $\infty$-categories, which induces on homotopy categories in view of \cref{rem:nulbbordism-cat} a functor
\[
	S^{\DiscInf}_\partial(-)\colon h\Manc_d\simeq  h\BordInf(d)^{(\infty,1)}_{\varnothing/}\lra h\cS.
\]

The goal of this section is to show that this functor satisfies the assumptions of \cref{thm:abstract-k-type-invariance} for $k=2$. This can be rephrased as follows:
 
\begin{prop}\label{prop:invariance-handles}
	Let $M\coloneqq \varnothing \leadsto P$ and $W\colon P \leadsto Q$ be $d$-dimensional bordisms. If $W$ is of handle type $[3,d]$, then the gluing map
	$(-\cup_P W)\colon S^{\DiscInf}_\partial(M)\ra S^{\DiscInf}_\partial(M\cup_PW)$
	is an equivalence. 
\end{prop}

Once this is proved, \cref{thm:abstract-k-type-invariance} implies the following refined version of \cref{bigthm:2-type-invariance}.

\begin{thm}\label{thm:2-type-invariance-detailed} Let $d\ge5$, and $M$, $N$ be two compact $d$-manifolds of the same tangential $2$-type.
\begin{enumerate}
	\item\label{enum:2-type-i} There exists an equivalence $S^{\DiscInf}_\partial(M)\simeq S^{\DiscInf}_\partial(N)$.
	\item\label{enum:2-type-ii} For any embedding $e\colon L\hookrightarrow M$ of a $d$-manifold $L$ with handle dimension $\le2$ there is an embedding $e'\colon L\hookrightarrow N$ so that the equivalence of \ref{enum:2-type-i} can be chosen to be compatible with
	\[
		e_*\colon S^{\DiscInf}_\partial(L)\ra S^{\DiscInf}_\partial(M)\quad\text{and}\quad e'_*\colon S^{\DiscInf}_\partial(L)\ra S^{\DiscInf}_\partial(N).
	\]
	\item\label{enum:2-type-iii} Any embedding $e\colon M\hookrightarrow N$ that induces an equivalence on tangential $2$-types induces an equivalence $S^{\DiscInf}_\partial(M)\simeq S^{\DiscInf}_\partial(N)$ as in \ref{enum:2-type-i}.
\end{enumerate}
\end{thm}

\begin{proof}[Proof of \cref{prop:invariance-handles}]Unravelling the statement using \cref{def:disc-structure-space}, the task is to show that
\[\begin{tikzcd}
	\BordInf(d)_{P}\dar[swap]{E}\rar{(-)\cup_PW}&[30pt] \BordInf(d)_{Q}\dar{E}\\
	\ModInf(d)^{\rep,\simeq}_{E_{P \times I}}\rar{(-)\cup_{E_{P \times I}}E_W}& \ModInf(d)^{\rep,\simeq}_{E_{Q \times I}}
\end{tikzcd}\]
is a pullback in $\cS$, where $\smash{\ModInf(d)^{\rep,\simeq}_{E_{P \times I}}\subset \ModInf(d)^{\simeq}_{E_{P \times I}}}$ and $\smash{\ModInf(d)^{\rep,\simeq}_{E_{Q \times I}}\subset \ModInf(d)^{\simeq}_{E_{Q \times I}}}$ are the $\infty$-groupoids given as the full subcategories of those objects in the image of the functor $E\colon \BordInf(d)_{P}\ra \ModInf(d)^{\simeq}_{E_{P \times I}}$ and in the image of its analogue for $P$ replaced by $Q$, respectively. We prove that it is a pullback by showing that the map on horizontal fibres are equivalences, for which we use that for any map $f\colon E\ra B$ in $\cS$ (thought of as a full subcategory of $\CatInf$) and a point $b\in B$, the fibre of $f$ over $b$ agrees with the colimit $\colim_{E}\,\Map_B(f(-),b)$. This follows from \cite[3.3.4.6]{LurieHTT} combined with the fact that the fibre over $b$ is the total space of the unstraightening of the functor $\Map_B(f(-),b)\colon E\ra \cS$ which in turn follows from \cite[3.3.2.8]{LurieHTT}. 

Applying this to the situation at hand and using the description of $E$ on mapping spaces from \cref{sec:functor-e-summary}, the claim follows once we show that for each nullbordism $N\in \BordInf(d)_{Q}$, the map
\vspace{-0.05cm}
\[
	\underset{(\BordInf(d)_{P})^\op}{\colim}\big[\Map_{\BordInf(d)_{Q}}\big((-)\cup_P W,N \big)\big]\xlra{E} \underset{(\ModInf(d)^{\rep,\simeq}_{E_{P \times I}})^\op}{\colim}\big[\Map_{\ModInf(d)^{\simeq}_{E_{Q \times I}}}\big((-)\cup_{E_{P \times I}}E_W,E_N\big)\big]
\]
is an equivalence. Using the factorisation $E\colon \BordInf(d)\to \ModInf(d)$ through the noncompact version of the bordism double $\infty$-category $\ncBordInf(d)$, this map fits into a commutative diagram
\[\hspace{-0.3cm}\begin{tikzcd}[column sep=0.2cm,row sep=0.5cm,ar symbol/.style = {draw=none,"\textstyle#1" description,sloped},
	equ/.style = {ar symbol={\simeq}}]
	\underset{(\BordInf(d)_{P})^\op}{\colim}\big[\Map_{\BordInf(d)_{Q}}\big((-)\cup_P W,N\big)\big]\dar{\circled{1}}\rar{E}& \underset{(\ModInf(d)^{\rep,\simeq}_{E_{P \times I}})^\op}{\colim}\big[\Map_{\ModInf(d)^{\simeq}_{E_{Q \times I}}}\big((-)\cup_{E_{P \times I}}E_W,E_N\big)\big]\dar{\circled{2}}\\
	\underset{(\ncBordInf(d)_{P})^\op}{\colim}\big[\Map_{\ncBordInf(d)_{Q}}\big((-)\cup_P W,N\big)\big]\rar{E} & \underset{(\ModInf(d)^{\rep}_{E_{P \times I}})^\op}{\colim}\big[\Map_{\ModInf(d)_{Q}}\big((-)\cup_{E_P}E_W,E_N\big)\big] \\[-0.1cm]
	\Map_{\ncBordInf(d)_{Q}}\big(P\times(-1,0]\cup_P W,N\big)\rar{E}\arrow[u,"\simeq",shorten >=-6pt] &\Map_{\ModInf(d)_{E_{Q \times I}}}\big(E_{P\times(-1,0]\cup_P W},E_N\big)\arrow[u,"\simeq",swap,shorten >=-8pt]
\end{tikzcd}\]
where the bottom vertical equivalences result from the fact that the bordism $(P\times (-1,0])\colon \varnothing\leadsto P$ is initial in $\ncBordInf(d)_{P}$ and its image under $E$ is initial in $\Mod(d)^{\rep}_{E_{P \times I}}$, by \cref{rem:initial-among-rep-presheaves}. By \cref{cor:convergence} the bottom map is an equivalence as the handle dimension of $P\times(-1,0]\cup_P W$ relative to $Q$ is $\le d-3$ by assumption. It thus suffices to show that $\circled{1}$ and $\circled{2}$ are equivalences.

We begin with $\circled{1}$. Since the mapping spaces in $\ncBordInf(d)_{P}$ are given by spaces of embeddings fixing the boundary and composition is given by composition of embeddings (see \cref{sec:details-ncbord}) and the same holds for $\BordInf(d)_{P}$ with embeddings replaced by diffeomorphisms (see \cref{sec:details-bord}), the map $\circled{1}$ is the map induced by restriction 
\[
	\underset{(\BordInf(d)_{P})^\op}\colim \, \Diff_\partial((-)\cup_P W,N)\lra \Emb_Q(W,N).
\]
Using the decomposition $\BordInf(d)_{P}=\bigsqcup_{M\in \pi_0\,\BordInf(d)_{P}}\BDiff_\partial(M)$ into path components (see \cref{sec:details-bord}), this can further be simplified as
\begin{equation}\label{equ:reduction-of-circled-1}
	\underset{M\in \pi_0\,\BordInf(d)_{P}}\bigsqcup \Diff_\partial(M\cup_P W,N)/ \Diff_\partial(M) \lra \Emb_Q(W,N).
\end{equation}
To show that the map \eqref{equ:reduction-of-circled-1} is an equivalence, we show  separately that it induces a bijection on components and that it is an equivalence on each component. To see that it is surjective on components, pick an embedding $e\in\Emb_Q(W,N)$. Up to changing $e$ within its isotopy class, we can assume that $P\subset W$ is mapped to the interior of $N$ and that the complement of $e(W\backslash P)\subset N$ defines a bordism $(N\backslash e(W\backslash P))\colon P\leadsto Q$. In this case the class in $\pi_0\,\Diff_\partial(M\cup_P W,N)/ \pi_0\,\Diff_\partial(M)=\pi_0(\Diff_\partial(M\cup_P W,N)/ \Diff_\partial(M))$ of the diffeomorphism $(N\backslash e(W\backslash P))\cup_P W\cong N$ obtained by extending $e$ by the identity provides a preimage of $[e]\in\pi_0\,\Emb_Q(W,N)$. Injectivity of \eqref{equ:reduction-of-circled-1} on $\pi_0$ follows from the isotopy extension theorem in the form of the homotopy fibre sequence
\vspace{-0.2cm}
\[
	\Diff_\partial(M)\xrightarrow{\phi\circ ((-)\cup_P{\id_W})} \Diff_\partial(M\cup_PW,N)\xra{\res} \Emb_Q(W,N)
\]
with fibre taken over the image of a diffeomorphism $\phi\colon M\cup_PW\cong N$. This sequence also implies that \eqref{equ:reduction-of-circled-1} is an equivalence on components, which finishes the proof for $\circled{1}$.

The argument for $\circled{2}$ is similar. Using Sections~\ref{sec:details-bimod} and~\ref{sec:functor-e-summary}, the reduction to showing that \eqref{equ:reduction-of-circled-1} is an equivalence applies also to the map $\circled{2}$ and shows that it agrees with the map
\[\begin{tikzcd}[row sep=0.33cm] 
	\underset{E_M\in \pi_0\,\Mod(d)^{\rep,\simeq}_{E_{P \times I}}}\bigsqcup \Map_{\Mod(d)_{E_{Q \times I}}^{\simeq}}(E_{M\cup_PW},E_N)/ \Aut_{\Mod(d)^{\simeq}_{E_{P \times I}}}(E_M) \dar[shorten <=-9pt] \\[-8pt] 
	\Map_{\Mod(d)_{E_{Q \times I}}^{\simeq}}(E_{P\times (-1,0]\cup_PW},E_N);
\end{tikzcd}\]
induced by the inclusion $P\times (-1,0]\cup_PW\subset M\cup_PW$. From the commutativity of the big diagram above and the fact that $\circled{1}$ and the bottom horizontal map are equivalences, we see that $\circled{2}$ is surjective on $\pi_0(-)$, so we must show it is injective on $\pi_0(-)$ and induces an equivalence on components. This follows as for $\circled{1}$ once we show that for $E_M\in\Mod(d)^{\rep,\simeq}_{E_{P \times I}}$ and $\phi\in \Map_{\Mod(d)_{E_{Q \times I}}^{\simeq}}(E_M\cup_{E_P} E_W,E_N)$ the sequence
\vspace{-0.2cm}
\[
	\Aut_{\Mod(d)^{\simeq}_{E_{P \times I}}}(E_M)\lra \Map_{\Mod(d)^{\simeq}_{E_{Q \times I}}}(E_{M\cup_PW},E_{N})\xra{\res} \Map_{\Mod(d)_{E_{Q \times I}}}(E_{P\times(-1,0]\cup_PW},E_{N}),
\]
whose left map is given by $\phi\circ ((-)\cup_{E_{P \times I}}{\id_{E_W}})$, is a homotopy fibre sequence when taking homotopy fibres over the image of $\phi$. By postcomposition with an inverse of $\phi$ it suffices to show this in the case $\phi=\id$. This follows from the second part of \cref{thm:isotopy-extension} (set $P=\varnothing$, $Q=P$, $R=Q$, $W=\varnothing$, $W'=W$, $M=M$, and $N=M$). The hypothesis to apply this result holds by \cref{cor:convergence}, since it follows from the assumption that $\underline{k}\times\bfR^d\sqcup P\times(-1,0]\cup_PW$ is the interior of a manifold obtained from a closed collar on $Q$ by attaching $(\le d-3)$-handles for all $k$.
\end{proof}

We conclude this section with a first application of the tangential $2$-type invariance. We will later use it to reduce the proof of the nontriviality result for $S^{\DiscInf}_\partial(M)$ to the case of $M=D^d$.

\begin{cor}\label{cor:homotopy-retract}For a compact spin $d$-manifold $M\neq\varnothing$ with $d\ge5$, the space $S^{\DiscInf}_\partial(M)$ contains $S^{\DiscInf}_\partial(D^d)$ as a homotopy retract. 
\end{cor}

\begin{proof}This essentially follows from the fact that any finitely presented group arises as the fundamental group of a compact connected codimension $0$-submanifold $N\subset D^k$ as long as $k\ge 5$ (in fact $k \geq 4$ is known to suffice, but we will not need this harder result). Indeed, apply this to $k=d$ and the fundamental group of each path component of $M$, to obtain a compact $d$-manifold $N\subset D^d$ whose fundamental groupoid is equivalent to that of $M$. Since $N$ admits an embedding into $D^d$, it is in particular spin, so the final discussion in \cref{rem:classification-2-types} shows that $M$ and $N$ have the same tangential $2$-type. Using the tangential $2$-type invariance of $S^{\DiscInf}_\partial(-)$ from \cref{thm:2-type-invariance-detailed}, it thus suffices to show the claim for $N$. The latter follows by choosing an embedded disc $D^d\subset N$ so that the composition $D^d\subset N\subset D^d$ is isotopic to the identity and applying $S_\partial^{\DiscInf}(-)$.
\end{proof}

\section{\cref{bigthm:infinite-loop-space}: infinite loop space} \label{sec:infinite-loop-space} 
The goal of this section is the proof of \cref{bigthm:infinite-loop-space}, or rather the following strengthening of it:

\begin{thm}\label{thm:oo-loop-general}For a compact manifold $M$ of dimension $d \geq 8$, $S^{\DiscInf}_\partial(M)$ admits the structure of an infinite loop space. If $M$ is $1$-connected spin, then the bound $d\ge8$ can be improved to $d\ge6$.\end{thm}

In \cref{sec:intr-infinite-loop-space}, we already gave an informal overview of the proof. We now make it precise.

\subsection{Operads with homological stability}
The proof of \cref{thm:oo-loop-general} relies on work of Basterra--Bobkova--Ponto--Tillmann--Yaekel \cite{BBPTY} on \emph{operads with homological stability} which generalises earlier work of Tillmann \cite{Tillmann}. We summarise their main result in this subsection.

\begin{rem}
\cite{BBPTY} is written in the setting of classical operads in topological spaces and algebras over them. To make it fit in our framework, we will rephrase their result in terms of (symmetric) $\infty$-operads and algebras over them (see \cref{sec:gen-infty-operads}). This translation is justified by the fact that there is an equivalence of $\infty$-categories between the $\infty$-category $\OpdInf$ of $\infty$-operads and the $\infty$-category underlying the model category $\sOp$ of classical coloured operads in simplicial sets (see \cite[p.\,858]{ChuHaugsengHeuts}) which is in turn Quillen equivalent to that of classical coloured operads in topological spaces; these equivalences do not affect the induced operad in the homotopy category. These equivalences extend to corresponding equivalences between categories of algebras over operads: \cite[Theorem 7.11]{PavlovScholbach} shows that (under mild conditions) the comparison functor from the $\infty$-category underlying the model category of algebras over a simplicial operad to the $\infty$-category of algebras over the associated $\infty$-operad is an equivalence, and applying $\Sing(-)$ induces (under mild conditions) a Quillen equivalence from the model category of algebras over a topological operad to the model category of algebras over the corresponding simplicial operad. The ``mild conditions'' in both steps are satisfied for all operads appearing in this section.
\end{rem}

Let $\bfN_0$ denote the set of non-negative integers. To state the main result of \cite{BBPTY}, we consider \emph{$\bfN_0$-graded $\infty$-operads} by which we mean (symmetric) $\infty$-operads $\cP$, together with a map of $\infty$-operads $\cP^{\otimes}\ra \gls*{finn}$ to the $\infty$-operad $\Fin^{\bfN_0}_\ast$ that is induced (via the operadic nerve \cite[2.1.1.27]{LurieHA}) by $\bfN_0$ under addition, considered as a symmetric monoidal category with a single object. Unpacking the definition, this amounts to an $\bfN_0$-indexed disjoint union decomposition $\Mul_\cP(x_1,\ldots,x_n;y)=\sqcup_{g\ge0}\Mul_\cP(x_1,\ldots,x_n;y)_g$ of all spaces of multi-operations that is additive under operadic composition. Every $\infty$-operad $\cO$ can be viewed as an $\bfN_0$-graded operad in grading $0$; formally this amounts to considering the composition $\cO^{\otimes}\ra\Fin_\ast\ra \Fin^{\bfN_0}_\ast$ where the second arrow is induced by the inclusion $\{0\}\subset \bfN_0$.

\begin{dfn}\label{dfn:operad-w-homstab}An \emph{operad with homological stability} is an $\bfN_0$-graded $\infty$-operad $\cP$ with a single colour (whose space of $k$-ary operations we write as $\Mul_{\cP}(\ast,\ldots,\ast;\ast)=\cP(k)=\sqcup_{g\ge0}\cP_g(k)$), together with
\begin{enumerate}
	\item\label{enum:data-homstab-operad-i} a map of $\bfN_0$-graded $\infty$-operads $\Assoc \ra \cP$ from the associative operad $\Assoc$ (see \cref{ex:associative-operad}) concentrated in degree $0$, and
	\item\label{enum:data-homstab-operad-ii} a distinguished element $s\in \cP_1(1)$, called the \emph{stabilising element},
\end{enumerate}
such that
\begin{enumerate}[(a)]
	\item\label{enum:cond-homstab-operad-i} the map on $2$-ary operations $\Assoc(2)\ra \cP_0(2)$ lands in a single path component, and
	\item\label{enum:cond-homstab-operad-ii} the map $\cP_\infty(k)\coloneqq \colim_g\,\cP_g(k)\ra\colim_g\,\cP_g(0)\eqcolon\cP_\infty(0)$ induced by taking horizontal colimits in the commutative diagram in $\cS$
	\[\begin{tikzcd}
		\cdots\rar&\cP_{g-1}(k)\rar{\circ_P(s;-)} \arrow[d,"{\circ_P(-;*,\ldots,*)}"]&[10pt]\cP_g(k)\rar{\circ_P(s;-)} \arrow[d,"{\circ_P(-;*,\ldots,*)}"]&[10pt]\cP_{g+1}(k)\dar{\circ_P(-;*,\ldots,*)}\rar&\cdots\\
		\cdots\rar&\cP_{g-1}(0)\rar{\circ_P(s;-)}&\cP_g(0)\rar{\circ_P(s;-)}\rar&\cP_{g+1}(0)\rar&\cdots
	\end{tikzcd}\]
	is an integral homology isomorphism for all $k\ge0$; here $\circ_P(-;-)$ is the operadic composition and $*\in \cP_0(0)$ is the image of $\ast \simeq\Assoc(0) \ra \cP_0(0)$.
\end{enumerate}
\end{dfn}

Given $\cP$ as in \cref{dfn:operad-w-homstab}, we may forget the grading and consider the composition 
\begin{equation}\label{equ:forgetfulgroup-completion-functor}
\Alg_\cP(\cS) \lra \Alg_{\Assoc}(\cS)\underset{\simeq}{\xrightarrow{\text{\cite[p\,465]{LurieHA}}}} \Mon(\cS)\xlra{\Omega B} \Mon^\grp(\cS)\xlra{U}\cS
\end{equation}
where the first arrow is the functor between $\infty$-categories of algebras in $\cS$ with its cartesian symmetric monoidal structure, induced by the morphism $\Assoc \ra \cP$ of $\infty$-operads (see \cref{sec:gen-infty-operads}), the second arrow is given by \emph{group-completion}, i.e.\,the left adjoint of the full subcategory inclusion $\Mon^\grp(\cS)\subset \Mon(\cS)$ of \emph{group-like objects}, i.e.\,those monoid objects $M\in\Mon(\cS)\subset \Fun(\Delta^{\op},\cS)$ in the sense of \cref{sec:cat-objects} for which the induced monoid of path components $\pi_0(M_{[1]})$ is a group, and the final arrow is the forgetful functor, given by evaluation at $[1]\in\Delta$. Recall (see e.g.\cite[5.2.6]{LurieHA}) that the composition of the final two arrows sends $M\in\Mon(\cS)$ to the pullback in $\cS$
\[\begin{tikzcd} 
	\Omega BM \rar \dar & M_{[0]} \simeq \ast \dar \\
	\ast\simeq M_{[0]} \rar & BM,
\end{tikzcd} 
\quad \text{with} \quad 
BM = \underset{\Delta^\op}\colim\,M.
\]
Writing $\Alg^\grp_{E_\infty}(\cS)\subset \Alg_{E_\infty}(\cS)$ for the full subcategory of group-like algebras in $\cS$ over the $E_\infty$-operad \cite[5.1.1.6]{LurieHA}, the main result of \cite{BBPTY} reads as follows:

\begin{thm}[Basterra--Bobkova--Ponto--Tillmann--Yeakel]\label{thm:stability-operads}
	For an operad with homological stability $\cP$, there exists a dashed functor fitting  into a commutative diagram of $\infty$-categories
	\[\begin{tikzcd}
		&\Alg^\grp_{E_\infty}(\cS)\dar{U}\\
		\Alg_{\cP}(\cS)\rar{\eqref{equ:forgetfulgroup-completion-functor}}\arrow[ur,dashed]&\cS
	\end{tikzcd}\]
	where $U$ is the forgetful functor.
\end{thm}
In other words, the underlying space of the group completion of an algebra over an operad with homological stability (considered as an ungraded operad) admits functorially the structure of a group-like $E_\infty$-algebra, or equivalently---by the recognition principle \cite[5.2.6.26]{LurieHA}---that of an infinite loop space.

\subsection{A manifold operad with homological stability}
The main example of an operad with homological stability considered in  \cite{BBPTY} is constructed out of the manifolds
\[
	W^{2n}_{g,k+l} \coloneqq W^{2n}_{0,k+l} \sharp (S^n \times S^n)^{\sharp g}\quad\text{with}\quad W^{2n}_{0,k+l}\coloneqq S^{2n}\backslash\interior((\sqcup^{k}D^{2n})\sqcup (\sqcup^{l}D^{2n}))
\]
for $k,l\ge0$ and $n\ge1$, considered as bordisms of the form $\sqcup^kS^{2n-1}\leadsto \sqcup^lS^{2n-1}$. Here $\smash{\sharp}$ denotes the connected sum operation. This is also the operad that is relevant for the proof of \cref{thm:oo-loop-general}, so we recall its construction in our setting. We omit the $2n$-superscripts for brevity.

Consider the tangential structure $\theta=\tau^*\Fr(\gamma)$ in the sense of \cref{sec:details-tangential-bord} given as the $\GL_{2n}(\bfR)$-space which is the pullback of the frame bundle of the universal bundle $\gamma\ra\BO(2n)$ along the $n$-connected cover map $\tau\colon \tau_{> n}\BO(2n)\ra \BO(2n)$. Since $S^{2n-1}$ is stably parallelisable, its once-stabilised tangent bundle admits a $\theta$-structure $\ell_0$ compatible its canonical orientation, unique up to equivalence of $\theta$-structures. We consider the symmetric monoidal $\infty$-category $\BordInf^{\theta}(2n)^{(\infty,1)}$ from \cref{sec:details-tangential-bord} and write $\BordInf^{\theta}(2n)^{(\infty,1),W}$ for the sub symmetric monoidal $\infty$-category (see \cref{ex:sub-sym-monoidal}) obtained by restricting objects to those equivalent to $\sqcup^k(S^{2n-1},\ell_0)$ for $k\ge0$ and restricting morphisms to those $\theta$-bordisms whose underlying bordism without $\theta$-structure is equivalent to a disjoint union of $W_{g,k+1}$'s for some $g,k\ge0$. Up to issues with components and different models, \cite[Theorem 1.3]{BBPTY} shows that the endomorphism operad
\[
	\cW\coloneq \End_{\BordInf^{\theta}(2n)^{(\infty,1),W}}(S^{2n-1},\ell_0)
\]
of $(S^{2n-1},\ell_0)$ in this category (see \cref{sec:end-operads}) can be enhanced to an operad with homological stability for all $2n\ge2$. For completeness and to deal with these issues, we give a proof in our setting by adapting their argument. As in \cite{BBPTY}, the main ingredient is a stable homological stability result of Galatius--Randal-Williams \cite{GRWII} (for the case $2n=2$ one can use \cite{Harer}). 

\begin{prop}\label{prop:manifold-operad-w-homstab}$\cW$ admits the structure of an operad with homological stability for all $2n\ge2$.\end{prop}

\begin{proof} By definition and \eqref{equ:mapping-cat-theta-bord-compact}, the space of $k$-ary operations
\[
	\cW(k)=\Map_{\BordInf^\theta(2n)^{(\infty,1),W}}\big({\sqcup^k} (S^{2n-1},\ell_0),(S^{2n-1},\ell_0)\big)
\] 
is the $\infty$-groupoid of $\theta$-bordisms $\sqcup^k (S^{2n-1},\ell_0)\leadsto (S^{2n-1},\ell_0)$ that are, after forgetting $\theta$-structures, equivalent to $W_{g,k+1}$ for some $g \geq 0$. As the manifolds $W_{g,k+1}$ are pairwise non-diffeomorphic for $g\ge0$, this induces a decomposition $\cW(k)=\sqcup_{g\ge0}\cW(k)_g$ which is compatible with the operad structure given by gluing bordisms with $\theta$-structures, so it gives rise to an $\bfN_0$-grading on $\cW$.

To construct a map $\Assoc\ra \cW$ from the associative operad (put in degree $0$), we first use \cref{ex:associative-operad} to recognise $\Assoc$ as a sub-operad in the sense of \cref{sec:suboperad} of the endomorphism operad $\End_{\BordInf^{\fr}(2)^{\partial,(\infty,1)}}(D^1,\stst)$ of the $1$-disc with the standard $1$-framing (that is, framing of its once-stabilised tangent bundle) considered as an object of the $2$-dimensional framed bordism category with boundary from \cref{sec:details-tangential-bord} (formally, the tangential structure involved is $\fr=(\id\colon \GL_{2}(\bfR)\ra \GL_{2}(\bfR)$). Namely, we restrict to those bordisms $(N,\ell)\colon {\sqcup^k}(D^1,\stst)\leadsto (D^1,\stst)$ for which $(N,\ell)$ is diffeomorphic (after smoothing corners) to $D^2$ with its standard framing such that ${\sqcup^k}(D^1,\stst)\subset \partial D^2$ is orientation preserving, $(D^1,\stst)\subset \partial D^2$ is orientation-reversing (see \cref{fig:bord-int} for an example). From \cref{ex:associative-operad}, one sees that this suboperad is equivalent to $\Assoc$ since its space of $k$-ary operations is homotopy discrete with components $\Sigma_k$ (with the regular $\Sigma_k$-action) as a consequence of the facts that (i) the diffeomorphism group of $D^2$ fixing some boundary intervals is contractible as a result of the equivalences $\Diff_\partial(D^1)\simeq\ast$ and $\Diff_\partial(D^2)\simeq\ast$ (the first is folklore, the latter is \cite[Theorem B]{Smale}) and that (ii) the space of framings of $D^2$ relative to fixed $1$-framings on collared intervals in the boundary is homotopy discrete (as $\Omega \GL_2(\bfR)$ is).

Now we consider the composition of symmetric monoidal $\infty$-categories
\vspace{-0.1cm}
\[
	\BordInf^{\fr}(2)^{\partial,(\infty,1)}\xra{(-)\times(D^{2n-1},\stst)} \BordInf^{\fr}(2n+1)^{\partial,(\infty,1)}\xlra{\partial} \BordInf^{\onefr}(2n)^{(\infty,1)}\lra\BordInf^\theta(2n)^{(\infty,1)}
\]
where the first arrow takes the product with $D^{2n-1}$ equipped with the standard framing and smooths corners (see \cref{ex:product-functor-with-framing}), the second arrow takes boundaries and lands in the bordism category with $1$-stabilised framings (see \cref{ex:framing-to-oneframing}), and the final arrow is induced by the naturality \eqref{eqn:bord-theta-naturality} in the tangential structure and the fact that there is a map of tangential structures $(\onefr) \to \theta$ since $\tau\colon \tau_{>n}\BO(2n)\ra\BO(2n)$ arises as the pullback of $\tau_{>n}\BO(2n+1)\ra \BO(2n+1)$ along $\BO(2n)\ra\BO(2n+1)$ and thus receives a map from the pullback of $* \to \BO(2n+1)$. Taking endomorphism operads and precomposing with the map from $\Assoc$, we have a composition \vspace{-0.1cm} 
\[
	\Assoc \overset{\subset}\lra \End_{\BordInf^{\fr}(2)^{\partial,(\infty,1)}}(D^1,\stst)\lra\End_{\BordInf^\theta(2n)^{(\infty,1)}}(S^{2n-1},\ell_0)
\]
which lands in the suboperad of $\smash{\End_{\BordInf^{\theta}(2n)^{(\infty,1)}}(S^{2n-1},\ell_0)}$ whose underlying bordisms are equivalent to $W_{0,k+1}$, using that $\partial(D^2\times D^{2n-1})\backslash \interior(\sqcup^{k+1}D^1\times D^{2n-1})\cong W_{0,k+1}$ after smoothing corners. In other words, it lands in the degree $0$-part of the operad $\cW$ and thus gives a map $\Assoc \ra \cW$ as in part \ref{enum:data-homstab-operad-i} of \cref{dfn:operad-w-homstab}. As $s\in \cW_1(1)$ in part \ref{enum:data-homstab-operad-ii}, we choose the bordism $W_{1,1}\colon S^{2n-1}\leadsto S^{2n-1}$ with an admissible $\theta$-structure as in \cite[p.\,130]{GRWII} that extends $\ell_0$ on the boundary spheres. 

It remains to check conditions \ref{enum:cond-homstab-operad-i} and \ref{enum:cond-homstab-operad-ii} of \cref{dfn:operad-w-homstab}. For \ref{enum:cond-homstab-operad-i}, one observes that already the composition $\smash{\Assoc(2)\ra\End_{\BordInf^\fr(2)^{\partial,(\infty,1)}}(D^1,\stst)(2) \ra \End_{\BordInf^\fr(2n+1)^{\partial,(\infty,1)}}(D^{2n},\stst)(2)}$
lands in a single path component, since the bordism $(D^{2n+1},\stst)\colon{\sqcup^2}(D^{2n},\stst)\leadsto (D^{2n},\stst)$ is for $n\ge1$ framed diffeomorphic to the same bordism with the two source components permuted using the isotopy extension theorem and the fact that the space of framed embeddings $\sqcup^2 D^{d}\hookrightarrow D^{d}$ is connected for $d\ge2$.
  
Finally, to verify \ref{enum:cond-homstab-operad-ii} we note that the image of $\ast\simeq\Assoc(0)\ra\cW$ is the bordism $D^{2n}\colon S^{2n-1}\leadsto \varnothing$, equipped with some $\theta$-structure, so the map $\cW_\infty(k)\ra \cW_{\infty}(0)$ is a homology equivalence as a result of applying \cite[Theorem 1.3]{GRWII} to the bordism (with some $\theta$-structure) 
\[
	(D^{2n})^{\sqcup k}\sqcup (S^{2n-1}\times [0,1])\colon (S^{2n-1})^{\sqcup k}\sqcup S^{2n-1}\leadsto S^{2n-1},
\]
which, being $(n-1)$-connected relative to its source, satisfies the condition of that theorem.
\end{proof}

\begin{figure}
	\begin{tikzpicture}[scale=.9]
		\foreach \i in {1,...,4} 
		{
			\draw[->,thick,Mahogany] (0,{.2+\i}) -- (0,{.8+\i});
			\draw (0,{-.2+\i}) to[out=0,in=0,looseness=2](0,{.2+\i});
		}
		\draw[->,thick,Mahogany] (0,.2) -- (0,.8);
		\draw[->,thick,Mahogany] (3,{.2+2}) -- (3,{.8+2});
		\draw (0,0.2) to[out=0,in=180] (3,2.2);
		\draw (0,4.8) to[out=0,in=180] (3,2.8);
	\end{tikzpicture}
	\caption{A $5$-ary operation in the $\infty$-operad $\cW$.}
	\label{fig:bord-int}
\end{figure}
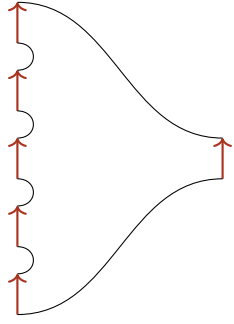

\subsection{Group completion and $\DiscInf$-structure spaces}\label{sec:map-of-w-algebras}
Fixing numbers $2\le 2n\le d$ and a closed $(d-2n)$-manifold $P$, we consider the sequence of symmetric monoidal $\infty$-categories
\begin{equation}\label{equ:theta-bord-to-morita}
	\BordInf^\theta(2n)^{(\infty,1),W}\subset \BordInf^\theta(2n)^{(\infty,1)}\xrightarrow{U}\BordInf(2n)^{(\infty,1)} \overset{P\times -}\lra \BordInf(d)^{(\infty,1)} \overset{E}\lra \Mod(d)^{(\infty,1)}
\end{equation}
where $U$ forgets tangential structures, $P\times(-)$ takes products (see \ref{step:product} in \cref{sec:the-functor}), and the final functor is discussed in \cref{sec:functor-e-summary}. \eqref{equ:theta-bord-to-morita} lands in the sub symmetric monoidal $\infty$-category 
\[
	\Mod(d)^{(\infty,1),W}\subset \Mod(d)^{(\infty,1)},
\] 
which is obtained by the restricting the objects to those equivalent to $E_{\sqcup^k P\times S^{2n-1}\times I}$ for $k\ge0$ and the morphisms to those bimodules equivalent to $E_{\sqcup^m P\times W_{g,k+1}}$ for some $m,k,g\ge0$. We write 
\[
	\BordInf(2n)^{(\infty,1),\overline{W}}\subset \BordInf(2n)^{(\infty,1)}
\] 
for the symmetric monoidal sub $\infty$-category obtained by restricting objects and morphisms to those that land in $\Mod(d)^{(\infty,1),W}\subset \Mod(d)^{(\infty,1)}$. Taking endomorphism operads, we obtain a composition of maps of $\infty$-operads
\[\hspace{-0.4cm}\Assoc \ra \cW=\End_{\BordInf^\theta(2n)^{(\infty,1),W}}(S^{2n-1},\ell_0) \ra \End_{\BordInf(d)^{(\infty,1),\overline{W}}}(P\times S^{2n-1})\ra \End_{\Mod(d)^{(\infty,1),W}}(P\times S^{2n-1}). 
\]
On $0$-ary operations, this in particular induces a map of $\cW$-algebras (see \cref{sec:map-as-algebra})
\begin{equation}\label{equ:map-of-endo-monoids}
	\Map_{\BordInf(2n)^{(\infty,1),\overline{W}}}(\varnothing, P\times S^{2n-1})\lra \Map_{\Mod(d)^{(\infty,1),W}}(E_\varnothing, E_{P\times S^{2n-1}\times I})
\end{equation}
which we can also view as a map of $\Assoc$-algebras in $\cS$, or equivalently, one of monoid objects in $\cS$. Going through the construction, the unit in $\Map_{\Mod(d)^{(\infty,1),W}}(\varnothing, E_{P\times S^{2n-1}\times I})$ is given by the bimodule $E_{P\times D^{2n}}$ and the fibre at that object of \eqref{equ:map-of-endo-monoids}, viewed as a map in $\cS$, is exactly $S^{\DiscInf}(P\times D^{2n})$ from \cref{sec:sdisc-for-manifolds}. Since the forgetful functor $\Mon(\cS)\ra \cS$ preserves limits, $S^{\DiscInf}(P\times D^{2n})$ inherits a monoid structure which fits into a pullback diagram in $\Mon(\cS)$
\begin{equation}\label{equ:sdisc-as-pullback}
	\begin{tikzcd} S^{\DiscInf}_\partial(P \times D^{2n}) \rar \dar & \Map_{\BordInf(d)^{(\infty,1),\overline{W}}}(\varnothing,P \times S^{2n-1}) \dar \\[-2pt]
	\ast \rar{E_{P \times D^{2n}}} & \Map_{\Mod(d)^{(\infty,1),W}}(E_\varnothing,E_{P \times S^{2n-1} \times I}).\end{tikzcd}
\end{equation}	
Under mild conditions, this square remains a pullback after group-completion. We show this as the first part of the following proposition.

\begin{prop}\label{prop:group-completion}Fix $2\le 2n\le d$ with $d\ge6$ and a closed $(d-2n)$-manifold $P$.
\begin{enumerate}
	\item \label{enum:group-completion-i}If also $2n\ge4$, then the pullback \eqref{equ:sdisc-as-pullback} in $\Mon(\cS)$ remains a pullback after group-completion.
	\item \label{enum:group-completion-ii}$S_\partial^{\DiscInf}(P\times D^{2n})$ is group-like when considered as a monoid object in $\cS$. 
\end{enumerate}
\end{prop}

\begin{proof}
The first part is an application of the following fact, which can be deduced from \cite[Theorem 2.11]{Steimle}: if a map $\varphi\colon X\ra Y$ of monoid objects in $\cS$ has the property that for all $y\in Y$ there is an $x\in X$ such that $\varphi(x)=y$ and the following squares are pullbacks in $\cS$
\[\begin{tikzcd}
	X\dar\rar{(-)\cdot x}&[10pt]X\dar\\[-2pt]
	Y\rar{(-)\cdot y}&Y
\end{tikzcd}
\quad \text{and}\quad
\begin{tikzcd}
	X\dar\rar{ x\cdot (-)}&[10pt]X\dar\\[-2pt]
	Y\rar{y\cdot(-)}&Y,
\end{tikzcd}\]
then group completion preserves pullbacks of monoid objects in $\cS$ along the map $\varphi\colon X\ra Y$.

To conclude \ref{enum:group-completion-i}, it thus suffices to check the condition for the right vertical map in \eqref{equ:sdisc-as-pullback} which amounts to showing that the square in $\cS$
\[
\begin{tikzcd}[column sep=2cm]
\Map_{\BordInf(d)^{(\infty,1),\overline{W}}}(\varnothing,P \times S^{2n-1}) \rar{(-)\Ydown (P\times W_{g,1})} \dar & \Map_{\BordInf(d)^{(\infty,1),\overline{W}}}(\varnothing,P \times S^{2n-1}) \dar \\
	\Map_{\Mod(d)^{(\infty,1),W}}(E_\varnothing,E_{P \times S^{2n-1} \times I})\rar{(-)\Ydown E_{P\times W_{g,1}}} & \Map_{\Mod(d)^{(\infty,1),W}}(E_\varnothing,E_{P \times S^{2n-1} \times I})
\end{tikzcd}
\]
is cartesian for all $g\ge0$, where $(-)\gls*{ppproduct}(-)$ denotes the monoid structure of the monoid objects in \eqref{equ:sdisc-as-pullback}, and that the same holds for the square where we take products from the left. We focus on the former; the latter is proved in the same way.

Going through the construction of the map $\Assoc\ra\cW$ in the proof of \cref{prop:manifold-operad-w-homstab}, we see that $(-)\Ydown (P\times W_{g,1})$ is given a ``pair of pants-product'': it sends a bordism $M\colon \varnothing \leadsto P\times S^{2n-1}$ to the disjoint union $(M\sqcup P\times W_{g,1})\colon \varnothing \leadsto \sqcup^2P\times S^{2n-1}$ and then takes composition with $(P\times W_{0,2+1})\colon \sqcup^2P\times S^{2n-1}\leadsto P\times S^{2n-1}$. By monoidality, this agrees with the map that sends $M\colon \varnothing\leadsto  P\times S^{2n-1}$ first to its composition with $([0,1]\times P\times S^{2n-1}\sqcup P\times W_{g,1})\colon P\times S^{2n-1}\leadsto \sqcup^2 P\times S^{2n-1}$ and then takes composition with $P\times W_{0,2+1}\colon \sqcup^2P\times S^{2n-1}\leadsto P\times S^{2n-1}$. The composition of the latter two bordisms is diffeomorphic, as a self-bordism of $P\times S^{2n-1}$, to $ P\times W_{g,1+1}$. The same argument applies to $(-)\Ydown E_{P\times W_{g,1}}$, so using monoidality of the functor $E\colon  \BordInf(d)^{(\infty,1)}\ra \Mod(d)^{(\infty,1)}$, we may replace the top and bottom maps in the previous square by the gluing maps $(-)\cup_{P\times S^{2n-1}}(P\times W_{g,1+1})$ and $(-)\cup_{E_{P\times S^{2n-1}\times I}}E_{P\times W_{g,1+1}}$ respectively. Taking vertical homotopy fibres, it thus suffices to show that for $h\ge0$ the gluing map
\[\big((-)\cup_{P\times S^{2n-1}}(P\times W_{g,1+1})\big)\colon S_\partial(P\times W_{h,1})\ra S_\partial(P\times W_{h+g,1})\]
is an equivalence. In the setting of \cref{thm:2-type-invariance-detailed}, this map is induced by the inclusion $(\id_P\times \inc)\colon P\times W_{h,1}\hookrightarrow P\times W_{h+g,1}$, so it is an equivalence by part \ref{enum:2-type-iii} of the theorem because the latter inclusion is an equivalence on tangential $2$-types as a result of $W_{h,1}$ for all $h\ge0$ being parallelisable and $1$-connected since we assumed $2n\ge4$.

To show \ref{enum:group-completion-ii}, we first recall from \cref{sec:disc-structure-spaces} that $\pi_0\, S_\partial^{\DiscInf}(P\times D^{2n})$ is the set of equivalence classes of pairs $(W,\phi)$ of a compact manifold $W$ whose boundary is identified with $P\times S^{2n-1}$, together with an equivalence $\phi\colon E_M\ra E_{P\times D^{2n}}$ in $\Mod(d)^{\simeq}_{P\times S^{2n-1}}$, and two such pairs are equivalent if there exists a diffeomorphism between the manifolds that makes the evident triangle in $\Mod(d)^{\simeq}_{P\times S^{2n-1}}$ homotopy commute. Forgetting $\phi$ induces an exact sequence of pointed sets
\begin{equation}\label{eqn:two-homomorphisms}
		\pi_0\,\Aut_{\Mod(d)^{\simeq}_{E_{P \times S^{2n-1} \times I}}}(E_{P \times D^{2n}}) \lra \pi_0\,S^{\DiscInf}_\partial(P \times D^{2n}) \lra M_\partial(P \times D^{2n}) \lra 0,
\end{equation}
where $M_\partial(P \times D^{2n})$ is the pointed set of compact $d$-manifolds $W$ with boundary identified with $P \times S^{2n-1}$ such that there exists an unspecified equivalence $E_W \simeq E_{P \times D^{2n}}$ in $\Mod(d)_{E_{P \times D^{2n}}}$, up to diffeomorphism relative to the boundary, and based at $P \times D^{2n}$. The monoid structure on $\pi_0\,S^{\DiscInf}_\partial(P \times D^{2n})$ given by the ``pair of pants''-product induces a compatible monoid structure on $M_\partial(P \times D^{2n})$, concretely given by
\begin{equation}\label{equ:pair-of-pants}
	W \Ydown W' \coloneqq (W \sqcup W') \cup_{P \times S^{2n-1} \sqcup P \times S^{2n-1}} P \times W_{0,2+1}.
\end{equation}
A priori, the leftmost pointed set in \eqref{eqn:two-homomorphisms} carries \emph{two} monoid structures---one induced by the ``pair of pants''-product $(-)\Ydown(-)$ and one by composition---but these agree by the Eckmann--Hilton argument. Thus \eqref{eqn:two-homomorphisms} is an exact sequence of monoids whose leftmost term is a group. Monoid-extensions of groups are groups, so it suffices to show that $M_\partial(P \times D^{2n})$ is a group. We do so by showing that every element has a right- and a left-inverse; since the two constructions are essentially identical, we will only explain the right-inverse. 
	
For this it is convenient to use the notion of \emph{relative bordism} $V \colon N_0 \leadsto N_1$ between two compact manifolds $N_0$ and $N_1$ with identified boundary $\partial N_0\cong \partial N_1$, by which we mean a compact manifold $V$ with a division of its boundary into three codimension zero submanifolds $\partial V = N_0 \cup (\partial N_0 \times I) \cup N_1$ that intersect at corners. Up to creating some corners, we can regard an element $W\in M_\partial(P\times D^{2n})$ as a relative bordism of the form $P \times D^{2n-1} \leadsto P \times D^{2n-1}$ by dividing the identified boundary $\partial W\cong P\times S^{2n-1}$ into $(P \times D^{2n-1}\times\{0\}) \cup (P \times \partial D^{2n-1}\times [0,1]) \cup (P \times D^{2n-1}\times\{1\})$. In these terms, the monoid structure is given by composition of relative bordisms. By definition of $M_\partial(P\times D^{2n})$, the manifold $W$ admits an equivalence $\phi \colon E_W \to E_{P \times D^{2n}}$ in $\Mod(d)_{E_{P\times S^{2n-1}\times I}}$. In general, for manifold $N$ viewed as a nullbordism $N\colon \varnothing\leadsto \partial N$, we can consider the composition  \[E_{\partial N\times I}\simeq E_{\varnothing}\otimes E_{\partial N\times I}\xra{E_{\iota}\otimes \id} E_{N}\otimes E_{\partial N\times I}\xra{\text{act}} E_N\]
in $\PSh(\DiscInf_d)$ using the Day convolution product $\otimes$, the unique embedding $\varnothing\ra N$ and the fact that $E_{\varnothing}$ is the monoidal unit. Evaluating this composition at $\bfR^d$ and taking quotients by the $\Diff(\bfR^d)\simeq\Emb(\bfR^d,\bfR^d)$-action by functoriality recovers the homotopy class of the boundary inclusion $\partial N\subset N$. Applying this principle to the equivalence $\phi$ above, we obtain a homotopy equivalence $W \simeq P \times D^{2n}$ under $P \times S^{2n-1}$. In terms of relative bordisms, this says that $W$ is a \emph{strongly inertial} relative $h$-cobordism: that is, not only are the inclusions of the incoming and outgoing boundary homotopy equivalences, but the induced homotopy equivalence between them is homotopic to a diffeomorphism relative to the boundary. Since we assumed $d \geq 6$, relative $h$-cobordisms $W \colon W_0 = P \times D^{2n-1} \leadsto W_1$, up to diffeomorphism relative to the incoming boundary $P \times D^{2n-1}$, are classified by their Whitehead torsion $\tau(W) \in \Wh_1(\pi_1\,P)$, and the Whitehead torsions of strongly inertial relative $h$-cobordisms form a subgroup (cf.\,the discussion in \cite[Section 3]{JahrenKwasik}). Thus we may find another strongly relative $h$-cobordism $W' \colon P \times D^{2n-1} \leadsto P \times D^{2n-1}$ with a diffeomorphism $W \cup_{P \times D^{2n-1}} W' \cong P \times D^{2n}$ that respects part of the boundary identification, namely $P \times D^{2n-1} \{0\}\cup (P \times \partial D^{2n-1} \times [0,1])$. By changing the identification of the outgoing boundary of $W'$ if necessary, we may assume that this diffeomorphism respects the full boundary identification. Smoothing corners, this gives a diffeomorphism $\psi \colon W \Ydown W' \to P \times D^{2n}$ relative to $P \times S^{2n-1}$. To show that $W'$ is a right inverse to $W$ in $M_\partial(P\times D^{2n})$ it thus suffices to produce an equivalence $E_{W'}\simeq E_{P \times D^{2n}}$ in $\Mod_{E_{P \times S^{2n-1} \times I}}$. This is given by: \vspace{-0.2cm}
\[
	E_{W'} \simeq E_{P \times D^{2n}} \Ydown E_{W'} \xrightarrow{\phi^{-1} \Ydown \id} E_W \Ydown E_{W'} \simeq E_{W \Ydown W'} \xrightarrow{E_\psi} E_{P \times D^{2n}}.\qedhere
\]
\end{proof}

\begin{cor}\label{cor:infinite-loop-space-for-products}For $4\le 2n\le d$ with $d\ge6$ and a closed $(d-2n)$-manifold $P$, the $\DiscInf$-structure space $S^{\DiscInf}_\partial(P\times D^{2n})$ admits the structure of an infinite loop space.
\end{cor}

\begin{proof}Combining both parts of \cref{prop:group-completion}, $S^{\DiscInf}_\partial(P\times D^{2n})$ agrees with the fibre of the group completion of the right vertical map in \eqref{equ:sdisc-as-pullback}. As the group completion of a map of $\cW$-algebras, this map can be enhanced to a map of infinite loop spaces by \cref{thm:stability-operads} and \cref{prop:manifold-operad-w-homstab}. Fibres of infinite loop maps carry infinite loop space structures, so the claim follows. 
\end{proof}

Combining \cref{cor:infinite-loop-space-for-products} with the invariance under the tangential $2$-type from \cref{thm:2-type-invariance-detailed}, we can complete the goal of this section: 

\begin{proof}[Proof of \cref{thm:oo-loop-general}] For $M$ a compact $d$-manifold with $d\ge8$, we pick an $2n\ge4$ such that $2n-d\ge4$ (the choice $2n=4$ always works) and use the case $k=2$ of \cref{lem:nice-representative-k-type} to pick a closed $(d-2n)$-manifold $P$ of the same tangential $2$-type as $M$. Both $P\times D^{2n}$ and $M$ are $d$-dimensional and have the same tangential $2$-type, so $S^{\DiscInf}_\partial(P\times D^{2n})\simeq S^{\DiscInf}_\partial(M)$ by \cref{thm:2-type-invariance-detailed} \ref{enum:2-type-i}. As $S^{\DiscInf}_\partial(P\times D^{2n})$ admits the structure of an infinite loop space by \cref{cor:infinite-loop-space-for-products}, the first part  follows. For the claimed improvement, one can replace the role of $P$ in the argument with $P=S^{d-2n}$ for any $2n\ge4$ with $d-2n\ge2$, using that any two $1$-connected spin manifolds have the same tangential $2$-type (see \cref{rem:classification-2-types}).
\end{proof}

\begin{rem}\label{rem:drawbacks-loop-space-structure}The construction of the infinite loop space structure on $\cS_\partial^{\DiscInf}(M)$ as presented in this section comes with several drawbacks:
\begin{enumerate}
	\item It depends on several choices, most notably: (a) the choice of $2n\ge4$ with $2n-d\ge4$ and (b) the choice of a closed $(d-2n)$-manifold $P$ of the same tangential $2$-type as $M$. In particular, the construction does not lift the functor $\smash{S_\partial^{\DiscInf}(-)\colon \BordInf(d)_{/\varnothing}^{(\infty,1)}\ra \cS}$ to a functor with values in $\smash{\Alg_{E_\infty}^\grp(\cS)}$, but it does enhances the $\Diff(P)$-action on the space $\cS_\partial^{\DiscInf}(D^{2n}\times P)$ for fixed $2n\ge4$ and a closed manifold $P$ to an action in $\smash{\Alg_{E_\infty}^\grp(\cS)}$.
	\item The restrictions on the dimension are likely not optimal.
	\item The space $S^{\DiscInf}_\partial(P\times D^{2n})$ ought to carry the structure of an $E_{2n}$-algebra and the infinite loop space structure we give ought to extend this $E_{2n}$-structure.
\end{enumerate}
We expect that there is a better construction of the infinite loop space structure on $\cS_\partial^{\DiscInf}(M)$ that does not suffer from these shortcomings.
\end{rem}

\section{Localisations of mapping spaces between operads} \label{sec:ed-operads} 
This section serves to prove general results on mapping spaces between (truncated) operads and their localisations at collections of primes. In particular, given $\infty$-operads $\cO$ and $\cP$, we rely on work of Göppl--Weiss \cite{Goppl} to study the effect on homotopy groups of a map
\begin{equation}\label{equ:t-localisation-mapping-spaces}
	\Map(\cO,\cP)_\bfQ\ra \Map(\cO_\bfQ,\cP_\bfQ)
\end{equation}
from the rationalisation of the mapping space between $\cO$ and $\cP$ to the mapping space between the respective rationalisations. In \cref{sec:nontrivial}, we use these results to prove Theorems~\ref{bigthm:nontrivial} and \ref{bigcor:top-vs-auted}. 

\begin{convention}\label{conv:no-more-infty}
Up to this point, we phrased our results and arguments in the language of $\infty$-categories. In this and the following section, we will use several intermediate results from various sources, none of which are written in this language. To stay close to these sources, we switch language for the remainder of this paper and work in the category of simplicial sets or the category of compactly generated weak Hausdorff spaces. We denote either of these categories by $\cat{S}$ and leave the necessary transitions based on the usual Quillen equivalence between the standard model structures on these categories to the reader. As a result of not working $\infty$-categorically, we have to derive all mapping spaces in various categories that appear (spaces, operads, etc.) with respect to a class of weak equivalences, e.g.\,using Dwyer--Kan's functorial simplicial localisation \cite{DwyerKanFunction,DwyerKanSimplicial}. We indicate various derived mapping spaces by adding an $h$-subscript, so write $\Map^h(-,-)$, and we will mention the class of weak equivalences with respect we derive whenever a new type of derived mapping space is considered.
\end{convention}

\subsection{Localisation of spaces and groups at a set of primes}\label{sec:localisation}
We first recall some facts about $T$-localisations of spaces for a set of primes $T$. Recall that a space $Z$ is \emph{$T$-local} if the map $(- \circ g) \colon \Map^h_\cat{S}(Y,Z) \to \Map^h_\cat{S}(X,Z)$ is a weak equivalence for any map $g \colon X \to Y$ that is an isomorphism on $\oH_*(-;\bfZ_T)$. Here the mapping spaces are derived with respect to  weak homotopy equivalences, and $\bfZ_T$ is the localisation of $\bfZ$ obtained by inverting all primes in $T$. Immediately from the definition, we see that the class of $T$-local spaces is closed under
\begin{enumerate}[(i)]
	\item taking homotopy limits,
	\item passing to collections of path components,
	\item applying $\Map^h_\cat{S}(X,-)$ for any space $X$.
\end{enumerate}
A map $f \colon X \to Y$ is a \emph{$T$-localisation} if $Y$ is $T$-local and $f$ is a $\bfZ_T$-homology isomorphism. Any space admits a $T$-localisation and suitably modelled, this yields an $\cat{S}$-enriched functor 
\begin{equation}
	\label{equ:localisation} \gls*{tloc} \colon \cat{S} \lra \cat{S}
\end{equation}
together with a natural transformation $r_T \colon \id \to (-)_T$ which enjoys the following properties (see e.g.\ \cite[1.A.3, 1.A.8, 1.B.2, 1.B.7, 1.C.9,  1.C.13, 1.E.4]{Farjoun}):
\begin{enumerate}[(a)]
	\item \label{enum:T-localisation} the map $r_T \colon X \to X_T$ is a $T$-localisation, so a weak equivalence if $X$ is $T$-local,
	\item \label{enum:weak-equ-T-localisation}$(-)_T$ preserves weak equivalences,
	\item \label{enum:product-T-localisation} the canonical map $(X\times Y)_T\ra X_T\times Y_T$ is a weak equivalence,
	\item \label{enum:precomp-T-localisation}the map $(-)\circ r_T\colon \Map^h_\cat{S}(X_T,Y)\ra \Map^h_\cat{S}(X,Y)$
is a weak equivalence if $Y$ is $T$-local.
\end{enumerate}
If $T$ is the set of all primes, $T$-localisation is \emph{rationalisation} which we denote as $\gls*{ratloc}$.

\subsubsection{Localisation of groups}\label{sec:localisation-groups}
Recall that a group $G$ is \emph{$T$-local} if the map $(- \circ g) \colon \Hom(H,G) \to \Hom(K,G)$ is an isomorphism for all $g \colon K \to H$ such that $\oH_1(g;\bfZ_T)$ is an isomorphism and $\oH_2(g;\bfZ_T)$ is surjective. The homotopy groups of a $T$-local space at any basepoint are $T$-local groups \cite[Theorem 5.5]{Bousfield}. A morphism of groups $f\colon H\ra G$ is a $T$-localisation if $G$ is $T$-local and $f$ has the property on $\oH_*(-;\bfZ_T)$ for $*=1,2$ just described. One way to construct $T$-localisations of groups is as follows: the functor \eqref{equ:localisation} has an analogue $(-)_T \colon \cat{S}_* \ra \cat{S}_*$ in the pointed setting, which agrees with \eqref{equ:localisation} on connected spaces \cite[A.7]{Farjoun}. Defining $G_T\coloneqq \pi_1((BG)_T)$, we obtain a functor $(-)_T \colon \cat{Grp} \ra \cat{Grp}$ on the category of groups with a natural transformation $\id\ra (-)_T$ which is a $T$-localisation \cite[Lemma 7.3]{Bousfield}. Note that we have $(G)^{\ab}\otimes\bfZ_T\cong (G_T)^{\ab}\otimes\bfZ_T$ by construction and the Hurewicz theorem. On nilpotent groups, $(-)_T$ agrees with the usual $T$-localisation of nilpotent groups in the algebraic sense.

\subsubsection{Localisation of nilpotent spaces}
Recall that a space $X$ is \emph{nilpotent} if it is connected, has nilpotent fundamental group, and its $\pi_1(X)$-action on $\pi_i(X)$ for $i\ge2$ is nilpotent. $T$-localisation preserves nilpotent spaces and can be characterised as follows  (see e.g.\,\cite[6.1.2]{MayPonto}):

\begin{lem}Let $f \colon X \to Y$ be a map from a nilpotent space $X$ to a $T$-local space $Y$. Then the following are equivalent:
	\begin{enumerate}
		\item $f \colon X \to Y$ is a $T$-localisation of spaces,
		\item $f_* \colon \widetilde{\oH}_k(X;\bfZ) \to \widetilde{\oH}_k(Y;\bfZ)$ is a $T$-localisation of abelian groups for all $k\ge1$,
		\item $f_* \colon \pi_k(X) \to \pi_k(Y)$ is a $T$-localisation of abelian and nilpotent groups for all $k\ge1$.
	\end{enumerate} 
\end{lem}

Localisations of nilpotent spaces behave well with respect to many constructions, such as:

\begin{lem}\label{lem:lem-hopb}Let $f \colon X \to A$ and $g \colon Y \to A$ be based maps between spaces with nilpotent basepoint component. Then
	\begin{enumerate}
		\item \label{enum:loc-hopb-nilpotent} the basepoint component $(X\times^h_AY)_0\subseteq X\times^h_AY$ of the homotopy pullback is nilpotent,
		\item \label{enum:loc-hopb-loc} the natural map $(X\times^h_AY)_0 \to (X_T\times^h_{A_T}Y_T)_0$ is a $T$-localisation of nilpotent spaces, and
		\item \label{enum:loc-hopb-fg} if $X_0$, $Y_0$, and $A_0$ have finitely generated homotopy groups then so does $(X\times^h_AY)_0$.
	\end{enumerate}
\end{lem}

\begin{proof}Since $(X_0\times^h_{A_0}Y_0)_0=(X\times^h_AY)_0$ and similarly for the localised version, we may assume that $X$, $Y$, and $A$ are connected. In this case, \ref{enum:loc-hopb-nilpotent} and \ref{enum:loc-hopb-loc} are \cite[6.2.5]{MayPonto}. For \ref{enum:loc-hopb-fg}, we use the long exact sequence for the homotopy groups of a homotopy pullback which exhibits $\pi_i(X\times^h_AY)$ for $i\ge1$ as a central extension of subquotients of finitely generated nilpotent groups. As the latter are closed under taking subgroups, quotients, and extensions, the statement follows.
\end{proof}

The next lemma involves equivariant mapping spaces $\Map_{G}(-,-) \coloneqq \Map_{\cat{S}^G}(-,-)$ between $G$-spaces for finite groups $G$ which we derive with respect to the $G$-equivariant maps whose underlying maps of spaces are weak homotopy equivalences.

\begin{lem}\label{lem:equivariant-mapping-spaces-nilpotent}Let $X$ and $Y$ be $G$-spaces for $G$ a finite group. If
\begin{itemize}
\item $X_{hG}$ is weakly equivalent to a finite CW complex and 
\item $Y$ has nilpotent path components,
\end{itemize}
then for any $f\in\Map^h_G(X,Y)$ the following holds:
\begin{enumerate}
\item the path component $\Map^h_G(X,Y)_f\subseteq \Map^h_G(X,Y) $ is nilpotent,
\item the postcomposition map $(r_T \circ (-))\colon \Map^h_G(X,Y)_f\ra\Map^h_G(X,Y_T)_{(r_T\circ f)}$ is a $T$-localisation,
\item\label{enum:Y-finitely-gen-homotopy} if $Y$ has finitely generated homotopy groups at all basepoints, then so does $\Map^h_G(X,Y)_f$.
\end{enumerate}
\end{lem}

\begin{proof}
By the assumption on $X_{hG}$, we may assume that $X$ is a finite $G$-CW complex consisting of free $G$-cells. This allows us to argue by induction on the number of cells: if $X$ is obtained from $X'$ by attaching a single free $G$-cell, there are commutative squares
\[\begin{tikzcd} 
	S^{d-1} \times G \rar \dar & X' \dar \\
	D^{d-1} \times G \rar & X\end{tikzcd}\qquad\text{and}\qquad\begin{tikzcd} \Map ^h_G(X,Y) \rar \dar & \Map ^h(D^d,Y) \dar \\
	\Map ^h_G(X',Y) \rar & \Map ^h(S^{d-1},Y).
\end{tikzcd}\]
The left square is a homotopy pushout of $G$-spaces and the right square is obtained from it by applying $\Map ^h_G(-,Y)$, so it is a homotopy pullback. By an induction over a principal Postnikov tower of the path components of $Y$, one sees that the conclusions hold for all components of the right-hand terms of the right-hand diagram. By induction, we may assume they hold for the bottom-left corner in the right diagram, so using \cref{lem:lem-hopb} and that subgroups of (finitely generated) nilpotent groups are (finitely generated) nilpotent, they also hold for all components of the top-left corner in the right diagram.
\end{proof}

Recall that a \emph{$\unl{k}$-cubical diagram} is a functor on the poset of subsets of $\unl{k}\coloneqq \{1,\ldots,k\}$.

\begin{lem}\label{lem:nilpotent-cube}
Let $X$ be a $\unl{k}$-cubical diagram of spaces with nilpotent path components.
\begin{enumerate}
	\item $\holim_{\varnothing\neq I\subseteq \unl{r}}X(I)$ has nilpotent components and the map $\holim_{\varnothing\neq I\subseteq \unl{k}}X(I)\ra \holim_{\varnothing\neq I\subseteq \unl{k}}(X(I)_T)$ induced by the $T$-localisations of the $X(I)$'s, is a $T$-localisation when restricted to any component of the source and the corresponding component of the target.
	\item If $X(I)$ has finitely generated homotopy groups at all basepoints for all $\varnothing \neq I\subseteq \unl{k}$, then $\holim_{\varnothing\neq I\subseteq \unl{k}}X(I)$ has finitely generated homotopy groups at all basepoints.
\end{enumerate}
\end{lem}

\begin{proof}
We prove the claim by induction on $k$. For $k=1$  the claim is vacuous as  $\holim_{\varnothing\neq I\subseteq \unl{1}}X(I)\simeq X(\underline{1})$. For larger $k$, we use that the homotopy limit fits into a homotopy cartesian square
\[\begin{tikzcd}
	\underset{\varnothing\neq I\subseteq \unl{k}}\holim\, X(I)\rar\dar&X(\unl{k})\dar\\
	\underset{\varnothing\neq I\subseteq \unl{k-1}}\holim\,X(I)\rar& \underset{\varnothing\neq I\subseteq \unl{k-1}}\holim\,X(I\cup\{k\}).
\end{tikzcd}\]
 By induction, the conclusion of the statement holds for two diagrams defining the bottom row, and by assumption also for $X(\unl{k})$, so \cref{lem:lem-hopb} gives the induction step.
\end{proof}

\subsection{Operads and dendroidal spaces}\label{section:operads}In this and the following sections, \emph{operads} $\cat{O}, \cat{P},\ldots$ are understood as single-coloured operads in $\bfS$ in the classical sense. Declaring a weak equivalence to be a levelwise weak equivalence gives rise to derived mapping spaces $\smash{\Map_{\cat{Opd}}^h(\cat{O},\cat{P})}$ between such operads. We will be mostly interested in \emph{$1$-reduced operads} which are operads $\cat{O}$ whose space of $0$- and $1$-ary operations $\cat{O}(0)$ and $\cat{O}(1)$ are weakly contractible. For such operads, there are equivalent point of views on their mapping spaces that we will make us of, related by natural maps
\begin{equation}\label{equ:different-models-mapping-spaces}
	\Map_{\cat{Opd}}^h(\cat{O},\cat{P})\xlra{\circled{1}}\Map^h_{\PSh(\Dend)}(N_d \cat{O},N_d \cat{P})\xlra{\circled{2}}\Map^h_{\PSh(\rDend)}(N_d \cat{O},N_d \cat{P})
\end{equation}
which we explain in the following two subsections. Part of our discussion in this and the following subsection is similar to that in \cite[Section 3.4]{WeissDalian}.

\subsubsection{Dendroidal spaces and the map $\circled{1}$}
The two alternative points of view stem from Moerdijk--Weiss' \emph{dendroidal spaces}. Briefly (see \cite{HeutsMoerdijk} for details), the category of \emph{dendroidal spaces} is the category of presheaves $\cO \colon {\Dend}^{\op} \to \cat{S}$ on a certain category $\gls*{dend}$ of finite rooted trees with specified subsets of leaves. More formally, an object $(t,\le,\ell(t))$ in $\Dend$ is a finite partially ordered set $(t,\le)$ of \emph{edges} together with a specified subset $\ell(t)\subset t$ of maximal elements (the \emph{leaves}) such that (a) $\{v\in t\,|\,w\le v\}$ is totally ordered for all $w\in T$ and (b) there is a unique maximal element $v\in T$ with respect to the partial order, the \emph{root} (see Section 3.2 loc.cit.). The subset $\nu(t)\coloneqq t\backslash \ell(T)\subset t$ is the set of \emph{vertices} of the tree. The \emph{incoming edges} $\mathrm{in}(v)\subset t$ of a vertex $v$ is the set of maximal elements in $\{w\in t\,|\,w<v\}$.  We refer to Section 3.2-3.3 loc.cit.\,for a description of the morphisms in $\Omega$. There is a functor $N_d(-)$ from operads to dendroidal spaces, the \emph{dendroidal nerve} (see Example 12.11 loc.cit.), given by
$\smash{\textstyle{\gls*{dendn} \cat{O}(t)\coloneqq  \bigsqcap_{v\in \nu(t)} \cat{O}(|\mathrm{in}(v)|)}}$. Declaring weak equivalences between dendroidal spaces to be levelwise weak equivalences gives rise to derived mapping spaces $\gls*{mapdend}$ of dendroidal spaces and as $N_d(-)$ preserves weak equivalence, we obtain the map $\circled{1}$. 

\subsubsection{($1$-reduced) dendroidal Segal spaces and the map $\circled{2}$}
There is a convenient class of dendroidal spaces that includes dendroidal nerves of $1$-reduced operads but is homotopically more flexible. To define it, we consider the \emph{$k$-corolla} which is the unique (up to isomorphism) tree in ${\Dend}$ with one vertex and $k$ leaves, denoted by $t_k$. The unique (up to isomorphism) tree in ${\Dend}$ with no vertices is denoted $\eta$. For each vertex $v$ in a tree $t$, there is a morphism $t_k\ra t$ (unique up to automorphism of $t_k$) that takes the root to $v$ and the leaves to $\mathrm{in}(v)$. Given a dendroidal space $\cO$ and a tree $t$, these morphisms assemble to a map
\begin{equation}\label{equ:dendroidal-segal-maps}
	\textstyle{\cO(t)\lra\bigsqcap_{v\in \nu(t)}\cO(t_{|\mathrm{in}(v)|})}.
\end{equation}
The following definition mimics the definition of a $1$-reduced operad on the level of dendroidal Segal spaces. The examples to keep in mind are dendroidal nerves $N_{d}\cat{O}$ of $1$-reduced operads.

\begin{dfn}
A \emph{$1$-reduced dendroidal Segal space} $\cO$ is a dendroidal space such that the values $\cO(t_0)$ and $\cO(t_1)$ at the $0$- and $1$-corollas are weakly contractible and such that \eqref{equ:dendroidal-segal-maps} is a weak equivalences for all trees $t\in \Dend$ (this says in particular that $\cO(\eta)$ is weakly contractible).
\end{dfn}

The full subcategory $\gls*{rdend} \subset {\Dend}$ of \emph{closed trees}, i.e.\ trees $t$ with $\ell(v)=\varnothing$ (see \cite[p.\,92, 97]{HeutsMoerdijk}, is often easier to work with. Presheaves $\cO \colon {\rDend}^{\op} \to \cat{S}$ are called \emph{closed dendroidal spaces}. Morphisms between those are still natural transformations and weak equivalences are levelwise; we denote the resulting derived mapping spaces by $\gls*{maprdend}$. Restriction along ${\rDend}\subset {\Dend}$ induces a map \[\smash{\Map^h_{\PSh(\Dend)}(\cO,\cP)\ra \Map^h_{\PSh(\rDend)}(\cO,\cP)}\] of which $\circled{2}$ is a special case. For $1$-reduced operads $\cat{O}$ and $\cat{P}$, both maps $\circled{1}$ and  $\circled{2}$ turn out to be weak equivalences (see \cite[Corollary 14.42]{HeutsMoerdijk} for $\circled{1}$ and \cite[Lemma 3.2.4]{Goppl} for $\circled{2}$):

\begin{prop}\label{prop:comparison-of-models}For $1$-reduced operads $\cat{O}$ and $\cat{P}$, the maps $\circled{1}$ and $\circled{2}$ are weak equivalences.\end{prop}

\subsection{A tower of derived mapping spaces}\label{sec:goppl-tower}
The category ${\rDend}$ of closed trees admits a filtration
\[
	{\rDend}_{\leq 0} \subset {\rDend}_{\leq 1} \subset \cdots \subset {\rDend},
\]
by the full subcategories $\gls*{rdendk}$ on those trees whose vertices $v$ have at most $k$ incoming edges. Denoting the restriction of a closed dendroidal space $\cO$ along ${\rDend}_{\leq k}\subset {\rDend}$ by the same symbol, we obtain a natural tower of derived mapping spaces
\begin{equation}\label{equ:tower}
	\begin{tikzcd}[row sep=0.6cm,column sep=2cm]
		&\ldots \dar\\
		&\Map^h_{\PSh(\rDend_{\le1})}(\cO,\cP)\dar\\
		\Map^h_{\PSh(\rDend)}(\cO,\cP)\arrow[ur]\arrow[uur]\rar&\Map^h_{\PSh(\rDend_{\le0})}(\cO,\cP),
	\end{tikzcd}
\end{equation}
all derived with respect to the levelwise weak equivalences. For simplicity we write
\begin{equation}\label{eqn:abbrevations-op-maps}
	\Map^h(\cO,\cP)\coloneqq \Map^h_{\PSh(\rDend)}(\cO,\cP)\quad\text{and}\quad \Map^h_{\leq k}(\cO,\cP)\coloneqq \Map^h_{\PSh(\rDend_{\le k})}(\cO,\cP).
\end{equation}
This tower was studied by Göppl and Weiss \cite{Goppl}. In Lemma 3.1.1 loc.cit.\,they note that it \emph{converges}, that is, we have a weak equivalence
\begin{equation}\label{equ:goppl-convergence}
	\Map^h(\cO,\cP)\xlra{\simeq} \underset{k}\holim\,\Map^h_{\leq k}(\cO,\cP).
\end{equation} 
To identify its \emph{layers}, i.e.\ the homotopy fibres of the vertical maps, they consider the $k$th \emph{matching} and \emph{latching} object of a $1$-reduced dendroidal space $\cO$
\[
	\gls*{latchk} \coloneqq \underset{(\overline{t}_k\ra t) \in ({\rDend}_{\leq k-1})_{\overline{t}_k/}}\hocolim \, \cO(t),\quad \text{and} \quad
	\gls*{matchk} \coloneqq \underset{{( t\ra \overline{t}_k) \in ({\rDend}_{\leq k-1})_{/\overline{t}_k}}}\holim\,  \cP(t).
\]
Here $\overline{t}_k\in \rDend$ is the \emph{closed $k$-corolla}, the unique (up to isomorphism) closed tree with $k+1$ vertices of which one has $k$ incoming edges and the others have none. Permuting incoming edges defines an action of the symmetric group $\Sigma_k$ on $\overline{t}_k$ in $\rDend$ which induces a natural $\Sigma_k$-action on $\Match_k(\cP)$ and $\Latch_k(\cP)$. These are related by $\Sigma_k$-equivariant maps 
\begin{equation}\label{equation:latch-match}
	\Latch_k(\cO)\lra \big(\cO(\overline{t}_k)\eqcolon \cO(k)\big)\lra \Match_k(\cO).
\end{equation} 
Göppl and Weiss used these maps to identify the vertical homotopy fibres in the above tower in terms of the matching and latching objects and spaces of derived maps between $\Sigma_k$-spaces, see Theorem 3.2.7 and Remark 3.2.15 loc.cit.:

\begin{thm}[Göppl--Weiss]\label{thm:goppl}
For $k\ge1$ and $1$-reduced dendroidal Segal spaces $\cO$ and $\cP$, there is a natural homotopy cartesian square whose left and top arrow is induced by restriction
\[\begin{tikzcd} 
	\Map^h_{\le k}(\cO,\cP) \dar \rar &  \Map^h_{\Sigma_k}(\cO(k),\cP(k)) \dar \\
	\Map^h_{\le k-1}(\cO,\cP) \rar & P_k(\cO,\cP).
\end{tikzcd}\]
The corner $P_k(\cO,\cP)$ fits into a natural homotopy cartesian square
\[\begin{tikzcd} 
	P_k(\cO,\cP) \rar \dar & \Map^h_{\Sigma_k}(\cO(k),\Match_k(\cP)) \dar \\
	\Map^h_{\Sigma_k}(\Latch_k(\cO),\cP(k)) \rar & \Map^h_{\Sigma_k}(\Latch_k(\cO),\Match_k(\cP))
\end{tikzcd}\]
	whose bottom and right maps are induced by \eqref{equation:latch-match}.
\end{thm}

\subsection{Localisations of dendroidal spaces}\label{sec:localisation-dendroidal-spaces}
\label{section:localisation-dendroidal-spaces}
Given a dendroidal space $\cO$, its \emph{$T$-localisation} $\cO_T$ for a set of primes $T$ is the dendroidal  space given as the composition of $\cO\colon {\Dend}^{\op} \ra \cat{S}$ with the localisation functor $(-)_T\colon \cat{S}\ra\cat{S}$. The natural transformation $\id_{\cat{S}}\ra (-)_T$ induces a map $r_T\colon \cO\ra \cO_T$ of dendroidal Segal spaces. It follows from properties \ref{enum:weak-equ-T-localisation} and \ref{enum:product-T-localisation} from \cref{sec:localisation} that if $\cO$ is a $1$-reduced dendroidal Segal space, then so is $\cO_T$.

\begin{lem}\label{lem:loc-dendr-map} For dendroidal spaces $\cO$ and $\cP$ such that $\cP$ is levelwise $T$-local,  $\Map^h_{\PSh(\Dend)}(\cO,\cP)$ is $T$-local and the natural zig-zag 
\[
	\Map^h_{\PSh(\Dend)}(\cO,\cP) \xrightarrow{(-)_T} \Map^h_{\PSh(\Dend)}(\cO_T,\cP_T) \xleftarrow{r_T\circ (-)}\Map^h_{\PSh(\Dend)}(\cO_T,\cP).
\]
consists of weak equivalences. The same holds when replacing $\Dend$ by $\rDend$ or $\rDend_{\le k}$. 
\end{lem}

\begin{proof}The derived mapping spaces appearing in the statement are formed in a category of space-valued presheaves with levelwise weak equivalences, so they can be computed as homotopy limits of a diagram of levelwise mapping spaces. We saw in \cref{sec:localisation} that $T$-local spaces are closed under taking homotopy limits and applying $\Map^h_\cat{S}(X,-)$ for any space $X$, so this implies the first part of the claim. Moreover, this argument reduces the second part to proving that the zigzag of derived mapping spaces in the category of spaces
\[
	\Map^h_\cat{S}(\cO(t),\cP(t)) \xrightarrow{(-)_T} \Map^h_\cat{S}(\cO(t)_T,\cP(t)_T) \xleftarrow{ r_T\circ (-)} \Map^h_\cat{S}(\cO(t)_T,\cP(t))
\]
consists of weak equivalences for all trees $t\in\Dend$. For the second map, this follows follows the fact that $r_T\colon \cP(t)\ra \cP(t)_T$ is a weak equivalence by property \ref{enum:T-localisation} of $T$-localisation. For the first map, we note
\[
	\Map^h_\cat{S}(\cO(t),\cP(t))\xra{(-)_T} \Map^h_\cat{S}(\cO(t)_T,\cP(t)_T)\xra{(-)\circ r_T} \Map^h_\cat{S}(\cO(t),\cP(t)_T)
\]
agrees with postcomposition with $r_T\colon \cP(t)\ra \cP(t)_T$, so is a weak equivalence. The second map is a weak equivalence by property \ref{enum:precomp-T-localisation}, so the first map is one too.
\end{proof}

\subsubsection{Localisations of derived mapping spaces}
Recalling the abbreviations of \eqref{eqn:abbrevations-op-maps}, denoting the path component of a derived map $f\in \Map^h_{\leq k}(\cO,\cP)$ by 
\[
	\Map^h_{\leq k}(\cO,\cP)_f\subseteq \Map^h_{\leq k}(\cO,\cP),
\]
and abbreviating $\cO(\overline{t}_k)$ to $\cO(k)$, we can now state the following result:

\begin{thm}\label{thm:truncated-operad-maps} Let $\cP$ and $\cO$ be $1$-reduced dendroidal Segal spaces such that for all $k\ge0$
\begin{itemize}
	\item $\cP(k)$ has nilpotent path components and
	\item $\cO(k)_{h\Sigma_k}$ and $\Latch_k(\cO)_{h\Sigma_k}$ are weakly equivalent to finite CW complexes,
\end{itemize}
then the following holds for all $k\ge0$ and any map $f\in \Map^h_{\le k}(\cO,\cP)$:
\begin{enumerate}[(i)]		
	\item \label{enum:truncated-nilp} the path component $\Map^h_{\le k}(\cO,\cP)_f$ is nilpotent,
	\item \label{enum:truncated-loc} for a set of primes $T$, the natural map induced by $T$-localisation $r_T\colon \cP\ra \cP_T$
	\[
		\Map^h_{\leq k}(\cO,\cP)_f \lra \Map^h_{\leq k}(\cO,\cP_{T})_{r_T\circ f}
	\]
	is a $T$-localisation of nilpotent spaces, and
	\item \label{enum:truncated-fg} if the spaces $\cP(k')$ have finitely generated homotopy groups at all basepoints for all $k'\ge0$, then so does $\Map^h_{\leq k}(\cO,\cP)_f$.
\end{enumerate}
\end{thm}

The first part of this result (and the strategy of proof) is similar to \cite[Proposition 5.2.4]{WeissDalian}. We start the proof with an auxiliary lemma:

\begin{lem}\label{lem:match-nilpotent}Let $T$ be a set of primes and $\cP$ a $1$-reduced dendroidal Segal space such that $\cP(k)$ has nilpotent path components for all $k\ge0$. The following holds for all $k\ge0$:
\begin{enumerate}
	\item $\Match_k(\cP)$ has nilpotent path components,
	\item the natural map $\Match_k(\cP)\ra \Match_k(\cP_T)$ is a $T$-localisation when restricted to a path component of the source and the corresponding path component of the target,
	\item if the spaces $\cP(k')$ has finitely generated homotopy groups at all basepoints for $0 \leq k' \leq k$, then so does $\Match_k(\cP)$.
\end{enumerate}
\end{lem}

\begin{proof}
Identifying the vertices of $\overline{t}_k$ with no incoming edges with $\unl{k} = \{1,2,\ldots,k\}$, every subset $I\subseteq \unl{k}$ defines a closed subcorolla $\overline{t}_I\subseteq \overline{t}_k$. This gives rise to an $\unl{k}$-cubical diagram $\underline{k} \supseteq I \mapsto \cP(\overline{t}_{\unl{k} \setminus I})$.
By the argument above Theorem 3.4.7 in \cite{WeissDalian}, there is a natural equivalence $\Match_k(\cP)\simeq \holim_{\varnothing\neq I\subseteq \unl{k}}\cP(\overline{t}_{\unl{k} \setminus I})$, so the claim follows from an application of \cref{lem:nilpotent-cube}.
\end{proof}

\begin{proof}[Proof of \cref{thm:truncated-operad-maps}]
We prove the claim by induction on $k$. The initial case $k=0$ is trivial since $\cP$ is assumed to be $1$-reduced, so the mapping spaces appearing in the statement are contractible. For the induction step we assume the claim for $k-1$ and prove it for $k$. To do so, we consider the homotopy cartesian squares of \cref{thm:goppl}. A choice of $f\in \Map^h_{\le k}(\cO,\cP)$ induces basepoints in all spaces participating in these squares; we denote these also by $f$. Now consider the maps
\begin{equation}\label{equ:match-latch-localisations}
	\begin{aligned}\Map^h_{\Sigma_k}(\cO(k),\cP(k))_f &\lra  \Map^h_{\Sigma_k}(\cO(k),\cP(k)_T)_{f} \\
	\Map^h_{\Sigma_k}(\cO(k),\Match_k(\cP))_f &\lra  \Map^h_{\Sigma_k}(\cO(k),\Match_k(\cP)_T)_{f}\\ 
	\Map^h_{\Sigma_k}(\Latch_k(\cO),\cP(k))_f &\lra  \Map^h_{\Sigma_k}(\Latch_k(\cO),\cP(k)_T)_{ f} \\
	\Map^h_{\Sigma_k}(\Latch_k(\cO),\Match_k(\cP))_f &\lra  \Map^h_{\Sigma_k}(\Latch_k(\cO),\Match_k(\cP)_T)_{f} \end{aligned}
\end{equation}
induced by postcomposition with the $T$-localisations of the codomains. Combining \cref{lem:equivariant-mapping-spaces-nilpotent} with \cref{lem:match-nilpotent}, all four maps are $T$-localisations of nilpotent spaces. Moreover, by the first part of \cref{lem:match-nilpotent}, we may replace $\Match_k(\cP)_T$ in the codomain of the second and fourth map by $\Match_k(\cP_T)$. An application of \cref{lem:lem-hopb} to the second square in \cref{thm:goppl} shows that the map $P_k(\cO,\cP)_f\ra P_k(\cO,\cP_T)_{f}$ between the components induced by $f$ is a $T$-localisation of nilpotent spaces. Combining this with the induction hypothesis, another application of \cref{lem:lem-hopb}---this time to the first square---shows that the natural map $\smash{\Map^h_{\leq k}}(\cO,\cP)_f \ra \smash{\Map^h_{\leq k}}(\cO,\cP_{T})_{f}$ is a $T$-localisation between nilpotent spaces, so \ref{enum:truncated-nilp} and \ref{enum:truncated-loc} hold. 
		
We argue similarly for \ref{enum:truncated-fg}: if the spaces $\cP(k)$ have finitely generated homotopy groups at all basepoints, then so does $\Match_k(\cP)$ by the second part of \cref{lem:match-nilpotent}. By the second part of \cref{lem:equivariant-mapping-spaces-nilpotent}, we conclude that the domains of the four maps have finitely generated homotopy groups, so the same holds for $P_k(\cO,\cP)_f$ by an application of the final part of  \cref{lem:lem-hopb} and thus also for $\smash{\Map^h_{\le k}}(\cO,\cP)_f$ by another application of that lemma and the induction hypothesis.
\end{proof}

This finishes the first part of this section as outlined in \cref{sec:intr-nontriviality} after \ref{enum:rationalisation-operads-intro}.

\subsection{Inverse limits and countability}The second part of this section begins with general results on the behaviour of homotopy groups of homotopy limits of towers of spaces.

\subsubsection{Towers of groups}
Following \cite[IX.2]{BousfieldKan}, we call a sequence of groups 
\[
	G_0\leftarrow G_1\leftarrow G_2\leftarrow  \cdots
\] 
a \emph{tower of groups} and abbreviate it by $\gls*{tower}$. We can assign to such a tower a limit \emph{group} $\lim_kG_k$ and a \emph{pointed $\lim^1$-set} $\lim^1_kG_k$ \cite[IX.2.1]{BousfieldKan}. If the tower consists of abelian groups, then $\lim^1_k G_k$ inherits an abelian group structure. A short exact sequence of towers of groups induces a long exact sequence as follows \cite[IX.2.3]{BousfieldKan}:

\begin{lem}\label{lem:les} A levelwise short exact sequence of towers of groups 
\[
	0\ra\{G_k\} \ra\{H_k\}\ra\{K_k\}\lra0
\]
induces a natural exact sequence of groups and pointed sets
\[
	\textstyle{0\ra \lim_kG_k\ra\lim_kH_k\ra\lim_kK_k\ra \lim^1_kG_k\ra\lim^1_kH_k\ra\lim^1_kK_k\ra0.}
\]
\end{lem}

Recall that a map $\{f_k\} \colon \{G_k\} \to \{H_k\}$ of towers of groups is called a \emph{pro-isomorphism} if for all $s \geq 0$ there exists a $t \geq s$ and a homomorphism $H_{t} \to G_s$ such that the diagram
\[\begin{tikzcd} 
	G_s \dar[swap]{f_s} & \lar G_{t} \dar{f_{t}} \\
	H_s & \lar H_{t} \arrow{lu} 
\end{tikzcd}\]
commutes. Pro-isomorphisms have the following property \cite[Proposition III.2.6]{BousfieldKan}:

\begin{lem}\label{lem:pro-iso-lim} For a pro-isomorphism $\{f_k\} \colon \{G_k\} \to \{H_k\}$ the induced map $\lim_kG_k \to \lim_kH_k$ is an isomorphism and the induced map $\lim^1_kG_k \to \lim^1_kH_k$ is a pointed bijection.
\end{lem}

For a tower of groups $\{G_k\}$ and $r\ge1$, the \emph{$r$th derived tower} $(G^{(r)}_k)$ is defined by \[G_k^{(r)}\coloneqq \im\big(G_{k+r}\ra G_k\big).\]
For each fixed $k$, this defines a tower $\{G^{(r)}_k\}_{r \in \bfN}$ of inclusions of subgroups. The tower $(G_k)$ is called \emph{Mittag--Leffler} if for each $k$ there is an $m<\infty$ so that $\lim_{m' \geq m} \smash{G_k^{(m')}}\ra \smash{G_k^{(m)}}$ is an isomorphism. Examples of Mittag--Leffler towers include towers of finite groups or finite dimensional vector spaces. Mittag--Leffler towers have the following property \cite[Corollary IX.3.5]{BousfieldKan}:

\begin{lem}\label{lem:ML-no-lim1} If a tower of groups $\{G_k\}$ is Mittag--Leffler, then $\lim^1_kG_k=*$.
\end{lem}

To recognise Mittag--Leffler towers, we use the following result from \cite[Theorem 2]{McGibbonMoller}:

\begin{lem}[McGibbon--Møller]\label{lem:McGibbon-Moller}
	For a tower $\{G_k\}$ of countable groups, the following statements are equivalent:
	\begin{enumerate}
		\item $\lim^1_k G_k$ is countable,
		\item $\lim^1_k G_k$ vanishes,
		\item the tower $\{G_k\}$ is Mittag--Leffler.
	\end{enumerate}
\end{lem}

The following lemma appears in \cite[Corollary 6.1.9]{DydakSegal}, but we include a proof for the convenience of the reader. For a group $G$ we denote the constant tower with value $G$ by $\{c\,G\}$.

\begin{lem}[Dydak--Segal] \label{lem:dydak-segal} If a tower of groups $\{G_k\}$ is Mittag--Leffler and $\lim_k G_k$ is countable, then the canonical map $\{c\lim_k G_k\} \to \{G_k\}$ is a pro-isomorphism.
\end{lem}

\begin{proof}Any Mittag--Leffler tower of groups $\{G_k\}$ is pro-isomorphic to one with surjective transition maps (consider the tower $\{G_k'\}$ of stable images $G_k'\subset G_k$, i.e. $G_k'=\im(G_{k+m}\rightarrow G_k)$ for $m\gg0$), so we may assume this is the case. This in particular ensures that the maps $\lim_k G_k\ra G_k$ are surjective, so $G_k$ is countable for all $k\ge0$, and it also shows that the claim is true if $\lim_k G_k = 0$ and thus $G_k=0$ for all $k$. We use this special case to prove the following claim, which implies the general statement when applied to the map $\{c\lim_k G_k\} \to \{G_k\}$:

\medskip

\noindent \textbf{Claim. } Let $\{G_k\}$ be a Mittag--Leffler tower of countable groups and $\{f_k\} \colon \{G_k\} \to \{H_k\}$ a levelwise surjective map of towers of groups. If $\lim_k f_k \colon \lim_k G_k \to \lim_k H_k$ is an isomorphism, then $\{f_k\}$ is a pro-isomorphism.

\medskip

\noindent \emph{Proof of claim.} Consider the short exact sequence of towers $1 \ra \{\ker(f_k)\} \ra \{G_k\} \ra \{H_k\} \ra 1$
and the associated long exact sequence of \cref{lem:les}. Since (a) the map $\lim_k f_k \colon \lim_k G_k \to \lim_k H_k$ is an isomorphism, (b) $\{G_k\}$ is Mittag--Leffler, and (c) \cref{lem:ML-no-lim1}, it follows that $\lim_k \ker(f_k)$ and $\lim^1_k \ker(f_k) $ both vanish. Invoking \cref{lem:McGibbon-Moller}, we see that $\{\ker(f_k)\}$ is Mittag--Leffler, so by the first part of the proof $\{\ker(f_k)\}$ is pro-isomorphic to $\{c\,0\}$. The result follows since a levelwise surjective map of towers of groups is a pro-isomorphism if its towers of levelwise kernels are pro-isomorphic to $\{c\,0\}$ \cite[Proposition III.2.2]{BousfieldKan}.\end{proof}

Mittag--Leffler towers often behave well with $T$-localisation in the sense of \cref{sec:localisation-groups}:

\begin{lem}\label{lem:lim-loc} Let $\{G_k\}$ be Mittag--Leffler. If $\lim_k G_k$ is countable, then the canonical map 
\[
	\textstyle{\big({\lim}_k\, G_k\big)_T \lra \lim_k\big((G_k)_T)}
\] is an isomorphism for any set of primes $T$.\end{lem}

\begin{proof}By \cref{lem:dydak-segal} the canonical map of towers $\{c\, \lim_kG_k\} \ra \{G_k\}$ is a pro-isomorphism, as $\lim_kG_k$ is countable. As $(-)_T$ preserves pro-isomorphisms and limits of constant towers, the canonical map from the constant tower on $\{\lim_kG_k\}_T$ to $\{(G_k)_T\}$ is a pro-isomorphism and the result follows from \cref{lem:pro-iso-lim}.\end{proof}

\begin{rem}We stated \cref{lem:lim-loc} in terms of $T$-localisation since this is what we will use, but the same proof applies to $(-)_T$ replaces by any endofunctor on the category of groups.
\end{rem}

\subsubsection{Towers of spaces}
Given a tower $X_0 \leftarrow X_1 \leftarrow \cdots$ of based spaces, taking homotopy groups results in a tower of pointed sets $\{\pi_i(X_k)\}$ (of groups for $i \geq 1$). The limits of these towers fit into the following \emph{Milnor exact sequence} \cite[Theorem IX.3.1]{BousfieldKan}.

\begin{lem}\label{lem:milnor-sequence}
For a tower $X_0 \leftarrow X_1 \leftarrow \cdots$ of based spaces and $i\ge0$, there is a natural short exact sequence of pointed sets (of groups for $i\ge1$)
\[
	\textstyle{0 \ra \lim^1_k\pi_{i+1}(X_k)\ra \pi_i(\holim_kX_k)\ra\lim_k\pi_i(X_k)\lra 0}.
\]
\end{lem}

Together with \cref{lem:McGibbon-Moller}, this has the following consequence.

\begin{prop}\label{prop:holim-tower}
	Fix $i\ge1$. For a tower of based spaces $X_0 \leftarrow X_1 \leftarrow \cdots$ such that $\pi_i(X_k)$ is countable for all $k\ge0$, at least one of the following statements holds:
	\begin{enumerate}[(i)]
		\item \label{enum:holim-tower-i-1-uncountable} \label{enum:holim-tower-i-uncountable} $\pi_*(\holim_k X_k)$  is uncountable in degree  $i-1$ or $i$,
		\item \label{enum:holim-tower-rat-iso} $({\lim}_k \pi_i(X_k))_T \ra {\lim}_k(\pi_i(X_k)_T)$ is an isomorphism for all sets of primes $T$.
	\end{enumerate}
	Moreover, if $\pi_{i+1}(X_k)$ is countable for all $k\ge0$, then at least one of the following is the case:
	\begin{enumerate}[(i')]	
		\item \label{enum:holim-tower-i-uncountable-prime} $\pi_i(\holim_k X_k)$ is uncountable,
		\item \label{enum-holim-tower-rat-iso-prime} the natural surjection $\pi_i(\holim_k X_k) \ra {\lim}_k \pi_i(X_k)$ is an isomorphism.
	\end{enumerate}
	\end{prop}

\begin{proof}By \cref{lem:McGibbon-Moller}, the assumption that $\pi_i(X_k)$ is countable for $k\ge0$ implies that either (a) $\lim^1_k(\pi_i(X_k))$ is uncountable or (b) the tower is Mittag--Leffler. In case (a), we apply \cref{lem:milnor-sequence} in degree $i-1$ to conclude that $\pi_{i-1}(\holim_k X_k)$ is uncountable, so \ref{enum:holim-tower-i-1-uncountable} holds. In case (b), either $\pi_i(\holim_k X_k)$ is uncountable and \ref{enum:holim-tower-i-uncountable} holds, or it is countable and \cref{lem:milnor-sequence} in degree $i$ implies that ${\lim_k}\, \pi_{i}(X_k)$ is countable, so \ref{enum:holim-tower-rat-iso} holds by \cref{lem:lim-loc}. Similarly if $\pi_{i+1}(X_k)$ is countable  for $k\ge0$, then either $\pi_i(\holim_k X_k)$ is uncountable and \ref{enum:holim-tower-i-uncountable-prime} holds, or this group is countable and so $\lim_k^1\pi_{i+1}(X_k)$ is countable by \cref{lem:milnor-sequence} and thus vanishes by \cref{lem:McGibbon-Moller}, so the claim follows from the Milnor exact sequence.
\end{proof}

\subsection{Applications to maps between operads}
Together with \cref{thm:truncated-operad-maps}, we use \cref{prop:holim-tower} to prove the following result about the map
\[
	\Map^h(\cO,\cP)\lra \Map^h(\cO_\bfQ,\cP_\bfQ)
\]
for $1$-reduced dendroidal Segal spaces $\cO$ and $\cP$ in the sense of \cref{section:operads}.

\begin{thm}\label{thm:uncountability-for-general-operads}Let $\cO$ and $\cP$ be $1$-reduced dendroidal Segal spaces such that for all $k\ge0$
\begin{itemize}
\item all components of $\cP(k)$ are nilpotent and have finitely generated homotopy groups,
\item $\cO(k)$ and $\Latch_k(\cO)_{h\Sigma_k}$ are weakly equivalent to finite CW complexes,
\end{itemize}
then for all $i\ge1$ and all basepoints $f\in\Map^h(\cO,\cP)$, at least one of the following is the case:
	\begin{enumerate}
		\item $\pi_{*}(\Map^h(\cO,\cP))$ is uncountable in degree $i-1$ or $i$,
		\item the canonical map $\pi_{i}(\Map^h(\cO,\cP))_\bfQ \to \pi_i(\Map^h(\cO_\bfQ,\cP_\bfQ))$
		is an isomorphism.
	\end{enumerate}
\end{thm}
\begin{proof}During the proof, we implicitly use the equivalence $\Map^h(\cO,\cP_\bfQ)\simeq \Map^h(\cO_\bfQ,\cP_\bfQ)$ and its truncated analogue (see \cref{lem:loc-dendr-map}). Then \eqref{equ:goppl-convergence} gives
\[
	\Map^h(\cO,\cP)\simeq \holim_k\,\Map^h_{\le k}(\cO,\cP) \quad \text{and} \quad \Map^h(\cO_\bfQ,\cP_\bfQ)\simeq \holim_k\,\Map^h_{\le k}(\cO_\bfQ,\cP_\bfQ),
\] so from the two Milnor sequences (see \cref{lem:milnor-sequence}) together with the fact that $\pi_i(\Map^h(\cO_\bfQ,\cP_\bfQ))$ is $\bfQ$-local as the homotopy group of a $\bfQ$-local space (see \cref{lem:loc-dendr-map}), we obtain a square	\[\begin{tikzcd}
	\pi_i(\Map^h(\cO,\cP))_\bfQ\rar{\circled{1}}\dar & \big({\lim}_k\pi_i(\Map^h_{\le k}(\cO,\cP))\big)_\bfQ \dar{\circled{2}}\\
	\pi_i(\Map^h(\cO_\bfQ,\cP_\bfQ))\rar{\circled{3}} & \lim_k\pi_i(\Map^h_{\le k}(\cO_\bfQ,\cP_\bfQ))
\end{tikzcd}\]
Assuming that $\pi_*(\Map^h(\cO,\cP))$ is countable in degrees $i-1$ and $i$, we need to show that the left vertical map is an isomorphism, which we do proving that this holds for the three circled maps. By \cref{thm:truncated-operad-maps} \ref{enum:truncated-fg}, the homotopy groups of  $\Map^h_{\le k}(\cO,\cP)$ are finitely generated for all $k$, so they are in particular countable. By  \cref{prop:holim-tower} \ref{enum-holim-tower-rat-iso-prime}, this implies that  $\circled{$1$}$ is an isomorphism, even before rationalisation. By \cref{thm:truncated-operad-maps} \ref{enum:truncated-loc}, we have $\pi_{k}(\Map^h_{\le k}(\cO,\cP))_\bfQ \cong \pi_{k}(\Map^h_{\le k}(\cO_\bfQ,\cP_\bfQ))$ for all $k\ge1$, so  $\circled{$2$}$ is an isomorphism by \cref{prop:holim-tower} \ref{enum:holim-tower-rat-iso}. Finally, by the Milnor sequence (see \cref{lem:milnor-sequence}), $\circled{$3$}$ is surjective and its kernel agrees with $\lim^1_k\pi_{i+1}(\Map^h_{\le k}(\cP_\bfQ,\cO_\bfQ))$, so is an isomorphism if $\{\pi_{i+1}(\Map^h_{\le k}(\cP_\bfQ,\cO_\bfQ))\}$ is Mittag--Leffler. For this it suffices that it is a tower of finite-dimensional vector spaces, which is indeed the case by \cref{thm:truncated-operad-maps} \ref{enum:truncated-loc} and \ref{enum:truncated-fg}.\end{proof}

\subsubsection{Applications to maps between $E_n$-operads}\label{sec:maps-between-En-operads}
Here and henceforth, we write $E_n$ for any operad weakly equivalent to the operad of little $n$-discs (the unital version, so $E_n(0)\simeq*$). As $E_n(1)\simeq\ast$, we can consider $E_n$ via the dendroidal nerve as a $1$-reduced dendroidal Segal space, denoted by the same symbol, and abbreviate its $T$-localisation (see \cref{sec:localisation-dendroidal-spaces}) by $E_n^T$. \cref{prop:comparison-of-models} says that the derived mapping spaces between $E_n$-operads do not depend on whether we consider them as operads or as $1$-reduced dendroidal Segal spaces. Keeping this in mind, we use Theorems~\ref{thm:truncated-operad-maps} and~\ref{thm:uncountability-for-general-operads} to prove the following two results:

\begin{thm}\label{thm:truncated-ed}Fix $n\ge1$ and $m\ge 3$, and a set of primes $T$. For $f\in\Map^h_{\le r}(E_n,E_{m})$ and $k\ge0$, the following holds:
\begin{enumerate}
\item the path component $\Map^h_{\le k}(E_n,E_{m})_f$ is nilpotent,
\item the map $(r_T\circ(-))\colon \Map^h_{\le k}(E_n,E_{m})_f\ra \Map^h_{\le k}(E_n,E_{m}^T)_{r_T\circ f}$
is a $T$-localisation,
\item the homotopy groups of $\Map^h_{\le k}(E_n,E_{m})_f$ are finitely generated.
\end{enumerate}
\end{thm}

The case of \cref{thm:truncated-ed} that will be relevant to the proof of \cref{bigthm:nontrivial} and \cref{bigcor:top-vs-auted} in the next section is $n=m$. For $m-n\ge2$ and $(-)_T$ being rationalisation, this result appears also in Section 10.2 of \cite{FTW} (see Remark 10.9 and Proposition 10.10 loc.\.cit.).

\begin{thm}\label{thm:haut-uncountable-or-iso} Fix $n\ge1$ and $m\ge 3$. For all $i\ge1$ and any basepoint in $\Map^h(E_n,E_m)$, at least one of the following statements holds:
\begin{enumerate}
	\item \label{enum:haut-u-o-i-unc} $\pi_{*}(\Map^h(E_n,E_m))$ is uncountable in degrees $i-1$ or $i$,
	\item \label{enum:haut-u-o-i-iso} the canonical map $\pi_{i}(\Map^h(E_n,E_m))_\bfQ \to \pi_{i}(\Map^h(E_n^\bfQ,E^\bfQ_m))$
	is an isomorphism.
\end{enumerate}
\end{thm}

\begin{proof}[Proof of Theorems~\ref{thm:truncated-ed} and \ref{thm:haut-uncountable-or-iso}]
This follows from Theorems~\ref{thm:truncated-operad-maps} and ~\ref{thm:uncountability-for-general-operads} once we checked the hypothesis. The space of $k$-ary operations $E_n(k)$ is weakly equivalent to the space of ordered configurations $F_k(\bfR^n)$, so $E_n$ is $1$-reduced for all $n\ge1$. Moreover, by transversality $E_n(k)$ is $1$-connected (so in particular nilpotent) as long as $n\ge3$, so its homotopy groups are finitely generated if its homology groups are. We are thus left to show that $E_n(k)_{h\Sigma_k}$ and $\Latch_k(E_n)_{h\Sigma_k}$ have the weak homotopy type of finite CW complexes for all $n\ge1$ and that $E_n(k)$ has degreewise finitely generated homology groups for $n\ge3$. By \cite[Examples 1.1.6, 2.1.13]{Goppl} (see also \cite[Proposition 3.4.6]{WeissDalian}), the map $\Latch_n(E_n)\ra E_n(k)$ agrees up to weak equivalence of $\Sigma_k$-spaces with the boundary inclusion $\partial \mathrm{FM}_n(k)\subset \mathrm{FM}_n(k)$ of the Fulton--MacPherson compactification of $F_k(\bfR^n)$. This is a compact manifold with corners and free $\Sigma_k$-action \cite{Sinha}, so we conclude (i) that $(E_n(k))_{h\Sigma_k}\simeq FM_n/\Sigma_k$ and $(\Latch_k(E_n))_{h\Sigma_k}\simeq \partial FM_n/\Sigma_k$ have the weak homotopy type of smooth compact manifolds with corners so are weakly equivalent to finite CW complexes, and (ii) that $E_n(k)\simeq \mathrm{FM}_n(k)$ has degreewise finitely generated homology.
\end{proof}

\section{\cref{bigthm:nontrivial}: nontriviality}\label{sec:nontrivial}
In this section we prove results on the homotopy groups of the homotopy fibre $\Aut^h(E_d)/\TOP(d)$ of the map $\BTOP(d)\ra\BAut^h(E_d)$ mentioned as \eqref{equ:topd-to-ed-intro} in the introduction (and explained below), and deduce results on the homotopy groups of $S^{\DiscInf}_\partial(D^d)$; \cref{bigthm:nontrivial} and \cref{bigcor:top-vs-auted} will follow as special cases. As explained in the outline in \cref{sec:intr-nontriviality}, the main ingredient besides \cref{thm:haut-uncountable-or-iso} is work of Boavida de Brito--Weiss' \cite{BdBWConf} and work of Fresse--Turchin--Willwacher \cite{FTW}. We also make use of results of Krannich, Kupers, Randal-Williams, and Watanabe \cite{KrRW,K-RWdiscs,WatanabeII}, though this can be avoided in most cases (see \cref{rem:classical-proof}).

\subsection{A theorem of Boavida de Brito--Weiss}\label{sec:conf-cats} 
We first extract the relevant parts of \cite{BdBWConf}. By Theorem 1.2 loc.cit., the space $\Map^h(E_d,E_d) = \smash{\Map^h_{\cat{Opd}}}(E_d,E_d)$ of derived operad maps is equivalent to a mapping space between certain $\infty$-categories of configurations spaces associated to $\bfR^d$. These configuration categories only depend on the underlying \emph{topological} manifold, so this in particular shows that the standard action of $\oO(d)$ on $E_d$ factors through an action of $\TOP(d)$ and thus gives a map
\begin{equation}\label{equ:top-to-auted-config}
	\BTOP(d)\lra \BAut^h(E_d).
\end{equation} 
We will explain below how a reformulation of further results of Boavida de Brito--Weiss relates the homotopy fibre $\Aut^h(E_d)/\TOP(d)$ of this map to the $\DiscInf$-structure space $S_\partial^{\DiscInf}(D^d)$ of a disc. To state the precise result, we denote by 
\[
	\Omega^{d+1}_{\oO(d)} \big({\Aut^h(E_d)}/{\TOP(d)}\big)\subseteq \Omega^{d+1}\big({\Aut^h(E_d)}/{\TOP(d)}\big)
\] 
the collection of those path components that are sent to classes in the image of the map $\pi_{d+1}(\BO(d))\ra \pi_{d+1}(\BTOP(d))$ under the map
\[
	\pi_0(\Omega^{d+1}({\Aut^h(E_d)}/{\TOP(d)}))=\pi_{d+1}({\Aut^h(E_d)}/{\TOP(d)})\lra\pi_{d+1}(\BTOP(d)).
\]

\begin{thm}[Boavida de Brito--Weiss]\label{thm:sdisc-auted-topd}
For $d \neq 4$ there exists a $0$-coconnected map of the form 
\[
	\Omega^{d+1}_{\oO(d)} ({\Aut^h(E_d)}/{\TOP(d)})\lra S^{\DiscInf}_\partial(D^d).
\]
\end{thm} 

Recall that being $0$-coconnected amounts to being an ``inclusion of path components'', meaning a map that induces an injection on $\pi_0(-)$ and an isomorphism on $\pi_i(-)$ for $i\ge1$.  After taking loop spaces, \cref{thm:sdisc-auted-topd} can also be deduced from work of Ducoulombier--Turchin \cite[(13)]{DucoulombierTurchin}.

\begin{rem}\label{rem:different-delooping-proof} After this work was finalised, a different proof of \cref{thm:sdisc-auted-topd} was given in \cite[Section 5.9.4]{KrannichKupersOperadic}. This proof is independent of \cite{BdBWConf} and \cite{DucoulombierTurchin} and shows the slightly stronger statement $\Omega^{d+1} ({\Aut^h(E_d)}/{\TOP(d)})\simeq S^{\DiscInf}_\partial(D^d)$.
\end{rem}

\begin{proof}[Proof of \cref{thm:sdisc-auted-topd}]This can be deduced from \cite{BdBWConf} as follows: combining their Theorems 1.2 and 1.4 with their Section 6 (see also Equation (1.3)), there is contractible space $\oC(D^d,D^d)$ (a certain mapping space of configuration categories) which fits into a homotopy cartesian square
\[\begin{tikzcd} T_\infty \Emb_\partial(D^d,D^d) \rar \dar & \oC(D^d,D^d) \dar \\[-2pt]
	\Bun_\partial(TD^d,TD^d) \rar & \Omega^d \Map^h(E_d,E_d) \end{tikzcd} \]
where $\Bun_\partial(TD^d,TD^d)$ is the space of vector bundle maps of $TD^d$ that are fixed on the boundary and $T_\infty \Emb_\partial(D^d,D^d)$ is the embedding calculus approximation to $\Emb_\partial(D^d,D^d)$. The bottom horizontal map admits a factorisation 
\begin{equation}
	\label{equ:factorisation-looped}\Bun_\partial(TD^d,TD^d)\ra \Bun_\partial(TD^d,TD^d)^{\TOP}\ra \Omega^d \Map^h(E_d,E_d)
\end{equation} 
through the space of topological microbundle maps (compare the proof of Theorem 1.6 loc.cit.). Under the equivalences $\Bun_\partial(TD^d,TD^d)\simeq \Omega^d\OO(d)$ and $\Bun_\partial(TD^d,TD^d)^{\TOP}\simeq \Omega^d\TOP(d)$, this agrees with the $d$-fold looping of the composition $\oO(d)\ra\TOP(d)\ra\Aut^h(E_d)$. 

Now the factorisation \eqref{equ:factorisation-looped} allows us to form the commutative diagram
\[\begin{tikzcd}[row sep=0.3cm, column sep=0.3cm] 
	\Emb_\partial(D^d,D^d) \arrow{dd}\arrow{rr} \arrow{rd} &[-15pt] &[-15pt] \Emb_\partial(D^d,D^d)^{\TOP} \arrow{dd} \arrow{rd} &[-15pt] \\[-5pt]
	& T_\infty \Emb_\partial(D^d,D^d) \arrow[crossing over]{rr} & & \oC(D^d,D^d) \arrow{dd} \\[-5pt]
	\Bun_\partial(TD^d,TD^d) \arrow{rr} \arrow[equal]{rd} & & \Bun_\partial(TD^d,TD^d)^{\TOP}\arrow{rd} & \\[-5pt]
	& \Bun_\partial(TD^d,TD^d) \arrow[from=uu,crossing over] \arrow{rr} & & \Omega^d \Map^h(E_d,E_d)
\end{tikzcd}\]
whose front and back face are homotopy cartesian; the former by what was said above and the latter by smoothing theory (see \cite[Essay V]{KirbySiebenmann}; this uses $d\neq 4$). Note that $\Emb_\partial(D^d,D^d) = \Diff_\partial(D^d)$ and $\Emb_\partial(D^d,D^d)^\TOP = \Homeo_\partial(D^d)$. From the constructions in \cite{BdBWConf}, one sees that this diagram is in fact a diagram of $A_\infty$-spaces if one uses the $A_\infty$-structure on $T_\infty \Emb_\partial(D^d,D^d)$ by composition induced by the model for embedding calculus with boundary condition from \cite[p.~379]{BdBWSheaf} which agrees with the $A_\infty$-structure provided by our model as a result of \cref{prop:comparison-to-pedromichael} (see the discussion at the beginning of \cref{sec:comparison-to-pedromichael}). Using contractibility of $\oC(D^d,D^d)$ and of $\Homeo_\partial(D^d)$ (using the Alexander trick) and \cref{thm:emb-calc}, the diagram becomes a map of homotopy fibre sequences
\[\begin{tikzcd}
	\Diff_\partial(D^d) \dar{E}\rar&\Omega^d\oO(d)\rar\arrow[d,equal]&\Omega^d\TOP(d)\dar\\[-2pt]
	\Map_{\Mod(d)_{E_{\partial D^d\times I}}}(E_{D^d},E_{D^d}) \rar&\Omega^d\oO(d)\rar&\Omega^d\Aut^h(E_d).
\end{tikzcd}\]
of $A_\infty$-spaces. Here we used the abbreviation from \eqref{equ:abbreviate-right-modules}. Aside from the bottom left fibre, all spaces in this diagram are visibly group-like, so this fibre is as well, i.e.
\[\Map_{\Mod(d)_{E_{\partial D^d\times I}}}(E_{D^d},E_{D^d})=\Aut_{\Mod(d)_{E_{\partial D^d\times I}}}(E_{D^d}),\]
using the notation from \cref{sec:disc-structure-spaces}. We may thus deloop the diagram once (after restricting the components of the rightmost spaces to those in the image of the maps from $\Omega^d\oO(d)$) and take vertical homotopy fibres to get
\[
	\Omega^{d+1}_{\oO(d)}\Aut(E_d)/\TOP(d) \simeq \Aut_{\Mod(d)_{E_{\partial D^d\times I}}}(E_{D^d})/\Diff_\partial(D^d).
\]
The right-hand space is a collection of components of $S^{\DiscInf}_\partial(D^d)$ by \eqref{equ:disjoint-union-description-sdisc}, so the claim follows.
\end{proof}

\subsection{Some results of Fresse--Turchin--Willwacher}Next, we recall part of work of Fresse--Turchin--Willwacher \cite{FTW}, who gave a complete description of the rational homotopy groups of $\Map^h(E_n,\smash{E_{m}^\bfQ})$ in terms of certain \emph{graph complexes}. We collect the parts of their results that are relevant to us below, after explaining why they are applicable in our setting.

\subsubsection{A comparison}The derived mapping spaces $\Map^h(E_n,E^\bfQ_{m})$ considered in \cite{FTW} differ a priori from those we considered in \cref{sec:maps-between-En-operads} in two ways:

Firstly, the derived mapping spaces between operads considered in \cite{FTW} are formed not in the usual category $\mathrm{s}\mathrm{Op}$ of simplicial operads as we did in \cref{section:operads}, but instead in a certain category $\mathrm{s}\Lambda\mathrm{Op}_{\varnothing*}$ of connected simplicial $\Lambda$-operads, equipped with levelwise weak equivalences. This category is isomorphic to the full subcategory $\mathrm{s}\mathrm{Op}_{*1}\subset \mathrm{s}\mathrm{Op}$ of the category of simplicial operads such that $\cat{P}(0)$ and $\cat{P}(1)$ are singletons (see the discussion following Proposition 4.4. loc.cit.). The inclusion functor $\mathrm{s}\Lambda\mathrm{Op}_{\varnothing*}\cong \mathrm{s}\mathrm{Op}_{*1}\to \mathrm{s}\mathrm{Op}$ induces weak equivalences on derived mapping spaces: by \cite[Theorem 1]{FTWSub} the inclusion $\mathrm{s}\mathrm{Op}_{*} \to \mathrm{s}\mathrm{Op}$ induces weak equivalences on derived mapping spaces, and the same holds for $\mathrm{s}\mathrm{Op}_{*1} \to \mathrm{s}\mathrm{Op}_*$ since this full subcategory inclusion preserves fibrant and cofibrant objects in suitable model categories on these categories with levelwise weak equivalences (see \cite[p.~369]{Fresse2} where this is explained in terms of the isomorphic categories $\mathrm{s}\Lambda\mathrm{Op}_{\varnothing*}\subset \mathrm{s}\Lambda\mathrm{Op}_{\varnothing}$). Hence the mapping spaces $\Map^h(E_n,E_m)$ considered in \cite{FTW} agree with any of the variants of mapping space we discussed in \eqref{equ:different-models-mapping-spaces} as a result of \cref{prop:comparison-of-models}, using that the space of $0$- and $1$-operations of $E_n$ are weakly contractible.

Secondly, the authors in \cite{FTW} rationalise operads differently than we do, namely via a rationalisation functor of Fresse \cite[Section 12.2]{Fresse2} (therein denoted $\smash{LG_\bullet\Omega_\sharp^*}(-)$ and phrased in terms of the isomorphism $\mathrm{s}\Lambda\mathrm{Op}_{\varnothing*}\cong\mathrm{s}\mathrm{Op}_{*1}$ mentioned above) that we denote
\[
	(-)_{\mathrm{F}\bfQ}\colon \mathrm{s}\mathrm{Op}_{*1}\lra \mathrm{s}\mathrm{Op}_{*1}.
\]
This functor comes with a natural transformation $r_{\mathrm{F}\bfQ}\colon \id\ra (-)_{\mathrm{F}\bfQ}$ and has the property that any operad $\cat{P}\in \mathrm{s}\mathrm{Op}_{*1}$, the induced maps $r_{\mathrm{F}\bfQ}\colon \cat{P}(k)\ra \cat{P}_{\mathrm{F}\bfQ}(k)$ agree up to weak equivalence with the Sullivan rationalisation as long as $\oH^*(\cat{P}(k);\bfQ)$ is degreewise finite dimensional for all $k\ge1$ (see Theorem 2.2.1 loc.cit.). We can compare this to the rationalisation $(-)_\bfQ$ we use (that is, levelwise $T$-localisation for $T$ the set of all primes) as follows:

\begin{lem}If $\cat{P}\in \mathrm{s}\mathrm{Op}_{*1}$ is levelwise connected nilpotent such that $\oH^*(\cat{P}(k);\bfQ)$ is degreewise finite dimensional for all $k\ge1$, then there exists a natural zig-zag of weak equivalences 
\[
	N_d(\cat{P}_{\mathrm{F}\bfQ})\simeq (N_d(\cat{P}))_{\bfQ}
\] 
between the dendroidal nerve of Fresse's rationalisation and the rationalisation of the dendroidal nerve in the sense of \cref{section:localisation-dendroidal-spaces}.
\end{lem}

\begin{proof}Consider the zigzag $N_d(\cat{P}_{\mathrm{F}\bfQ})\xlra{r_\bfQ} (N_d (\cat{P}_{\mathrm{F}\bfQ}))_\bfQ \xleftarrow{N_d(r_{\mathrm{F}\bfQ})_\bfQ} (N_d (\cat{P}))_\bfQ$. To check both these maps are weak equivalences, it suffice to do so levelwise. Using the dendroidal Segal condition and the fact that rationalisation commutes with products of connected nilpotent spaces by \cref{lem:lem-hopb} \ref{enum:loc-hopb-loc}, we may verify this on corollas. For those, the zig-zag becomes
\[
	\cat{P}_{\mathrm{F}\bfQ}(k)\xlra{r_\bfQ} \cat{P}_{\mathrm{F}\bfQ}(k)_\bfQ \xlla{r_{\mathrm{F}\bfQ}}\cat{P}(k)_\bfQ.
\]
Under the assumption on $\cat{P}(k)$, Sullivan rationalisation agrees with the rationalisation in \cref{sec:localisation}, so all three spaces in the zig-zag are $\bfQ$-local and the two maps are weak equivalences.
\end{proof}

\subsubsection{Homotopy groups of spaces of maps between rationalised $E_n$-operads}\label{sec:FTW}
The ingredient from \cite{FTW} required for the proofs of \cref{bigthm:nontrivial} and \cref{bigcor:top-vs-auted} is a computation of the homotopy groups of the derived mapping space $\smash{\Map^h(E_d,E^\bfQ_d)}$ based at the rationalisation map $\smash{r_\bfQ\colon E_d\ra E_d^\bfQ}$ in a range of degrees, which we summarise as the first two items in the following theorem. In its statement, we write $\bfQ[k]$ for the $\bfZ$-graded $1$-dimensional vector space concentrated in degree $k$ and we write $\iota\colon E_d\ra E_{d+k}$ for the standard inclusion.

\begin{thm}[Fresse--Turchin--Willwacher]\label{thm:homotopy-Aut-En}\
	\begin{enumerate}
		\item \label{enum:fwt-list-even} For $2n\ge4$, we have an inclusion of graded rational vector spaces
		\begin{align*}
			\pi_{*>0}(\Map^h(E_{2n},E^\bfQ_{2n}),r_\bfQ)\supset&\textstyle{\Big(\bigoplus_{i\ge0}\bfQ[2n-4i-1]\Big)}\oplus\\
			&\bfQ[6n-6]\oplus\bfQ[10n-10]\oplus\bfQ[12n-15]\oplus\\
			&\bfQ[14n-14]\oplus \bfQ[16n-16]\oplus \bfQ[16n-19]\oplus  \\
			& \bfQ[18n-18]\oplus\bfQ[18n-21].
		\end{align*}
		 This inclusion is an equality in degrees $*\le20n-28$.
		\item \label{enum:fwt-list-odd} For $2n+1\ge3$, we have an inclusion of graded rational vector spaces
		\begin{align*}
			\pi_{*>0}(\Map^h(E_{2n+1},E^\bfQ_{2n+1}),r_\bfQ)\supset&\textstyle{\Big(\bigoplus_{i\ge0}\bfQ[2n-4i-2]\Big)}\oplus\\
			&\bfQ[4n-1]\oplus\bfQ[6n-3]\oplus\bfQ[8n-5]\oplus\\
			&\bfQ^2[10n-7]\oplus\bfQ^2[12n-9]\oplus \bfQ[12n-6]\oplus \\
			& \bfQ^3[14n-11]\oplus\bfQ[14n-8].
		\end{align*}
		 This inclusion is an equality in degrees $*\le16n-14$.
			\item \label{enum:fwt-cerf-lemma} For $d \ge 2$, $(-)\circ\iota \colon \Map^h(E_d,E^\bfQ_{d})_{r_\bfQ}\ra \Map^h(E_{d-1},E^\bfQ_{d})_{r_\bfQ\circ \iota}$ is a weak equivalence.
	\item \label{enum:homotopy-codim-2} For $d \ge 1$, $\pi_{d+1}(\Map^h(E_d,E^\bfQ_{d+2}),r_\bfQ \circ \iota)$ is an infinite-dimensional $\bfQ$-vector space.
	\end{enumerate}
\end{thm}
\begin{proof}By \cite[Corollary 5]{FTW}, there is an isomorphism of graded vector spaces of the form $\pi_{*>0}(\Map^h(E_{d},E^\bfQ_{d}),r_\bfQ)\cong H_{*>0}(\GC_d^2)$ for $d\ge3$ where $\GC_d^2$ is a certain graph complex introduced by Kontsevich (see loc.cit.\,for details). This complex splits into subcomplexes according to the number of loops of the graphs. The subspaces in \ref{enum:fwt-list-even} and \ref{enum:fwt-list-odd} are the homologies of the subcomplexes of loop order $\le 9$ and $\le 7$ depending on the parity of $d$ (see Equation (4) loc.cit.). The fact that this subspace spans the full homology in the claimed ranges appears as Corollary 6 loc.cit., which proves \ref{enum:fwt-list-even} and \ref{enum:fwt-list-odd} (see also \cite{BrunWillwacher} for more computations in this direction). Part \ref{enum:fwt-cerf-lemma} is \cite[Equation (12)]{FTW}. Part \ref{enum:homotopy-codim-2} follows from the isomorphism $\smash{\pi_k(\Map^h(E_d,E^\bfQ_{d+2}),r_\bfQ \circ \inc)\cong H_{k-1}(\HCG_{d,d+2})}$ of Corollary 3 loc.cit.\,by considering the $1$-loop contribution to degree $n$ of $\HCG_{d,d+2}$ explained in Equation (2) loc.cit.\,and noting that the graph $H_k$ in that equation has degree $d$ for all $k$.
\end{proof}

\begin{rem}\label{rem:rationalising-source-does-not-matter}Note that we have $\Map^h(E_{n},E^\bfQ_{m})_{r_\bfQ\circ \iota}\simeq \Map^h(E^\bfQ_{n},E^\bfQ_{m})_{\iota^\bfQ}$ by \cref{lem:loc-dendr-map}.
\end{rem}

\subsection{Homotopy groups of $\Aut^h(E_d)/\TOP(d)$ }\label{sec:homotopy-auted-topd}
We now state our main technical result on the homotopy groups of the fibre $\Aut^h(E_d)/\TOP(d)$ of the map $\BTOP(d)\ra \BAut^h(E_d)$ from \eqref{equ:top-to-auted-config}. We phrase the result in terms of the following statement that we will refer to as $(\operatorname{\mathbf{H}}^d_{k,m})$. It depends on a choice of dimension $d\ge1$ and degrees $k,m\ge 2$.

\medskip

\begin{center} 
\quad\quad\quad\begin{minipage}{11.5cm}At least one of the following two scenarios is the case:
\begin{enumerate}
\item $\pi_*(\Aut^h(E_d)/\TOP(d))$ is uncountable in degree $k-2$ or $k-1$, or
\item $\pi_{m}(\Aut^h(E_d)/\TOP(d))_\bfQ$ is nontrivial.
\end{enumerate}
\end{minipage} \hfill $(\operatorname{\mathbf{H}}^d_{k,m})$\end{center}

\begin{thm}\label{thm:nontriviality-general}The statement $(\operatorname{\mathbf{H}}^d_{k,m})$ holds in the following cases:
\begin{enumerate}
\item\label{nontriviality-general:ii} dimension $d=3$ and degrees $k=7$ and $m=6$,
\item\label{nontriviality-general:iii} dimension $d=4$ and degrees $k=4$ and $m=4$,
\item\label{nontriviality-general:iv} dimension $d=2n+1\ge5$, degrees $k\le 8n-12$ with $k\equiv0\modulo{4}$ and $k\neq 6n-2$, and $m=k$. For $2n+1=5$, the bound $k\le 8n-12$ can be weakened to $k\le 8n-8$,
\item\label{nontriviality-general:v} dimension $d=2n\ge6$, degrees $2n\le k\le 8n-12$ with $k\equiv0\modulo{4}$, and $m=k$. If $n$ is odd then the condition $2n\le k$ can be removed.
\end{enumerate}
\end{thm}

This in particular shows that the map $\BTOP(d)\ra \BAut^h(E_d)$ is not a weak equivalence for $d\ge3$, so proves the first part of \cref{bigcor:top-vs-auted} in these cases (the second part follows by combining Theorems~\ref{thm:oo-loop-general} and~\ref{thm:sdisc-auted-topd}). In the low-dimensional case $d\le2$, the map is an equivalence which one can see by combining the facts that in these dimensions $\BO(d)\ra \BTOP(d)$ and $\BO(d)\ra\BAut^h(E_d)$ are weak equivalences, the first by \cite[Essay V.\S 5.0(7)]{KirbySiebenmann} and the latter by work of Horel for $d=2$ \cite[Theorem 8.5]{Horel} and a folklore result for $d=1$.

\medskip
 
To prepare the proof of \cref{thm:nontriviality-general}, we extract two results on the homotopy groups of the space $\BTOP(d)$ from the literature. The first says they are countable, and its proof requires the following lemma which is likely known to experts but for which we do not know a reference.

\begin{lem}\label{lem:homeo-countable}For a compact topological manifold $M$, possible with boundary, or the interior of such a manifold, the homotopy groups of $\BHomeo_\partial(M)$ are countable.\end{lem}

\cref{lem:homeo-countable} will be a consequence of the following point-set topological fact. Recall that a topological space is \emph{second countable} if its topology has a countable basis, and \emph{locally weakly-contractible} if for every neighbourhood $U$ of a point $p$ there exists a weakly-contractible open neighbourhood $V \subseteq U$ of $p$.

\begin{lem}\label{lem:countable-homotopy-criterion} If $X$ is a locally weakly-contractible second countable space, then the homotopy groups of $X$ based at any basepoint are countable.\end{lem}

\begin{proof}Recall (for instance from \cite[VIII.6.3]{Dugundji}) that every second countable space $X$ is \emph{Lindel\"of}, i.e.\,every open cover has a countable subcover. For locally weakly contractible $X$, we apply this to the collection of all weakly-contractible open subsets to see that $X$ admits a countable open cover by weakly-contractible subsets. As being a locally weakly-contractible second countable space is preserved by passing to an open subset, the same is true for open subsets of $X$. This allows one to inductively construct an open hypercover $U_\bullet \to X$ such that each $U_\bullet$ has countable many components, each of which is weakly-contractible. Now consider the zigzag
$X \leftarrow \hocolim\,U_\bullet \rightarrow \hocolim\,\pi_0(U_\bullet)$
whose left map is the weak homotopy equivalence of \cite[Theorem 1.3]{DuggerIsaksen} and whose right map is induced by taking path components, so it is also a weak homotopy equivalence since homotopy colimits take objectwise weak homotopy equivalences to weak homotopy equivalences. Now observe that the right term is equivalent to a countable CW complex, e.g.~using the formula in \cite[Proposition 3.2]{DuggerIsaksen} exhibiting the homotopy colimit as the geometric realisation of a simplicial set with countable sets of $k$-simplices for all $k$, and hence has countable homotopy groups.
\end{proof}

\begin{proof}[Proof of \cref{lem:homeo-countable}] For $M$ compact, restriction to the boundary induces a fibration sequence $\Homeo_\partial(M) \ra \Homeo(M) \ra \Homeo(\partial M)$ as a result of the existence of collars. Hence it suffices to prove the result for the topological group of homeomorphisms of a compact manifold with boundary or the interior of such a manifold, with no boundary condition. This space is second countable in the compact-open topology \cite[Proposition 5.4]{GleasonPalais} and locally contractible by \cite[Theorem 1, Theorem 2]{Cernavskii} (or \cite[Corollary]{CernavskiiRn} for the case $\bfR^d$, which also serves an erratum for the previous reference) or \cite[Corollary 1.1, Corollary 6.1]{EdwardsKirby} (or \cite[Theorem 4]{Kirby} for the case $\bfR^d$), so the claim follows from \cref{lem:countable-homotopy-criterion}.
\end{proof}

Applying \cref{lem:homeo-countable} to $\bfR^d=\interior(D^d)$ we conclude:

\begin{cor}\label{cor:btopd-countable} The homotopy groups of $\BTOP(d)$ are countable.\end{cor}

\begin{rem}For $d \neq 4$, \cref{cor:btopd-countable} also follows by combining \cite[Lemma 10, p.\,188]{MilnorCollectionIV} with \cite[Essay V.\S 5.0(1)]{KirbySiebenmann}. The advantage of the proof above is that it applies to $d=4$.\end{rem}

The second result on $\BTOP(d)$ we will use follows from works of Krannich, Kupers, Randal-Williams, and Watanabe \cite{KrRW,K-RWdiscs,WatanabeII}. It concerns two commutative squares
\begin{equation}\label{equ:topd-to-top}
	\begin{tikzcd}
		\BO(2n)\dar{\iota}\rar{(e,\mathrm{stab})}\dar&[5pt]K(\bfQ,2n)\times \BO\dar{\id\times\iota}[swap]{\simeq_\bfQ}&[-10pt]\BO(2n+1)\rar{(E,\mathrm{stab})}\dar{\iota}&[5pt] K(\bfQ,4n)\times \BO\dar{\id\times\iota}[swap]{\simeq_\bfQ}\\
		\BTOP(2n)\rar{(e,\mathrm{stab})}&K(\bfQ,2n)\times\BTOP&\BTOP(2n+1)\rar{(E,\mathrm{stab})}&K(\bfQ,4n)\times\BTOP
	\end{tikzcd}
\end{equation}
where the vertical arrows are induced by the inclusion $\oO(d)\subset\TOP(d)$ and the horizontal arrows by the stabilisation map, the Euler class $e\in\oH^{2n}(\BTOP(2n);\bfQ)$, and the odd-dimensional analogue of its square $E\in\oH^{2n+1}(\BTOP(2n+1);\bfQ)$ (see \cite[Sections 1.2.2 and 8.1.1]{KrRW} for further information on this class). That the right vertical maps are rational equivalences follows from the finiteness of the groups $\pi_*(\TOP/\OO)$ \cite[Essay V.\S 5.0(5)]{KirbySiebenmann}.
		
		\begin{thm}\label{thm:topd-surjectivity}
The maps induced by the bottom horizontal arrows
\[\pi_k(\BTOP(2n))_\bfQ\ra\pi_k(K(\bfQ,2n)\times \BTOP)_\bfQ\quad\pi_k(\BTOP(2n+1))_\bfQ\ra\pi_k(K(\bfQ,4n)\times\BTOP)_\bfQ\]	
are surjective in degrees $k\le 4n-1$ for all $n$, and in degrees $k\le 8n-12$ as long as $n\ge3$. Moreover, the right-hand map for $n=2$ is also surjective in degree $4n$.
\end{thm}
	
\begin{proof}In degrees $*\le 4n-1$ the claimed surjectivity follows from the classical fact that the upper horizontal arrows are rationally surjective in exactly this range. 
	
In order to show the claim for the bottom horizontal map in the left square of \eqref{equ:topd-to-top} for $n\ge3$ in the range $*\le 8n-12$, it thus suffices to show that the map $\Omega^{2n}_0\BTOP(2n)\ra \Omega^{2n}_0\BTOP$ is surjective on $\pi_*(-)_\bfQ$ for $*\le 6n-12$ which can be further reduced to showing that the map $\BDiff_\partial(D^{2n})\simeq \Omega^{2n}_0\TOP(2n)/\oO(2n)\ra \Omega^{2n}_0\TOP/\OO(2n)$ is surjective on $\pi_*(-)_\bfQ$ for $*\le 6n-13$; here we have used Morlet's smoothing theory equivalence \cite[p\,241]{KirbySiebenmann}. This surjectivity was proved in \cite[Corollary 6.7]{K-RWdiscs}. By precomposing the map $\BTOP(2n+1)\ra \BTOP$ with $\BTOP(2n)\ra\BTOP(2n+1)$, this argument also shows that the bottom horizontal map in the right square of \eqref{equ:topd-to-top} for $n\ge3$ is surjective on $\pi_*(-)_\bfQ$ for $*\le 6n-12$ as long as $*\neq 4n$.

This leaves us with showing that for all $n\ge2$, the bottom horizontal map of the right square of \eqref{equ:topd-to-top} is surjective on $\pi_{4n}(-)_\bfQ$. Since the pullback of the class $E\in\oH^{4n}(\BTOP(2n+1);\bfQ)$ to $\BO(2n)$ agrees with $e^2$ by definition of $E$ and hence is decomposable, evaluation of the pullback of $E$ on the image of the Hurewicz map $\pi_{4n}(\BO(2n))_\bfQ \to H_{4n}(\BO(2n);\bfQ)$ is trivial. Hence the fact that the map $\BO(2n)\ra \BTOP$ is surjective on $\pi_{4n}(-)_\bfQ$ implies that the direct summand $\pi_{4n}(\BTOP)_\bfQ\subset \pi_{4n}(K(\bfQ,4n)\times \BTOP)_\bfQ$ is in the image. So we are left with showing that the map $E\colon \BTOP(2n+1)\ra K(\bfQ,4n)$ is nontrivial for all $n\ge2$. Using the smoothing theory equivalence $\BDiff^\fr_\partial(D^{2n+1})_0\simeq \Omega^{2n+1}_0\TOP(2n+1)$ involving the framed diffeomorphism group, the composition \vspace{-0.8em} \[\pi_{4n}(\BDiff^\fr_\partial(D^{2n+1}))_\bfQ\cong \pi_{4n}(\BTOP(2n+1))_\bfQ\xlra{E}\bfQ\subset\bfR\] agrees by \cite[Theorem B.4, Remark B.5]{KrRW} up to a constant with the ``Kontsevich class'' $\zeta_{2,3}$ from \cite[p.\,631]{WatanabeII}, so it is nontrivial for $n\ge2$ by Theorem 3.1 loc.cit and \cite{WatanabeIIerr}.\end{proof}

\begin{proof}[Proof of \cref{thm:nontriviality-general}]
Throughout the proof, we  use the facts that $\pi_{k>0}(\BTOP)_\bfQ$ is $1$-dimensional for $k\equiv0\modulo{4}$ and trivial otherwise, and that $\BTOP(d)$ has countable homotopy groups by \cref{cor:btopd-countable}. We divide the proof into three cases.

\begin{itemize}[leftmargin=9mm]
	\item[$d=3$]Applying Theorem \ref{thm:haut-uncountable-or-iso} for $n=m=3$ and $i=6$, we see that either 
	\begin{enumerate}[label=(\alph*)]
		\item $\pi_*(\BAut^h(E_3))$ is uncountable in degrees $6$ or $7$, or 
		\item $\pi_7(\BAut^h(E_3))_\bfQ \cong \pi_7(\BAut^h(E_3^\bfQ))$. 
	\end{enumerate}
	By the long exact sequence of $\Aut^h(E_3)/\TOP(3)\ra\BTOP(3)\ra\BAut^h(E_3)$, there is nothing left to show in the first case since $\BTOP(d)$ has countable homotopy groups. In the second case, we use that firstly the map $\BO(3)\ra \BTOP(3)$ is a weak equivalence by \cite[p.\,605]{Hatcher} and thus $\pi_7(\BTOP(3))_\bfQ \cong \pi_7(\BO(3))_\bfQ$ vanishes, and that secondly \cref{thm:homotopy-Aut-En} \ref{enum:fwt-list-odd} combined with \cref{rem:rationalising-source-does-not-matter} shows that $\pi_7(\BAut^h(\smash{E_3^\bfQ}))\cong \pi_6(\Map^h(\smash{E_3^\bfQ},\smash{E_3^\bfQ});\id)$ is nontrivial, in fact at least $3$-dimensional (since $12n-6=14n-8$ for $n=3$). Using the same long exact sequence as before, this shows the claim in the second case.
		
	\item[$d=4$] The logic is the same as in the case $d=3$: we again apply Theorem \ref{thm:haut-uncountable-or-iso}, this time for $n=m=4$ and $i=3$, to see that either 
	\begin{enumerate}[label=(\alph*)]
		\item $\pi_*(\BAut^h(E_4))$ is uncountable in degrees $3$ or $4$, or 
		\item $\pi_4(\BAut^h(E_4))_\bfQ \cong \pi_4(\BAut^h(E_4^\bfQ))$. 
	\end{enumerate}
	As before, there is nothing left to show in the first case. In the second case, we use that firstly $\pi_4(\BTOP(4))_\bfQ$ is at least $2$-dimensional as a result of \cref{thm:topd-surjectivity} and that secondly $\pi_4(\BAut^h({E_4}^\bfQ))$ is $1$-dimensional as a result of \cref{thm:homotopy-Aut-En} \ref{enum:fwt-list-even} (since $2n-4i-1 = 3$ for $i=0$ and all other terms are in degree $\geq 7$).
		
	\item[$d\ge5$] Theorem \ref{thm:haut-uncountable-or-iso} for $n=m=d$ and $i=k-1$ shows that either 
	\begin{enumerate}[label=(\alph*)]
		\item $\pi_*(\BAut^h(E_{d}))$ is uncountable in degrees $k$ or $k-1$, or 
		\item $\pi_{k}(\BAut^h(E_{d}))_\bfQ \cong \pi_{k}(\BAut^h(E_{d}^\bfQ))$. 
	\end{enumerate}
	As previously, nothing is left to show in the first case. In the second case, we first consider odd $d$. If $d=2n+1\ge 5$ and $1\le k\le 8n-12$ (or $1\le k\le 8n-8$ if $n=2$) such that $k\neq 6n-2$ and $k\equiv0\modulo{4}$, then we use firstly that $\pi_{k}(\BTOP(2n+1))_\bfQ$ is at least $2$-dimensional if $k=4n$ and otherwise at least $1$-dimensional by \cref{thm:topd-surjectivity}, and secondly that \cref{thm:homotopy-Aut-En} \ref{enum:fwt-list-odd} shows that $\pi_k(\BAut^h(E_{2n+1}^\bfQ))$ is trivial for $k\neq 4n$ and $1$-dimensional for $k=4n$. Finally, for even $d=2n\ge6$ and $k\equiv 0\modulo{4}$ with $2n\le k\le 8n-2$ for $n$ even and $k\le 8n-2$ for $n$ odd, we use a) that $\pi_{k}(\BTOP(2n))_\bfQ$ is at least $1$-dimensional and at least $2$-dimensional for $k=2n$ if $n$ is even by \cref{thm:topd-surjectivity}, and b) that \cref{thm:homotopy-Aut-En} \ref{enum:fwt-list-even} shows that $\pi_{k}(\BAut^h(E_{2n}^\bfQ))$ is trivial for $k\neq 2n$ and $1$-dimensional for $k=2n$. \qedhere
\end{itemize}
\end{proof}
		
\begin{rem}\label{rem:map-doesnt-matter}
The proof of \cref{thm:nontriviality-general} simply compares the homotopy groups of $\TOP(d)$ and $\Aut^h(E_d)$ abstractly. It does not use anything about the specific map $\TOP(d)\ra \Aut^h(E_d)$.
\end{rem}

\subsubsection{Applications to $S^{\DiscInf}_\partial(D^d)$}
In view of the $0$-coconnected map of \cref{thm:sdisc-auted-topd} \[\Omega^{d+1}_{\oO(d)}\Aut^h(E_d)/\TOP(d)\lra S_\partial^{\DiscInf}(D^d),\] as long as $k-d-3\ge 0$ the statement $(\operatorname{\mathbf{H}}^d_{k,m})$ implies the following variant for $S^{\DiscInf}_\partial(D^d)$:

\medskip

\begin{center} 
\quad\quad\quad\begin{minipage}{11cm}At least one of the following two scenarios is the case:
\begin{enumerate}
	\item $\pi_*(S^{\DiscInf}_\partial(D^d))$ is uncountable in degree $k-d-3$ or $k-d-2$, or
	\item $\pi_{m-d-1}(S^{\DiscInf}_\partial(D^d))_\bfQ$ is nontrivial.
\end{enumerate}
\end{minipage} \hfill $(\operatorname{\mathbf{H}}^{d,\DiscInf}_{k,m})$\end{center}

\medskip

\noindent For $k-d-3=0$, this implication uses that if $\pi_{0}(\Omega^{d+1}\Aut^h(E_d)/\TOP(d))$ is uncountable, then so is $\pi_{0}(\Omega^{d+1}_{\oO(d)}\Aut^h(E_d)/\TOP(d))$. This is because $\pi_{d+1}(\BTOP(d))$ is countable, so if the domain of the map $\pi_{d+1}(\Aut(E_d)/\TOP(d))\ra \pi_{d+1}(\BTOP(d))$ is uncountable, then so is its kernel. Combined with \cref{thm:nontriviality-general} we therefore obtain:

\begin{cor}\label{cor:precise-nontriviality-sdisc}
Under the additional assumption $k-d-3\ge 0$, the statement $(\operatorname{\mathbf{H}}^{d,\DiscInf}_{k,m})$ holds for all choices of triples $(d,k,m)$ to which  \cref{thm:nontriviality-general} applies. 
\end{cor}

We now use \cref{cor:precise-nontriviality-sdisc} to prove that $S^{\DiscInf}_\partial(D^d)$ is not contractible for all $d\ge5$ with $d\neq3$.
		
\begin{thm}\label{thm:thm-for-discs}
For $d=3$ or $d\ge5$, the space $S^{\DiscInf}_\partial(D^d)$ is not contractible.
\end{thm}

\begin{proof}
For $d=3$, the claim follows from $(\operatorname{\mathbf{H}}^{d,\DiscInf}_{k,m})$ for the triple $(d,k,m)=(3,7,6)$ since this statement holds true in this case $k-d-3\ge 0$ and \cref{thm:nontriviality-general} applies to this triple. In the case $d\ge5$ the claim follows similarly as long as we ensure that there exists a $k$ such that $k-d-3\ge0$ and \cref{thm:nontriviality-general} applies to the triple $(d,k,k)$. For $d=2n+1$ with $n\ge 4$, we pick the unique $k\equiv0\modulo{4}$ with $2n+5\le k\le 2n+8$. This satisfies the requirements because $k-d-3\ge 0$ and $k\neq 6n-2$ as $2n+8<6n-2$ and $2n+8\le 8n-12$. For $d=2n+1$ with $n=3$, we choose $k=12$. This works because $k-d-3=2\ge0$ and $12\le 8n-12=12$. For $d=2n+1$ with $n=2$, we choose $k=8$ which works using the improvement of the bound since $k-d-3=0\ge0$ and $k\le 8n-8=8$. For $d=2n$ with $n\ge4$, we can pick the unique $k\equiv0\modulo{4}$ with $2n+4\le k\le 2n+7$ which is valid since $k-d-3\ge0$ and $k\le 2n+7\le 8n-12$. Finally for $d=2n$ with $n=3$ we pick $k=12$ which works because $k-d-3=3\ge0$ and $k\le 8n-12=12$.
\end{proof}
		
\begin{rem}\label{rem:classical-proof}
If one relaxes the range $k\le 8n-12$ in \cref{thm:nontriviality-general} \ref{nontriviality-general:iv} and \ref{nontriviality-general:v} to $k\le 4n-1$, then the proof we gave does not rely on the recent works \cite{KrRW,K-RWdiscs,WatanabeII}, since the proof of \cref{thm:topd-surjectivity} does not use them in this range. This is sufficient to deduce \cref{bigcor:top-vs-auted}. It also gives a weaker version of \cref{cor:precise-nontriviality-sdisc} that does not rely on these works. The latter is good enough to conclude \cref{thm:thm-for-discs} \emph{except in dimensions $d=5,6,7$}.
\end{rem}

Combining \cref{thm:thm-for-discs} with \cref{cor:homotopy-retract} implies \cref{bigthm:nontrivial}. 

\begin{rem}\label{rem:3-manifolds}
Even though \cref{thm:thm-for-discs} applies to $d=3$ and all orientable $3$-manifolds $M$ are spin, we cannot conclude that $S^{\DiscInf}_\partial(M)$ is nontrivial in this case, because our tangential $2$-type invariance result does not apply if $d=3$, so \cref{cor:homotopy-retract} is not available. Nonetheless, $S^{\DiscInf}_\partial(M)$ is nontrivial if $M$ embeds into $D^3$ after removing finitely many codimension $0$ discs, since 
\begin{enumerate}
	\item removing discs does not change the homotopy type of $S^{\DiscInf}_\partial(M)$ by \cref{prop:invariance-handles}, 
	\item $S^{\DiscInf}_\partial(D^3)$ is a homotopy retract of $S^{\DiscInf}_\partial(M)$ if $M$ embeds into $D^3$ by the same argument as in the second part of proof of \cref{cor:homotopy-retract}, and
	\item $S^{\DiscInf}_\partial(D^3)$ is nontrivial by \cref{thm:thm-for-discs}. 
\end{enumerate} This applies in particular to $S^3$ or to the handlebodies $(S^1\times D^2)^{\natural g}\natural(S^2\times D^1)^{\natural g}$ for $g,h\ge0$, with $\natural$ denoting the boundary connected sum operation.
\end{rem}

\subsection{Positive codimension} We conclude this section with a brief discussion of an analogue of the nontriviality results of the previous section in positive codimension, by which we mean the following: the subgroup $\oO(c)\subset \oO(d)$ acting on the last $c$ coordinates stabilises the standard inclusion $E_{d-c}\ra E_d$ for $c\ge0$ under the $\oO(d)$-action on $\Map^h(E_{d-c},E_d)$, so we have a map $\oO(d)/\oO(c)\ra \Map^h(E_{d-c},E_d)$. In the same way as in the case $c=d$ discussed in \cref{sec:conf-cats}, Boavida de Brito--Weiss' work \cite{BdBWConf} shows that this action factors as a composition
\[
	\oO(d)/\oO(c)\lra \TOP(d)/\TOP(d,d-c) \lra \Map^h(E_{d-c},E_d)
\]
where $\TOP(d,d-c)\subset \TOP(d)$ is the subgroup of those homeomorphism that fix $\{0\}\times\bfR^{d-c}\subset \bfR^d$. Generalising from the codimension $c=0$ case of \cref{bigcor:top-vs-auted}, one might wonder whether
\begin{equation}\label{equ:relative-comparison-top-en}
	\TOP(d)/\TOP(d,d-c) \lra \Map^h(E_{d-c},E_{d}),
\end{equation}
is a weak equivalence. In codimension $c\ge 3$, this was shown to be the case by Boavida de Brito--Weiss \cite[Theorem 1.6]{BdBWConf} after taking $(d-c+1)$-fold loop spaces. Adapting the methods of the previous subsection, we consider the remaining codimensions $c=1,2$. As before, we phrase the result in terms of the following placeholder statement involving dimension $d\ge1$, codimension $c\in\{1,2\}$, and degrees $k\ge3$ and $m\ge1$.

\medskip

\begin{center}
\quad\begin{minipage}{11.5cm}At least one of the following two scenarios is the case:
\begin{enumerate}
	\item The homotopy groups \[\pi_*\Big(\hofib_\iota\big(\TOP(d)/\TOP(d,d-c) \ra \Map^h(E_{d-c},E_{d})\big)\Big)\] are uncountable in degree $k-3$ or $k-2$, or
	\item The homotopy group \[\pi_{m-1}\Big(\hofib_\iota\big(\TOP(d)/\TOP(d,d-c) \ra \Map^h(E_{d-c},E_{d})\big)\Big)_\bfQ\] is nontrivial.
\end{enumerate}
\end{minipage} \hfill $(\operatorname{\mathbf{H}}^{d,c}_{k,m})$\end{center}

\medskip

\noindent The proof requires some results about the homotopy groups of the spaces $\TOP(d,d-c)$:

\begin{lem}\label{lem:topdd-c} \quad
	\begin{enumerate}
		\item \label{enum:c-is-1} The map $(-) \times \bfR^{d-1} \colon \OO(1) \simeq \TOP(1) \to \TOP(d,d-1)$ is a homotopy equivalence.
		\item \label{enum:c-is-2} The map $(-) \times \bfR^{d-2} \colon \OO(2) \simeq \TOP(2) \to \TOP(d,d-2)$ is $(d-2)$-connected.
	\end{enumerate}
\end{lem}

\begin{proof}
Part \ref{enum:c-is-1} admits an elementary argument: if $f(-,-) \colon \bfR^{d-1}\times \bfR=\bfR^d \to \bfR^d$ is an orientation-preserving homeomorphism fixing $\bfR^{d-1}\times\{0\}$ pointwise, then
\[
	[0,\infty)\times \bfR^{d-1}\times \bfR \ni (t,x,s) \longmapsto f_t(x,s) \coloneqq \begin{cases} (x,s) & \text{if $|s| \leq t$}, \\
	f(x,s-t) & \text{if $s>t$,} \\
	f(x,s+t) & \text{if $s<-t$,}\end{cases}
\]
gives an isotopy of homeomorphisms that extends continuously to $t=\infty$ with value $\id_{\bfR^d}$ and depends continuously on $f$. If $f$ is orientation-reversing a similar formula works. 
	
Part \ref{enum:c-is-2} is due to Kirby--Siebenmann \cite[Theorem B]{KirbySiebenmannCodim2} for $d\neq 4$, who deduce it using immersion theory from an existence and uniqueness result for normal bundles of codimension $2$ locally flat embeddings into $d$-manifolds. In the remaining case $d=4$, the necessary results on normal bundles of locally flat embeddings of surfaces into $4$-manifolds were established later by Freedman--Quinn \cite[Section 9.4]{FreedmanQuinn}. 
\end{proof}

\begin{rem}
The proof of \cref{cor:btopd-countable} extends to show that $\TOP(d,d-c)$ has countable homotopy groups: use \cref{lem:countable-homotopy-criterion}, that it is second countable being a subspace of $\TOP(d)$, and that it is locally weakly-contractible by the variant of \cite[Corollary 7.3]{EdwardsKirby} for this group.
\end{rem}

\begin{thm}\label{equ:technical-poscodim}The statement $(\operatorname{\mathbf{H}}^{d,c}_{k,m})$ holds in the following cases.
\begin{enumerate}
	\item\label{codim1} For $c=1$, it holds for all choices of $(d,k,m)$ to which \cref{thm:nontriviality-general} applies. 
	\item\label{codim2} For $c=2$, it holds for $d\ge3$, $k=d$, and $m=d-1$.
\end{enumerate}
\end{thm}

\begin{proof}
For \ref{codim1}, we consider the following zig-zag of maps \vspace{-0.2em}
\[
	\TOP(d) \ra \TOP(d)/\TOP(d,d-1) \ra \Map^h(E_{d-1},E_d)\xrightarrow{r_\bfQ\circ(-)} \Map^h(E_{d-1},E_d^\bfQ)\xleftarrow{(-)\circ \iota}\Map^h(E_{d},E_d^\bfQ)
\]
After taking loop spaces the leftmost and the rightmost arrow become weak equivalences; the former by \cref{lem:topdd-c} \ref{enum:c-is-1} and the latter by \cref{thm:homotopy-Aut-En} \ref{enum:fwt-cerf-lemma}. Since the homotopy groups of $\TOP(d)$ are countable by \cref{cor:btopd-countable}, it thus suffices to prove that for choices $k\ge3$ and $m\ge1$ as in the claim either
\begin{enumerate}[label=(\alph*)]
\item $\pi_{*}(\Map(E_d,E_d);\id)$ is uncountable in degrees $k-2$ or $k-1$, or
\item the dimension of $\pi_{m-1}(\TOP(d);\id)_\bfQ$ is larger than that of $\pi_{m-1}(\Map(E_d,E_d);\id)_\bfQ$.
\end{enumerate}
But we already showed this, as part of the proof of \cref{thm:nontriviality-general}. To establish \ref{codim2}, we apply \cref{thm:haut-uncountable-or-iso} to degree $i=d-1$ to conclude that either 
\begin{enumerate}[label=(\alph*)]
	\item $\pi_*(\Map^h(E_{d-c},E_{d}),\iota)$ is uncountable in degrees $d-2$ or $d-1$, or 
	\item $\pi_{d-1}(\Map^h(E_{d-c},E_{d}),\iota)_\bfQ \cong \pi_{d-1}(\Map^h(E_{d-c},E^\bfQ_{d}),r_\bfQ\circ \iota)$. 
\end{enumerate}
Since $\pi_{d-1}(\Map^h(E_{d-2},E^\bfQ_{d}),r_\bfQ\circ\iota)$ is infinite-dimensional by \cref{thm:homotopy-Aut-En} \ref{enum:homotopy-codim-2}, it suffices to show that the groups $\pi_*(\TOP(d)/\TOP(d,d-2))$ are finitely generated in degrees $*\le d-1$. The latter follows from a combination of the following facts:
\begin{enumerate}
	\item $\pi_*(\TOP(d))$ is finitely generated in degrees $*\le d-1$ for all $d$.
	\item The map $(-)\times \bfR^{d-2}\colon \oO(2)\simeq \TOP(2)\ra\TOP(d,d-2)$ is $(d-2)$-connected, so in particular $\pi_*(\TOP(d,d-2))$ is finitely generated in degrees $*\le d-2$.
\end{enumerate}
The first statement follows from \cite[Essay V.\S 5.0]{KirbySiebenmann} for $d\neq 4$ and from \cite[Theorem 8.7A]{FreedmanQuinn} for $d=4$, and the second is \cref{lem:topdd-c} \ref{enum:c-is-2}.
\end{proof}

Unwrapping the statement, \cref{equ:technical-poscodim} in particular implies the following:

\begin{cor}The map \eqref{equ:relative-comparison-top-en} is not an equivalence if $d \geq 3$ and $c=\{1,2\}$.\end{cor}

\begin{rem}There are no maps of the form $E_{d-c} \to E_d$ for $c<0$. Indeed, by restricting to $2$-ary operations such a map would induce an equivariant map $\smash{S^{d-c-1}} \to \smash{S^{d-1}}$ with respect to the antipodal action, which implies $c\ge0$ by the Borsuk--Ulam theorem.
\end{rem}

\newpage
\appendix
\printglossary[title=\listofsymbolsname,type=symbols,style=long4col,nogroupskip]

\bibliographystyle{amsalpha}
\bibliography{literature}

\end{document}